\documentclass[11pt,reqno]{article}
\usepackage{amssymb,amscd,amsmath,amsthm,color}

\usepackage[dvipdfmx]{hyperref}

\newtheorem{thm}{Theorem}[section]
\newtheorem{prop}[thm]{Proposition}
\newtheorem{lemma}[thm]{Lemma}
\newtheorem{cor}[thm]{Corollary}
\newtheorem{definition}[thm]{Definition}
\newtheorem{remark}[thm]{Remark}
\newtheorem{example}[thm]{Example}

\numberwithin{equation}{section}

\def\dis{\displaystyle}
\def\bR{\mathbb{R}}
\def\bC{\mathbb{C}}
\def\bN{\mathbb{N}}
\def\bM{\mathbb{M}}
\def\cH{\mathcal{H}}
\def\cB{\mathcal{B}}

\def\cP{\mathcal{P}}
\def\cC{\mathcal{C}}
\def\cD{\mathcal{D}}
\def\cR{\mathcal{R}}
\def\Tr{\mathrm{Tr}\,}

\def\eps{\varepsilon}
\def\<{\langle}
\def\>{\rangle}
\def\ffi{\varphi}
\def\tr{\mathrm{tr}}
\def\Re{\mathrm{Re}\,}
\def\Im{\mathrm{Im}\,}
\def\sa{\mathrm{sa}}

\def\cA{\mathcal{A}}
\def\cM{\mathcal{M}}
\def\cX{\mathcal{X}}

\def\id{\mathrm{id}}
\def\cF{\mathcal{F}}

\def\Aut{\mathrm{Aut}}
\def\bZ{\mathbb{Z}}
\def\Proj{\mathrm{Proj}}
\def\cO{\mathcal{O}}
\def\cN{\mathcal{N}}
\def\cL{\mathcal{L}}
\def\fF{\mathfrak{F}}
\def\fN{\mathfrak{N}}
\def\fM{\mathfrak{M}}
\def\lin{\mathrm{span}}
\def\cY{\mathcal{Y}}
\def\cK{\mathcal{K}}
\def\cE{\mathcal{E}}
\def\fA{\mathfrak{A}}
\def\fT{\mathfrak{T}}
\def\Ad{\mathrm{Ad}}
\def\Int{\mathrm{Int}}

\begin{document}
\baselineskip=14pt
\allowdisplaybreaks

\centerline{\LARGE Concise lectures on selected topics of}
\medskip
\centerline{\LARGE von Neumann algebras}

\bigskip
\bigskip
\centerline{\large F.\ Hiai}
\medskip
\centerline{Graduate School of Information Sciences, Tohoku University}
\centerline{E-mail: hiai.fumio@gmail.com}

\bigskip
\begin{abstract}
A breakthrough took place in the von Neumann algebra theory when the Tomita-Takesaki
theory was established around 1970. Since then, many important issues in the theory were
developed through 1970's by Araki, Connes, Haagerup, Takesaki and others, which are already
very classics of the von Neumann algebra theory. Nevertheless, it seems still difficult for
beginners to access them, though a few big volumes on the theory are available.

These lecture notes are delivered as an intensive course in 2019, April at Department of
Mathematical Analysis, Budapest University of Technology and Economics. The course was aimed
at giving a fast track study of those main classics of the theory, from which people gain an
enough background knowledge so that they can consult suitable volumes when more details are
needed.
\end{abstract}

{\baselineskip12pt
\tableofcontents
}

\newpage
\section{von Neumann algebras -- An overview}

The overview\footnote{
This is the English translation of a part of my article in {\it Encyclopedic Dictionary of
Mathematics}, 4th edition (in Japanese), Iwanami Publisher, Japan.}
gives a brief survey on topics of von Neumann algebra theory mostly developed through 1970's,
many (not all) of which are explained in detail in the main body of these lecture notes.

\subsection{Preliminaries}

The set $B(\cH)$ of all bounded linear operators on a Hilbert space $\cH$ with the inner
product $\<\cdot,\cdot\>$ is a vector space with the operator sum $a+b$ and the scalar
multiplication $\lambda a$ ($a,b\in B(\cH)$, $\lambda\in\bC$) and is a Banach space with
the operator norm
$$
\|a\|:=\sup\{\|a\xi\|:\xi\in H,\,\|\xi\|\le1\}.
$$
Moreover, $B(\cH)$ becomes a Banach *-algebra with the operator product $ab$ and the adjoint
operation $a\mapsto a^*$. A subspace of $B(\cH)$ is called a subalgebra if it is closed under
the product, and a *-subalgebra if it is further invariant under the *-operation. In general,
\emph{operator algebras} mean *-subalgebras of $B(\cH)$.

The (operator) norm topology, the strong operator topology and the weak operator topology are
defined on $B(\cH)$, which are weaker in this order. Since $B(\cH)$ is the dual Banach space
of the Banach space $\cC_1(\cH)$ consisting of trace-class operators with the trace-norm, the
weak topology $\sigma(B(\cH),\cC_1(\cH))$ is also defined on $B(\cH)$, which is called the
\emph{$\sigma$-weak topology} and particularly important in studies of von Neumann algebras.

We write $B(\cH)_\sa$ for the set of all self-adjoint $a=a^*$ in $B(\cH)$. The order $a\le b$
for $a,b\in B(\cH)_\sa$ means that $\<\xi,a\xi\>\le\<\xi,b\xi\>$ for all $\xi\in\cH$, which
is a partial order on $B(\cH)_\sa$. When a net $\{a_\alpha\}$ in $B(\cH)_\sa$ is increasing
and bounded above, it has the supremum $a\in B(\cH)_\sa$ and $a_\alpha\to a$ strongly; in this
case, we write $a_\alpha\nearrow a$.

\subsection{Basics of von Neumann algebras}

A *-subalgebra of $B(\cH)$ is called a \emph{von Neumann algebra} (also \emph{$W^*$-algebra})
if it contains the identity operator $1$ and closed in the weak topology. For
$S\subset B(\cH)$, define the \emph{commutant} $S'$ of $S$ by
$$
S':=\{a\in B(H):ab=ba\ \mbox{for all $b\in S$}\},
$$
and also $S'':=(S')'$. For a *-subalgebra $M$ of $B(\cH)$ with $1\in M$, the \emph{double
commutation theorem} or \emph{von Neumann's density theorem} says that the following three
conditions are equivalent:
\begin{itemize}
\item[(i)] $M$ is a von Neumann algebra (i.e., weakly closed);
\item[(ii)] $M$ is strongly closed;
\item[(iii)] $M''=M$.
\end{itemize}
From this, the polar decompositions and the spectral decompositions of operators in a von
Neumann algebra $M$ are taken inside $M$ itself, so $M$ contains plenty of projections.
Starting from this theorem (1929), J.~von Neumann developed basics (including the
classification in Sec.~1.4) of von Neumann algebra theory in a series of joint papers
with F.~J.~Murray.

\emph{$C^*$-algebras} are another major subject of operator algebras, which are faithfully
represented as norm-closed subalgebras of $B(\cH)$. Sakai \cite{Sa} gave the abstract
characterization of von Neumann algebras in such a way that a $C^*$-algebra is isomorphic to
a von Neumann algebra if and only if it is the dual Banach space of some Banach space. In
this case, the predual space is unique in a strong sense. The term $W^*$-algebra is often
used to stress this abstract (or Hilbert space-free) situation. Although von Neumann algebras
are special $C^*$-algebras, both categories of operator algebras are quite different
theoretically and methodologically.

The \emph{Kaplansky density theorem} is particularly useful\footnote{
In his book Pedersen wrote ``The density theorem is Kaplansky's great gift to mankind. It can
be used every day, and twice on Sundays."}
in study of von Neumann algebras,
saying that if $\cA$ is a *-subalgebra $\cA$ of $B(\cH)$ containing $1$, then
$\{a\in\cA:\|a\|\le1\}$ is strongly dense in $\{a\in\cA'':\|a\|\le1\}$.

Let $M,N$ be von Neumann algebras, and $\pi:M\to N$ be a *-homomorphism. If
$a_\alpha\nearrow a$ in $M$ implies $\pi(a_\alpha)\nearrow\pi(a)$, it is said that $\pi$ is
\emph{normal}, which is equivalent to the continuity of $\pi$ with respect to the
$\sigma$-weak topologies on $M,N$. A *-homomorphism $\pi:M\to B(\cK)$ ($\cK$ is a Hilbert space)
is called a \emph{*-representation} (or simply representation). The range $\pi(M)$ of a normal
representation $\pi$ is a von Neumann algebra and its kernel is represented as $Me$ for some
central projection $e$ (i.e., a projection in the \emph{center} $Z(M):=M\cap M'$), so $\pi$
induces a *-isomorphism between $M(1-e)$ and $\pi(M)$. Note that a faithful representation is
normal automatically.

For two Hilbert spaces $\cH_1,\cH_2$, the \emph{tensor product Hilbert space}
$\cH_1\otimes\cH_2$ is defined by completing the algebraic (as an complex vector space) of
$\cH_1,\cH_2$ with respect to the inner product determined by
$\<\xi\otimes\xi_2,\eta_1\otimes\eta_2\>=\<\xi_1,\eta_1\>\<\xi_2,\eta_2\>$
($\xi_i,\eta_i\in\cH_i$). For any $a_1\in B(\cH_1)$ and $a_2\in B(\cH_2)$ the tensor product
$a_1\otimes a_2\in B(\cH_1\otimes\cH_2)$ is uniquely determined by
$(a_1\otimes a_2)(\xi_1\otimes\xi_2)=a_1\xi_1\otimes a_2\xi_2$ ($\xi_i\in\cH_i$). For von
Neumann algebras $M_i\subset B(\cH_i)$, the von Neumann algebra generated by
$\{a_1\otimes a_2:a_1\in M_1,\,a_2\in M_2\}$ is denoted by $M_1\otimes M_2$ and called the
\emph{tensor product} of $M_1,M_2$. The \emph{commutant theorem}
$$
(M_1\otimes M_2)'=M_1'\otimes M_2'
$$
holds for tensor products of von Neumann algebras.

\subsection{States, weights, and traces}

We write $M_*$ for the set of all $\sigma$-weakly continuous linear functionals on a von
Neumann algebra $M$, which is a Banach space as a closed subspace of the dual Banach space
$M^*$. Then the dual Banach space of $M_*$ is isometric to $M$. In fact, $M_*$ is a unique
Banach space with this property, so it is called the \emph{predual} of $M$. A positive linear
functional $\ffi$ on $M$ is in $M_*$ if and only if $\ffi$ is normal (i.e.,
$a_\alpha\nearrow a$ $\implies$ $\ffi(a_\alpha)\nearrow\ffi(a)$). In particular, a normal
positive linear functional $\ffi$ on $M$ with $\ffi(1)=1$ is called a \emph{normal state} of
$M$. Any $\ffi\in M_*$ is represented as a linear combination of normal states.

A functional $\ffi:M_+=\{a\in M:a\ge0\}\to[0,+\infty]$ is called a \emph{weight} on $M$ if
it satisfies, for all $a,b\in M_+$ and $\lambda\ge0$,
\begin{itemize}
\item additivity: $\ffi(a+b)=\ffi(a)+\ffi(b)$,
\item positive homogeneity: $\ffi(\lambda a)=\lambda\ffi(a)$, where $0(+\infty):=0$.
\end{itemize}
A weight $\ffi$ is faithful if $\ffi(a)>0$ for all $a\in M_+\setminus\{0\}$, and normal if
$a_\alpha\nearrow a$ $\implies$ $\ffi(a_\alpha)\nearrow\ffi(a)$. Let
$\fN_\ffi=\{a\in M:\ffi(a^*a)<\infty\}$, and $\fM_\ffi$ be the linear span of
$\fN_\ffi^*\fN_\ffi$. If $\fM_\ffi$ is $\sigma$-weakly dense in $M$, $\ffi$ is said to be
semifinite. The weight $\ffi$ extends to $\fM_\ffi$ as a linear functional. Similarly to
the GNS construction of a $C^*$-algebra with respect to a state, the GNS construction
$(\cH_\ffi,\pi_\ffi,\eta_\ffi)$ of $M$ with respect to a semifinite normal weight $\ffi$ is
defined as follows: $\pi_\ffi:M\to B(\cH_\ffi)$ is a normal representation,
$\eta_\ffi:\fN_\ffi\to\cH_\ffi$ is a linear map, $\overline{\eta_\ffi(\fN_\ffi)}=\cH_\ffi$,
$\<\eta_\ffi(a),\eta_\ffi(b)\>=\ffi(a^*b)$ and $\pi_\ffi(x)\eta_\ffi(a)=\eta_\ffi(xa)$ for all
$a,b\in\fN_\ffi$, $x\in M$.

A weight $\tau$ satisfying $\tau(a^*a)=\tau(aa^*)$ for all $a\in M$ is called a \emph{trace}.
A finite ($\tau(1)<+\infty$) trace $\tau$ uniquely extends to $M$ as a linear functional
satisfying $\tau(ab)=\tau(ba)$ ($a,b\in M$).

\subsection{Classification of von Neumann algebras}

The notion of the \emph{Murray-von Neumann equivalence} on the set $\Proj(M)$ of all
projections in a von Neumann algebra $M$ is defined as follows: $e,f\in\Proj(M)$ is said to be
equivalent ($e\sim f$) if there is a $v\in M$ such that $v^*v=e$ and $vv^*=f$. A projection
$e\in\Proj(M)$ is called an abelian projection if $eMe$ is an abelian algebra, and a finite
projection if for $f\in\Proj(M)$, $e\sim f\le e$ $\implies$ $f=e$.

The von Neumann algebra $M$ is said to be \emph{finite} if $1$ is a finite projection, and
\emph{semifinite} if, for every central projection $e\ne0$ (i.e., a projection in the center
of $M$), there is a finite projection $f\in\Proj(M)$ such that $0\ne f\le e$. If $M$ has no
finite central projection ($\ne0$), then $M$ is said to be \emph{properly infinite}. If $M$
has no finite projection ($\ne0$), then $M$ is said to be \emph{purely infinite} or
\emph{of type III}. A von Neumann algebra $M$ is properly infinite if and only if
$M\cong M\otimes B(\cH)$ for separable Hilbert spaces $\cH$, and $M$ is semifinite if and only
if $M$ has a faithful semifinite normal trace.

A von Neumann algebra $M$ is called a \emph{factor} if the center is trivial (i.e.,
$Z(M)=\bC1$). The factors are classified into one of the  following types:
\begin{itemize}
\item type I$_n$ ($n\in\bN$) -- $M$ is isomorphic to the matrix algebra $\bM_n(\bC)$; $M$ is
finite and has a finite abelian projection $\ne0$,
\item type I$_\infty$ -- $M$ is isomorphic to $B(\cH)$ with $\dim\cH=\infty$; $M$ is properly
infinite and has an abelian projection $\ne0$,
\item type II$_1$ -- $M$ is finite and has no abelian projection $\ne0$,
\item type II$_\infty$ -- $M$ is semifinite and properly infinite, and has no abelian
projection $\ne0$,
\item type III -- $M$ has no finite projection $\ne0$.
\end{itemize}

A finite factor has a faithful finite normal trace (unique up to positive constants). A factor
of type II$_\infty$ is represented as $(\mbox{a factor of type II$_1$})\otimes\cB(H)$.
Corresponding to the types of factors, the quotient set $\Proj(M)/\sim$ is identified with one
of the following:
\begin{itemize}
\item type I$_n$ -- $\{0,1,\dots,n\}$,
\item type I$_\infty$ -- $\{0,1,\dots,\infty\}$,
\item type II$_1$ -- $[0,\infty]$,
\item type III -- $\{0,\infty\}$.
\end{itemize}

Von Neumann algebras of type I (i.e., a direct sum of factors of type I) is said to be
\emph{discrete} or \emph{atomic}. For each type of II$_1$, II$_\infty$ and III, there are
continuously many non-isomorphic classes (due to McDuff, Sakai for type II$_1$, see Sec.~1.6
for type III).

Any von Neumann algebra $M$ on a separable Hilbert space $\cH$ is decomposed into factors as
a direct integral (called von Neumann's reduction theory) as follows: There exists a measurable
field of factors $\{M(\gamma),\cH(\gamma)\}_{\gamma\in\Gamma}$ on a standard Borel space
$(\Gamma,\mu)$ with a finite measure $\mu$ such that
$$
\cH=\int_\Gamma^\oplus H(\gamma)\,d\mu,\qquad
M=\int_\Gamma^\oplus M(\gamma)\,d\mu,\qquad
Z(M)=\int_\Gamma^\oplus\bC1_\gamma\,d\mu
\ \ (\cong L^\infty(\Gamma,\mu)).
$$
Many issues of von Neumann algebras can be reduced to the case of factors by using this
reduction.

\subsection{Tomita-Takesaki theory}

The study of type III von Neumann algebras was extensively developed in 1970's, whose start
point is the \emph{modular theory} constructed by M.~Tomita and more developed by M.~Takesaki,
so the theory is called the \emph{Tomita-Takesaki theory}.

Let $(\cH_\ffi,\pi_\ffi,\eta_\ffi)$ be the GNS construction of a von Neumann algebra $M$ with
respect to a faithful semifinite normal weight $\ffi$ on $M$. Define a conjugate-linear
operator $S_\ffi$ on a dense subspace $\eta_\ffi(\fN_\ffi\cap\fN_\ffi^*)$ by
$S_\ffi\eta_\ffi(a):=\eta_\ffi(a^*)$, which is closable. By taking the polar decomposition
$\overline{S}_\ffi=J_\ffi\Delta_\ffi^{1/2}$ of the closure of $S_\ffi$, we define a
positive self-adjoint operator $\Delta_\ffi$ called the \emph{modular operator}, and a
conjugate-linear unitary involution $J_\ffi$ ($J_\ffi^2=1$) called the \emph{modular
conjugation}. The following two statements are \emph{Tomita's fundamental theorem}:
\begin{itemize}
\item[(i)] $J_\ffi M J_\ffi=M'$,
\item[(ii)] $\Delta_\ffi^{it}M\Delta_\ffi^{-it}=M$ ($t\in\bR$),
\end{itemize}
where $M=\pi_\ffi(M)$ with deleting the faithful representation $\pi_\ffi$. By (ii),
$\sigma_t^\ffi(a)=\Delta_\ffi^{it}a\Delta_\ffi^{-it}$ ($a \in M$) defines a strongly
continuous one-parameter automorphism group on $M$ called the \emph{modular automorphism group}
associated with $\ffi$. Takesaki's theorem says that $\ffi$ satisfies the the \emph{KMS}
(Kubo-Martin-Schwinger) \emph{condition} at $\beta=-1$ as follows: For every
$a,b\in\fN_\ffi\cap\fN_\ffi^*$, there exists a bounded continuous function $f_{a,b}(z)$ on
$\beta\le\Im z\le0$, analytic in $\beta<\Im z\le0$, such that
$$
f_{a,b}(t)=\ffi(a\sigma_t^\ffi(b)),\qquad f_{a,b}(t+i\beta)=\ffi(\sigma_t^\ffi(b)a)
\qquad(t\in\bR).
$$
Furthermore, the modular automorphism group $\sigma_t^\ffi$ is uniquely determined by this
condition for $\ffi$. The KMS condition was introduced by Haag-Hugenholtz-Winnink to
characterize the equilibrium states in quantum systems in the $C^*$-algebraic approach to
quantum statistical mechanics. (In statistical mechanics, $\beta$ corresponds to the inverse
temperature.) The link between the modular theory and the KMS condition was quite a remarkable
occurrence.

\subsection{Classification of factors of type III}

In the same period of time as the appearance of the Tomita-Takesaki theory, R.~T.~Powers
showed the existence of continuously many non-isomorphic factors of type III, called the
\emph{Powers factors}. Next, H.~Araki and E.~J.~Woods classified the so-called
\emph{ITPFI factors} (also called the {\emph{Araki-Woods factors}) defined from infinite tensor
products of factors of type I. When an infinite sequence of pairs $\{\bM_{k_n},\ffi_n\}$ of
matrix algebras and states, the ITPFI factor $\bigotimes_{n=1}^\infty\{\bM_{k_n},\ffi_n\}$ is
given by making the GNS construction of the infinite tensor product
$\bigotimes_1^\infty\bM_{k_n}$ with respect to the tensor state $\bigotimes_1^\infty\ffi_n$.
In particular, for $0<\lambda\le1$ consider the state of $\bM_2$ defined by
$\omega_\lambda=\Tr(D_\lambda\,\cdot)$, where $D_\lambda=\begin{bmatrix}{1\over1+\lambda}&0\\
0&{\lambda\over1+\lambda}\end{bmatrix}$. Then
$\cR_\lambda=\bigotimes_1^\infty\{\bM_2,\omega_\lambda\}$, $0<\lambda<1$, are the Powers
factors, which are non-isomorphic factors of type III for different $\lambda$. On the other
hand, this infinite tensor product when $\lambda=1$ is the so-called hyperfinite factor of
type II$_1$ (see Sec.~1.8).

Motivated by the idea to reconstruct the Araki-Woods factors by the Tomita-Takesaki theory,
A.~Connes construct the classification theory of type III factors. To do so, he introduced
the \emph{$T$-set} and the \emph{$S$-set} of $M$ by
\begin{align*}
T(M)&:=\{t\in\bR:\sigma_t^\ffi\in\Int(M)\}\quad
(\mbox{$\Int(M)=$ the inner automorphisms of $M$}), \\
S(M)&:=\bigcap\,\{\mathrm{Sp}(\Delta_\ffi):
\ffi\ \mbox{is a faithful semifinite normal weight on $M$}\},
\end{align*}
which are invariants for isomorphism classes of von Neumann algebras. For any two
faithful semifinite normal weights $\ffi,\psi$ on $M$ there exists a unitary cocycle
$u_t=(D\psi:D\ffi)_t$ ($\in M$) with respect to $\sigma_t^\ffi$ for which
$\sigma_t^\psi(a)=u_t\sigma_t^\ffi(a)u_t^*$ ($a\in M$, $t\in\bR$). Therefore, $T(M)$ is a
subgroup of $\bR$ determined independently of the choice of $\ffi$. On the other hand,
$S(M)\setminus\{0\}$ is a closed subgroup of the multiplicative group $\bR_+$\,($=(0,\infty)$).
A von Neumann algebra $M$ is semifinite if and only if $S(M)=\{1\}$. The factors of type III
are classified in terms of the $S$-set as
\begin{itemize}
\item type III$_1$ -- $S(M)=\{1\}$,
\item type III$_\lambda$ ($0<\lambda<1$) -- $S(M)=\{0\}\cup\{\lambda^n:n\in\bZ\}$,
\item type III$_0$ -- $S(M)=\{0,1\}$.
\end{itemize}
The $T$-set of type III$_1$ factors is $T(M)=\{0\}$, that of III$_\lambda$ factors
($0<\lambda<1$) is $T(M)=(2\pi/\log\lambda)\bZ$, and $T(M)$ of III$_0$ factors is not
unique. From the result of Araki-Wood, there is a unique ITPFI factor for each of type
III$_\lambda$ ($0<\lambda<1$) and type III$_1$, and that of type III$_\lambda$ is the
Powers factor $\cR_\lambda$. There are continuously many ITPFI factors of type III$_0$.

\subsection{Crossed products and type III structure theory}

Let $\alpha=\{\alpha_g\}_{g\in G}$ be a continuous action of a locally compact group $G$ on
$M\subset B(\cH)$, i.e., $g\in G\mapsto\alpha_g\in\Aut(M)$ is a homomorphism and
$g\mapsto\alpha_g(a)$ is strongly continuous for any $a\in M$. On the Hilbert space
$L^2(G,\cH)=L^2(G)\otimes\cH$ (with respect to the Haar measure on $G$), a representation
$\pi_\alpha$ of $M$ and a unitary representation $\lambda$ of $G$ are defined by
$$
(\pi_\alpha(a)\xi)(h):=\alpha_{h^{-1}}(a)\xi(h),\qquad
(\lambda(g)\xi)(h):=\xi(g^{-1}h)\qquad(\xi\in L^2(G,\cH),\ g,h\in G).
$$
The \emph{crossed product} $M\rtimes_\alpha G$ of $M$ by the action $\alpha$ is the von
Neumann algebra generated by $\pi_\alpha(M)\cup\lambda(G)$. When $G$ is abelian, a unitary
representation of the dual group $\widehat G$ on $L^2(G,\cH)$ is defined by
$$
(\mu(p)\xi)(h):=\overline{\<h,p\>}\xi(h)\qquad(\xi\in L^2(G,\cH)).
$$
Then an action $\widehat\alpha$ of $\widehat G$ of $M\rtimes_\alpha G$ is defined by
$$
\widehat\alpha_p(x):=\mu(p)x\mu(p)^*\qquad(x\in M\rtimes_\alpha G,\ p\in\widehat G),
$$
which is called the \emph{dual action}. \emph{Takesaki's duality theorem} holds:
$$
(M\rtimes_\alpha G)\rtimes_{\widehat\alpha}\widehat G
\ \cong\ M\otimes B(L^2(G)).
$$
In particular, if $M$ is properly infinite and $G$ satisfies the second axiom of countability,
then $(M\rtimes_\alpha G)\rtimes_{\widehat\alpha}\widehat G \cong M$ holds. This duality
theorem was extended to the case of actions of non-abelian locally compact groups and more
general Kac algebras (Hopf algebras with *-structure).

Connes and Takesaki established the type III structure theory by use of crossed products.
\begin{itemize}
\item {\it The structure of type III$_\lambda$ factors} (Connes):\enspace
For any factor of type III$_\lambda$ ($0<\lambda<1$), there exist a factor $N$ of
type II$_\infty$ and a $\theta\in\Aut(N)$ such that
$$
M\ \cong\ N\rtimes_\theta\bZ,\qquad\tau\circ\theta=\lambda\tau,
$$
where $\tau$ is a faithful semifinite normal trace on $N$. Moreover, $(\cN,\theta)$ is
unique up to conjugacy.
\end{itemize}

Connes showed a similar (but a bit more complicated) structure theorem for type III$_0$
factors.
\begin{itemize}
\item {\it The structure of type III von Neumann algebras} (Takesaki):\enspace
For any von Neumann algebra of type III, there exist a von Neumann algebra $N$ of type
II$_\infty$, a faithful semifinite normal trace $\tau$ on $N$, and a continuous one-parameter
action $\theta_t$ on $N$ such that
$$
M\ \cong\ N\rtimes_\theta\bR,\qquad\tau\circ\theta_t=e^{-t}\tau\quad (t\in\bR).
$$
Moreover, $(\cN,\theta)$ is unique up to conjugacy. Note that $(N,\theta)$ is realized as
$N=M\rtimes_{\sigma^\ffi}\bR$ the crossed product by the modular automorphism group
$\sigma_t^\ffi$ and the dual action $\theta=\widehat{\sigma^\ffi}$ in Takesaki's duality.
\end{itemize}

The above structure theorems give the crossed product decompositions of type III von Neumann
algebras $M$ by $\bZ$ or $\bR$-action of type II von Neumann algebras. Thus, the study of
type III structure may be reduced to that of type II and actions with trace-scaling properties
as $\tau\circ\theta=\lambda\tau$ and $\tau\circ\theta_t=e^{-t}\tau$.

When a properly infinite factor $M$ is decomposed as $M\cong N\rtimes_\theta\bR$ as above,
Connes-Takesaki introduced the \emph{flow of weights} of $M$ as
$$
(X,F_t^M):=(Z(N),\theta_t|_{Z(N)}),
$$
which is a non-singular ergodic flow and classify the type of $M$ as follows:
\begin{itemize}
\item if $X$ is a single point (i.e., $N$ is a factor), then $M$ is of type III$_1$,
\item if $(X,F_t^M)$ is a translation on $[0,-\log\lambda)$ with a period $-\log\lambda$,
then $M$ is of type III$_\lambda$,
\item if $(X,F_t^M)$ is aperiodic and not the $\bR$-translation, then $M$ is of type III$_0$,
\item if $(X,F_t^M)$ is the $\bR$-translation, then $M$ is semifinite.
\end{itemize}

\subsection{Classification of AFD factors}

Here, assume that a von Neumann algebra $M$ is on a separable Hilbert space, or equivalently,
$M$ has the separable predual $M_*$. If $M$ is generated by an increasing sequence of
finite-dimensional *-subalgebras, then $M$ is said to be \emph{hyperfinite} or \emph{AFD}
(\emph{approximately finite dimensional}). The uniqueness of type II$_1$ AFD factors is an
old result of Murray-von Neumann. ITPFI factors (see Sec.~1.6) is obviously AFD. A von
Neumann algebra $M\subset B(\cH)$ is said to be \emph{injective} if there exists a norm one
projection from $B(\cH)$ onto $M$. Connes (1976) proved that a von Neumann algebra is AFD
if and only if it is injective, and in the same time, he proved that injective factors of
types II$_1$, II$_\infty$, III$_\lambda$ ($0<\lambda<1$) are unique for each type, which
are, respectively,
$$
\cR,\qquad \cR_{0,1}:=\cR\otimes\cB(H),\qquad\cR_\lambda\qquad\mbox{(see Sec.~1.6)}.
$$
It was also shown that injective factors of type III$_0$ are only \emph{Krieger factors},
where a Krieger factor is realized as the crossed product $L^\infty(\Omega)\rtimes_T\bZ$ by
a non-singular ergodic transformation $T$ acting freely on a Lebesgue space $(\Omega,\mu)$.
This construction is called the \emph{group measure space construction}, known since the
early stage of von Neumann algebra study. Two Krieger factors are isomorphic if and only if
two transformations are weakly equivalent, i.e., orbit equivalent (Dye, Krieger). The class
of Krieger factors is bigger than that of ITPFI factors. More general Krieger type
construction is known for ergodic action of countable groups and for ergodic countable
equivalence relations (Feldman-Moore).

U.~Haagerup (1984) proved the uniqueness of injective factors of type III$_1$, that is,
the III$_1$ ITPFI factor is a unique injective factor of type III$_1$, thus completing the
classification of AFD (= injective) factors. In other words, the conjugacy class of the flow
of weights (see Sec.~1.7) is a complete invariant for injective factors of type III.

Classification of group actions on AFD factors was also well developed. When $M$ is a factor,
the \emph{outer period} $p_0(\alpha)$ and the \emph{obstruction} $\gamma(\alpha)$ of
$\alpha\in\Aut(M)$ are defined as follows: $p_0(\alpha)$ is the smallest integer $n>0$ such
that $\alpha^n\in\Int(M)$ ($p_0(\alpha)=0$ if such $n$ does not exist). Then, for a unitary
$u\in M$ such that $\alpha^{p_0(\alpha)}(a)=uau^*$, define $\gamma(\alpha)$ as $\gamma\in\bC$
satisfying $\alpha(u)=\gamma u$, where $\gamma(\alpha):=1$ if $p_0(\alpha)=0$. Note that
$\gamma(\alpha)$ is a $p_0(\alpha)$th root of $1$ (a cohomological quantity). Connes proved
around 1985 that $(p_0(\alpha),\gamma(\alpha))$ is a complete invariant of outer conjugacy
(conjugacy modulo inner automorphisms) for automorphism of the AFD type II$_1$ factor $\cR$.
Moreover, for automorphisms of the AFD II$_\infty$ factor $\cR_{0,1}$,
$(p_0(\alpha),\gamma(\alpha),\mathrm{mod}(\alpha))$ is a complete invariant of outer
conjugacy, where the \emph{module} $\mathrm{mod}(\alpha)$ is the $\lambda>0$ satisfying
$\tau\circ\alpha=\lambda\tau$ for the trace on $\cR_{0,1}$. Since then, there were many
studies on classification of amenable group actions on AFD factors, by Jones, Ocneanu,
Takesaki, and others.

\subsection{Standard form and natural positive cone}

Let $\ffi$ be a faithful semifinite normal weight on a general von Neumann algebra $M$.
The closure $\cP_\ffi$ of $\{\Delta_\ffi^{1/4}\eta_\ffi(a):0\le a\in\fM_\ffi\}$ is a
self-dual cone of $\cH_\ffi$, which is called the \emph{natural positive cone} of $M$. The
quadruple $(M,\cH,J,\cP)$ of a von Neumann algebra $M$, its faithful representing Hilbert
space $\cH=\cH_\ffi$, the conjugate-linear unitary involution $J=J_\ffi$, and the self-dual
cone $\cP=\cP_\ffi$ satisfies the following properties:
\begin{itemize}
\item $JMJ'=M'$,
\item $JcJ=c^*$ ($c\in Z(M)$),
\item $J\xi=\xi$ ($\xi\in\cP$),
\item $aJaJ(\cP)\subset\cP$ ($a\in M$).
\end{itemize}
Such a $(M,\cH,J,\cP)$ is unique up to unitary equivalence for any $M$, and it is called the
\emph{standard form} of $M$. Theory of standard form obtained independently by Araki, Connes,
and Haagerup is important in studies of von Neumann algebras. The map
$\xi\in\cP\mapsto\omega_\xi=\<\xi,\cdot\,\xi\>\in M_*^+:=\{\ffi\in M_*:\ffi\ge0\}$ is a
homeomorphism in the norm topologies and satisfies
$$
\|\xi-\eta\|^2\le\|\omega_\xi-\omega_\eta\|\le\|\xi-\eta\|\,\|\xi+\eta\|\qquad
(\xi,\eta\in\cP).
$$
For any $g\in\Aut(M)$ there exists a unique unitary $u_g\in B(\cH)$ such that
$u_gJ=Ju_g$, $u_g(\cP)=\cP$ and $g(a)=u_gau_g^*$ ($a\in M$). The $u_g$ is a unitary
representation of the group $\Aut(M)$ and satisfies
$\omega_{u_g\xi}=\omega_\xi\circ g^{-1}$ ($g\in\Aut(M)$, $\xi\in\cP$). When $M$ is a
semifinite von Neumann algebra with a faithful semifinite normal trace $\tau$, its standard
form is realized as $(M,L^2(M,\tau),J,L^2(M,\tau)_+)$, where the Hilbert space $L^2(M,\tau)$
is the non-commutative $L^2$-space with respect to $\tau$ (consisting of $\tau$-measurable
operators $a$ with $\tau(a^*a)<+\infty$), $M$ is represented on $L^2(M,\tau)$ by left
multiplications, and $Ja=a^*$ for $a\in L^2(M,\tau)$.

We end the overview with a list of a few standard textbooks on von Neumann algebras.

\begin{itemize}
\item
R.\ V.\ Kadison and J.\ R.\ Ringrose, Fundamentals of the theory of operator algebras, I, II,
Academic Press, 1983, 1986.
\item
G.\ T.\ Pedersen, $C^*$-Algebras and Their Automorphism Groups, Academic Press, 1979.
\item
S.\ Sakai, $C^*$-algebras and $W^*$-algebras, Springer-Verlag, 1971.
\item
S. Str\v atil\v a and L. Zsid\'o, Lectures on von Neumann Algebras,
Abacus Press, 1979.
\item
M.\ Takesaki, Theory of Operator Algebras, I, II, Springer-Verlag, 1979, 2003.
\end{itemize}

\section{Tomita-Takesaki modular theory}

After the appearance of Tomita's unpublished paper on the subject in 1960's, a readable
monograph was published by Takesaki \cite{Ta} and a simplified proof was given by van Daele
\cite{vD}. Then a proof minimizing the use of unbounded operators was also presented in
\cite{RvD}.
Below, following the expositions in \cite{Ta2,vD2}, we present a proof of Tomita's theorem
in the Tomita-Takesaki theory of von Neumann algebras, in the setting with a cyclic and
separating vector.

\subsection{Tomita's fundamental theorem}

Let $M$ be a von Neumann algebra. We assume that there is a faithful $\omega\in M_*^+$, that
is equivalent to that $M$ is \emph{$\sigma$-finite}, i.e., mutually orthogonal projections in
$M$ are at most countable. By making the \emph{GNS cyclic representation}
$\{\pi_\omega,\cH_\omega,\Omega\}$ of $M$ with respect to $\omega$, we have an isomorphism
$\pi_\omega:M\to B(\cH_\omega)$ with a cyclic and separating vector $\Omega\in\cH_\omega$ for
$\pi_\omega(M)$ such that
$$
\omega(x)=\<\Omega,\pi_\omega(x)\Omega\>,\qquad x\in M.
$$
Thus, by identifying $M$ with $\pi_\omega(M)$, we may assume that $M$ itself is a von Neumann
algebra on $\cH$ with a cyclic and separating vector $\Omega$ for $M$. Here, $\Omega$ is
\emph{cyclic} for $M$ if $\overline{M\Omega}=\cH$, and $\Omega$ is \emph{separating} for $M$
if $x\in M$, $x\Omega=0$ $\implies$ $x=0$, equivalently $\overline{M'\Omega}=\cH$. When $M$
is not $\sigma$-finite, the construction of the modular theory is essentially similar to below
(but technically more complicated) with a faithful semifinite normal weight on $M$ based on
the left Hilbert algebra theory.

We begin with the two conjugate-linear operators $S_0$ and $F_0$ with the dense domains
$\cD(S_0)=M\Omega$ and $\cD(F_0)=M'\Omega$ defined by
\begin{align*}
S_0x\Omega&:=x^*\Omega,\qquad\ \,x\in M, \\
F_0x'\Omega&:=x'^*\Omega,\qquad x'\in M'.
\end{align*}
For any $x\in M$ and $x'\in M'$ note that
$$
\<x'\Omega,S_0x\Omega\>=\<x'\Omega,x^*\Omega\>
=\<x\Omega,x'^*\Omega\>=\<x\Omega,F_0x'\Omega\>,
$$
which implies that $S_0$ and $F_0$ are closable, $\overline{F_0}\subset S_0^*$ and
$\overline{S_0}\subset F_0^*$. So we set $S:=\overline{S_0}$ and $F:=S^*=S_0^*$, and take the polar
decomposition of $S$ as
$$
S=J\Delta^{1/2},\qquad\Delta:=S^*S=FS.
$$
Since the ranges of $S$ and $S^*$ are dense, it follows that $J$ is a
conjugate-linear unitary and $\Delta$ is a non-singular positive self-adjoint operator.

\begin{lemma}\label{L-2.1}
We have:
\begin{itemize}
\item[\rm(i)] $J=J^*$ and $J^2=1$.
\item[\rm(ii)] $\Delta=FS$ and $\Delta^{-1}=SF$.
\item[\rm(iii)] $S=J\Delta^{1/2}=\Delta^{-1/2}J$ and $F=J\Delta^{-1/2}=\Delta^{1/2}J$.
\item[\rm(iv)] $\Delta^{-1}=J\Delta J$ and $J\Delta^{it}=\Delta^{-it}J$ for all $t\in\bR$.
\item[\rm(v)] $J\Omega=\Omega$ and $\Delta\Omega=\Omega$.
\end{itemize}
\end{lemma}

\begin{proof}
Since $S^2\supset S_0^2=1_{M\Omega}$, it is easy to see that $S=S^{-1}$, which implies that
$S=J\Delta^{1/2}=\Delta^{-1/2}J^*=J^*J\Delta^{-1/2}J^*$. Therefore, $J=J^*$ and
$\Delta^{1/2}=J\Delta^{-1/2}J^*$ so that (i) and (iv) hold. Since $F=\Delta^{1/2}J$, (ii)
and (iii) hold. Moreover, since $S\Omega=F\Omega=\Omega$, (v) holds as well.
\end{proof}

The next theorem is \emph{Tomita's fundamental theorem}.

\begin{thm}[Tomita]\label{T-2.2}
With $J$ and $\Delta$ given above, we have
\begin{align}
JMJ&=M', \label{F-2.1}\\
\Delta^{it}M\Delta^{-it}&=M,\qquad t\in\bR. \label{F-2.2}
\end{align}
\end{thm}

\begin{definition}\label{D-2.3}\rm
The operator $\Delta$ is called the \emph{modular operator} with respect to $\Omega$ (or
$\omega)$, and $J$ is called the \emph{modular conjugation} with respect to $\Omega$ (or
$\omega)$. By \eqref{F-2.2} the one-parameter automorphism group $\sigma_t=\sigma_t^\omega$ of
$M$ is defined by $\sigma_t^\omega(x):=\Delta^{it}x\Delta^{-it}$ ($x\in M$, $t\in\bR$), which
is called the \emph{modular automorphism group} with respect to $\Omega$ (or $\omega)$.
\end{definition}

\begin{example}\label{E-2.4}\rm
Consider the simple case where $\omega=\tau$ is a faithful normal finite trace of $M$, so $M$
is a finite von Neumann algebra of type II$_1$. Since
$\|x\Omega\|^2=\tau(x^*x)=\tau(xx^*)=\|x^*\Omega\|^2$ for all $x\in M$, $S$ is a
conjugate-linear unitary, which means that $S=J$ and $\Delta=1$. Hence \eqref{F-2.2} trivially
holds. For every $x,y,y_1\in M$ note that
$$
JxJyy_1\Omega=Jxy_1^*y^*\Omega=yy_1x^*\Omega=yJxy_1^*\Omega=yJxJy_1\Omega
$$
so that $JxJy=yJxJ$. Hence $JMJ\subset M'$. Moreover, for every $x\in M$ and $x'\in M'$ note
that
$$
\<x\Omega,Jx'\Omega\>=\<x'\Omega,Jx\Omega\>
=\<x'\Omega,x^*\Omega\>=\<x\Omega,x'^*\Omega\>
$$
so that $Jx'\Omega=x'^*\Omega$. Hence, similarly to the above, $JM'J\subset M$, so
$M'\subset JMJ$. Therefore, \eqref{F-2.1} holds. In this way, Tomita's theorem in this case
is quite easy.
\end{example}

The proof of Theorem \ref{T-2.2} taken from \cite{vD2} needs several lemmas as below.

\begin{lemma}\label{L-2.5}
Let $\lambda\in\bC$ with $|\lambda|=1$ and $\lambda\ne-1$. For every $x'\in M'$ there exists
a (unique) $x\in M$ such that
\begin{align}\label{F-2.3}
x'\Omega=(\Delta+\lambda)x\Omega,\quad\mbox{i.e.},\quad
(\Delta+\lambda)^{-1}x'\Omega=x\Omega.
\end{align}
Hence $(\Delta+\lambda)^{-1}M'\Omega\subset M\Omega$.
\end{lemma}

\begin{proof}
We may assume that $0\le x'\le1$. Let $\alpha:={1\over1+\lambda}$; then
$\alpha+\overline\alpha=1$. Define $\psi,\psi_x\in M_*$ for $x\in M_\sa$ by
$$
\psi(y):=\<x'\Omega,y\Omega\>,\qquad
\psi_x(y):=\omega(\alpha xy+\overline\alpha yx),\qquad y\in M.
$$
Clearly, $0\le\psi\le\omega$ and $\psi_x\in M_*^{sa}$ since
$\overline{\psi_x(y)}=\omega(\overline\alpha y^*x+\alpha xy^*)=\psi_x(y^*)$. Let us first prove
that there exists an $x\in M_\sa$ such that $\psi=\psi_x$. Since
$x\in M_\sa\mapsto\psi_x\in M_*^{sa}$ is $\sigma(M,M_*)$-$\sigma(M_*,M)$-continuous, it
follows that
$$
\mathcal{V}:=\{\psi_x:x\in M_\sa,\,\|x\|\le1\}
$$
is $\sigma(M_*,M)$-compact and convex in $M_*^{sa}$. Now, assume on the contrary that
$\psi\not\in\mathcal{V}$. Then by the Hahn-Banach separation theorem, there is a
$y_0\in M_\sa$ such that
$$
\psi(y_0)>\sup\{\psi_x(y_0):x\in M_\sa,\,\|x\|\le1\}.
$$
Take the Jordan decomposition $y_0:=y_0^+-y_0^-$ and let $x_0:=s(y_0^+)-s(y_0^-)$, where
$s(y_0^+)$ is the support projection of $y_0^+$. Since $\|x_0\|\le1$ and
$x_0y_0=y_0x_0=y_0^++y_0^-$, we have
$$
\psi(y_0)\le\psi(y_0^++y_0^-)\le\omega((\alpha+\overline\alpha)x_0y_0)
=\omega(\alpha x_0y_0+\overline\alpha y_0x_0)=\psi_{x_0}(y_0)<\psi(y_0),
$$
a contradiction. Hence an $x\in M_\sa$ with $\psi=\psi_x$ exists, so for every $y\in M$,
$$
\<x'\Omega,y\Omega\>=\<\Omega,\alpha xy\Omega\>+\<\Omega,\overline\alpha yx\Omega\>
$$
so that
$$
\<x'\Omega-\overline\alpha x\Omega,y\Omega\>
=\<y^*\Omega,\overline\alpha x\Omega\>=\<Sy\Omega,\overline\alpha x\Omega\>.
$$
This means that $x\Omega\in\cD(F)$ and
$$
\alpha F(x\Omega)=F(\overline\alpha x\Omega)=x'\Omega-\overline\alpha x\Omega.
$$
Since $x\Omega=x^*\Omega=Sx\Omega$, by Lemma \ref{L-2.1}\,(ii) we have $x\Omega\in\cD(\Delta)$
and $\alpha\Delta x\Omega=x'\Omega-\overline\alpha x\Omega$ so that
$$
x'\Omega=(\alpha\Delta+\overline\alpha)x\Omega={\Delta+\lambda\over1+\lambda}\,x\Omega,
$$
which gives \eqref{F-2.3} by replacing ${1\over1+\lambda}\,x$ with $x$.
\end{proof}

\begin{lemma}\label{L-2.6}
Let $\lambda\in\bC$ with $|\lambda|=1$ and $\lambda\ne-1$. Assume that $x'\in M'$ and $x\in M$
satisfy \eqref{F-2.3}. Then for every
$\xi_1,\xi_2\in\cD(\Delta^{1/2})\cap\cD(\Delta^{-1/2})$,
\begin{align}\label{F-2.4}
\<\xi_1,x'\xi_2\>=\<\Delta^{1/2}\xi_1,Jx^*J\Delta^{-1/2}\xi_2\>
+\lambda\<\Delta^{-1/2}\xi_1,Jx^*J\Delta^{1/2}\xi_2\>.
\end{align}
\end{lemma}

\begin{proof}
Let $x'\in M'$ and $x\in M$ be as stated above. For every $y_1,y_2\in M$ we find that
\begin{align*}
\<y_1\Omega,x'y_2\Omega\>&=\<y_2^*y_1\Omega,x'\Omega\>
=\<y_2^*y_1\Omega,\Delta x\Omega\>+\lambda\<y_2^*y_1\Omega,x\Omega\> \\
&=\<Sx\Omega,Sy_2^*y_1\Omega\>+\lambda\<y_1\Omega,y_2x\Omega\> \\
&=\<y_1x^*\Omega,y_2\Omega\>+\lambda\<y_1\Omega,y_2x\Omega\> \\
&=\<SxSy_1\Omega,y_2\Omega\>+\lambda\<y_1\Omega,Sx^*Sy_2\Omega\> \\
&=\<\Delta^{-1/2}JxJ\Delta^{1/2}y_1\Omega,y_2\Omega\>
+\lambda\<y_1\Omega,\Delta^{-1/2}Jx^*J\Delta^{1/2}y_2\Omega\>
\end{align*}
thanks to Lemma \ref{L-2.1}\,(iii).

Now, assume that $y_1\Omega,y_2\Omega\in(\Delta+1)^{-1}M'\Omega$ ($\subset M\Omega$ by (1)).
Since $M'\Omega=\cD(F_0)\subset\cD(F)=\cD(\Delta^{-1/2})$, it follows that
$y_1\Omega,y_2\Omega\in\cD(\Delta^{-1/2})$ and
$$
\<y_1\Omega,x'y_2\Omega\>
=\<\Delta^{1/2}y_1\Omega,Jx^*J\Delta^{-1/2}y_2\Omega\>
+\lambda\<\Delta^{-1/2}y_1\Omega,Jx^*J\Delta^{1/2}y_2\Omega\>.
$$
Thus, what remains to show is that, for any $\xi\in\cD(\Delta^{1/2})\cap\cD(\Delta^{-1/2})$,
there is a sequence $y_n\in M$ such that
$$
y_n\Omega\in(\Delta+1)^{-1}M'\Omega,\quad y_n\Omega\to\xi,\quad
\Delta^{1/2}y_n\Omega\to\Delta^{1/2}\xi,\quad\Delta^{-1/2}y_n\Omega\to\Delta^{-1/2}\xi.
$$
Since $\Delta^{-1/2}M'\Omega=JFM'\Omega=JF_0M'\Omega=JM'\Omega$, we see that
$\Delta^{-1/2}M'\Omega$ is dense in $\cH$. For any
$\xi\in\cD(\Delta^{1/2})\cap\cD(\Delta^{-1/2})$, there is a sequence $y_n'\in M'$
such that $\Delta^{-1/2}y_n'\Omega\to(\Delta^{1/2}+\Delta^{-1/2})\xi$. By Lemma \ref{L-2.5}
choose $y_n\in M$ such that $y_n\Omega=(\Delta+1)^{-1}y_n'\Omega$. Then
\begin{align*}
y_n\Omega&={1\over\Delta^{1/2}+\Delta^{-1/2}}\,\Delta^{-1/2}y_n'\Omega
\ \longrightarrow\ \xi, \\
\Delta^{1/2}y_n\Omega
&={\Delta^{1/2}\over\Delta^{1/2}+\Delta^{-1/2}}\,\Delta^{-1/2}y_n'\Omega
\ \longrightarrow\ \Delta^{1/2}\xi, \\
\Delta^{-1/2}y_n\Omega
&={\Delta^{-1/2}\over\Delta^{1/2}+\Delta^{-1/2}}\,\Delta^{-1/2}y_n'\Omega
\ \longrightarrow\ \Delta^{-1/2}\xi,
\end{align*}
as desired.
\end{proof}

\begin{lemma}\label{L-2.7}
Let $\lambda=e^{i\theta}$ with $-\pi<\theta<\pi$. Assume that $x,y\in B(\cH)$ satisfy
\begin{align}\label{F-2.5}
\<\xi_1,x\xi_2\>=\<\Delta^{1/2}\xi_1,y\Delta^{-1/2}\xi_2\>
+\lambda\<\Delta^{-1/2}\xi_1,y\Delta^{1/2}\xi_2\>
\end{align}
for all $\xi_1,\xi_2\in\cD(\Delta^{1/2})\cap\cD(\Delta^{-1/2})$. Then
$$
y=i\int_{-\infty}^\infty
{e^{-\theta t}e^{-i\theta/2}\over e^{\pi t}+e^{-\pi t}}\,\Delta^{-it}x\Delta^{it}\,dt,
$$
where the integral converges in the strong operator topology.
\end{lemma}

\begin{proof}
For each $n\in\bN$ let $E_n$ be the spectral projection of $\Delta$ corresponding to
$[n^{-1},n]$. Define
$$
f(z):={e^{i\theta z}\over\sin\pi z}\,\Delta^{-z}E_nxE_n\Delta^z,\qquad z\in\bC,
$$
which is analytic in the operator norm in $\bC$ except at $z=k$ ($k\in\bZ$). Note that
$f$ has simple poles at $z=k$ ($k\in\bZ$). Consider the contour integral along the
following closed curve:
$$
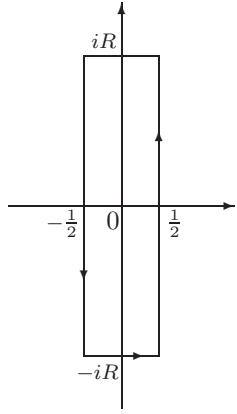
\setlength{\unitlength}{1mm}
\begin{picture}(30,60)(0,-30)
\put(0,0){\vector(1,0){30}}
\put(15,-27){\vector(0,1){54}}
\put(13,-3){\small$0$}
\put(10,-20){\line(0,1){40}}
\put(10,-5){\vector(0,-1){5}}
\put(20,-20){\line(0,1){40}}
\put(20,5){\vector(0,1){5}}
\put(10,-20){\line(1,0){10}}
\put(15,-20){\vector(1,0){3}}
\put(10,20){\line(1,0){10}}
\put(5,-3){\scriptsize$-{1\over2}$}
\put(21,-3){\scriptsize${1\over2}$}
\put(9,-23){\scriptsize$-iR$}
\put(11,21){\scriptsize$iR$}
\end{picture}
$$
Note that, for $R>0$ and $-1/2\le s\le1/2$,
$$
\|f(s\pm iR)\|\le{2e^{|\theta|R}\over e^{\pi R}-e^{-\pi R}}
\,\|\Delta^{-s}E_n\|\,\|x\|\,\|\Delta^sE_n\|
\le{2e^{|\theta|R}\over e^{\pi R}-e^{-\pi R}}\,n^2\|x\|\ \longrightarrow\ 0
\quad(R\to\infty),
$$
and the residue of $f$ at $z=0$ is
$$
\lim_{z\to0}zf(z)={1\over\pi}\,E_nxE_n.
$$
Use the residue theorem and let $R\to\infty$ to obtain
$$
{1\over\pi}\,E_nxE_n={1\over2\pi i}\biggl[
\int_{-\infty}^\infty f\biggl(it+{1\over2}\biggr)i\,dt
-\int_{-\infty}^\infty f\biggl(it-{1\over2}\biggr)i\,dt\biggr].
$$
Since
\begin{align*}
f\biggl(it\pm{1\over2}\biggr)
&={2ie^{i\theta(it\pm{1\over2})}\over
e^{i\pi(it\pm{1\over2})}-e^{-i\pi(it\pm{1\over2})}}\,
\Delta^{-(it\pm{1\over2})}E_nxE_n\Delta^{it\pm{1\over2}} \\
&={\pm2ie^{-\theta t}e^{\pm i\theta/2}\over e^{\pi t}+e^{-\pi t}}\,
\Delta^{\mp1/2}E_n\Delta^{-it}x\Delta^{it}E_n\Delta^{\pm1/2},
\end{align*}
one has
\begin{align}\label{F-2.6}
-iE_nxE_n=\Delta^{1/2}E_n\widetilde xE_n\Delta^{-1/2}
+\lambda\Delta^{-1/2}E_n\widetilde xE_n\Delta^{1/2},
\end{align}
where
$$
\widetilde x:=\int_{-\infty}^\infty
{e^{-\theta t}e^{-i\theta/2}\over e^{\pi t}+e^{-\pi t}}\,
\Delta^{it}x\Delta^{-it}\,dt.
$$
On the other hand, since \eqref{F-2.5} holds for all $\xi_1,\xi_2\in E_n\cH$, one has
\begin{align}\label{F-2.7}
E_nxE_n=\Delta^{1/2}E_nyE_n\Delta^{-1/2}+\lambda\Delta^{-1/2}E_nyE_n\Delta^{1/2}.
\end{align}
It follows from \eqref{F-2.6} and \eqref{F-2.7} that
\begin{align}\label{F-2.8}
0=\Delta^{1/2}E_n(y-i\widetilde x)E_n\Delta^{-1/2}
+\lambda\Delta^{-1/2}E_n(y-i\widetilde x)E_n\Delta^{1/2}
\end{align}

Now, let $a$ and $b$ denote the left and the right multiplication operators by $\Delta E_n$
and $\Delta^{-1}E_n$, respectively, acting on $B(\cH)$. Let $\fA$ be the unital commutative
Banach subalgebra of $B(B(\cH))$ generated by $a,b$. Then \eqref{F-2.8} means that
$abv=-\lambda v$, where $v:=E_n(y-i\widetilde x)E_n\in B(\cH)$. Assume that $v\ne0$; then
$-\lambda\in\sigma(ab)$, the spectrum of $ab$ in $\fA$. In view of the Gelfand transform, we
can easily see that $\sigma(ab)\subset\sigma(a)\sigma(b)$. Since
$\sigma(a),\sigma(b)\subset[0,\infty)$ obviously, we have
$-\lambda\in\sigma(ab)\subset[0,\infty)$, contradicting $|\lambda|=1$ with $\lambda\ne-1$.
Therefore, $v=0$, i.e., $E_n(y-i\widetilde x)E_n=0$. Letting $n\to\infty$ gives
$y=i\widetilde x$.
\end{proof}

\begin{lemma}\label{L-2.8}
Let $\lambda=e^{i\theta}$ with $-\pi<\theta<\pi$. Then
\begin{align}\label{F-2.9}
\Delta^{1/2}(\Delta+\lambda)^{-1}
=i\int_{-\infty}^\infty{e^{-\theta t}e^{-i\theta/2}\over e^{\pi t}+e^{-\pi t}}
\,\Delta^{-it}\,dt,
\end{align}
where the integral converges in the strong operator topology.
\end{lemma}

\begin{proof}
Apply the proof of Lemma \ref{L-2.7} to
$$
f(z):={e^{i\theta z}\over\sin\pi z}\,\Delta^{-z}E_n.
$$
Since $\lim_{z\to0}zf(z)={1\over\pi}\,E_n$, one obtains
$$
{1\over\pi}\,E_n={1\over2\pi i}\biggl[
\int_{-\infty}^\infty f\biggl(it+{1\over2}\biggr)i\,dt
-\int_{-\infty}^\infty f\biggl(it-{1\over2}\biggr)i\,dt\biggr],
$$
which is specified as
$$
-iE_n=\int_{-\infty}^\infty{e^{-\theta t}e^{-i\theta/2}\over e^{\pi t}+e^{-\pi t}}
(\Delta^{1/2}+\lambda\Delta^{-1/2})\Delta^{-it}E_n\,dt.
$$
Since $\Delta^{1/2}+\lambda\Delta^{-1/2}=\Delta^{-1/2}(\Delta+\lambda)$, one has
$$
\Delta^{1/2}(\Delta+\lambda)^{-1}E_n
=i\biggl(\int_{-\infty}^\infty{e^{-\theta t}e^{-i\theta/2}\over e^{\pi t}+e^{-\pi t}}
\,\Delta^{-it}\,dt\biggr)E_n.
$$
Letting $n\to\infty$ gives \eqref{F-2.9}.
\end{proof}

\begin{proof}[Proof of Theorem \ref{T-2.2}]
Let $x'\in M'$ and $\lambda=e^{i\theta}$ with $-\pi<\theta<\pi$. By Lemma \ref{L-2.5} there
exists an $x\in M$ satisfying \eqref{F-2.3}, so by Lemma \ref{L-2.6}, \eqref{F-2.4} holds for
all $\xi_1,\xi_2\in\cD(\Delta^{1/2})\cap\cD(\Delta^{-1/2})$. Then by Lemma \ref{L-2.7} applied
to $x'$, $Jx^*J$ (for $x,y$), it follows that
\begin{align}\label{F-2.10}
Jx^*J=i\int_{-\infty}^\infty{e^{-\theta t}e^{-i\theta/2}\over e^{\pi t}+e^{-\pi t}}
\,\Delta^{-it}x'\Delta^{it}\,dt.
\end{align}
Therefore, for every $y'\in M'$ we have by \eqref{F-2.3} and \eqref{F-2.10}
\begin{align*}
y'J\Delta^{1/2}(\Delta+\lambda)^{-1}x'\Omega
&=y'J\Delta^{1/2}x\Omega=y'x^*\Omega=x^*y'\Omega \\
&=-ie^{i\theta/2}\int_{-\infty}^\infty{e^{-\theta t}\over e^{\pi t}+e^{-\pi t}}\,
J\Delta^{-it}x'\Delta^{it}Jy'\Omega\,dt.
\end{align*}
On the other hand, by Lemma \ref{L-2.8} we have
$$
y'J\Delta^{1/2}(\Delta+\lambda)^{-1}x'\Omega
=-ie^{i\theta/2}\int_{-\infty}^\infty{e^{-\theta t}\over e^{\pi t}+e^{-\pi t}}\,
y'J\Delta^{-it}x'\Omega\,dt.
$$
Combining the above two identities gives
$$
\int_{-\infty}^\infty{e^{-\theta t}\over e^{\pi t}+e^{-\pi t}}
\bigl(J\Delta^{-it}x'\Delta^{it}Jy'\Omega-y'J\Delta^{-it}x'\Omega\bigr)\,dt=0,
\qquad-\pi<\theta<\pi.
$$
Since
$$
z\ \longmapsto\ \int_{-\infty}^\infty{e^{-zt}\over e^{\pi t}+e^{-\pi t}}
\bigl(J\Delta^{-it}x'\Delta^{it}Jy'\Omega-y'J\Delta^{-it}x'\Omega\bigr)\,dt=0
$$
is analytic in $-\pi<\Re z<\pi$ as easily verified, it follows from analytic continuation that
$$
\int_{-\infty}^\infty{e^{-ist}\over e^{\pi t}+e^{-\pi t}}
\bigl(J\Delta^{-it}x'\Delta^{it}Jy'\Omega-y'J\Delta^{-it}x'\Omega\bigr)\,dt=0,\qquad s\in\bR.
$$
The injectivity of the Fourier transform yields
\begin{align}\label{F-2.11}
J\Delta^{-it}x'\Delta^{it}Jy'\Omega=y'J\Delta^{-it}x'\Omega,\qquad t\in\bR.
\end{align}

Now, for every $y_1',y_2'\in M'$, using \eqref{F-2.11} twice we have
$$
y_1'J\Delta^{-it}x'\Delta^{it}Jy_2'\Omega=y_1'y_2'J\Delta^{-it}x'\Omega
=J\Delta^{-it}x'\Delta^{it}Jy_1'y_2'\Omega,
$$
which implies that $y_1'J\Delta^{-it}x'\Delta^{it}J=J\Delta^{-it}x'\Delta^{it}Jy_1'$ so that
$J\Delta^{-it}x'\Delta^{it}J\in M''=M$. Letting $t=0$ gives $Jx'J\in M$. Therefore,
$JM'J\subset M$ and so $M'\subset JMJ$. Furthermore, for every $x,y\in M$ and $x'\in M'$ we
find that
\begin{align*}
\<x'^*\Omega,yJx\Omega\>&=\<y^*\Omega,x'Jx\Omega\>=\<y^*\Omega,J(Jx'J)x\Omega\> \\
&=\<y^*\Omega,JSx^*Jx'^*J\Omega\>\qquad\mbox{(since $Jx'J\in M$)} \\
&=\<y^*\Omega,FJx^*Jx'^*\Omega\>\qquad \mbox{(by Lemma \ref{L-2.1}\,(iii))} \\
&=\<Jx^*Jx'^*\Omega,Sy^*\Omega\>=\<x'^*\Omega,JxJy\Omega\>.
\end{align*}
Therefore,
$$
JxJy\Omega=yJx\Omega.
$$
For every $y_1,y_2\in M$ we hence have
$$
y_1JxJy_2\Omega=y_1y_2Jx\Omega=JxJy_1y_2\Omega,
$$
which implies that $JxJ\in M'$, so $JMJ\subset M'$. Thus, \eqref{F-2.1} follows. Moreover,
since $J\Delta^{-it}M'\Delta^{it}J\subset M$ as proved above, we have
$\Delta^{-it}M'\Delta^{it}\subset JMJ=M'$ and hence $M'\subset\Delta^{it}M'\Delta^{-it}$ as
well. Thus, $\Delta^{it}M'\Delta^{-it}=M'$ and \eqref{F-2.2} follows.
\end{proof}

\begin{remark}\label{R-2.9}\rm
From \eqref{F-2.1} note that $JM\Omega=JMJ\Omega=M'\Omega$ and for every $x'\in M'$,
$$
JS_0Jx'\Omega=JS_0(Jx'J)\Omega=JJx'^*J\Omega=x'^*\Omega=F_0x'\Omega
$$
so that $F_0=JS_0J$ and hence $\overline{F_0}=JSJ=F=S_0^*$. In this way, we have the
complete symmetry between $S_0$ and $F_0$. If we start from $(M',\Omega,F_0)$ in place of
$(M,\Omega,S_0)$, then the modular operator is $\Delta^{-1}$ and the modular conjugation is
the same $J$. A direct proof of $\overline{F_0}=S_0^*$ without the use of \eqref{F-2.1} is
found in e.g., \cite[Proposition 2.5.9]{BR}.
\end{remark}

\begin{prop}\label{P-2.10}
If $x\in M\cap M'$ (the center of $M$), then
$$
JxJ=x^*,\qquad\sigma_t(x)=\Delta^{it}x\Delta^{-it}=x,\qquad t\in\bR.
$$
\end{prop}

\begin{proof}
Assume that $x\in M\cap M'$; then $Sx\Omega=x^*\Omega$ and $Fx^*\Omega=x\Omega$. Hence by
Lemma \ref{L-2.1}\,(ii), $\Delta x\Omega=FSx\Omega=x\Omega$. Hence, since
$\sigma_t(x)\Omega=\Delta^{it}x\Omega=x\Omega$, one has $\sigma_t(x)=x$.
Moreover, since $JxJ\Omega=J\Delta^{1/2}x\Omega=Sx\Omega=x^*\Omega$, one has $JxJ=x^*$. 
\end{proof}

\subsection{KMS condition}

In the rest of the section we present Takesaki's theorem on the KMS condition of the modular
automorphism group, an important ingredient of the Tomita-Takesaki theory in addition to
Tomita's theorem. Let $\alpha_t$ ($t\in\bR$) be a one-parameter weakly continuous automorphism
group of $M$. It is worth noting that $\alpha_t$ is automatically strongly* continuous.
Indeed, for any $x\in M$ and $\xi\in\cH$,
\begin{align*}
\|(\alpha_t(x)-x)\xi\|^2&=\<\xi,(\alpha_t(x)-x)^*(\alpha_t(x)-x)\xi\> \\
&=\<\xi,\alpha_t(x^*x)\xi\>-\<\alpha_t(x)\xi,x\xi\>
-\<x\xi,\alpha_t(x)\xi\>+\<\xi,x^*x\xi\> \\
&\longrightarrow\ 0\quad(t\to0),
\end{align*}
and also $\|(\alpha_t(x)^*-x^*)\xi\|\to0$ ($t\to0$).

For $\beta\in\bR$ with $\beta\ne0$, if $\beta<0$,
$$D_\beta:=\{z\in\bC:0<\Im z<-\beta\},\qquad
\overline D_\beta:=\{z\in\bC:0\le\Im z\le-\beta\},
$$
and if $\beta>0$, $D_\beta:=\{z\in\bC:-\beta<\Im z<0\}$ and $\overline D_\beta$ is similar.

\begin{definition}\label{D-2.11}\rm
A functional $\ffi\in M_*^+$ is said to satisfy the \emph{KMS (Kubo-Martin-Schwinger) condition}
with respect to $\alpha_t$ at $\beta$, or the $(\alpha_t,\beta)$-KMS condition, if for every
$x,y\in M$ there is a bounded continuous function $f_{x,y}(z)$ on $\overline D_\beta$,
that is analytic in $D_\beta$, such that
$$
f_{x,y}(t)=\ffi(\alpha_t(x)y),\qquad
f_{x,y}(t-i\beta)=\ffi(y\alpha_t(x)),\qquad t\in\bR.
$$
\end{definition}

This condition proposed by Haag, Hugenholtz and Winnink serves as a mathematical formulation
of equilibrium states in the quantum statistical mechanics, which is also defined and more
useful in $C^*$-algebraic dynamical systems (see \cite{BR2}). To illustrate this, given a
Hamiltonian $H\in B(\cH)_\sa$ in a finite-dimensional $\cH$, consider the Gibbs state
$\ffi(x)=\Tr e^{-\beta H}x/\Tr e^{-\beta H}$ and the corresponding dynamics
$\alpha_t(x)=e^{it H}xe^{-it H}$, $t\in\bR$. For any $x,y\in B(\cH)$ the entire function
$f_{x,y}(z):=\Tr(e^{-\beta H}e^{izH}xe^{-izH}y)/\Tr e^{-\beta H}$ satisfies
\begin{align*}
f_{x,y}(t)&=\Tr(e^{-\beta H}e^{itH}xe^{-itH}y)/\Tr e^{-\beta H}=\ffi(\alpha_t(x)y), \\
f_{x,y}(t-i\beta)&=\Tr(e^{itH}xe^{-itH}e^{-\beta H}y)/\Tr e^{-\beta H}=\ffi(y\alpha_t(x)),
\qquad t\in\bR,
\end{align*}
so that $\ffi$ satisfies the $(\alpha_t,\beta)$-KMS condition. Moreover, the Gibbs state
$\ffi$ is a unique state satisfying the $(\alpha_t,\beta)$-KMS condition (an exercise).

The following lemma is another justification for the KMS condition to describe equilibrium
states.

\begin{lemma}\label{L-2.12}
If $\ffi\in M_*^+$ satisfies the $(\alpha_t,\beta)$-KMS condition, then
$\ffi\circ\alpha_t=\ffi$ for all $t\in\bR$.
\end{lemma}

\begin{proof}
Let $x\in M$. From the KMS condition applied to $x$ and $y=1$, the function
$f(t)=\ffi(\alpha_t(x))$ can extend to a bounded continuous function $f(z)$ on
$\overline D_\beta$, analytic in $D_\beta$, such that $f(t)=f(t-i\beta)$, $t\in\bR$. From the
Schwarz reflection principle, $f$ can further extend to an entire function with a period
$i\beta$, which is bounded. Hence the Liouville theorem says that $f$ is a constant function,
so $f(t)\equiv f(0)$, i.e., $\ffi(\alpha_t(x))=\ffi(x)$, $t\in\bR$.
\end{proof}

An element $x\in M$ is said to be \emph{$\alpha_t$-analytic (entire)} if there is an
$M$-valued entire function $x(z)$ in the strong* topology such that $x(t)=\alpha_t(x)$ for all
$t\in\bR$. In this case, we write $\alpha_z(x)$ for $x(z)$, $z\in\bC$.

\begin{lemma}\label{L-2.13}
The set of $\alpha_t$-analytic elements is an $\alpha_t$-invariant subalgebra of $M$
that is strongly* dense in $M$. Moreover, if $x\in M$ is $\alpha_t$-analytic, then
$\alpha_\zeta(x)$ is also $\alpha_t$-analytic for every $\zeta\in\bC$.
\end{lemma}

\begin{proof}
Write $M(\alpha)$ for the set of $\alpha_t$-analytic elements $x\in M$. If $x,y\in M(\alpha)$
with analytic extensions $\alpha_z(x)$ and $\alpha_z(y)$, then we have
$xy,x^*,\alpha_t(x)\in M(\alpha)$ with $\alpha_z(xy)=\alpha_z(x)\alpha_z(y)$,
$\alpha_z(x^*)=\alpha_{\overline z}(x)^*$ and $\alpha_z(\alpha_t(x))=\alpha_{z+t}(x)$. Hence
$M(\alpha)$ is an $\alpha_t$-invariant subalgebra of $M$. Since
$\alpha_t(\alpha_\zeta(x))=\alpha_{t+\zeta}(x)$, we have also $\alpha_\zeta(x)\in M(\alpha)$
with $\alpha_z(\alpha_\zeta(x))=\alpha_{z+\zeta}(x)$. To prove the strong* denseness of
$M(\alpha)$, for every $x\in M$ and $n\in\bN$ define
$$
x_n:=\sqrt{n\over\pi}\int_{-\infty}^\infty e^{-nt^2}\alpha_t(x)\,dt.
$$
It is easy to verify that $x_n\to x$ strongly* as $n\to\infty$. Moreover, define
$$
x_n(z):=\sqrt{n\over\pi}\int_{-\infty}^\infty e^{-n(t-z)^2}\alpha_t(x)\,dt,
\qquad z\in\bC.
$$
It then follows that $x_n(z)$ is entirely analytic in the strong* topology and
$x_n(s)=\alpha_s(x)$ for all $s\in\bR$.
\end{proof}

Takesaki's theorem in \cite{Ta} is the following:

\begin{thm}[Takesaki]\label{T-2.14}
In the same situation as in Theorem \ref{T-2.2}, the $\omega$ satisfies the
$(\sigma_t,-1)$-KMS condition. Furthermore, $\sigma_t^\omega$ is uniquely determined as a
weakly continuous one-parameter automorphism group of $M$ for which $\omega$ satisfies the
KMS condition at $\beta=-1$.
\end{thm}

\begin{proof}
For any $x,y\in M$, since $x\Omega,y\Omega\in\cD(\Delta^{1/2})$, it follows from Theorem
\ref{T-A.7} of Appendix~A that $\Delta^{-iz/2}x\Omega$ and $\Delta^{-iz/2}y\Omega$ are bounded
continuous on $0\le\Im z\le1$ and analytic in $0<\Im z<1$ in the norm on $\cH$. Therefore,
$$
f(z):=\<\Delta^{i\overline z/2}x\Omega,\Delta^{-iz/2}y\Omega\>
$$
is bounded continuous on $0\le\Im z\le1$ and analytic in $0<\Im z<1$. For every $t\in\bR$
compute
\begin{align*}
f(t)&=\<\Delta^{it/2}x\Omega,\Delta^{-it/2}y\Omega\>
=\<\Delta^{it}x\Delta^{-it}\Omega,y\Omega\>\quad\mbox{(by Lemma \ref{L-2.1}\,(v))} \\
&=\<\sigma_t(x)\Omega,y\Omega\>=\omega(\sigma_t(x^*)y), \\
f(t+i)&=\<\Delta^{1/2}\Delta^{it/2}x\Omega,\Delta^{-it/2}\Delta^{1/2}y\Omega\>
=\<\Delta^{1/2}\Delta^{it}x\Delta^{-it}\Omega,\Delta^{1/2}y\Omega\> \\
&=\<J\Delta^{1/2}y\Omega,J\Delta^{1/2}\sigma_t(x)\Omega\>
=\<y^*\Omega,\sigma_t(x^*)\Omega\>=\omega(y\sigma_t(x^*)).
\end{align*}
Hence $\omega$ satisfies the $(\sigma_t,-1)$-KMS condition.

To prove the uniqueness assertion, let $\alpha_t$ be a weakly (hence strongly) continuous
one-parameter automorphism group of $M$ such that $\omega$ satisfies the $(\alpha_t,-1)$-KMS
condition. By Lemma \ref{L-2.12}, $\omega$ is $\alpha_t$-invariant, so one can define
$U_tx\Omega:=\alpha_t(x)\Omega$ for $x\in M$ to obtain a continuous one-parameter unitary
group $U_t$ on $\cH$ such that $\alpha_t(x)=U_txU_t^*$ for all $x\in M$ and $t\in\bR$. Let
$x\in M$ be an $\alpha_t$-analytic element with the analytic extension $\alpha_z(x)$, and
$y\in M$ be a $\sigma_t$-analytic element with $\sigma_z(y)$. Define
$f(z):=\omega(\alpha_z(x)\sigma_z(y))$, which is an entire function. For any
$s\in\bR$ let $\widetilde y:=\sigma_s(y)$. Since the entire function
$f_{x,\widetilde y}(z):=\omega(\alpha_z(x)\widetilde y)$ satisfies
$$
f_{x,\widetilde y}(t):=\omega(\alpha_t(x)\widetilde y),\qquad
f_{x,\widetilde y}(t+i):=\omega(\widetilde y\alpha_t(x)),\qquad t\in\bR,
$$
from the $(\alpha,-1)$-KMS condition of $\omega$, we have
$\omega(\alpha_z(x)\widetilde y)=\omega(\widetilde y\alpha_{z-i}(x))$ for all $z\in\bC$.
Therefore,
\begin{align}\label{F-2.12}
\omega(\alpha_t(x)\sigma_s(y))=\omega(\sigma_s(y)\alpha_{t-i}(x)),\qquad t,s\in\bR.
\end{align}
Similarly, for any $s\in\bR$ let $\widetilde x:=\alpha_{s-i}(x)$. Since the entire function
$f_{y,\widetilde x}(z):=\omega(\sigma_z(y)\widetilde x)$ satisfies
$$
f_{y,\widetilde x}(t)=\omega(\sigma_t(y)\widetilde x),\qquad
f_{y,\widetilde x}(t+i)=\omega(\widetilde x\sigma_t(y)),\qquad t\in\bR,
$$
from the $(\sigma,-1)$-KMS condition of $\omega$, we have
$\omega(\sigma_z(y)\widetilde x)=\omega(\widetilde x\sigma_{z-i}(y))$ for all
$z\in\bC$. Therefore,
\begin{align}\label{F-2.13}
\omega(\sigma_t(y)\alpha_{s-i}(x))
=\omega(\alpha_{s-i}(x)\sigma_{t-i}(y)),\qquad t,s\in\bR.
\end{align}
Combining \eqref{F-2.12} and \eqref{F-2.13} gives
\begin{align*}
f(t)&=\omega(\alpha_t(x)\sigma_t(y))=\omega(\sigma_t(y)\alpha_{t-i}(x)) \\
&=\omega(\alpha_{t-i}(x)\sigma_{t-i}(y))=f(t-i),\qquad t\in\bR,
\end{align*}
that is, $f$ has a period $i$. From this and the three-lines theorem it follows that $f$ is
bounded on $0\le\Im z\le1$ and hence bounded on the whole $\bC$. By the Liouville theorem,
$f\equiv f(0)$ so that $\omega(\alpha_t(x)\sigma_t(y))=\omega(xy)$ for all $t\in\bR$. By
Lemma \ref{L-2.13} (applied to $\alpha_t$ and $\sigma_t$), this equality can extend to all
$x,y\in M$, so we have
\begin{align*}
\<x^*\Omega,U_t^*\Delta^{it}y\Omega\>
&=\<\alpha_t(x^*)\Omega,\sigma_t(y)\Omega\>=\omega(\alpha_t(x)\sigma_t(y)) \\
&=\omega(xy)=\<x^*\Omega,y\Omega\>
\end{align*}
for all $x,y\in M$. Therefore, $U_t^*\Delta^{it}=1$, i.e., $U_t=\Delta^{it}$, which implies
that $\alpha_t=\sigma_t$.
\end{proof}

\begin{definition}\label{D-2.15}\rm
The \emph{centralizer} of $\omega$ is defined as
$$
M_\omega:=\{x\in M:\omega(xy)=\omega(yx),\ y\in M\}.
$$
It is obvious that $M_\omega$ is a subalgebra of $M$ including the center $M\cap M'$.
\end{definition}

\begin{prop}\label{P-2.16}
The centralizer $M_\omega$ coincides with the fixed-point algebra of $\sigma_t^\omega$, i.e.,
$$
M_\omega=\{x\in M:\sigma_t^\omega(x)=x,\ t\in\bR\}.
$$
\end{prop}

\begin{proof}
Let $x\in M$. For every $\sigma_t$-analytic element $y\in M$, the entire function
$f_{y,x}(z):=\omega(\sigma_z(y)x)$ satisfies
$$
f_{y,x}(t)=\omega(\sigma_t(y)x)=\omega(y\sigma_{-t}(x)),\qquad
f_{y,x}(t+i)=\omega(x\sigma_t(y)),\qquad t\in\bR,
$$
from $\omega\circ\sigma_t=\omega$ (by Lemma \ref{L-2.12}) and the $(\sigma,-1)$-KMS condition
of $\omega$. If $\sigma_t(x)=x$ for all $t\in\bR$, then $f_{y,x}(z)\equiv\omega(yx)$ so that
$\omega(xy)=\omega(yx)$. By Lemma \ref{L-2.13}, $x\in M_\omega$ follows. Conversely, if
$x\in M_\omega$, then $f_{y,x}(t)=f_{y,x}(t+i)$ for all $t\in\bR$ so that
$f_{y,x}\equiv f_{y,x}(0)$ as in the proof of Theorem \ref{T-2.14}. Therefore,
$$
\<y^*\Omega,\sigma_{-t}(x)\Omega\>=\omega(y\sigma_{-t}(x))
=\omega(yx)=\<y^*\Omega,x\Omega\>,
$$
which implies by Lemma \ref{L-2.13} that $\sigma_{-t}(x)=x$ for all $t\in\bR$.
\end{proof}

By Proposition \ref{P-2.16} we see that $\sigma_t(x)=x$ for all $x\in M\cap M'$ and $t\in\bR$,
which was shown in Proposition \ref{P-2.10}. Also, it follows that $\omega$ is a trace if and
only if $\sigma_t=\id$ for all $t\in\bR$.

\section{Standard form}

The theory of the standard form of von Neumann algebras was developed, independently, by
Araki \cite{Ar} and Connes \cite{Co3} in the case of $\sigma$-finite von Neumann algebras,
and by Haagerup \cite{Haa0,Haa1} in the general case. In this section we present a concise
exposition of the theory, mainly following \cite{Haa0,Haa1} with certain simplifications in
\cite[\S2.5.4]{BR2}.

\subsection{Definition and basic properties}

Let $M$ be a $\sigma$-finite von Neumann algebra, thus represented on a Hilbert space $\cH$
with a cyclic and separating vector $\Omega$, for which we have the modular operator $\Delta$
and the modular conjugation $J$, as in Sec.~2. Let $j:M\to M'$ be the conjugate-linear
*-isomorphism defined by $j(x):=JxJ$, $x\in M$.

\begin{definition}\label{D-3.1}\rm
The \emph{natural positive cone} $\cP=\cP^\natural$ in $\cH$ associated with $(M,\Omega)$ is
defined by
\begin{align}\label{F-3.1}
\cP:=\overline{\{xj(x)\Omega:x\in M\}}=\overline{\{xJx\Omega:x\in M\}}.
\end{align}
In addition, define
$$
\cP^\sharp:=\overline{M_+\Omega},\qquad\cP^\flat:=\overline{M_+'\Omega}.
$$
\end{definition}

\begin{thm}\label{T-3.2}
We have:
\begin{itemize}
\item[\rm(i)] $\cP=\overline{\Delta^{1/4}M_+\Omega}=\overline{\Delta^{1/4}\cP^\sharp}
=\overline{\Delta^{-1/4}M_+'\Omega}=\overline{\Delta^{-1/4}\cP^\flat}$. In particular,
$\cP$ is a closed cone.
\item[\rm(ii)] $J\xi=\xi$ for all $\xi\in\cP$.
\item[\rm(iii)] $\Delta^{it}\cP=\cP$ for all $t\in\bR$.
\item[\rm(iv)] $xj(x)\cP\subset\cP$ for all $x\in M$.
\item[\rm(v)] If $f$ is a positive definite function on $\bR$, then
$f(\log\Delta)\cP\subset\cP$.
\item[\rm(vi)] $\cP$ is self-dual, i.e.,
$$
\cP=\{\eta\in\cH:\<\xi,\eta\>\ge0,\ \xi\in\cP\}.
$$
\end{itemize}
\end{thm}

\begin{proof}
(i)\enspace
Let $\sigma_t$ be the modular automorphism group for $(M,\Omega)$. write $M(\sigma)$ for the
set of $\sigma$-analytic elements $x\in M$.  For every $x\in M(\sigma)$ let
$y:=\sigma_{i/4}(x)$; then $y\in M(\sigma)$ and $x=\sigma_{-i/4}(y)$. Note by Lemma
\ref{L-2.13} and Theorem \ref{T-A.7} that, for any $\zeta\in\bC$,
$\sigma_\zeta(y)\Omega\in\bigcap_{z\in\bC}\cD(\Delta^z)$ and
$\Delta^z\sigma_\zeta(y)\Omega=\sigma_{-iz+\zeta}(y)\Omega$. We hence have
\begin{align}
xj(x)\Omega&=\sigma_{-i/4}(y)J\sigma_{-i/4}(y)\Omega
=\sigma_{-i/4}(y)J\sigma_{-i/2+i/4}(y)\Omega \nonumber\\
&=\sigma_{-i/4}(y)J\Delta^{1/2}\sigma_{i/4}(y)\Omega
=\sigma_{-i/4}(y)\sigma_{i/4}(y)^*\Omega \nonumber\\
&=\sigma_{-i/4}(y)\sigma_{-i/4}(y^*)\Omega
=\sigma_{-i/4}(yy^*)=\Delta^{1/4}yy^*\Omega. \label{F-3.2}
\end{align}
This implies by Lemma \ref{L-2.13} and the Kaplansky density theorem that
$\cP\subset\overline{\Delta^{1/4}M_+\Omega}\subset\overline{\Delta^{1/4}\cP^\sharp}$.
Conversely, for every $\eta\in\cP^\sharp$, from Lemma \ref{L-2.13} and the Kaplansky density
theorem again, there is a sequence $\{y_n\}$ in $M(\sigma)$ such that $y_ny_n^*\Omega\to\eta$.
Then $Sy_ny_n^*\Omega=y_ny_n^*\Omega\to\eta$, so
$\eta\in\cD(S)=\cD(\Delta^{1/2})\subset\cD(\Delta^{1/4})$ and $S\eta=\eta$. Hence
$$
\Delta^{1/2}y_ny_n^*\Omega=Jy_ny_n^*\Omega\ \longrightarrow\ J\eta=\Delta^{1/2}\eta,
$$
which implies that
$$
\|\Delta^{1/4}(y_ny_n^*\Omega-\eta)\|^2
=\<y_ny_n^*\Omega-\eta,\Delta^{1/2}(y_ny_n^*\Omega-\eta)\>\ \longrightarrow\ 0.
$$
Letting $x_n:=\sigma_{-i/4}(y_n)$ we have $\Delta^{1/4}y_ny_n^*\Omega=x_nj(x_n)\Omega$
similarly to \eqref{F-3.2}, so $\Delta^{1/4}\eta\in\cP$. Therefore,
$\overline{\Delta^{1/4}\cP^\sharp}\subset\cP$, implying
$\cP=\overline{\Delta^{1/4}M_+\Omega}=\overline{\Delta^{1/4}\cP^\sharp}$. When we replace
$(M,\Omega)$ with $(M',\Omega)$, the modular conjugation is the same $J$ and the modular
operator is $\Delta^{-1}$ (see Remark \ref{R-2.9}). Moreover, the natural positive cone is
the same $\cP$ due to \eqref{F-2.1}. Therefore,
$\cP=\overline{\Delta^{-1/4}M_+'\Omega}=\overline{\Delta^{-1/4}\cP^\flat}$ as well.

(ii) follows since
$$
Jxj(x)\Omega=Jj(x)x\Omega=xj(x)\Omega,\qquad x\in M.
$$

(iii) follows from (i) and
$$
\Delta^{it}\Delta^{1/4}M_+\Omega=\Delta^{1/4}\Delta^{it}M_+\Delta^{-it}\Omega
=\Delta^{1/4}M_+\Omega.
$$

(iv)\enspace
For every $x,y\in M$,
$$
xj(x)yj(y)\Omega=xyj(x)j(y)\Omega=xyj(xy)\Omega\in\cP.
$$

(v)\enspace
Bochner's theorem says that $f$ has the form $f(t)=\int_{-\infty}^\infty e^{ist}\,d\mu(s)$
with a finite positive Borel measure $\mu$ on $\bR$. We hence write
$f(\log\Delta)=\int_{-\infty}^\infty\Delta^{is}\,d\mu(s)$, which implies that
$f(\log\Delta)\cP\subset\cP$ by (iii).

Before proving (vi) we give the following lemma.

\begin{lemma}\label{L-3.3}
$\cP^\sharp$ and $\cP^\flat$ are mutually the dual cones in $\cH$, i.e.,
$$
\cP^\flat=\{\eta\in\cH:\<\xi,\eta\>\ge0,\ \xi\in\cP^\sharp\},\qquad
\cP^\sharp=\{\eta\in\cH:\<\xi,\eta\>\ge0,\ \xi\in\cP^\flat\}.
$$
\end{lemma}

\begin{proof}
Write $\cP^{\sharp\vee}$ and $\cP^{\flat\vee}$ for the dual cones of $\cP^\sharp$ and
$\cP^\flat$, respectively. For every $x\in M_+$ and $x'\in M_+'$,
$$
\<x\Omega,x'\Omega\>=\<\Omega,x^{1/2}x'x^{1/2}\Omega\>\ge0,
$$
which implies that $\cP^\flat\subset\cP^{\sharp\vee}$ and $\cP^\sharp\subset\cP^{\flat\vee}$.
To prove the converse, let $\eta\in\cP^{\sharp\vee}$ and define an operator $T_0:M\Omega\to\cH$
by $T_0x\Omega:=x\eta$ for $x\in M$. Since
$$
\<x\Omega,T_0x\Omega\>=\<x\Omega,x\eta\>=\<x^*x\Omega,\eta\>\ge0,
$$
it follows $T_0$ is a densely-defined positive symmetric operator. So one has the Friedrichs
extension (the largest positive self-adjont extension) $T$ of $T_0$ (see \cite[\S124]{RN}).
For every unitary $u\in M$, $T_0ux\Omega=ux\eta=uT_0x\Omega$ for $x\in M$, which means that
$u^*T_0u=T_0$. From the construction of the Friedrichs extension, it follows that $u^*Tu=T$
for all unitaries $u\in M$, so $T$ is affiliated with $M'$. Taking the spectral projection
$e_n'$ of $T$ corresponding to $[0,n]$, one has $x_n':=Te_n'\in M_+'$ and
$x_n'\Omega\to T\Omega=\eta$ so that $\eta\in\cP^\flat$. Hence $\cP^\flat=\cP^{\sharp\vee}$,
and $\cP^\sharp=\cP^{\flat\vee}$ is similar.
\end{proof}

\noindent
{\it Proof of Theorem \ref{T-3.2}\,(vi).}\enspace
Since
$$
\<\Delta^{1/4}x\Omega,\Delta^{-1/4}x'\Omega\>
=\<x\Omega,x'\Omega\>\ge0,\qquad x\in M,\ x'\in M',
$$
it follows from (i) that $\<\xi,\eta\>\ge0$ for all $\xi,\eta\in\cP$. Conversely, assume that
$\eta\in\cH$ and $\<\xi,\eta\>\ge0$ for all $\xi\in\cP$. For each $n\in\bN$ let
$f_n(t):=e^{-t^2/2n^2}$ ($t\in\bR$). Since $0<f_n(\log t)\nearrow1$ ($n\nearrow\infty$) for
every $t>0$, we have $\eta_n:=f_n(\log\Delta)\eta\to\eta$. For any $\alpha\in\bR$, note that
\begin{align*}
\int_0^\infty t^{2\alpha}\,d\|E(t)\eta_n\|^2
&=\int_0^\infty t^{2\alpha}\exp\biggl(-{(\log t)^2\over n^2}\biggr)\,d\|E(t)\eta\|^2 \\
&=\int_0^\infty\exp\biggl(2\alpha\log t-{(\log t)^2\over n^2}\biggr)\,d\|E(t)\eta\|^2
<+\infty,
\end{align*}
where $E(\cdot)$ is the spectral measure of $\Delta$. Therefore, by Theorem \ref{T-A.7},
$\eta_n\in\bigcap_{z\in\bC}\cD(\Delta^z)$. Furthermore, it is well-known that $f_n$ is a
positive definite function on $\bR$, so by (v) and assumption,
$\<\xi,\eta_n\>=\<f_n(\log\Delta)\xi,\eta\>\ge0$ for all $\xi\in\cP$. Therefore, for every
$x\in M_+$, since $\Delta^{1/4}x\Omega\in\cP$ by (i), we have
$\<x\Omega,\Delta^{1/4}\eta_n\>=\<\Delta^{1/4}x\Omega,\eta_n\>\ge0$. By Lemma \ref{L-3.3}
this implies that $\Delta^{1/4}\eta_n\in\cP^\flat$. So
$\eta_n\in\Delta^{-1/4}\cP^\flat\subset\cP$ by (i). Letting $n\to\infty$ gives $\eta\in\cP$.
\end{proof}

\begin{remark}\label{R-3.4}\rm
The cones $\cP^\sharp$ and $\cP^\flat$ were first introduced by Takesaki \cite{Ta}, where
Lemma \ref{L-3.3} was proved. In \cite{Ar}, Araki introduced a one-parameter family of cones
$V_\Omega^\alpha:=\overline{\Delta^\alpha M_+\Omega}$ for $\alpha\in[0,1/2]$. Note that
$V_\Omega^0=\cP^\sharp$, $V_\Omega^{1/4}=\cP^\natural$ and $V_\Omega^{1/2}=\cP^\flat$. In
\cite{Ar} it was shown, among others, that the dual of $V_\Omega^\alpha$ is
$V_\Omega^{{1\over2}-\alpha}$.
\end{remark}

Summing up the discussions so far with Sec.~2, we conclude that any ($\sigma$-finite)
von Neumann algebra is faithfully represented on a Hilbert space $\cH$ with a conjugate-linear
involution $J$ and a self-dual cone $\cP$ such that
\begin{itemize}
\item[\rm(a)] $JMJ=M'$\quad (Theorem \ref{T-2.2}),
\item[\rm(b)] $JxJ=x^*$ for all $x\in M\cap M'$\quad (Proposition \ref{P-2.10}),
\item[\rm(c)] $J\xi=\xi$ for all $\xi\in\cP$,
\item[\rm(d)] $xj(x)\cP\subset\cP$ for all $x\in M$, where $j(x):=JxJ$.
\end{itemize}

\begin{definition}\label{D-3.5}\rm
A quadruple $(M,\cH,J,\cP)$ satisfying the above conditions (a)--(d) is called a
\emph{standard form} of a von Neumann algebra $M$. This is the abstract (or axiomatic)
definition, and we have shown the existence of a standard form for any $\sigma$-finite von
Neumann algebra.
\end{definition}

\begin{example}\label{E-3.6}\rm
(1)\enspace
Let $(X,\cX,\mu)$ be a $\sigma$-finite (or more generally, localizable) measure space. The
commutative von Neumann algebra $M=L^\infty(X,\mu)$ is faithfully represented on the Hilbert
space $L^2(X,\mu)$ as multiplication operators $\pi(f)\xi:=f\xi$ for $f\in L^\infty(X,\mu)$
and $\xi\in L^2(X,\mu)$. The standard form of $M$ is
$$
(L^\infty(X,\mu),L^2(X,\mu),J\xi=\overline\xi,L^2(X,\mu)_+),
$$
where $L^2(X,\mu)_+$ is the cone of non-negative functions $\xi\in L^2(X,\mu)$.

(2)\enspace
Let $M=B(\cH)$, a factor of type I. Let $\cC_2(\cH)$ be the space of Hilbert-Schmidt
operators (i.e., $a\in B(\cH)$ with $\Tr a^*a<+\infty$), which is a Hilbert space with the
Hilbert-Schmidt inner product $\<a,b\>:=\Tr a^*b$ for $a,b\in\cC_2(\cH)$. Then $M=B(\cH)$
is faithfully represented on $\cC_2(\cH)$ as left multiplication operators $\pi(x)a:=xa$ for
$x\in B(\cH)$ and $a\in\cC_2(\cH)$. The standard form of $M$ is
$$
(B(\cH),\cC_2(\cH), J=\,^*,\cC_2(\cH)_+),
$$
where $J=\,^*$ is the adjoint operation and $\cC_2(\cH)_+$ is the cone of positive
$a\in\cC_2(\cH)$. In this case, note that, for any $x\in B(\cH)$, $JxJ$ is the right
multiplication of $x^*$ on $\cC_2(\cH)$ and
$xj(x)\cC_2(\cH)_+=x\cC_2(\cH)_+x^*\subset\cC_2(\cH)_+$.
\end{example}

The following gives geometric properties of the cone $\cP$.

\begin{prop}\label{P-3.7}
Let $(M,\cH,J,\cP)$ be a standard form. Then:
\begin{itemize}
\item[\rm(1)] $\cP$ is a pointed cone, i.e., $\cP\cap(-\cP)=\{0\}$.
\item[\rm(2)] If $\xi\in\cH$ and $J\xi=\xi$, then $\xi$ has a unique decomposition
$\xi=\xi_1-\xi_2$ with $\xi_1,\xi_2\in\cP$ and $\xi_1\perp\xi_2$.
\item[\rm(3)] $\cH$ is linearly spanned by $\cP$.
\end{itemize}
\end{prop}

\begin{proof}
(1)\enspace
If $\xi\in\cP\cap(-\cP)$, then the self-duality of $\cP$ implies that $\<\xi,-\xi\>\ge0$,
hence $\xi=0$.

(2)\enspace
Assume that $\xi\in\cH$ and $J\xi=\xi$. Since $\cP$ is a closed convex set in $\cH$, there is
a unique $\xi_1\in\cP$ such that
$$
\|\xi_1-\xi\|=\inf\{\|\eta-\xi\|:\eta\in\cP\}.
$$
Set $\xi_2:=\xi_1-\xi$. For any $\eta\in\cP$ and $\lambda>0$, since $\xi_1+\lambda\eta\in\cP$,
one has $\|\xi_1-\xi\|^2\le\|\xi_1+\lambda\eta-\xi\|^2$, i.e., 
$\|\xi_2\|^2\le\|\xi_2+\lambda\eta\|^2$ so that
$$
2\lambda\Re\<\xi_2,\eta\>+\lambda^2\|\eta\|^2\ge0,\qquad\lambda>0.
$$
Therefore, $\Re\<\xi_2,\eta\>\ge0$. Since $J\xi_2=J\xi_1-J\xi=\xi_2$,
$\<\xi_2,\eta\>=\<J\xi_2,J\eta\>=\<\eta,\xi_2\>$ so that $\<\xi_2,\eta\>\in\bR$ and
$\<\xi_2,\eta\>\ge0$ for all $\eta\in\cP$. Hence $\xi_2\in\cP$. Next, show that
$\xi_1\perp\xi_2$. For any $\lambda\in(0,1)$, since $(1-\lambda)\xi_1\in\cP$, one has
$\|\xi_1-\xi\|^2\le\|(1-\lambda)\xi_1-\xi\|^2$, i.e., $\|\xi_2\|^2\le\|\xi_2-\lambda\xi_1\|^2$
so that
$$
-2\lambda\<\xi_2,\xi_1\>+\lambda^2\|\xi_1\|^2\ge0,\qquad\lambda\in(0,1).
$$
Therefore, $\<\xi_2,\xi_1\>\le0$. But $\<\xi_2,\xi_1\>\ge0$ since $\xi_1,\xi_2\in\cP$, hence
$\xi_1\perp\xi_2$. To show the uniqueness of the decomposition, besides $\xi=\xi_1-\xi_2$,
let $\xi=\eta_1-\eta_2$ with $\eta_1,\eta_2\in\cP$ and $\eta_1\perp\eta_2$. Then
$\xi_1-\eta_1=\xi_2-\eta_2$ and so
$$
\|\xi_1-\eta_1\|^2=\<\xi_1-\eta_1,\xi_2-\eta_2\>
=-\<\xi_1,\eta_2\>-\<\eta_1,\xi_2\>\le0,
$$
implying $\xi_1=\eta_1$ and $\xi_2=\eta_2$.

(3)\enspace
For any $\eta\in\cH$, let $\xi:=(\eta+J\eta)/2$ and $\xi':=(\eta-J\eta)/2i$; then $J\xi=\xi$,
$J\xi'=\xi'$ and $\eta=\xi+i\xi'$. By (2), $\eta=(\xi_1-\xi_2)+i(\xi_3-\xi_4)$ with
$\xi_i\in\cP$.
\end{proof}

The next proposition is the description of the standard form of a reduced von Neumann algebra
$eMe$, where $e$ is a projection in $M$.

\begin{prop}\label{P-3.8}
Let $(M,\cH,J,\cP)$ be a standard form. Let $e\in M$ be a projection and set $q:=ej(e)$. Then:
\begin{itemize}
\item[\rm(1)] $exe\mapsto qxq$ is a *-isomorphism of $eMe$ onto $qMq$. In particular, $e\ne0$
$\iff$ $q\ne0$.
\item[\rm(2)] $(qMq,q\cH,qJq,q\cP)$ is a standard form of $qMq\cong eMe$.
\end{itemize}
\end{prop}

\begin{proof}
(1)\enspace
Note that the commutant of $eMe$ on $e\cH$ is $M'e$. As is well-known, the central support
$c_{M'e}(q)$ of $q\in M'e$ is the projection onto $\overline{M'eqe\cH}=\overline{eM'j(e)\cH}$.
Hence $c_{M'e}(q)=ec_{M'}(j(e))$, where $c_{M'}(j(e))$ is the central support of $j(e)\in M'$.
Since $J$ commutes with projections in $M\cap M'$ (see Proposition \ref{P-2.10}), we have
$$
c_{M'}(j(e))=Jc_{M'}(j(e))J\ge Jj(e)J=e
$$
so that $c_{M'e}(q)=e$. This is equivalent to that $x\in eMe\mapsto xq\in(eMe)q$ is a
*-isomorphism.

(2)\enspace
Since $Jq=JeJeJ=qJ$, $J$ leaves $q\cH$ invariant. Hence $qJ=Jq$ is a conjugate-linear
involution on $q\cH$. Since $q\cP\subset\cP$ by (d), it is obvious that $\<\xi,\eta\>\ge0$ for
all $\xi,\eta\in q\cP$. Assume that $\eta\in q\cH$ and $\<\xi,\eta\>\ge0$ for all
$\xi\in q\cP$. Then for every $\zeta\in\cP$,
$$
0\le\<q\zeta,\eta\>=\<\zeta,q\eta\>=\<\zeta,\eta\>
$$
so that $\eta\in\cP$ and $\eta=q\eta\in q\cP$. Therefore, $q\cP$ is a self-dual cone in
$q\cH$. Now we may show the conditions (a)--(d) for $(qMq,q\cH,qJq,q\cH)$.

(a) follows since
\begin{align*}
(qJ)(qMq)(qJ)&=qJMJq=qM'q=ej(e)(eMe)'ej(e) \\
&=(eMej(e))'=(qMq)'.
\end{align*}

(b)\enspace
Note that $Z(qMq)=Z(eMe)ej(e)=Z(M)ej(e)=Z(M)q$, where $Z(N)$ denotes the center of a von
Neumann algebra $N$. For every $z\in Z(qMq)$, writing $z=xq$ for some $x\in Z(M)$ one has
$$
(qJ)z(qJ)=q(JxJ)q=qx^*q=z^*.
$$

(c)\enspace
For every $\xi\in\cP$ one has $qJ(q\xi)=qJ\xi=q\xi$.

(d)\enspace
For every $x\in M$ one has
\begin{align*}
(qxq)(qJ)(qxq)(qJ)q\cP&=qxqJxJq\cP=qxej(e)j(x)ej(e)\cP \\
&=qexej(exe)\cP\subset q\cP.
\end{align*}
\end{proof}

Here, let us introduce a few simple notations for later use. For each $\xi\in\cH$ we denote
by $\omega_\xi$ ($\in M_*^+$) the vector functional $x\mapsto\<\xi,x\xi\>$ on $M$. We write
$e(\xi)$ for the projection onto $\overline{M'\xi}$. Note that $e(\xi)\in M$ and
$e(\xi)=s(\omega_\xi)$, the support projection of $\omega_\xi$ (an exercise).

\begin{lemma}\label{L-3.9}
Let $(M,\cH,J,\cP)$ be a standard form.
\begin{itemize}
\item[\rm(1)] If $\xi\in\cP$, then $\xi$ is cyclic for $M$ if and only if $\xi$ is separating
for $M$.
\item[\rm(2)] If $M$ is $\sigma$-finite, then there exists a $\xi\in\cP$ that is cyclic and
separating for $M$.
\end{itemize}
\end{lemma}

\begin{proof}
(1)\enspace
Let $\xi\in\cP$. If $\xi$ is cyclic for $M$, then $\xi=J\xi$ is cyclic for $JMJ=M'$, so $\xi$
is separating for $M$. The converse is similar.

(2)\enspace
Take a maximal family $(\xi_k)_{k\in K}$ of non-zero vectors in $\cP$ such that
$e(\xi_k)$ ($k\in K$) are mutually orthogonal. Assume that $e:=1-\sum_{k\in K}e(\xi_k)\ne0$.
By Proposition \ref{P-3.8}, $q:=ej(e)\ne0$ and $q\cP$ is a self-dual cone in $q\cH$, so
$q\cP\ne\{0\}$. Since $q\cP\subset\cP$, one can choose a $\xi\in\cP$ such that $\xi=q\xi\ne0$,
so $\xi=e\xi$. Since $eM'\xi=M'e\xi=M'\xi$, one has $e(\xi)\le e$, which contradicts the
maximality of $(\xi_k)_{k\in K}$. Therefore, $\sum_{k\in K}e(\xi_k)=1$. Since $M$ is
$\sigma$-finite, the index set $K$ is at most countable. So we may assume that
$\sum_{k\in K}\|\xi_k\|^2<+\infty$. Now, set $\xi:=\sum_{k\in K}\xi_k\in\cP$. Since
$M'\xi_k\perp M'\xi_j$ ($k\ne j$) and $M'=JMJ$, we have $M\xi_k\perp M\xi_j$ ($k\ne j$).
Hence it follows that $\omega_\xi=\sum_{k\in K}\omega_{\xi_k}$ so that
$$
e(\xi)=s(\omega_\xi)=\bigvee_{k\in K}s(\omega_{\xi_k})
=\bigvee_{k\in K}e(\xi_k)=\sum_{k\in K}e(\xi_k)=1,
$$
which means that $\xi$ is cyclic for $M'$, or equivalently, $\xi$ is separating for $M$.
By (1), $\xi$ is cyclic and separating for $M$.
\end{proof}

The next proposition says the universality of $(J,\cP)$ for the choice of a cyclic and
separating vector $\xi$ in $\cP$.

\begin{prop}\label{P-3.10}
Let $(M,\cH,J,\cP)$ be a standard form. If $\xi\in\cP$ is cyclic and separating for $M$, then
$$
J_\xi=J,\qquad\cP_\xi=\cP,
$$
where $J_\xi$ is the modular conjugation and $\cP_\xi$ is the natural positive cone
associated with $(M,\xi)$.
\end{prop}

\begin{proof}
Let $S_\xi$ be the closure of $x\xi\mapsto x^*\xi$ for $x\in M$, and $F_\xi$ is the closure
of $x'\xi\mapsto x'^*\xi$ for $x'\in M'$; then $F_\xi=S_\xi^*$ (see Remark \ref{R-2.9}).
For every $x\in M$ one has
$$
JF_\xi Jx\xi=JF_\xi(JxJ)\xi=J(JxJ)^*\xi=x^*\xi=S_\xi x\xi,
$$
so $S_\xi\subset JF_\xi J$, and $F_\xi\subset JS_\xi J$ by a symmetric argument. Hence
$JS_\xi=F_\xi J=(JS_\xi)^*$, so $JS_\xi$ is self-adjoint. Moreover, for every $x\in M$ one
has
$$
\<x\xi,JS_\xi x\xi\>=\<x\xi,Jx^*\xi\>=\<\xi,x^*j(x^*)\xi\>\ge0,
$$
since $\xi,x^*j(x^*)\in\cP$ by (d). Since $M\xi$ is a core of $JS_\xi$, it follows that
$JS_\xi$ is a positive self-adjoint operator. Take the polar decomposition
$S_\xi=J_\xi\Delta_\xi^{1/2}$. Since $J(JS_\xi)=J_\xi\Delta_\xi^{1/2}$, it follows from the
uniqueness of the polar decomposition that $J=J_\xi$. From the definition of $\cP_\xi$ as in
\eqref{F-3.1} and $J_\xi=J$, we have
$$
\cP_\xi=\overline{\{xJx\xi:x\in M\}}\subset\cP
$$
due to (d). Thanks to the self-duality of $\cP_\xi$ (see Theorem \ref{T-3.2}\,(vi)) and $\cP$,
we hence have $\cP_\xi\supset\cP$ as well, so $\cP_\xi=\cP$.
\end{proof}

In later sections we will sometimes consider the tensor product $M\otimes\bM_2(\bC)=\bM_2(M)$
of $M$ with the $2\times2$ matrix algebra $\bM_2=\bM_2(\bC)$. The next example gives a
description of the standard form of $\bM_2(M)$.

\begin{example}\label{E-3.11}\rm
Let $(M,\cH,J,\cP)$ be a standard form. We write $M^{(2)}$ for the tensor product
$M\otimes\bM_2$ of $M$ and $\bM_2$. Choose a cyclic and separating vector $\Omega\in\cP$ and
set $\omega(x):=\<\Omega,x\Omega\>$, $x\in M$. Consider a faithful
$\omega^{(2)}\in(M^{(2)})_*^+$ defined by
$\omega^{(2)}\Biggl(\begin{bmatrix}x_{11}&x_{12}\\x_{21}&x_{22}\end{bmatrix}\Biggr)
;=\omega(x_{11})+\omega(x_{22})$. Then the GNS cyclic representation
$\{\pi^{(2)},\cH^{(2)},\Omega^{(2)}\}$ of $M^{(2)}$ with respect to $\omega^{(2)}$ is given as
follows:
\begin{align}\label{F-3.3}
\cH^{(2)}=\cH\oplus\cH\oplus\cH\oplus\cH
=\Biggl\{\begin{bmatrix}\xi_{11}&\xi_{12}\\\xi_{21}&\xi_{22}\end{bmatrix}:\xi_{ij}
\in\cH\Biggr\},\qquad
\Omega^{(2)}=\begin{bmatrix}\Omega&0\\0&\Omega\end{bmatrix},
\end{align}
and $\pi^{(2)}\Biggl(\begin{bmatrix}x_{11}&x_{12}\\x_{21}&x_{22}\end{bmatrix}\Biggr)$
acts like $2\times2$ matrix product as
$\begin{bmatrix}x_{11}&x_{12}\\x_{21}&x_{22}\end{bmatrix}
\begin{bmatrix}\xi_{11}&\xi_{12}\\\xi_{21}&\xi_{22}\end{bmatrix}$,
whose $4\times4$ representation is
\begin{align}\label{F-3.4}
\begin{bmatrix}x_{11}&0&x_{12}&0\\0&x_{11}&0&x_{12}\\x_{21}&0&x_{22}&0\\
0&x_{22}&0&x_{22}\end{bmatrix}
\begin{bmatrix}\xi_{11}\\\xi_{12}\\\xi_{21}\\\xi_{22}\end{bmatrix}\quad\mbox{for}\quad
\begin{bmatrix}\xi_{11}\\\xi_{12}\\\xi_{21}\\\xi_{22}\end{bmatrix}\in\cH^{(2)}.
\end{align}
Write $S^{(2)}$ and $\Delta^{(2)}$ for $(M^{(2)},\Omega^{(2)})$ as well as
$S$ and $\Delta$ for $(M,\Omega)$, see Sec.~2.1. Since
$$
S^{(2)}\Biggl(\begin{bmatrix}x_{11}&x_{12}\\x_{21}&x_{22}\end{bmatrix}\Omega^{(2)}\Biggr)
=\begin{bmatrix}x_{11}^*\Omega&x_{21}^*\Omega\\
x_{12}^*\Omega&x_{22}^*\Omega\end{bmatrix},
$$
one can write
\begin{align}\label{F-3.5}
S^{(2)}=\begin{bmatrix}S&0&0&0\\0&0&S&0\\0&S&0&0\\0&0&0&S\end{bmatrix}\quad
\mbox{and}\quad\Delta^{(2)}=\begin{bmatrix}\Delta&0&0&0\\0&\Delta&0&0\\
0&0&\Delta&0\\0&0&0&\Delta\end{bmatrix}.
\end{align}
From this with $S=J\Delta^{1/2}$ one has the polar decomposition
$S^{(2)}=J^{(2)}(\Delta^{(2)})^{1/2}$ where
\begin{align}\label{F-3.6}
J^{(2)}=\begin{bmatrix}J&0&0&0\\0&0&J&0\\0&J&0&0\\0&0&0&J\end{bmatrix}.
\end{align}
Therefore, the standard form of $M^{(2)}$ is given as $(M^{(2)},\cH^{(2)},J^{(2)},\cP^{(2)})$
with identifications \eqref{F-3.3}, \eqref{F-3.4} and \eqref{F-3.6}. Moreover, by
\eqref{F-3.1} one has $\cP^{(2)}=\overline{\bigl\{xJ^{(2)}x\Omega^{(2)}:x\in M^{(2)}\bigr\}}$.
In particular, restricting to $x=\begin{bmatrix}x_1&0\\0&x_2\end{bmatrix}$ one has
$\cP^{(2)}\supset\Biggl\{\begin{bmatrix}\xi&0\\0&\eta\end{bmatrix}:\xi,\eta\in\cP\Biggr\}$.
\end{example}

\subsection{Uniqueness theorem}

The following is the most important property of the standard form, establishing the relation
between $\cP$ and $M_*^+$.

\begin{thm}\label{T-3.12}
Let $(M,\cH,J,\cP)$ be a standard form. For every $\ffi\in M_*^+$ there exists a $\xi\in\cP$
such that $\ffi(x)=\<\xi,x\xi\>$ for all $x\in M$. Furthermore, the following estimates hold:
$$
\|\xi-\eta\|^2\le\|\omega_\xi-\omega_\eta\|\le\|\xi-\eta\|\,\|\xi+\eta\|,
\qquad\xi,\eta\in\cP.
$$
Consequently, the map $\xi\mapsto\omega_\xi$ is a homeomorphism from $\cP$ onto $M_*^+$ when
$\cP$ and $M_*^+$ are equipped with the norm topology.
\end{thm}

An essential property of the standard form is the universality (uniqueness) in the following
strict sense:

\begin{thm}\label{T-3.13}
Let $(M,\cH,J,\cP)$ and $(M_1,\cH_1,J_1,\cP_1)$ be standard forms of von Neumann algebras of
$M$ and $M_1$, respectively. If $\Phi:M\to M_1$ is a *-isomorphism, then there exists a
unique unitary $U:\cH\to\cH_1$ such that
\begin{itemize}
\item[\rm(1)] $\Phi(x)=UxU^*$ for all $x\in M$,
\item[\rm(2)] $J_1=UJU^*$,
\item[\rm(3)] $\cP_1=U\cP$.
\end{itemize}
\end{thm}

The above theorems were proved in \cite{Haa0,Haa1} for general von Neumann algebras, but
below we assume that von Neumann algebras are $\sigma$-finite. We first prove Theorem
\ref{T-3.13}, assuming that Theorem \ref{T-3.12} holds true. The proof of the latter will
be presented later on.

\begin{proof}[Proof of Theorem \ref{T-3.13}]
First, we prove the uniqueness of $U$. Assume that $U,V:\cH\to\cH_1$ are unitaries
satisfying (1)--(3). For every $\xi\in\cP$, by (3) and (1) we have $U\xi,V\xi\in\cP_1$ and
$$
\omega_{U\xi}(\Phi(x))=\<\xi,U^*\Phi(x)U\xi\>=\<\xi,x\xi\>
=\<\xi,V^*\Phi(x)V\xi\>=\omega_{V\xi}(\Phi(x)),\qquad x\in M.
$$
Since $\eta\in\cP_1\mapsto\omega_\eta\in(M_1)_*^+$ is injective by Theorem \ref{T-3.12},
$U\xi=V\xi$ for all $\xi\in\cP$. Hence $U=V$ by Proposition \ref{P-3.7}\,(3). (Note that
condition (2) is unnecessary for the uniqueness of $U$.)

Next, we prove the existence. By Lemma \ref{L-3.9}\,(2) there exists a cyclic and separating
vector $\xi_0\in\cP$ for $M$. By Theorem \ref{T-3.12} there exists an $\eta_0\in\cP_1$ such
that $\omega_{\eta_0}=\omega_{\xi_0}\circ\Phi^{-1}$, i.e.,
$$
\omega_{\eta_0}(\Phi(x))=\omega_{\xi_0}(x),\qquad x\in M.
$$
Then $\eta_0$ is separating for $M_1$, so it is also cyclic by Lemma \ref{L-3.9}\,(1). Note
that
$$
\|\Phi(x)\eta_0\|^2=\omega_{\eta_0}(\Phi(x^*x))=\omega_{\xi_0}(x^*x)=\|x\xi_0\|^2,
\qquad x\in M.
$$
Hence an isometry $U:M\xi_0\to M_1\eta_0$ is defined by $Ux\xi_0:=\Phi(x)\eta_0$ for $x\in M$,
which can extend by continuity to a unitary $U:\cH\to\cH_1$. We now show that $U$ satisfies
(1)--(3).

(1)\enspace
Let $\zeta\in M_1\eta_0$, so $\zeta=\Phi(y)\eta_0=Uy\xi_0$ for some $y\in M$. For every
$x\in M$,
$$
\Phi(x)\zeta=\Phi(xy)\eta_0=Uxy\xi_0=UxU^*\zeta.
$$
Since $\overline{M_1\eta_0}=\cH_1$, one has $\Phi(x)=UxU^*$.

(2)\enspace
Let $S_{\xi_0}$ (resp., $S_{\eta_0}$) be the closure of $S_{\xi_0}^0x\xi_0=x^*\xi_0$ for
$x\in M$ (resp., $S_{\eta_0}^0y\eta_0=y^*\eta_0$ for $y\in M_1$). Since
$$
US_{\xi_0}^0U^*\Phi(x)\eta_0=US_{\xi_0}^0x\xi_0=Ux^*\xi_0
=\Phi(x)^*\eta_0=S_{\eta_0}^0\Phi(x)\eta_0,
$$
one has $S_{\eta_0}^0=US_{\xi_0}^0U^*$ and hence $S_{\eta_0}=US_{\xi_0}U^*$. Taking the polar
decompositions $S_{\xi_0}=J_{\xi_0}\Delta_{\xi_0}^{1/2}$ and
$S_{\eta_0}=J_{\eta_0}\Delta_{\eta_0}^{1/2}$, one has $J_{\eta_0}=UJ_{\xi_0}U^*$. Since
$J=J_{\xi_0}$ and $J_1=J_{\eta_0}$ by Proposition \ref{P-3.10}, $J_1=UJU^*$ follows.

(3)\enspace
By Proposition \ref{P-3.10} one has
\begin{align*}
\cP_1&=\cP_{\eta_0}=\overline{\{yJ_1y\eta_0:y\in M_1\}}
=\overline{\{(UxU^*)(UJU^*)(UxU^*)\eta_0:x\in M\}} \\
&=U\overline{\{xJx\xi_0:x\in M\}}=U\cP_{\xi_0}=U\cP.
\end{align*}
\end{proof}

In the rest of the section we give the proof of Theorem \ref{T-3.12}, which we divide into
several lemmas.

\begin{lemma}\label{L-3.14}
The linear map $\Phi:M_\sa\to\cH_\sa$, where $\cH_\sa:=\cP-\cP$, defined by
$$
\Phi(x):=\Delta^{1/4}x\Omega
$$
is an order isomorphism from $M_\sa$ onto the set
$$
\cL:=\{\xi\in\cH_\sa:-\alpha\Omega\le\xi\le\alpha\Omega\ \mbox{for some $\alpha>0$}\},
$$
where the orders on $M_\sa$ and $\cH_\sa$ are induced by the cones $M_+$ and $\cP$,
respectively.
\end{lemma}

\begin{proof}
By Theorem \ref{T-3.2}\,(i), if $x\in M_+$, then $\Delta^{1/4}x\Omega\in\cP$. Conversely, if
$x\in M_\sa$ and $\Delta^{1/4}x\Omega\in\cP$, then for any $x'\in M'$,
$$
\<x'\Omega,xx'\Omega\>=\<|x'|^2\Omega,x\Omega\>
=\<\Delta^{-1/4}|x'|^2\Omega,\Delta^{1/4}x\Omega\>\ge0,
$$
hence $x\ge0$. Since $\Phi(1)=\Omega$ and $\Phi$ is clearly injective, it follows that
$\Phi:M_\sa\to\Phi(M_\sa)\subset\cL$ is an order isomorphism. Hence it remains to show that
$\Phi(M_\sa)=\cL$. First we show that $\Phi$ is continuous with respect to the
$\sigma(M,M_*)$-topology and the weak topology on $\cH$. For any $x\in M$ note that
$(1+\Delta^{1/2})x\Omega=x\Omega+Jx^*\Omega$ so that
$$
\Delta^{1/4}x\Omega=\Delta^{1/4}(1+\Delta^{1/2})^{-1}(x\Omega+Jx^*\Omega)
=(\Delta^{1/4}+\Delta^{-1/4})^{-1}(x\Omega+Jx^*\Omega).
$$
Since $(\Delta^{1/4}+\Delta^{-1/4})^{-1}$ is bounded and
$$
\<\eta,\Delta^{1/4}x\Omega\>=\<(\Delta^{1/4}+\Delta^{-1/4})^{-1}\eta,x\Omega+Jx^*\Omega\>,
\qquad\eta\in\cH,
$$
the above stated continuity of $\Phi$ follows.

Now, let $\xi\in\cL$. We may assume that $0\le\xi\le\Omega$. Put $\xi_n:=f_n(\log\Delta)\xi$,
where $f_n(t):=e^{-t^2/2n}$. Since $f_n(\log\Delta)\cP\subset\cP$ by Theorem \ref{T-3.2}\,(v),
one has
$$
0\le\xi_n=f_n(\log\Delta)\xi\le f_n(\log\Delta)\Omega=\Omega.
$$
Furthermore, one has $\xi_n\to\xi$ and $\xi_n\in\bigcap_{z\in\bC}\cD(\Delta^z)$ as in the proof
of Theorem \ref{T-3.2}\,(vi). For any $\eta\in\cP^\flat$, since $\Delta^{-1/4}\eta\in\cP$ by
Theorem \ref{T-3.2}\,(i), $\<\eta,\Delta^{-1/4}\xi_n\>=\<\Delta^{-1/4}\eta,\xi_n\>\ge0$. Hence
$\Delta^{-1/4}\xi_n\in\cP^\sharp$ by Lemma \ref{L-3.3}. Similarly,
$\Omega-\Delta^{-1/4}\xi_n=\Delta^{-1/4}(\Omega-\xi_n)\in\cP^\sharp$. Let
$\zeta_n:=\Delta^{-1/4}\xi_n$ and, as in the proof of Lemma \ref{L-3.3}, define an operator
$T_n:M'\Omega\to\cH$ by $T_nx'\Omega:=x'\zeta_n$ for $x'\in M'$. Since
\begin{align*}
&\<x'\Omega,T_nx'\Omega\>=\<x'^*x'\Omega,\zeta_n\>\ge0, \\
&\<x'\Omega,x'\Omega\>-\<x'\Omega,T_nx'\Omega\>=\<x'^*x'\Omega,\Omega-\zeta_n\>\ge0,
\end{align*}
one has $0\le\<x'\Omega,T_nx'\Omega\>\le\<x'\Omega,x'\Omega\>$ for all $x'\in M'$. Hence the
Friedrichs extension of $T_n$ is indeed an $x_n\in M$ with $0\le x_n\le1$, so
$\xi_n=\Delta^{1/4}x_n\Omega=\Phi(x_n)$. Therefore,
$$
\{\xi_n\}\subset\cK:=\{\Phi(x):x\in M_\sa,\,0\le x\le1\}.
$$
But $\cK$ is weakly compact due to the continuity of $\Phi$ shown above. Thus
$\xi=\lim_n\xi_n\in\cK\subset\Phi(M_\sa)$.
\end{proof}

\begin{lemma}\label{L-3.15}
For every $\xi,\eta\in\cP$,
$$
\|\xi-\eta\|^2\le\|\omega_\xi-\omega_\eta\|\le\|\xi-\eta\|\,\|\xi+\eta\|.
$$
\end{lemma}

\begin{proof}
The second inequality holds for all $\xi,\eta\in\cH$ since
$$
(\omega_\xi-\omega_\eta)(x)={1\over2}
[\<\xi-\eta,x(\xi+\eta)\>+\<\xi+\eta,x(\xi-\eta)\>].
$$

Prove the first inequality. First, assume that $\xi+\eta$ is cyclic and separating for $M$,
so $\cP_{\xi+\eta}=\cP$ by Proposition \ref{P-3.10}. Since
$$
-(\xi+\eta)\le\xi-\eta\le\xi+\eta,
$$
there exists, by Lemma \ref{L-3.14} applied to $\Omega=\xi+\eta$, an $x\in M_\sa$ with
$-1\le x\le1$ such that
$$
\xi-\eta=\Delta_{\xi+\eta}^{1/4}(\xi+\eta),
$$
where $\Delta_{\xi+\eta}$ is the modular operator with respect to $\xi+\eta$. Therefore,
\begin{align*}
\|\omega_\xi-\omega_\eta\|&\ge(\omega_\xi-\omega_\eta)(x)
=\<\xi,x\xi\>-\<\eta,x\eta\> \\
&=\Re\<\xi-\eta,x(\xi+\eta)\>
=\bigl\<\xi-\eta,\Delta_{\xi+\eta}^{-1/4}(\xi-\eta)\bigr\>.
\end{align*}
Since $J\Delta_{\xi+\eta}^{1/4}=\Delta_{\xi+\eta}^{-1/4}J$ (see Proposition \ref{P-3.10}),
we have
$$
\bigl\<\xi-\eta,\Delta_{\xi+\eta}^{-1/4}(\xi-\eta)\bigr\>
=\bigl\<\xi-\eta,\Delta_{\xi+\eta}^{1/4}(\xi-\eta)\bigr\>,
$$
so
$$
\|\omega_\xi-\omega_\eta\|\ge\Bigl\<\xi-\eta,
{1\over2}\bigl(\Delta_{\xi+\eta}^{1/4}+\Delta_{\xi+\eta}^{-1/4}\bigr)(\xi-\eta)\Bigr\>
\ge\|\xi-\eta\|^2,
$$
thanks to ${1\over2}\bigl(\Delta_{\xi+\eta}^{1/4}+\Delta_{\xi+\eta}^{-1/4}\bigr)\ge1$.

Next, let $\xi,\eta\in\cP$ be arbitrary. Let $\sigma_t'$ be the modular automorphism group for
$(M',\Omega)$. From Theorem \ref{T-3.2}\,(i) and Lemma \ref{L-2.13}, one can choose sequences
$x_n',y_n'\in M_+'$ such that $x_n',y_n'$ are $\sigma_t'$-analytic and
$\|\xi_n-\xi\|\to0$, $\|\eta_n-\eta\|\to0$, where
$\xi_n:=\Delta^{-1/4}x_n'\Omega=\sigma_{-i/4}'(x_n')$ and
$\eta_n:=\Delta^{-1/4}y_n'\Omega=\sigma_{-i/4}(y_n')$.
Here, by adding $\eps_n1$ ($\eps_n\searrow0$) to $x_n'$, we may assume that
$x_n'+y_n'\ge\eps_n1$, so $\sigma_{-i/4}'(x_n'+y_n')\in M'$ are invertible. Hence
$\xi_n+\eta_n=\Delta^{-1/4}(x_n'+y_n')\Omega=\sigma_{-i/4}'(x_n'+y_n')\Omega\in\cP$ is
separating and also cyclic for $M$ by Lemma \ref{L-3.9}\,(1). Therefore, from the first part
it follows that $\|\omega_{\xi_n}-\omega_{\eta_n}\|\ge\|\xi_n-\eta_n\|^2$. Letting $n\to\infty$
gives the asserted inequality.
\end{proof}

\begin{lemma}\label{L-3.16}
The map $\xi\mapsto\omega_\xi$ is a homeomorphism from $\cP$ onto a closed subset
$\cE:=\{\omega_\xi:\xi\in\cP\}$ of $M_*^+$ with respect to the norm topologies on $\cP$ and
$M_*^+$.
\end{lemma}

\begin{proof}
It follows from Lemma \ref{L-3.15} that $\xi\in\cP\mapsto\omega_\xi\in\cE$ is a homeomorphism.
If $\{\omega_{\xi_n}\}$ is a Cauchy sequence in $M_*^+$, then
$\|\xi_m-\xi_n\|^2\le\|\omega_{\xi_m}-\omega_{\xi_n}\|\to0$ as $m,n\to\infty$, so $\{\xi_n\}$
is Cauchy in $\cP$. Hence $\xi_n\to\xi\in\cP$ and $\omega_{\xi_n}\to\omega_{\xi}$ for some
$\xi\in\cP$. Hence $\cE$ is closed in $M_*^+$.
\end{proof}

\begin{lemma}\label{L-3.17}
Let $\xi\in\cP$. If $\psi\in M_*^+$ and $\psi\le\omega_\xi$, then there exists an $\eta\in\cP$
such that $\eta\le\xi$ and $\psi(x)={1\over2}(\<\eta,x\xi\>+\<\xi,x\eta\>)$ for all $x\in M$.
\end{lemma}

\begin{proof}
First, assume that $\xi$ is cyclic and separating for $M$. There exists a $b'\in M'$ such that
$0\le b'\le1$ and $\psi(x)=\<\xi,b'x\xi\>$ for all $x\in M$. Note that
$b'\xi,(1-b')\xi\in\cP_\xi^\flat$, where $\cP_\xi^\flat$ is $\cP^\flat$ for $\Omega=\xi$.
Since $\Delta_\xi^{-1/4}\cP_\xi^\flat\subset\cP_\xi=\cP$ by Theorem \ref{T-3.2}\,(i) and
Proposition \ref{P-3.10}, we have $\zeta:=\Delta_\xi^{-1/4}b'\xi\in\cP$ and
$\xi-\zeta=\Delta_\xi^{-1/4}(1-b')\xi\in\cP$, i.e., $0\le\zeta\le\xi$. Set
$$
\eta:=2\bigl(1+\Delta_\xi^{1/2}\bigr)^{-1}b'\xi
=2\bigl(1+\Delta_\xi^{1/2}\big)^{-1}\Delta_\xi^{1/4}\zeta
=2\bigl(\Delta_\xi^{1/4}+\Delta_\xi^{-1/4}\bigr)^{-1}\zeta=f(\log\Delta_\xi)\zeta,
$$
where $f(t):=2/(e^{t/4}+e^{-t/4})=[\cosh(t/4)]^{-1}$. Note that $f(t)$ is the Fourier
transform of $4\pi[\cosh(2\pi s)]^{-1}$, so $f$ is a positive definite function on $\bR$.
Hence by Theorem \ref{T-3.2}\,(v),
$$
0\le\eta=f(\log\Delta_\xi)\zeta\le\xi=f(\log\Delta_\xi)\xi.
$$
Furthermore, since $\eta=J\eta=J_\xi\eta$, we find that
$b'\xi={1\over2}\bigl(1+\Delta_\xi^{1/2}\bigr)\eta={1\over2}(\eta+F_\xi\eta)$, see Lemma
\ref{L-2.1}\,(iii). Hence, for every $x\in M$ we have
\begin{align*}
\psi(x)&=\<b'\xi,x\xi\>={1\over2}(\<\eta,x\xi\>+\<F_\xi\eta,x\xi\>)
={1\over2}(\<\eta,x\xi\>+\<S_\xi x\xi,\eta\>) \\
&={1\over2}(\<\eta,x\xi\>+\<x^*\xi,\eta\>)={1\over2}(\<\eta,x\xi\>+\<\xi,x\eta\>).
\end{align*}

Next, let $\xi\in\cP$ be arbitrary. Let $e:=e(\xi)=s(\omega_\xi)$ and $q:=ej(e)$. Then
$\xi\in q\cP$ is cyclic and separating for $qMq$ whose standard form is
$(qMq,q\cH,qJq,q\cP)$, see Proposition \ref{P-3.8}\,(2). By Proposition \ref{P-3.8}\,(1) one
can define $\psi_q\in(qMq)_*^+$ by $\psi_q(qxq)=\psi(exe)=\psi(x)$ for $x\in M$. Since
$\psi\le\omega_\xi$ where $\omega_\xi$ is regarded as an element of $(qMq)_*^+$, it follows
from the first part of the proof that there exists an $\eta\in q\cP$ ($\subset\cP$ by (d))
such that $\eta\le\xi$ and $\psi_q(x)={1\over2}(\<\eta,x\xi\>+\<\xi,x\eta\>)$ for all
$x\in qMq$. Therefore, for every $x\in M$,
$$
\psi(x)=\psi_q(qxq)={1\over2}(\<\eta,x\xi\>+\<\xi,x\eta\>).
$$
\end{proof}

\begin{lemma}\label{L-3.18}
If $\xi_0\in\cP$, $\psi\in M_*^+$ and $\psi\le\omega_{\xi_0}$, then $\psi=\omega_\eta$ for
some $\eta\in\cP$.
\end{lemma}

\begin{proof}
Put $\psi_1:=\omega_{\xi_0}-\psi$. By Lemma \ref{L-3.17} one can find an $\eta_1\in\cP$ such
that $\eta_1\le{1\over2}\xi_0$ and $\psi_1(x)=\<\eta_1,x\xi_0\>+\<\xi_0,x\eta_1\>$ for all
$x\in M$. Set $\xi_1:=\xi_0-\eta_1\in\cP$; then
\begin{align*}
\omega_{\xi_1}(x)
&=\<\xi_0,x\xi_0\>-\<\xi_0,x\eta_1\>-\<\eta_1,x\xi_0\>+\<\eta_1,x\eta_1\> \\
&=\omega_{\xi_0}(x)-\psi_1(x)+\omega_{\eta_1}(x)=\psi(x)+\omega_{\eta_1}(x),
\qquad x\in M,
\end{align*}
so $\omega_{\xi_1}-\psi=\omega_{\eta_1}$ and
$$
\|\psi_1\|=\psi_1(1)=2\Re\<\eta_1,\xi_0\>=2\<\eta_1,\xi_0\>.
$$
Therefore,
$$
\|\omega_{\xi_1}-\psi\|=\|\eta_1\|^2\le\Bigl\<\eta_1,{1\over2}\xi_0\Bigr\>
={1\over4}\|\psi_1\|={1\over4}\|\omega_{\xi_0}-\psi\|.
$$
Apply the above argument to $\xi_1$ in place of $\xi_0$ to obtain $\xi_2\in\cP$ such that
$\|\omega_{\xi_2}-\psi\|\le{1\over4}\|\omega_{\xi_1}-\psi\|$. Repeating the argument we find
a sequence $\xi_n\in\cP$ such that $\|\omega_{\xi_n}-\psi\|\to0$. By Lemma \ref{L-3.16},
$\psi=\omega_\eta$ with $\eta:=\lim_n\xi_n\in\cP$.
\end{proof}

\begin{lemma}\label{L-3.19}
The set $\{\psi\in M_*^+:\psi\le\alpha\omega_\Omega\ \mbox{for some $\alpha>0$}\}$ is
norm-dense in $M_*^+$.
\end{lemma}

\begin{proof}
Recall that each $\psi\in M_*^+$ has the form $\psi(x)=\sum_{n=1}^\infty\<\xi_n,x\xi_n\>$ for
some sequence $\xi_n\in\cH$ with $\sum_n\|\xi_n\|^2<+\infty$. By approximating $\psi$ by
$\psi_m(x):=\sum_{n=1}^m\<\xi_n,x\xi_n\>$ and then $\xi_n$ by $x_n'\Omega$ with $x_n'\in M'$,
one can approximate $\psi$ in norm by $\widetilde\psi\in M_*^+$ of the form
$\widetilde\psi(x):=\sum_{n=1}^m\<x_n'\Omega,xx_n'\Omega\>$. Since
$$
\widetilde\psi(x)=\sum_{n=1}^m\<x^{1/2}\Omega,x_n'^*x_n'x^{1/2}\Omega\>
\le\Biggl(\sum_{n=1}^m\|x_n'\|^2\Biggr)\<\Omega,x\Omega\>,\qquad x\in M_+,
$$
it follows that $\widetilde\psi\le\alpha\omega_\Omega$ for some $\alpha>0$.
\end{proof}

\begin{proof}[End of Proof of Theorem \ref{T-3.12}]
By Lemma \ref{L-3.18} applied to $\xi_0=\Omega$,
$$
\cE=\{\omega_\xi:\xi\in\cP\}\supset
\{\psi\in M_*^+:\psi\le\alpha\omega_\Omega\ \mbox{for some $\alpha>0$}\}.
$$
By Lemmas \ref{L-3.16} and \ref{L-3.19} we have $\cE=\overline{\cE}=M_*^+$, which shows the
first assertion of the theorem. The remaining are Lemmas \ref{L-3.15} and \ref{L-3.16}.
\end{proof}

\section{$\tau$-measurable operators}

The non-commutative integration theory was created by Segal \cite{Se}, where measurable
operators affiliated with a von Neumann algebra were discussed. Later in \cite{Ne}, Nelson
introduced the notion of $\tau$-measurable operators in a stricter connection with a
given trace $\tau$, whose notion is quite tractable to develop the non-commutative integration,
in particular, non-commutative $L^p$-spaces associated with $\tau$. This section is a study
of the theory of $\tau$-measurable operators, mainly based on \cite[Chap.~I]{Te} whose
exposition is considerably more readable than that in \cite{Ne}. Throughout the section we
assume that $M$ is a semifinite von Neumann algebra on a Hilbert space $\cH$ and $\tau$ is a
faithful semifinite normal trace on $M$.

\subsection{$\tau$-measurable operators}

Let $M$ be a von Neumann algebra on a Hilbert space $\cH$. Let $a:\cD(a)\to\cH$ be a linear
operator whose domain $\cD(a)$ is a linear subspace of $\cH$. We say that $a$ is
\emph{affiliated with} $M$, denoted by $a\,\eta M$, if $x'a\subset ax'$ for all $x'\in M'$, or
equivalently, if $u'au'^*=a$ for all unitaries $u'\in M'$. The following facts are easy to
verify (exercises):
\begin{itemize}
\item[\rm(a)] If $a,b$ are linear operators affiliated with $M$, then $a+b$ with
$\cD(a+b)=\cD(a)\cap\cD(b)$ and $ab$ with $\cD(ab)=\{\xi\in\cD(b):b\xi\in D(a)\}$ are
affiliated with $M$.
\item[\rm(b)] If $a$ is densely defined and $a\,\eta M$, then $a^*\,\eta M$.
\item[\rm(c)] If $a$ is closable and $a\,\eta M$, then $\overline a\,\eta M$.
\item[\rm(d)] Assume that $a$ is densely defined and closed, so we have the polar decomposition
$a=w|a|$ and the spectral decomposition $|a|=\int_0^\infty\lambda\,e_\lambda$. Then $a\,\eta M$
if and only if $w,e_\lambda\in M$ for all $\lambda\ge0$. 
\end{itemize}

Hereafter, let $M$ be a semifinite von Neumann algebra with a faithful semifinite normal trace
$\tau$, that is, $\tau$ is a faithful semifinite normal weight on $M$ (see Sec.~1.3) satisfying
the trace condition that $\tau(x^*x)=\tau(xx^*)$ for all $x\in M$.

For each $\eps,\delta>0$ define
$$
\cO(\eps,\delta):=\{a\,\eta M:e\cH\subset\cD(a),\ \|ae\|\le\eps\ \mbox{and}
\ \tau(e^\perp)\le\delta\ \mbox{for some $e\in\Proj(M)$}\},
$$
where $\Proj(M)$ is the set of projections in $M$.

\begin{lemma}\label{L-4.1}
For any $\eps_1,\eps_2,\delta_1,\delta_2>0$,
\begin{itemize}
\item[\rm(1)] $\cO(\eps_1,\delta_1)+\cO(\eps_2,\delta_2)\subset
\cO(\eps_1+\eps_2,\delta_1+\delta_2)$,
\item[\rm(2)] $\cO(\eps_1,\delta_1)\cO(\eps_2,\delta_2)\subset
\cO(\eps_1\eps_2,\delta_1+\delta_2)$.
\end{itemize}
\end{lemma}

\begin{proof}
Let $a\in\cO(\eps_1,\delta_1)$ and $b\in\cO(\eps_2,\delta_2)$, so there are $e,f\in\Proj(M)$
such that
\begin{align*}
&e\cH\subset\cD(a),\quad\|ae\|\le\eps_1,\quad\tau(e^\perp)\le\delta_1, \\
&f\cH\subset\cD(b),\quad\|bf\|\le\eps_1,\quad\tau(f^\perp)\le\delta_1.
\end{align*}

(1)\enspace
Letting $p:=e\wedge f\in\Proj(M)$ one has
\begin{align*}
&p\cH\subset e\cH\cap f\cH\subset\cD(a)\cap\cD(b)=\cD(a+b), \\
&\|(a+b)p\|\le\|ap\|+\|bp\|\le\|ae\|+\|bf\|\le\eps_1+\eps_2, \\
&\tau(p^\perp)=\tau(e^\perp\vee f^\perp)\le\tau(e^\perp)+\tau(f^\perp)\le\delta_1+\delta_2.
\end{align*}
Hence $a+b\in\cO(\eps_1+\eps_2,\delta_1+\delta_2)$.

(2)\enspace
For every $x'\in M'$ one has
$$
x'e^\perp bf=e^\perp x'bf\subset e^\perp bx'f=e^\perp bfx',
$$
which implies that $e^\perp bf\in M$. Let $g$ be the projection onto the kernel of
$e^\perp bf$, so $g\in\Proj(M)$ and $bfg=ebfg$. If $\xi\in g\cH$, then $bf\xi\in e\cH$ and
hence $\xi\in\cD(abf)$, so $f\xi\in\cD(ab)$. Let $q:=f\wedge g\in\Proj(M)$. Then
$q\cH\subset\cD(ab)$ and
$$
abq=abfgq=aebfgq=aebfq
$$
so that $\|abq\|\le\|ae\|\,\|bf\|\le\eps_1\eps_2$. Since
\begin{align*}
g^\perp&=\mbox{[the projection onto the range closure of $(e^\perp bf)^*$]} \\
&\sim\mbox{[the projection onto the range closure of $e^\perp bf$]}\le e^\perp,
\end{align*}
one has
$$
\tau(q^\perp)=\tau(f^\perp\vee g^\perp)\le\tau(f^\perp)+\tau(g^\perp)\le\delta_2+\delta_1.
$$
Hence $ab\in\cO(\eps_1\eps_2,\delta_1+\delta_2)$.
\end{proof}

\begin{definition}\label{D-4.2}\rm
A linear subspace $\cL$ of $\cH$ is said to be \emph{$\tau$-dense} if, for any $\delta>0$,
there exists an $e\in\Proj(M)$ such that $e\cH\subset\cL$ and $\tau(e^\perp)\le\delta$.
\end{definition}

\begin{lemma}\label{L-4.3}
A $\tau$-dense linear subspace $\cL$ of $\cH$ is dense in $\cH$.
\end{lemma}

\begin{proof}
Let $\cL$ be $\tau$-dense. For $\delta=1/2^k$ ($k\in\bN$) choose a $q_k\in\Proj(M)$ such that
$q_k\cH\subset\cL$ and $\tau(q_k^\perp)\le1/2^k$. Let
$e_n:=\bigwedge_{k=n}^\infty q_k\in\Proj(M)$. Then $e_n\nearrow$ and $e_n\cH\subset\cL$.
Since $e_n^\perp=\bigvee_{k=n}^\infty q_k^\perp$, one has
$$
\tau(e_n^\perp)=\lim_{m\to\infty}\tau\Biggl(\bigvee_{k=n}^mq_k^\perp\Biggr)
\le\sum_{k=n}^\infty\tau(q_k^\perp)\le{1\over2^{n-1}}\ \longrightarrow\ 0\quad(n\to\infty).
$$
Hence $e_n^\perp\searrow0$, i.e., $e_n\nearrow1$. This implies that $\cL$ is dense in $\cH$.
\end{proof}

\begin{lemma}\label{L-4.4}
Let $a$ be a densely-defined closed operator with $a\,\eta M$. Let $a=w|a|$ and
$|a|=\int_0^\infty\lambda\,de_\lambda$ be as in (d) above. Then for any $\eps,\delta>0$,
$$
a\in\cO(\eps,\delta)\ \iff\ \tau(e_\eps^\perp)\le\delta.
$$
\end{lemma}

\begin{proof}
Assume that $a\in\cO(\eps,\delta)$, so there is an $e\in\Proj(M)$ such that
$\|\,|a|e\|=\|ae\|\le\eps$ and $\tau(e^\perp)\le\delta$. For any $t>\eps$ and
$\xi\in e_t^\perp\cH$,
$$
\|\,|a|\xi\|^2=\int_{(t,\infty)}\lambda^2\,d\|e_\lambda\xi\|^2\ge t^2\|\xi\|^2
$$
so that $e\wedge e_t^\perp=0$. So one has
$$
e_t^\perp=e_t^\perp-e\wedge e_t^\perp\sim e\vee e_t^\perp-e\le e^\perp,
$$
which implies that $\tau(e_t^\perp)\le\tau(e^\perp)\le\delta$. Since
$e_t^\perp\nearrow e_\eps^\perp$ as $t\searrow\eps$, $\tau(e_\eps^\perp)\le\delta$ follows.
The converse is obvious by taking $e=e_\eps$.
\end{proof}

\begin{lemma}\label{L-4.5}
Let $a$ and $b$ be densely-defined closed operator with $a,b\,\,\eta M$. If there exists a
$\tau$-dense linear subspace $\cL$ of $\cH$ such that $\cL\subset\cD(a)\cap\cD(b)$ and
$a|_\cL=b|_\cL$, then $a=b$.
\end{lemma}

\begin{proof}
Consider the von Neumann algebra
$$
M^{(2)}:=M\otimes\bM_2(\bC)
=\biggl\{\begin{bmatrix}x_{11}&x_{12}\\x_{21}&x_{22}\end{bmatrix}:x_{ij}\in M\biggr\}
$$
on $\cH^{(2)}:=\cH\oplus\cH$ with a faithful semifinite normal trace
$$
\tau^{(2)}\biggl(\begin{bmatrix}x_{11}&x_{12}\\x_{21}&x_{22}\end{bmatrix}\biggr)
:=\tau(x_{11})+\tau(x_{22})\quad
\mbox{for}\ \begin{bmatrix}x_{11}&x_{12}\\x_{21}&x_{22}\end{bmatrix}\in(M^{(2)})_+.
$$
Let $p_a,p_b$ be the projections from $\cH^{(2)}$ onto the graphs (closed subspaces)
$G(a),G(b)$ of $a,b$, respectively. Note that
$$
(M^{(2)})'=\biggl\{\begin{bmatrix}x'&0\\0&x'\end{bmatrix}:x'\in M'\biggr\}.
$$
For every $x'\in M'$ and $\xi\in\cD(a)$ one has
$(x'\oplus x')(\xi\oplus a\xi)=x'\xi\oplus ax'\xi\in G(a)$, so
$(x'\oplus x')p_a=p_a(x'\oplus x')p_a$. This implies that $p_a\in(M^{(2)})''=M^{(2)}$.
Similarly, $p_b\in M^{(2)}$. For any $\delta>0$ there is an $e\in\Proj(M)$ such that
$e\cH\subset\cL$ and $\tau(e^\perp)\le\delta/2$. Set
$e^{(2)}:=\begin{bmatrix}e&0\\0&e\end{bmatrix}\in\Proj(M^{(2)})$;
then $\tau^{(2)}(e^{(2)})\le\delta$. Since $a|_\cL=b|_\cL$, one has
\begin{align*}
G(a)\cap e^{(2)}\cH^{(2)}&=\{\xi\oplus a\xi:\xi\in e\cH,\,a\xi\in e\cH\} \\
&=\{\xi\oplus b\xi:\xi\in e\cH,\,b\xi\in e\cH\}=G(b)\cap e^{(2)}\cH^{(2)},
\end{align*}
which means that $p_a\wedge e^{(2)}=p_b\wedge e^{(2)}$. Let $p_0:=p_a-p_a\wedge p_b$. Since
$p_a\wedge e^{(2)}=p_a\wedge p_b\wedge e^{(2)}$, it follows that $p_0\wedge e^{(2)}=0$ so that
$$
p_0=p_0-p_0\wedge e^{(2)}\sim p_0\vee e^{(2)}-e^{(2)}\le e^{(2)\perp}.
$$
Therefore, $\tau^{(2)}(p_0)\le\tau^{(2)}(e^{(2)\perp})\le\delta$. Since $\delta>0$ is arbitrary,
$\tau^{(2)}(p_0)=0$ so $p_0=0$, i.e., $p_a=p_a\wedge p_b$. Similarly, $p_b=p_a\wedge p_b$, so
$p_a=p_b$, i.e., $G(a)=G(b)$ or $a=b$.
\end{proof}

\begin{definition}\label{D-4.6}\rm
Let $a$ be a densely-defined closed operator such that $a\,\eta M$. We say that $a$ is
\emph{$\tau$-measurable} if, for any $\delta>0$, there exists an $e\in\Proj(M)$ such that
$e\cH\subset\cD(a)$ and $\tau(e^\perp)\le\delta$. Since $e\cH\subset\cD(a)$ $\iff$
$\|ae\|<+\infty$ due to the closed graph theorem, the condition is equivalent to that
for any $\delta>0$ there is an $\eps>0$ such that $a\in\cO(\eps,\delta)$. We denote by
$\widetilde M$ the set of such $\tau$-measurable operators.
\end{definition}

\begin{prop}\label{P-4.7}
Let $a$ be a densely-defined closed operator affiliated with $M$ with $a=w|a|$ and
$|a|=\int_0^\infty\lambda\,de_\lambda$ as above. Then the following conditions are equivalent:
\begin{itemize}
\item[\rm(i)] $a\in\widetilde M$;
\item[\rm(ii)] $|a|\in\widetilde M$;
\item[\rm(iii)] $\tau(e_\lambda^\perp)\to0$ as $\lambda\to\infty$;
\item[\rm(iv)] $\tau(e_\lambda^\perp)<+\infty$ for some $\lambda>0$.
\end{itemize}
\end{prop}

\begin{proof}
(i)\,$\iff$\,(ii) is obvious and (iii)\,$\implies$\,(iv) is trivial. (ii)\,$\iff$\,(iii)
immediately follows from Definition \ref{D-4.6} and Lemma \ref{L-4.4}.

(iv)\,$\implies$\,(iii).\enspace
Assume that $\tau(e_\lambda^\perp)<+\infty$ for some $\lambda>0$. Since
$e_t-e_\lambda\nearrow1-e_\lambda$ as $\lambda<t\to\infty$, one has
$\tau(e_t-e_\lambda)\nearrow\tau(1-e_\lambda)<+\infty$ so that
$$
\tau(e_t^\perp)=\tau((1-e_\lambda)-(e_t-e_\lambda))
=\tau(1-e_\lambda)-\tau(e_t-e_\lambda)\ \longrightarrow\ 0
\quad\mbox{as $t\to\infty$}.
$$
Hence (iii) follows.
\end{proof}

For each $\eps,\delta>0$ define
\begin{align*}
\cN(\eps,\delta)&:=\widetilde M\cap\cO(\eps,\delta) \\
&\ =\{a\in\widetilde M:\|ae\|\le\eps\ \mbox{and}\ \tau(e^\perp)\le\delta
\ \mbox{for some $e\in\Proj(M)$}\}.
\end{align*}

\begin{lemma}\label{L-4.8}
If $a\in\widetilde M$, then $a^*\in\widetilde M$. Moreover, $a\in\cN(\eps,\delta)$ $\iff$
$a^*\in\cN(\eps,\delta)$.
\end{lemma}

\begin{proof}
Let $a\in\widetilde M$ with $a=w|a|$ and $|a|=\int_0^\infty\lambda\,de_\lambda$. One can
define a spectral resolution $\{\widehat e_\lambda\}_{\lambda\ge0}$ by
$\widehat e_\lambda^\perp:=we_\lambda^\perp w^*\in\Proj(M)$. Then
$|a^*|=w|a|w^*=\int_0^\infty\lambda\,d\widehat e_\lambda$. Since
$\tau(\widehat e_\lambda^\perp)=\tau(e_\lambda^\perp)$, $a^*\in\widetilde M$ follows from
Proposition \ref{P-4.7}. Moreover, this implies by Lemma \ref{L-4.4} that
$a\in\cO(\eps,\delta)$ $\iff$ $a^*\in\cO(\eps,\delta)$. Hence the latter assertion follows.
\end{proof}

\begin{lemma}\label{L-4.9}
If $a,b\in\widetilde M$, then $a+b$ and $ab$ are densely defined and closable, and
$\overline{a+b},\overline{ab}\in\widetilde M$. Moreover, if $a\in\cN(\eps_1,\delta_1)$ and
$b\in\cN(\eps_2,\delta_2)$, then $\overline{a+b}\in\cN(\eps_1+\eps_2,\delta_1+\delta_2)$ and
$\overline{ab}\in\cN(\eps_1\eps_2,\delta_1+\delta_2)$.
\end{lemma}

\begin{proof}
Let $a,b\in\widetilde M$. For any $\delta_1,\delta_2>0$ there are $\eps_1,\eps_2>0$ such that
$a\in\cO(\eps_1,\delta_1)$ and $b\in\cO(\eps_2,\delta_2)$. By Lemma \ref{L-4.1} one has
$a+b\in\cO(\eps_1+\eps_2,\delta_1+\delta_2)$ and $ab\in\cO(\eps_1\eps_2,\delta_1+\delta_2)$.
Since $\delta_1+\delta_2$ is arbitrarily small, it follows that $\cD(a+b)$ and $\cD(ab)$ are
$\tau$-dense. So $a+b$ and $ab$ are densely defined by Lemma \ref{L-4.3}. Since
$a^*,b^*\in\widetilde M$ by Lemma \ref{L-4.8}, $a^*+b^*$ and $b^*a^*$ are also densely
defined so that $(a^*+b^*)^*$ and $(b^*a^*)^*$ exist. Note that $a+b\subset(a^*+b^*)^*$ and
$ab\subset(b^*a^*)^*$ (exercises). Therefore, $a+b$ and $ab$ are closable, so
$\overline{a+b}\in\cO(\eps_1+\eps_2,\delta_1+\delta_2)$ and
$\overline{ab}\in\cO(\eps_1\eps_2,\delta_1+\delta_2)$. Since $\delta_1+\delta_2$ is arbitrarily
small, it follows that $\overline{a+b},\overline{ab}\in\widetilde M$. Moreover, the latter
assertion follows from the above proof.
\end{proof}

\begin{prop}\label{P-4.10}
$\widetilde M$ is a *-algebra with respect to the adjoint $*$, the strong sum $\overline{a+b}$
and the strong product $\overline{ab}$.
\end{prop}

\begin{proof}
First, note by Lemmas \ref{L-4.8} and \ref{L-4.9} that $\widetilde M$ is closed under the
adjoint $*$, the strong sum and product. Let $a,b,c\in\widetilde M$. Since
$\overline{\overline{a+b}+c},\overline{a+\overline{b+c}}\supset a+b+c$ and $\cD(a+b+c)$ is
$\tau$-dense by Lemma \ref{L-4.1}\,(1), it follows from Lemma \ref{L-4.5} that
$$
\overline{\overline{a+b}+c}=\overline{a+\overline{b+c}}.
$$
Since $\overline{\overline{ab}c},\overline{a\overline{bc}}\supset abc$ and $\cD(abc)$ is
$\tau$-dense by Lemma \ref{L-4.1}\,(2), it also follows that
$$
\overline{\overline{ab}c}=\overline{a\overline{bc}}.
$$
Since $\overline{(\overline{a+b})c},\overline{\overline{ac}+\overline{bc}}\subset(a+b)c$
and $\cD((a+b)c)$ is $\tau$-dense by Lemma \ref{L-4.1}, one has
$\overline{(\overline{a+b})c}=\overline{\overline{ac}+\overline{bc}}$ and similarly
$\overline{a(\overline{b+c})}=\overline{\overline{ab}+\overline{ac}}$. Also, since
$(\overline{a+b})^*=(a+b)^*\supset a^*+b^*$, one has $(\overline{a+b})^*=\overline{a^*+b^*}$.
Since $(\overline{ab})^*=(ab)^*\supset b^*a^*$, one has $(\overline{ab})^*=\overline{b^*a^*}$.
Moreover, $a^{**}=a$ holds. 
\end{proof}

In view of Proposition \ref{P-4.10}, for every $a,b\in\widetilde M$ we will use the
convention that $a+b$ and $ab$ mean the strong sum $\overline{a+b}$ and the strong product
$\overline{ab}$, respectively. A big advantage of $\tau$-measurable operators is that we can
freely take adjoint, sum and product in $\widetilde M$. So the domain problem never occurs,
which is an annoying problem in the case of more general measurable operators as in \cite{Se}.

\begin{lemma}\label{L-4.11}
For any $\eps,\eps_1,\eps_2,\delta,\delta_1,\delta_2>0$,
\begin{itemize}
\item[\rm(1)] $\cN(\eps,\delta)^*=\cN(\eps,\delta)$,
\item[\rm(2)] $\lambda\cN(\eps,\delta)=\cN(|\lambda|\eps,\delta)$ for all $\lambda\in\bC$,
$\lambda\ne0$,
\item[\rm(3)] $\eps_1\le\eps_2$, $\delta_1\le\delta_2$ $\implies$
$\cN(\eps_1,\delta_1)\subset\cN(\eps_2,\delta_2)$,
\item[\rm(4)] $\cN(\eps\wedge\eps_2,\delta_1\wedge\delta_2)\subset
\cN(\eps_1,\delta_1)\cap\cN(\eps_2,\delta_2)$,
\item[\rm(5)] $\cN(\eps_1,\delta_1)+\cN(\eps_2,\delta_2)\subset
\cN(\eps_1+\eps_2,\delta_1+\delta_2)$,
\item[\rm(6)] $\cN(\eps_1,\delta_1)\cN(\eps_2,\delta_2)\subset
\cN(\eps_1\eps_2,\delta_1+\delta_2)$.
\end{itemize}
\end{lemma}

\begin{proof}
(1), (5) and (6) are in Lemmas \ref{L-4.8} and \ref{L-4.9}, while (2)--(4) are obvious.
\end{proof}

The main result of the section is the following:

\begin{thm}\label{T-4.12}
$\widetilde M$ is a complete metrizable Hausdorff topological *-algebra with
$\{\cN(\eps,\delta):\eps,\delta>0\}$ as a neighborhood basis of $0$. Moreover, $M$ is dense
in $\widetilde M$.
\end{thm}

\begin{proof}
From (2), (4) and (5) of Lemma \ref{L-4.11} it follows that $\{\cN(\eps,\delta):\eps,\delta>0\}$
defines a linear topology on $\widetilde M$ wth it as a neighborhood basis of $0$. Assume that
$a\in\bigcap_{\eps,\delta>0}\cN(\eps,\delta)$ with the spectral decomposition
$|a|=\int_0^\infty\lambda\,de_\lambda$. By Lemma \ref{L-4.4}, $\tau(e_\eps^\perp)\le\delta$
for all $\eps,\delta>0$, which implies that $e_\eps^\perp=0$ for all $\eps>0$, so $|a|=0$ or
$a=0$. Hence the defined topology is Hausdorff, which is metrizable since there is a countable
neighborhood basis $\{\cN(1/n,1/n):n\in\bN\}$ of $0$. By Lemma \ref{L-4.11}\,(1),
$a\mapsto a^*$ is continuous on $\widetilde M$. For any $a_0,b_0\in\widetilde M$ and any
$\eps,\delta>0$, take $r,s>0$ such that $a_0\in\cN(r,\delta/6)$ and $b_0\in\cN(s,\delta/6)$.
Choose an $\eps_1>0$ with $\eps_1(\eps_1+r+s)\le\eps$. If $a-a_0,b-b_0\in\cN(\eps_1,\delta/6)$,
then
\begin{align*}
&ab-a_0b_0=(a-a_0)(b-b_0)+a_0(b-b_0)+(a-a_0)b_0 \\
&\quad\in\cN(\eps_1,\delta/6)\cN(\eps_1,\delta/6)+\cN(r,\delta/6)\cN(\eps_1,\delta/6)
+\cN(\eps_1,\delta/6)\cN(s,\delta/6)\subset\cN(\eps,\delta)
\end{align*}
thanks to (5) and (6) of Lemma \ref{L-4.11}. Hence $(a,b)\mapsto ab$ is continuous on
$\widetilde M\times\widetilde M$.

Let $a=w|a|\in\widetilde M$ with $|a|=\int_0^\infty\lambda\,de_\lambda$. For any
$\eps,\delta>0$ choose an $r>0$ such that $\tau(e_r^\perp)\le\delta$, and let
$a_1:=w\int_{[0,r]}\lambda\,de_\lambda\in M$. Then
$$
(a-a_1)e_r=w\biggl(\int_{(r,\infty)}\lambda\,de_\lambda\biggr)e_r=0,
$$
and so $a-a_1\in\cN(\eps,\delta)$. Hence $M$ is dense in $\widetilde M$.

Finally, to prove the completeness, let $\{a_n\}$ be a Cauchy sequence in $\widetilde M$.
By taking a subsequence we may assume that $a_{n+1}-a_n\in\cN(2^{-n},2^{-n})$ for all
$n\in\bN$. Choose a sequence $p_n\in\Proj(M)$ such that $\|(a_{n+1}-a_n)p_n\|\le2^{-n}$ and
$\tau(p_n)\le2^{-n}$. Let $e_n:=\bigwedge_{k=n}^\infty p_k\in\Proj(M)$; then $e_n\nearrow$
and $\tau(e_n^\perp)\le2^{-n+1}$. When $l>m\ge n$,
$$
\|(a_l-a_m)e_n\|\le\sum_{k=m}^{l-1}\|(a_{k+1}-a_k)p_k\|\le2^{-n+1}.
$$
Hence one can define $a_0\xi:=\lim_{m\to\infty}a_m\xi$ for $\xi\in\cD(a_0)=\bigcup_ne_n\cH$.
Similarly, for $\{a_n^*\}$ choose a sequence $q_n\in\Proj(M)$ such that
$\|(a_{n+1}^*-a_n^*)q_n\|\le2^{-n}$ and $\tau(q_n)\le2^{-n}$. Let
$f_n:=\bigwedge_{k=n}^\infty q_k$ and define $b_0\zeta:=\lim_{m\to\infty}b_m\zeta$ for
$\zeta\in\cD(b_0)=\bigcup_nf_n\cH$. For every $\xi\in\cD(a_0)$ and $\zeta\in\cD(b_0)$,
$$
\<a_0\xi,\zeta\>=\lim_{m\to\infty}\<a_m\xi,\zeta\>
=\lim_{m\to\infty}\<\xi,a_m^*\zeta\>=\<\xi,b_0\zeta\>,
$$
which implies that $a_0\subset b_0^*$, so $a_0$ is closable. Now, let $a:=\overline a_0$.
Since $a_0\,\eta M$ as easily verified, we have $a\,\eta M$. Since
$e_n\cH\subset\cD(a_0)\subset\cD(a)$ and $\tau(e_n^\perp)\le2^{-n+1}$ for all $n\in\bN$,
we have $a\in\widetilde M$. Furthermore, for any $\eps,\delta>0$ choose an $n_0$ with
$2^{-n_0+1}\le\eps\wedge\delta$. Then $\tau(e_{n_0}^\perp)\le\delta$. When $l>m\ge n_0$,
since $\|(a_l-a_m)e_{n_0}\|\le\eps$, we have $\|(a_l-a_m)e_{n_0}\xi\|\le\eps$ for all
$\xi\in\cH$ with $\|\xi\|\le1$. Letting $l\to\infty$ gives $\|(a-a_m)e_{n_0}\|\le\eps$, so
$a-a_m\in\cN(\eps,\delta)$ for all $m\ge n_0$, implying $a_m\to a$ as $m\to\infty$.
\end{proof}

Let $\widetilde M_+$ be the set of positive self-adjoint $a\in\widetilde M$ (denoted by
$a\ge0$). Note that $\widetilde M_+$ is a closed convex cone in $\widetilde M$. In fact,
assume that $a_n\in\widetilde M_+$ ($n\in\bN$) and $a_n\to a\in\widetilde M$. Then in the
last paragraph of the proof of Theorem \ref{T-4.12}, one can let $e_n=f_n$ and $a_0=b_0$ so
that $\<\xi,a_0\xi\>\ge0$ for all $\xi\in\cD(a_0)$. Since $a=\overline a_0$, $a\ge0$ follows.
Hence $\widetilde M$ is an ordered topological space with the order $a\ge b$ defined by
$a-b\ge0$.

The topology on $\widetilde M$ given in Theorem \ref{T-4.12} is called the
\emph{measure topology}. This topology is not necessarily locally convex. Indeed, it is an
exercise to show that if $M$ is a finite non-atomic von Neumann algebra with a faithful
normal finite trace $\tau$, then a non-empty open convex set in $\widetilde M$ is only the
whole $\widetilde M$ and there is no non-zero continuous functional on $\widetilde M$.

\begin{example}\label{E-4.13}\rm
(1)\enspace
When $M=B(\cH)$ and $\tau$ is the usual trace $\Tr$, $\widetilde M=B(\cH)$ and the measure
topology is the operator norm topology. Indeed, for $a\in\widetilde M$ there is a projection
$e$ such that $\|ae\|<+\infty$ and $\Tr(e^\perp)<1$. The latter implies that $e^\perp=0$ or
$e=1$, so $a\in M$. Moreover, when $\delta<1$, $\cN(\eps,\delta)=\{a\in M:\|a\|\le\eps\}$.

(2)\enspace
Let $M$ be finite with a faithful normal finite trace $\tau$. Then $\widetilde M$ is the
set of all densely-defined closed operators $x\,\eta M$. Indeed, if $a$ is in the latter set
with $|a|=\int_0^\infty\lambda\,de_\lambda$, then $\tau(e_\lambda^\perp)\to0$ as
$\lambda\to\infty$ automatically.

(3)\enspace
Let $(X,\cX,\mu)$ be a localizable measure space, where $(X,\cX,\mu)$ is \emph{localizable}
if for every $A\in\cX$ there is a $B\in\cX$ such that $B\subset A$ and $\mu(B)<+\infty$. For
an abelian von Neumann algebra $\fA=L^\infty(X,\mu)=L^1(X,\mu)^*$ with $\tau(f):=\int_Xf\,d\mu$
for $f\in L^\infty(X,\mu)_+$, $\widetilde\fA$ is the space of measurable functions $f$ on $X$
such that there is an $A\in\cX$ such that $\mu(A)<+\infty$ and $f$ is bounded on
$X\setminus A$, where $f=g$ in $\widetilde\fA$ means $f(x)=g(x)$ $\mu$-a.e.
\end{example}

\subsection{Generalized $s$-numbers}

In the rest of this section we present a brief exposition of the generalized $s$-numbers of
$\tau$-measurable operators. The more detailed accounts are found in \cite{FK} that is the best
literature on the topic.

Let $a\in\widetilde M$. For an interval $I$ of $[0,\infty)$ let $e_I(|a|)$ denote the spectral
projection of $|a|$ corresponding to $I$. For example, $e_{[0,s]}(|a|)=e_s$ and
$e_{(s,\infty)}(|a|)=e_s^\perp$ when $|a|=\int_0^\infty\lambda\,de_\lambda$ is the spectral
decomposition.

\begin{definition}\label{D-4.14}\rm
For $a\in\widetilde M$ and $t>0$ the ($t$th) \emph{generalized $s$-number} $\mu_t(a)$ is
defined by
$$
\mu_t(a):=\inf\{s\ge0:\tau(e_{(s,\infty)}(|a|)\le t\}.
$$
Note that $\mu_t(a)<+\infty$ for all $t>0$ since $\tau(e_{(s,\infty)}(|a|))\to0$ as
$s\to\infty$ by Proposition \ref{P-4.7}.
\end{definition}

\begin{example}\label{E-4.15}\rm
(1)\enspace
Let $M=B(\cH)$ with $\tau=\Tr$. Let $a$ be a compact operator on $\cH$, and take the spectral
decomposition $|a|=\sum_{n=1}^\infty\lambda_n|\xi_n\>\<\xi_n|$ into rank one projections with
$\lambda_1\ge\lambda_2\ge\cdots\to0$ and orthonormal vectors $\xi_n$. For each $n\in\bN$ with
$n-1\le t<n$, since $\Tr(e_{(s,\infty)}(|a|))\le t$ $\iff$ $s\ge\lambda_n$, we have
$$
\mu_t(a)=\lambda_n\qquad(n-1\le t<n,\ n\in\bN),
$$
which is the $n$th \emph{singular value} of $a$.

(2)\enspace
Let $(X,\cX,\mu)$ be a localizable measure space, and $M=L^\infty(X,\mu)$ with
$\tau(f)=\int_Xf\,d\mu$ for $f\in L^\infty(X,\mu)_+$. For $f\in\widetilde M$ (see Example
\ref{E-4.13}\,(3)), since $e_{(s,\infty)}(|f|)=\chi_{\{x:|f(x)|>s\}}$, we have
$$
\mu_t(f)=\inf\{s\ge0:\mu(\{x:|f(x)|>s\})\le t\},\qquad t>0,
$$
which is the \emph{decreasing rearrangement} of $|f|$.
\end{example}

In the following, we use the operator norm $\|a\|$ for any $a\in\widetilde M$ with the
convention that $\|a\|=+\infty$ unless $a\in M$.

\begin{lemma}\label{L-4.16}
Let $a,a_n\in\widetilde M$ for $n\in\bN$.
\begin{itemize}
\item[\rm(1)] For every $s,t>0$,
$$
\mu_t(a)\le s\ \iff\ \tau(e_{(s,\infty)}(|a|)\le t\ \iff\ a\in\cN(s,t).
$$
Hence,
\begin{align}\label{F-4.1}
\mu_t(a)=\inf\{s>0:a\in\cN(s,t)\}.
\end{align}
\item[\rm(2)] For every $t>0$,
\begin{align}\label{F-4.2}
\mu_t(a)=\inf\{\|ae\|:e\in\Proj(M),\,\tau(e^\perp)\le t\}.
\end{align}
Moreover, if $\fA$ is a von Neumann subalgebra of $M$ containing all spectral projections of
$|a|$, then
\begin{align}\label{F-4.3}
\mu_t(a)=\inf\{\|ae\|:e\in\Proj(\fA),\,\tau(e^\perp)\le t\}.
\end{align}
\item[\rm(3)] When $a\in\widetilde M_+$, for every $t>0$,
\begin{align}\label{F-4.4}
\mu_t(a)=\inf\Biggl\{\sup_{\xi\in e\cH,\,\|\xi\|=1}\<\xi,a\xi\>:
e\in\Proj(M),\,\tau(e^\perp)\le t\Biggr\}.
\end{align}
Here, $\<\xi,a\xi\>$ is defined to be $\int_0^\infty\lambda\,d\|e_\lambda\xi\|^2$, where
$a=\int_0^\infty\lambda\,de_\lambda$ is the spectral decomposition of $a$.
\item[\rm(4)] $a_n\to a$ in the measure topology if and only if $\mu_\eps(a_n-a)\to0$ for any
$\eps>0$.
\end{itemize}
\end{lemma}

\begin{proof}
(1)\enspace
The first equivalence is immediately seen from the definition of $\mu_t(a)$. The second
equivalence follows from Lemma \ref{L-4.4}.

(2)\enspace
For any $s\ge0$, let $e:=e_{[0,s]}(|a|)\in\fA$. Then $\|ae\|=\|\,|a|e\|\le s$ and
$e^\perp=e_{(s,\infty)}(|a|)$. Hence $\mu_t(a)\ge\mbox{the RHS of \eqref{F-4.3}}$. Conversely,
let $r:=\mbox{the RHS of \eqref{F-4.2}}$. For any $\eps>0$ there is an $e\in\Proj(M)$ such that
$\|ae\|\le r+\eps$ and $\tau(e^\perp)\le t$. Hence $a\in\cN(r+\eps,t)$. So $\mu_t(a)\le r+\eps$
by (1). Letting $\eps\searrow0$ gives $\mu_t(a)\le r$.

(3)\enspace
Similarly to the proof of (2) one has $\mu_t(a)\ge\mbox{the RHS of \eqref{F-4.4}}$. Hence, by
(2) it suffices to show that $\|ae\|=\sup\{\<\xi,a\xi\>:\xi\in eH,\,\|\xi\|=1\}$ for any
$e\in\Proj(\fA)$, where $\fA$ is as in (2). If $ae$ is bounded, then
$\|ae\|=\|eae\|=\sup\{\<\xi,a\xi\>:\xi\in eH,\,\|\xi\|=1\}$. If $ae$ is not unbounded, then
both sides are $+\infty$.

(4)\enspace
By (1), $\mu_\eps(a_n-a)\to0$ for any $\eps>0$ if and only if, for any $\eps,\delta>0$, there
is an $n_0$ such that $a_n-a\in\cN(\eps,\delta)$ for all $n\ge n_0$, which means that
$a_n\to a$ in the measure topology.
\end{proof}

\begin{prop}\label{P-4.17}
Let $a,b,c\in\widetilde M$.
\begin{itemize}
\item[\rm(1)] The function $t\in(0,\infty)\mapsto\mu_t(a)\in[0,\infty)$ is non-increasing and
right-continuous.
\item[\rm(2)] $\mu_t(a)\nearrow\|a\|$ ($\in[0,+\infty]$) as $t\searrow0$.
\item[\rm(3)] $\mu_t(a)=\mu_t(|a|)=\mu_t(a^*)$ for all $t>0$.
\item[\rm(4)] $\mu_t(\alpha a)=|\alpha|\mu_t(a)$ for all $\alpha\in\bC$ and $t>0$.
\item[\rm(5)] If $0\le a\le b$, then $\mu_t(a)\le\mu_t(b)$ for all $t>0$.
\item[\rm(6)] $\mu_t(bac)\le\|b\|\,\|c\|\mu_t(a)$ for all $t>0$.
\item[\rm(7)] $\mu_{t+t'}(a+b)\le\mu_t(a)+\mu_{t'}(b)$ for all $t,t'>0$.
\item[\rm(8)] $\mu_{t+t'}(ab)\le\mu_t(a)\mu_{t'}(b)$ for all $t,t'>0$.
\item[\rm(9)] If $f$ is a continuous non-decreasing function on $[0,\infty)$ with $f(0)=0$,
then $f(|a|)\in\widetilde M$ and $\mu_t(f(|a|))=f(\mu_t(a))$ for all $t>0$.
\item[\rm(10)] $|\mu_t(a)-\mu_t(b)|\le\|a-b\|$ for all $t>0$.
\end{itemize}
\end{prop}

\begin{proof}
(1) is easy.

(2)\enspace
By \eqref{F-4.1}, $\mu_f(a)\le\|a\|$ is clear. Let $s:=\lim_{t\searrow0}\mu_t(s)$; then
$\tau(e_{(s,\infty)}(|a|))=0$ and so $\|a\|\le s$. Hence the assertion follows.

(3)\enspace
That $\mu_t(a)=\mu_t(|a|)$ is obvious. By Lemma \ref{L-4.11}\,(1) and \eqref{F-4.1},
$\mu_t(a)=\mu_t(a^*)$ follows.

(4)\enspace
By Lemma \ref{L-4.11}\,(2).

(5)\enspace
It suffices to show that $\tau(e_{(s,\infty)}(a))\le\tau(e_{(s,\infty)}(b))$ for all $s\ge0$.
Let $a=\int_0^\infty\lambda\,de_\lambda$ and $b=\int_0^\infty\lambda\,df_\lambda$ be the
spectral decompositions. Let $0<s<s'$ and $\xi$ be in the range of
$e_{(s',\infty)}(a)\wedge e_{[0,s]}(b)$. Since
$$
s\|\xi\|^2\ge\int_{[0,s]}\lambda\,d\|f_\lambda\xi\|^2=\|b^{1/2}\xi\|^2
\ge\|a^{1/2}\xi\|^2=\int_{(s',\infty)}\lambda\,d\|e_\lambda\xi\|^2\ge s'\|\xi\|^2
$$
so that $\xi-0$. Hence $e_{(s',\infty)}(a)\wedge e_{[0,s]}(b)=0$. Since
\begin{align*}
e_{(s',\infty)}(a)&=e_{(s'\infty)}(a)-e_{(s',\infty)}(a)\wedge e_{[0,s]}(b) \\
&\sim e_{(s',\infty)}(a)\wedge e_{[0,s]}(b)-e_{[0,s]}(b)\le e_{(s,\infty)}(b),
\end{align*}
one has $\tau(e_{(s',\infty)}(a))\le\tau(e_{(s,\infty)}(b))$. Letting $s'\searrow s$ gives
the desired inequality.

(6)\enspace
By \eqref{F-4.2} and (3) one has
$$
\mu_t(bac)\le\|b\|\mu_t(ac)=\|b\|\mu_t(c^*a^*)\le\|b\|\,\|c^*\|\mu_t(a^*)
=\|b\|\,\|c\|\mu_t(a).
$$

(7)\enspace
By \eqref{F-4.1} and Lemma \ref{L-4.11}\,(5) one has
\begin{align*}
\mu_t(a)+\mu_{t'}(b)&=\inf\{s+s':s,s'>0,\,a\in\cN(s,t),\,b\in\cN(s',t')\} \\
&\ge\inf\{s>0:a+b\in\cN(s,t+t')\}=\mu_{t+t'}(a+b).
\end{align*}

(8)\enspace
By \eqref{F-4.1} and Lemma \ref{L-4.11}\,(6) one has
\begin{align*}
\mu_t(a)\mu_{t'}(b)&=\inf\{ss':s,s'>0,\,a\in\cN(s,t),\,b\in\cN(s',t')\} \\
&\ge\inf\{s>0:ab\in\cN(s,t+t')\}=\mu_{t+t'}(ab).
\end{align*}

(9)\enspace
Let $\fA$ be the commutative von Neumann subalgebra of $M$ generated by the spectral
projections of $|a|$. Let $|a|=\int_0^\infty\lambda\,de_\lambda$ be the spectral decomposition.
Then $f(|a|)=\int_0^\infty f(\lambda)\,de_\lambda$ and hence
$e_{(s,\infty)}(f(|a|))=e_{f^{-1}((s,\infty))}(|a|)\in\fA$. This implies that
$f(|a|)\in\widetilde M$. By \eqref{F-4.3} we have
$$
\mu_t(f(|a|))=\inf\{\|f(|a|)e\|:e\in\Proj(\fA),\,\tau(e^\perp)\le t\}.
$$
When $e\in\Proj(\fA)$, note that
$$
f(|a|)e=\int_0^\infty f(\lambda)\,d(e_\lambda e)
=f\biggl(\int_0^\infty\lambda\,d(e_\lambda e)\biggr)=f(|a|e),
$$
where we have used $f(0)=0$ for the second equality. Therefore, $\|f(|a|)e\|=f(\|\,|a|e\|)$
holds so that
$$
\mu_t(f(|a|))=\inf\{f(\|\,|a|e\|):e\in\Proj(\fA),\,\tau(e^\perp)\le t\}=f(\mu_t(a)).
$$

(10)\enspace
For every $t>0$ and $\eps>0$, by (7) and (2) one has
$$
\mu_{t+\eps}(a)=\mu_{t+\eps}(b+(a-b))\le\mu_t(b)+\mu_\eps(a-b)
\le\mu_t(b)+\|a-b\|.
$$
Letting $\eps\searrow0$ gives, by (1), $\mu_t(a)\le\mu_t(b)+\|a-b\|$, and similarly
$\mu_t(b)\le\mu_t(a)+\|a-b\|$. Hence the assertion follows.
\end{proof}

We extend the trace $\tau$ on $M_+$ to $a\in\widetilde M_+$ by
\begin{align}\label{F-4.5}
\tau(a):=\int_0^\infty\lambda\,d\tau(e_\lambda),
\end{align}
where $a=\int_0^\infty\lambda\,de_\lambda$ is the spectral decomposition.

\begin{prop}\label{P-4.18}
For every $a\in\widetilde M_+$,
\begin{align}\label{F-4.6}
\tau(a)=\int_0^\infty\mu_t(a)\,dt.
\end{align}
Moreover, for any continuous non-decreasing function $f$ on $[0,\infty)$ with $f(0)\ge0$,
\begin{align}\label{F-4.7}
\tau(f(a))=\int_0^\infty f(\mu_t(a))\,dt.
\end{align}
\end{prop}

\begin{proof}
For each $n\in\bN$ define
$$
f_n(\lambda):=\sum_{k=0}^\infty{k\over2^n}
\,\chi_{\bigl[{k\over2^n},{k+1\over2^n}\bigr)}(\lambda),\qquad
a_n:=f_n(a)=\sum_{k=0}^\infty{k\over2^n}
\,e_{\bigl[{k\over2^n},{k+1\over2^n}\bigr)}(a).
$$
Then we have
$$
\tau(a_n)=\sum_{k=0}^\infty{k\over2^n}
\,\tau\Bigl(e_{\bigl[{k\over2^n},{k+1\over2^n}\bigr)}(a)\Bigr)
=\int_0^\infty f_n(\lambda)\,d\tau(e_\lambda)
\ \longrightarrow\ \int_0^\infty f(\lambda)\,d\tau(e_\lambda)=\tau(a).
$$
Since $\|a-a_n\|\le1/2^n$, we have $|\mu_t(a)-\mu_t(a_n)|\le1/2^n$ by Proposition
\ref{P-4.17}\,(10), so $\mu_t(a_n)\nearrow\mu_t(a)$ for all $t>0$. Hence
$\int_0^\infty\mu_t(a_n)\,dt\nearrow\int_0^\infty\mu_t(a)$\,dt. Note that
$$
\mu_t(a_n)={k\over2^n}\qquad\mbox{if}\quad
\tau\Bigl(e_{\bigl[{k+1\over2^n},\infty\bigr)}(a)\Bigr)\le
t<\tau\Bigl(e_{\bigl[{k\over2^n},\infty\bigr)}(a)\Bigr),
$$
which implies that
$$
\int_0^\infty\mu_t(a_n)\,dt
=\sum_{k=0}^\infty{k\over2^n}
\,\tau\Bigl(e_{\bigl[{k\over2^n},{k+1\over2^n}\bigr)}(a)\Bigr)
=\tau(a_n).
$$
Letting $n\to\infty$ gives \eqref{F-4.6}. Furthermore, \eqref{F-4.7} follows from Proposition
\ref{P-4.17}\,(9).
\end{proof}

\begin{lemma}\label{L-4.19}
Let $a,a_n\in\widetilde M$ ($n\in\bN$) be such that $a_n\to a$ in the measure topology. Then:
\begin{itemize}
\item[\rm(1)] $\mu_t(a)\le\liminf_{n\to\infty}\mu_t(a_n)$ for all $t>0$.
\item[\rm(2)] $\mu_t(a)=\lim_{n\to\infty}\mu_t(a_n)$ if $\mu_t(a)$ is continuous at $t$.
\end{itemize}
\end{lemma}

\begin{proof}
(1)\enspace
For any $\eps>0$ Proposition \ref{P-4.17}\,(7) implies that
$\mu_{t+\eps}(a)\le\mu_t(a_n)+\mu_\eps(a-a_n)$. Since $\mu_\eps(a-a_n)\to0$ as $n\to\infty$ by
Lemma \ref{L-4.16}\,(4), one has
$$
\mu_{t+\eps}(a)\le\liminf_{n\to\infty}\mu_t(a_n).
$$
Letting $\eps\searrow0$ gives the assertion.

(2)\enspace
If $0<\eps<t$, then
$$
\mu_t(a_n)\le\mu_{t-\eps}(a)+\mu_\eps(a_n-a),
$$
and so $\limsup_{n\to\infty}\mu_t(a_n)\le\mu_{t-\eps}(a)$. Letting $\eps\searrow0$ gives
$\mu_t(a)\ge\limsup_{n\to\infty}\mu_t(a_n)$, implying the assertion.
\end{proof}

\begin{prop}\label{P-4.20}
Let $a,a_n\in\widetilde M$, $a_n\ge0$ ($n\in\bN$), be such that $a_n\to a$ in the measure
topology. Then:
\begin{itemize}
\item[\rm(1)] $\tau(a)\le\liminf_{n\to\infty}\tau(a_n)$. (\emph{Fatou's lemma})
\item[\rm(2)] If $a_n\le a$ for all $n$ (in particular, $a_n$ is increasing), then
$\tau(a)=\lim_{n\to\infty}\tau(a_n)$. (\emph{Monotone convergence theorem})
\end{itemize}
\end{prop}

\begin{proof}
(1)\enspace
We have
\begin{align*}
\tau(a)&=\int_0^\infty\mu_t(a)\,dt\quad\mbox{(by \eqref{F-4.6})} \\
&\le\int_0^\infty\liminf_{n\to\infty}\mu_t(a_n)\,dt\quad\mbox{(by Lemma \ref{L-4.19}\,(1))} \\
&\le\liminf_{n\to\infty}\int_0^\infty\mu_t(a_n)\,dt\quad\mbox{(by Fatou's lemma)} \\
&=\liminf_{n\to\infty}\tau(a_n).
\end{align*}

(2) follows from (1) and $\tau(a_n)=\int_0^\infty\mu_t(a_n)\,dt\le\int_0^\infty\mu_t(a)\,dt
=\tau(a)$.
\end{proof}

\begin{prop}\label{P-4.21}
Let $a,a_n\in\widetilde M$ with $a_n\to a$ in the measure topology. Assume that there is an
$f\in L^1((0,\infty),dt)$ such that $\mu_t(a_n)\le f(t)$ a.e.\ for all $n$ (in particular,
there is a $b\in\widetilde M_+$ such that $\tau(b)<+\infty$ and $|a_n|\le b$ for all $n$).
Then
$$
\lim_{n\to\infty}\tau(|a_n-a|)=0,\qquad\lim_{n\to\infty}\tau(|a_n|)=\tau(|a|).
$$
(\emph{Lebesgue's convergence theorem})
\end{prop}

\begin{proof}
Note that
$$
\mu_t(a_n-a)\le\mu_{t/2}(a_n)+\mu_{t/2}(a)\le2f(t/2)\ \ a.e.
$$
and $\int_0^\infty f(t/2)\,dt=2\int_0^\infty f(t)\,dt<+\infty$. Since $\mu_t(a_n-a)\to0$ for
all $t>0$ by Lemma \ref{L-4.16}\,(4), it follows from Proposition \ref{P-4.18} and the
Lebesgue's convergence theorem that $\tau(|a_n-a|)=\int_0^\infty\mu_t(a_n-a)\,dt\to0$. Since
$\mu_t(a_n)\le f(t)$ a.e.\ and $\mu_t(a_n)\to\mu_t(a)$ a.e.\ by Lemma \ref{L-4.19}\,(2),
$\tau(|a_n|)=\int_0^\infty\mu_t(a_n)\,dt\to\int_0^\infty\mu_t(a)\,dt=\tau(|a|)$ similarly.
\end{proof}

We end the section with the next lemma, which will be used to prove Proposition \ref{P-11.26}
of Sec.~11.2.

\begin{lemma}\label{L-4.22}
Let $a\in\widetilde M_+$ and $v\in M$ be a contraction. If $f$ is a non-negative continuous
and convex function on $[0,\infty)$ with $f(0)=0$, then
$$
\mu_t(f(vav^*))=f(\mu_t(vav^*))\le\mu_t(vf(a)v^*)
$$
for all $t>0$. 
\end{lemma}

\begin{proof}
It is clear that the function $f$ stated is non-decreasing on $[0,\infty)$. Hence
$f(vav^*)\in\widetilde M_+$ and the first equality follows from Proposition \ref{P-4.17}\,(9).
First, assume that $a$ is bounded with the spectral decomposition
$a=\int_0^{\|a\|}\lambda\,de_\lambda$. For any vector $\xi\in\cH$, $\|\xi\|=1$, since
$\|v^*\xi\|\le1$ and $f(0)=0$, one finds that
\begin{align*}
\<\xi,vf(a)v^*\xi\>&=\int_0^\infty f(\lambda)\,d\|e_\lambda v^*\xi\|^2 \\
&=\int_0^\infty f(\lambda)\,d\|e_\lambda v^*\xi\|^2+f(0)(1-\|v^*\xi\|^2) \\
&\ge f\biggl(\int_0^\infty\lambda\,d\|e_\lambda v^*\xi\|^2\biggr)
=f(\<\xi,vav^*\xi\>).
\end{align*}
In the above, $vf(a)v^*\xi$ and $vav^*\xi$ make sense and the inequality follows from the
convexity of $f$. Hence by \eqref{F-4.4} one has
\begin{align*}
\mu_t(vf(a)v^*)
&=\inf\Biggl\{\sup_{\xi\in e\cH,\,\|\xi\|=1}\<\xi,vf(a)v^*\xi\>:
e\in\Proj(M),\,\tau(e^\perp)\le t\Biggr\} \\
&\ge f\Biggl(\inf\Biggl\{\sup_{\xi\in e\cH,\,\|\xi\|=1}\<\xi,vav^*\xi\>:
e\in\Proj(M),\,\tau(e^\perp)\le t\Biggr\}\Biggr) \\
&=f(\mu_t(vav^*)).
\end{align*}

Next, let $a\in\widetilde M_+$ be arbitrary with $a=\int_0^\infty\lambda\,de_\lambda$. Set
$a_n:=\int_0^n\lambda\,de_\lambda$ for $n\in\bN$; then
$f(a_n)=\int_0^nf(\lambda)\,de_\lambda\le\int_0^\infty f(\lambda)\,de_\lambda=f(a)$ and
$vf(a_n)v^*\le vf(a)v^*$. From the above case one has
$$
f(\mu_t(va_nv^*))\le\mu_t(vf(a_n)v^*)\le\mu_t(vf(a)v^*),\qquad t>0.
$$
Since $va_nv^*\to vav^*$ in the measure topology, it follows from Lemma \ref{L-4.19}\,(1) that
$\mu_t(vav^*)\le\liminf_{n\to\infty}\mu_t(va_nv^*)$ so that
$$
f(\mu_t(vav^*))\le\liminf_{n\to\infty}f(\mu_t(va_nv^*))\le\mu_t(vf(a)v^*)
$$
thanks to $f$ being non-decreasing and continuous on $[0,\infty)$.
\end{proof}

\section{$L^p$-spaces with respect to a trace}

In this section we assume as in Sec.~4 that $M$ is a semifinite von Neumann algebra with
a faithful semifinite normal trace $\tau$. The section gives a concise but self-contained
exposition of the non-commutative $L^p$-spaces with respect to a trace, as a nice application
of the topics of Sec.~4. The non-commutative $L^p$-spaces $L^p(M,\tau)$ on $(M,\tau)$ was
first developed in \cite{Dix,Ku,Se}, which was later discussed in \cite{Ne} in a simpler
approach based on $\tau$-measurable operators.

\begin{definition}\label{D-5.1}\rm
For each $a\in\widetilde M$ define
$$
\|a\|_p:=\tau(|a|^p)^{1/p}\in[0,+\infty],\qquad0<p<\infty.
$$
By \eqref{F-4.7} note that $\|a\|_p=\bigl(\int_0^\infty\mu_t(a)^p\,dt\bigr)^{1/p}$.
Furthermore, define $\|a\|_\infty:=\|a\|\in[0,+\infty]$. By Proposition \ref{P-4.17}\,(3) note
that $\|a\|_p=\|a^*\|_p$ for every $a\in\widetilde M$ and $0<p\le\infty$. The
\emph{non-commutative $L^p$-space} on $(M,\tau)$ is defined as
$$
L^p(M)=L^p(M,\tau):=\{a\in\widetilde M:\|a\|_p<+\infty\},\qquad0<p\le\infty.
$$
(In particular, $L^\infty(M)=L^\infty(M,\tau)=M$.)
\end{definition}

\begin{example}\label{E-5.2}\rm
(1)\enspace
When $M=B(\cH)$ and $\tau=\Tr$, $L^p(M)$ is the \emph{Schatten and von Neumann $p$-class}
$\cC_p(\cH):=\{a\in B(\cH):\Tr|a|^p<+\infty\}$ with $\|a\|_p:=(\Tr|a|^p)^{1/p}$.
In particular, the case $p=1$ is the \emph{trace-class} and the case $p=2$ is the
\emph{Hilbert-Schmidt class}.

(2)\enspace
When $M=L^\infty(X,\mu)$ on a localizable measure space $(X,\cX,\mu)$, $L^p(M)$ is the usual
$L^p$-space $L^p(X,\mu):=\{f:\mbox{measurable},\,\int_X|f|^p\,d\mu<+\infty\}$ with
$\|f\|_p:=\bigl(\int_X|f|^p\,d\mu\bigr)^{1/p}$.
\end{example}

\begin{lemma}\label{L-5.3}
Let $0<p\le\infty$.
\begin{itemize}
\item[\rm(1)] $L^p(M)$ is a linear subspace of $\widetilde M$.
\item[\rm(2)] If $a\in L^p(M)$, $0<p\le\infty$, and $x,y\in M$, then $xay\in L^p(M)$ and
$\|xay\|_p\le\|x\|\,\|y\|\,\|a\|_p$.
\end{itemize}
\end{lemma}

\begin{proof}
(1)\enspace
Since the case $p=\infty$ is obvious, let $0<p<\infty$ and $a,b\in L^p(M)$. For
$\alpha\in\bC$, since
$$
\|\alpha a\|_p=\biggl[\int_0^\infty\mu_t(\alpha a)^p\,dt\biggr]^{1/p}
=\biggl[\int_0^\infty(|\alpha|\mu_t(a))^p\,dt\biggr]^{1/p}=|\alpha|\,\|a\|_p<+\infty,
$$
one has $\alpha a\in L^p(M)$. Since
$$
\mu_t(a+b)\le\mu_{t/2}(a)+\mu_{t/2}(b)\le2\max\{\mu_{t/2}(a),\mu_{t/2}(b)\},
$$
one has $\mu_t(a+b)^p\le2^p\bigl[\mu_{t/2}(a)^p+\mu_{t/2}(b)\bigr]$ so that
\begin{align}
\|a+b\|_p^p&=\int_0^\infty\mu_t(a+b)^p\,dt
\le2^p\biggl[\int_0^\infty\mu_{t/2}(a)^p\,dt+\int_0^\infty\mu_{t/2}(b)^p\,dt\biggr]
\nonumber\\
&=2^{p+1}(\|a\|_p^p+\|b\|_p^p)<+\infty, \label{F-5.1}
\end{align}
so $a+b\in L^p(M)$.

(2)\enspace
The case $p=\infty$ is obvious. For $0<p<\infty$, since $\mu_t(xay)\le\|x\|\,\|y\|\mu_t(a)$
by Proposition \ref{P-4.17}\,(6), we may take the $p$th power of both sides and integrate to
have the assertion.
\end{proof}

\begin{lemma}\label{L-5.4}
Set
\begin{align*}
\fF_\tau&:=\{x\in M_+:\tau(x)<+\infty\}=M_+\cap L^1(M), \\
\fN_\tau&:=\{x\in M:\tau(x^*x)<+\infty\}=M\cap L^2(M), \\
\fM_\tau&:=\lin\{y^*x:x,y\in\fN_\tau\}.
\end{align*}
Then:
\begin{itemize}
\item[\rm(1)] $\fN_\tau$ is a two-sided ideal of $M$.
\item[\rm(2)] $\fM_\tau=\lin\,\fF_\tau$ ($\subset L^1(M)$) and $\fF_\tau=M_+\cap\fM_\tau$.
\item[\rm(3)] $\tau|_{\fF_\tau}$ uniquely extends by linearity to a positive linear functional
on $\fM_\tau$, and the extended $\tau$ satisfies $\tau(xy)=\tau(yx)$ for all $x,y\in\fN_\tau$.
\end{itemize}
\end{lemma}

\begin{proof}
(1)\enspace
If $x,y\in\fN_\tau$, then $(x+y)^*(x+y)\le2(x^*x+y^*y)$ and so $x+y\in\fN_\tau$. If
$x\in\fN_\tau$ and $y\in M$, then $(yx)^*(yx)\le\|y\|^2x^*x$, so $yx\in\fN_\tau$. Also
$\tau((xy)^*(xy))=\tau((xy)(xy)^*)\le\|y\|^2\tau(xx^*)=\|y\|^2\tau(x^*x)<+\infty$, so
$xy\in\fN_\tau$. Hence $\fN_\tau$ is a two-sided ideal of $M$.

(2)\enspace
For every $x,y\in M$ recall the polarization
\begin{align}\label{F-5.2}
y^*x={1\over4}\sum_{k=0}^3i^k(x+i^ky)^*(x+i^ky).
\end{align}
If $x,y\in\fN_\tau$, then $x+i^ky\in\fN_\tau$ by (1) so that $y^*x\in\lin\,\fF_\tau$. Hence
$\fM_\tau\subset\lin\,\fF_\tau$. The converse is obvious. Next, if $a=\sum_{j=1}^ny_j^*x_j
\in M_+\cap\fM_\tau$ ($x_j,y_j\in\fN_\tau$), then from the polarization and $a^*=a$, one has
$$
a={1\over4}\sum_{j=1}^n\{(x_j+y_j)^*(x_j+y_j)-(x_j-y_j)^*(x_j-y_j)\}
\le{1\over4}\sum_{j=1}^n(x_j+y_j)^*(x_j+y_j),
$$
which implies that $a\in\fF_\tau$. Hence $\fF_\tau=M_+\cap\fM_\tau$.

(3) From (2) we can uniquely extend $\tau|_{\fF_\tau}$ to a linear functional
$\tau:\fM_\tau\to\bC$. For every $x,y\in\fN_\tau$, using the polarization in \eqref{F-5.2} we
have
\begin{align*}
\tau(xy)&={1\over4}\sum_{k=0}^3\bigl[
\tau((y+x^*)^*(y+x^*))-\tau((y-x^*)^*(y-x^*)) \\
&\qquad\qquad+i\tau((y+ix^*)^*(y+ix^*))-i\tau((y-ix^*)^*(y-ix^*))\bigr] \\
&={1\over4}\sum_{k=0}^3\bigl[
\tau((y+x^*)(y+x^*)^*)-\tau((y-x^*)(y-x^*)^*) \\
&\qquad\qquad+i\tau((y+ix^*)(y+ix^*)^*)-i\tau((y-ix^*)(y-ix^*)^*)\bigr] \\
&={1\over4}\sum_{k=0}^3\bigl[
\tau((x+y^*)^*(x+y^*))-\tau((x-y^*)^*(x-y^*)) \\
&\qquad\qquad+i\tau((x+iy^*)^*(x+iy^*))
-i\tau((x-iy^*)^*(x-iy^*))\bigr] \\
&=\tau(yx),
\end{align*}
as asserted.
\end{proof}

\begin{prop}\label{P-5.5}
The functional $\tau$ on $\fM_\tau$ (defined in Lemma \ref{L-5.4}\,(3)) uniquely extends to
a positive linear functional on $L^1(M)$ such that
\begin{align}\label{F-5.3}
|\tau(a)|\le\|a\|_1,\qquad a\in L^1(M).
\end{align}
Moreover, if $a\in L^1(M)_+$, then $\tau(a)$ coincides with the definition in \eqref{F-4.5}.
\end{prop}

\begin{proof}
Since $(y,x)\in\fN_\tau\times\fN_\tau\mapsto\tau(y^*x)$ is an inner product, the Schwarz
inequality says that
$$
|\tau(y^*x)|\le\tau(x^*x)^{1/2}\tau(y^*y)^{1/2}=\|x\|_2\|y\|_2,\qquad x,y\in\fN_\tau.
$$
Let $a\in M\cap L^1(M)$ with the polar decomposition $a=w|a|$. Set $x:=|a|^{1/2}$ and
$y:=|a|^{1/2}w^*$; then $x,y\in M$. Since $\tau(x^*x)=\tau(|a|)$ and
$\tau(y^*y)=\tau(w|a|w^*)=\tau(|a|^{1/2}w^*w|a|^{1/2})=\tau(|a|)$, one has $x,y\in\fN_\tau$,
$\|x\|_2=\|y\|_2=\|a\|_1^{1/2}$ and $a=y^*x\in\fM_\tau$, so
$|\tau(a)|\le\|x\|_2\|y\|_2=\|a\|_1$. Therefore, $M\cap L^1(M)\subset\fM_\tau$ and
\begin{align}\label{F-5.4}
|\tau(a)|\le\|a\|_1,\qquad a\in M\cap L^1(M).
\end{align}

Next, let $a\in L^1(M)$ with $a=w|a|$ and $|a|=\int_0^\infty\lambda\,de_\lambda$. Set
$a_n:=w\int_0^n\lambda\,de_\lambda$ for $n\in\bN$; then $|a_n|=\int_0^n\lambda\,de_\lambda$
and $|a-a_n|=\int_{(0,n)}\lambda\,de_\lambda$. Hence $a_n\to a$ in the measure topology and
$\mu_t(a_n)\le\mu_t(a)\in L^1((0,\infty),dt)$, so we have $a_n\in M\cap L^1(M)$ and
$\|a_n-a\|_1\to0$ by Proposition \ref{P-4.21}. Moreover, by \eqref{F-5.4} and \eqref{F-5.1}
for $p=1$,
\begin{align*}
|\tau(a_m)-\tau(a_n)|&=|\tau(a_m-a_n)|\le\|a_m-a_n\|_1=\|(a_m-a)+(a-a_n)\|_1 \\
&\le4(\|a_m-a\|_1+\|a_n-a\|_1)\ \longrightarrow\ 0\quad\mbox{as $m,n\to\infty$}.
\end{align*}
So one can define
$$
\tau(a):=\lim_{n\to\infty}\tau(a_n).
$$
Since $|\tau(a_n)|\le\|a_n\|_1$ by \eqref{F-5.4} and $\|a_n\|_1\to\|a\|_1$ by Proposition
\ref{P-4.21}, \eqref{F-5.3} holds. The linearity of $\tau$ on $L^1(M)$ and the uniqueness of
$\tau$ with \eqref{F-5.3} are easy to see, so the details are omitted.

Finally, let $a\in L^1(M)_+$ and set $a_n:=\int_0^n\lambda\,de_\lambda\in M_+\cap L^1(M)$ as
above. Since
$$
\sum_{k=1}^{n2^m}{k\over2^m}\,e_{\bigl[{k\over2^m},{k+1\over2^m}\bigr)}
\,\nearrow\,a_n\quad\mbox{as $m\to\infty$},
$$
it follows from the normality of $\tau$ that
$$
\tau(a_n)=\lim_{m\to\infty}\tau\Biggl(
\sum_{k=1}^{n2^m}{k\over2^m}\,e_{\bigl[{k\over2^m},{k+1\over2^m}\bigr)}\Biggr)
=\lim_{m\to\infty}\sum_{k=1}^{n2^m}{k\over2^m}
\,\tau\biggl(e_{\bigl[{k\over2^m},{k+1\over2^m}\bigr)}\biggr)
=\int_0^n\lambda\,d\tau(e_\lambda).
$$
Therefore,
$$
\tau(a)=\lim_{n\to\infty}\tau(a_n)=\lim_{n\to\infty}\int_0^n\lambda\,d\tau(e_\lambda)
=\int_0^\infty\lambda\,de(e_\lambda),
$$
which is the definition in \eqref{F-4.5}.
\end{proof}

\begin{prop}[H\"older's inequality]\label{P-5.6}
Let $1\le p<\infty$ and $1/p+1/q=1$, If $a\in L^p(M)$ and $b\in L^q(M)$, then $ab\in L^1(M)$
and
\begin{align}\label{F-5.5}
|\tau(ab)|\le\|ab\|_1\le\|a\|_p\|b\|_q.
\end{align}
\end{prop}

\begin{proof}
Since the case $p=1$ holds by Lemma \ref{L-5.3}\,(2), assume that $1<p<\infty$. Let
$a=u|a|\in L^p(M)$ and $b=v|b|\in L^q(M)$ in the polar decompositions. For each $n\in\bN$
define
\begin{align*}
x_n&:=\sum_{k=0}^{n2^n}{k\over2^n}\,e_k,\qquad\mbox{where}
\quad e_k:=e_{\bigl[{k\over2^n},{k+1\over2^n}\bigr)}(|a|), \\
y_n&:=\sum_{k=0}^{n2^n}{k\over2^n}\,f_k,\qquad\mbox{where}
\quad f_k:=e_{\bigl[{k\over2^n},{k+1\over2^n}\bigr)}(|b|).
\end{align*}
Then $x_n\le|a|$, $x_n\to|a|$ in the measure topology, and $y_n\le|b|$, $y_n\to|b|$ in the
measure topology. Since
$$
x_n^p=\sum_{k=0}^{n2^n}\biggl({k\over2^n}\biggr)^pe_k\le|a|^p,
$$
one has
$$
+\infty>\tau(x_n^p)=\sum_{k=0}^{n2^n}\biggl({k\over2^n}\biggr)^p\tau(e_k)
$$
so that $\tau(e_k)<+\infty$ for $1\le k\le n2^n$. Similarly, $\tau(f_k)<+\infty$ for
$1\le k\le n2^n$. For any $n$ fixed, take the polar decomposition $ux_nvy_n=w|ux_nvy_n|$, and
note that
$$
x_n^{pz}=\sum_{k=1}^{n2^n}\biggl({k\over2^n}\biggr)^{pz}e_k,\qquad
y_n^{q(1-z)}=\sum_{k=1}^{n2^n}\biggl({k\over2^n}\biggr)^{q(1-z)}f_k
$$
are in $\fN_\tau$ for $0\le\Re z\le1$. Hence by Lemma \ref{L-5.4} we can define
$$
f(z):=\tau(w^*ux_n^{pz}vy_n^{q(1-z)})
=\sum_{j,k=1}^{n2^n}\biggl({j\over2^n}\biggr)^{pz}\biggl({k\over2^n}\biggr)^{q(1-z)}
\tau(w^*ue_jvf_k)
$$
is bounded continuous on $0\le\Re z\le1$ and analytic in $0<\Re z<1$. Moreover, we find that
\begin{align*}
|f(it)|&\le\|w^*ux_n^{ipt}vy_n^qy_n^{-iqt}\|_1\le\|y_n^q\|_1\le\|\,|b|^q\|_1=\|b\|_q^q, \\
|f(1+it)|&\le\|w^*ux_n^px_n^{ipt}vy_n^{-iqr}\|_1\le\|x_n^p\|_1\le\|\,|a|^p\|_1=\|a\|_P^p,
\qquad t\in\bR.
\end{align*}
From the three-lines theorem it follows that
$$
|f(1/p)|\le\bigl(\|b\|_q^q\bigr)^{1-1/p}\bigl(\|a\|_p^p)^{1/p}=\|a\|_p\|b\|_q.
$$
Since $f(1/p)=\tau(w^*ux_nvy_n)=\|ux_nvy_n\|_1$, we have $\|ux_nvy_n\|_1\le\|a\|_p\|b\|_q$.
Note that $ux_nvy_n\to u|a|v|b|=ab$ in the measure topology and
$$
\mu_t(ux_nvy_n)\le\mu_{t/2}(ux_n)\mu_{t/2}(vy_n)\le\mu_{t/2}(x_n)\mu_{t/2}(y_n)
\le\mu_{t/2}(a)\mu_{t/2}(b),
$$
\begin{align*}
\int_0^\infty\mu_{t/2}(a)\mu_{t/2}(b)\,dt&=2\int_0^\infty\mu_t(a)\mu_t(b)\,dt \\
&\le2\biggl[\int_0^\infty\mu_t(a)^p\,dt\biggr]^{1/p}
\biggl[\int_0^\infty\mu_t(b)^q\,dt\biggr]^{1/q}
=2\|a\|_p\|b\|_q<+\infty.
\end{align*}
By Proposition \ref{P-4.21} we obtain
$\|ab\|_1=\lim_{n\to\infty}\|ux_nvy_n\|_1\le\|a\|_p\|b\|_q$ so that $ab\in L^1(M)$ and
$|\tau(ab)|\le\|ab\|_1$ follows from \eqref{F-5.3}.
\end{proof}

\begin{prop}\label{P-5.7}
Let $1\le p<\infty$ and $1/p+1/q=1$. If $a\in L^p(M)$ and $b\in L^q(M)$, then
$\tau(ab)=\tau(ba)$.
\end{prop}

\begin{proof}
By linearity we may assume that $a,b\ge0$. When $1<p<\infty$, with the spectral decompositions
$a=\int_0^\infty\lambda\,de_\lambda$ and $b=\int_0^\infty\lambda\,df_\lambda$ define
$$
a_n:=\int_{1/n}^n\lambda\,de_\lambda=qe_{[1/n,n]}(a),\qquad
b_n:=\int_{1/n}^n\lambda\,df_\lambda=be_{1/n,n]}(b).
$$
Since
$$
e_{[1/n,n]}(a)\le n^p\int_{1/n}^n\lambda^p\,de_\lambda\le n^pa^p,
$$
one has $\tau(e_{[1/n,n]}(a))<+\infty$, and similarly $\tau(e_{[1/n,n]}(b))<+\infty$. Let
$e:=e_{[1/n,n]}(a)\vee e_{[1/n,n]}(b)\in\Proj(M)$; then $\tau(e)<+\infty$. Since
$a_n,b_n\in eMe\subset\fN_\tau$, one has $\tau(a_nb_n)=\tau(b_na_n)$ by Lemma \ref{L-5.4}\,(3).
Since $\|a_n-a\|_p\to0$ and $\|b_n-b\|_1\to0$, we find by \eqref{F-5.5} that
\begin{align*}
|\tau(a_nb_n)-\tau(ab)|&\le|\tau((a_n-a)b_n)|+|\tau(a(b_n-b))| \\
&\le\|a_n-a\|_p\|b_n\|_q+\|a\|_p\|b_n-b\|_q\ \longrightarrow\ 0.
\end{align*}
Hence we have $\tau(a_nb_n)\to\tau(ab)$, and similarly $\tau(b_na_n)\to\tau(ba)$, so
$\tau(ab)=\tau(ba)$. When $p=1$ and $q=\infty$, since $a^{1/2},a^{1/2}b,ba^{1/2}\in L^2(M)$,
we have
$$
\tau(ab)=\tau(a^{1/2}a^{1/2}b)=\tau(a^{1/2}ba^{1/2})=\tau(ba^{1/2}a^{1/2})=\tau(ba).
$$
\end{proof}

\begin{prop}[Minkowski's inequality]\label{P-5.8}
Let $1\le p\le\infty$. For every $a,b\in L^p(M)$,
\begin{align}\label{F-5.6}
\|a+b\|_p\le\|a\|_p+\|b\|_p.
\end{align}
\end{prop}

\begin{proof}
Since the case $p=\infty$ is trivial, assume that $1\le p<\infty$ and $1/p+1/q=1$. It suffices
to show that
\begin{align}\label{F-5.7}
\|a\|_p=\sup\{|\tau(ac)|:c\in L^q(M),\,\|c\|_q\le1\}.
\end{align}
It follows from \eqref{F-5.5} that $\|a\|_p\ge\mbox{the RHS of \eqref{F-5.7}}$. To prove the
converse, let $a=w|a|$ be the polar decomposition of $a$, so $a^*=|a|w^*$. When $p=1$, let
$c:=w^*\in M$; then $\|c\|_\infty\le1$ and $\tau(ac)=\tau(ca)=\tau(|a|)=\|a\|_1$. When
$1<p<\infty$, let $c:=|a|^{p-1}w^*$. Since $|c^*|^2=|a|^{p-1}w^*w|a|^{p-1}=|a|^{2(p-1)}$,
we have $|c^*|=|a|^{p-1}$ and $|c^*|^q=|a|^p$, so $\|c\|_q=\|c^*\|_q=\|a\|_p^{p/q}$. Since
$ca=|a|^p$, it follows that $\tau(ac)=\tau(ca)=\|a\|_p^p$. We may assume that $\|a\|_p>0$
(i.e., $a\ne0$), and let $c_1:=\|a\|_p^{-p/q}c$. We then find that $\|c_1\|_q=1$ and
$\tau(ac_1)=\|a\|_p$. Hence \eqref{F-5.7} follows.
\end{proof}

\begin{remark}\label{R-5.9}\rm
As explained in \cite{FK} a more systematic approach to Minkowski's and H\"older's inequalities
is to develop majorization such as
\begin{align}
\int_0^s\mu_t(a+b)\,dt&\le\int_0^s\{\mu_t(a)+\mu_t(b)\}\,dt,\qquad s>0, \nonumber\\
\int_0^s\log\mu_t(ab)\,dt&\le\int_0^s\log\{\mu_t(a)\mu_t(b)\}\,dt,\qquad s>0, \label{F-5.8}
\end{align}
for $a,b$ in $\widetilde M$ (or its certain subclass). In particular, when $M=B(\cH)$, this
approach was fully adopted in \cite{Hi0}. The majorization theory is a major subject in
matrix theory, whose version in (semi)finite von Neumann algebras is also worth discussing
(see, e.g., \cite{HN1,HN2}).
\end{remark}

\begin{thm}\label{T-5.10}
For every $p\in[1,\infty]$, $L^p(M)$ is a Banach space with respect to the norm $\|\cdot\|_p$.
In particular, $L^2(M)$ is a Hilbert space with the inner product $\<a,b\>_\tau=\tau(a^*b)$.
Moreover, $M\cap L^1(M)$ is dense in $L^p(M)$ for any $p\in[1,\infty)$.
\end{thm}

\begin{proof}
We may assume that $1\le p<\infty$. From Lemma \ref{L-5.3}\,(1) and \eqref{F-5.6} it follows
that $L^p(M)$ is a normed space, by noting that $\|a\|_p=0$ $\implies$ $\mu_t(a)=0$ for all
$t>0$ $\implies$ $a=0$. So it remains to prove the completeness of $\|\cdot\|_p$. Let $\{a_n\}$
be a Cauchy sequence in $L^p(M)$. By \eqref{F-4.7} note that
\begin{align}\label{F-5.9}
\|a_m-a_n\|_p^p=\tau(|a_m-a_n|^p)=\int_0^\infty\mu_t(a_m-a_n)^p\,dt.
\end{align}
For any $\delta>0$, since
$$
\delta\mu_\delta(a_m-a_n)^p\le\int_0^\delta\mu_t(a_m-a_n)^p\,dt
\le\int_0^\infty\mu_t(a_m-a_n)^p\,dt\ \longrightarrow\ 0\quad\mbox{as $m,n\to\infty$},
$$
it follows from Lemma \ref{L-4.16}\,(4) that $\{a_n\}$ is Cauchy in $\widetilde M$, so by
Theorem \ref{T-4.12} there is an $a\in\widetilde M$ such that $a_n\to a$ in the measure
topology. For any $\eps>0$ there is an $n_0\in\bN$ such that $\|a_m-a_n\|_p\le\eps$ for all
$m,n\ge n_0$. Since $a_m-a_n\to a-a_n$ as $m\to\infty$ in the measure topology, we have
$\mu_t(a_m-a_n)\to\mu_t(a-a_n)$ a.e.\ by Lemma \ref{L-4.19}\,(2). By Fatou's lemma and
\eqref{F-5.9} we have
$$
\int_0^\infty\mu_t(a-a_n)^p\,dt\le\liminf_{m\to\infty}\int_0^\infty\mu_t(a_m-a_n)^p\,dt
\le\eps^p,\qquad n\ge n_0.
$$
Therefore, if $n\ge n_0$, then $a-a_n\in L^p(M)$ and $\|a-a_n\|_p\le\eps$, which implies that
$a\in L^p(M)$ and $\|a-a_n\|_p\to0$.

When $p=2$, it is clear that $\<a,b\>_\tau:=\tau(a^*b)$ is an inner product on $L^2(M)$.
Since $\|a\|_2=\tau(|a|^2)^{1/2}=\<a,a\>_\tau$, it follows that $L^2(M)$ is a Hilbert space.

If $a\in M\cap L^1(M)$, then it is clear that $|a|^p\le\|a\|^{p-1}|a|$ and so
$\tau(|a|^p)<+\infty$. Hence $M\cap L^1(M)\subset L^p(M)$ for all $p\in[1,\infty)$. Let
$1\le p<\infty$ and $a\in L^p(M)$ with $a=w|a|$ and $|a|=\int_0^\infty\lambda\,de_\lambda$.
For each $n\in\bN$ let $a_n:=w\int_{1/n}^n\lambda\,de_\lambda$; then
$$
|a_n|=\int_{1/n}^n\lambda\,de_\lambda\le n^{p-1}\int_{1/n}^n\lambda^p\,de_\lambda
\le n^{p-1}|a|^p,
$$
so that $\tau(|a_n|)<+\infty$ and hence $a_n\in M\cap L^1(M)$. Note that
$$
|a-a_n|^p=\int_{(0,1/n)}\lambda^p\,de_\lambda+\int_{(n,\infty)}\lambda^p\,de_\lambda
\le|a|^p\in L^1(M),
$$
$$
\bigg\|\int_{(0,1/n)}\lambda^p\,de_\lambda\bigg\|_\infty\le(1/n)^p\ \longrightarrow\ 0,\quad
\biggl(\int_{(n,\infty)}\lambda^p\,de_\lambda\biggr)e_n=0,\quad
\tau(e_n^\perp)\ \longrightarrow\ 0,
$$
which imply that $|a-a_n|^p\to0$ as $n\to\infty$ in the measure topology, so
$\|a-a_n\|_p^p=\tau(|a-a_n|^p)\to0$ by Proposition \ref{P-4.21}. Hence the last assertion
follows.
\end{proof}

The following are famous inequalities for the $\|\cdot\|_p$-norms with $1<p<\infty$, whose
proofs are omitted here. But Clarkson's inequality will be proved in Proposition \ref{P-11.26}
for more general Haagerup's $L^p$-spaces.

\begin{prop}\label{P-5.11}
When $2\le p<\infty$, for every $a,b\in L^p(M)$,
$$
\|a+b\|_p^p+\|a-b\|_p^p\le2^{p-1}\bigl(\|a\|_p^p+\|b\|_p^p\bigr).\qquad
\mbox{(\emph{Clarkson's inequality})}
$$
When $1<p\le2$ and $1/p+1/q=1$, for every $a,b\in L^p(M)$,
$$
\|a+b\|_p^q+\|a-b\|_p^q\le2\bigl(\|a\|_p^p+\|b\|_P^p\bigr)^{q/p}.\qquad
\mbox{(\emph{McCarthy's inequality})}
$$
\end{prop}

Recall that a Banach space $X$ is said to be \emph{uniformly convex} if, for any
$\eps\in(0,2)$,
$$
\delta(\eps):=\inf\biggl\{1-\bigg\|{x+y\over2}\bigg\|:
x,y\in X,\,\|x\|=\|y\|=1,\,\|x-y\|\ge\eps\biggr\}>0.
$$
As is well-known, a uniformly convex Banach space $X$ is reflexive, i.e., $X^{**}=X$.

The next result follows from Proposition \ref{P-5.11}.

\begin{cor}\label{C-5.12}
When $1<p<\infty$, $L^p(M)$ is uniformly convex (hence reflexive).
\end{cor}

The last theorem of the section is the $L^p$-$L^q$-duality for $L^p(M)$, $1\le p<\infty$.

\begin{thm}\label{T-5.13}
Let $1\le p<\infty$ and $1/p+1/q=1$. Then the dual Banach space of $L^p(M)$ is $L^q(M)$ under
the duality pairing $(a,b)\in L^p(M)\times L^q(M)\mapsto\tau(ab)\in\bC$.
\end{thm}

\begin{proof}
First, assume that $1<p<\infty$ and $1/p+1/q=1$. Define $\Phi:L^q(M)\to L^p(M)^*$ by
$$
\Phi(b)(a):=\tau(ab),\qquad a\in L^p(M),\ b\in L^q(M),
$$
where $\Phi(b)\in L^p(M)^*$ is seen from \eqref{F-5.5}. Clearly $\Phi$ is linear. Furthermore,
it follows from \eqref{F-5.7} (with $p,q$ exchanged) that $\|\Phi(b)\|=\|b\|_q$, so $\Phi$ is
a linear isometry. Now, we prove that $\Phi$ is surjective. Since $\Phi(L^q(M))$ is
norm-closed in $L^p(M)^*$ and $L^p(M)^*$, as well as $L^p(M)$, is reflexive by Corollary
\ref{C-5.12} we have
$$
\Phi(L^q(M))=\overline{\Phi(L^q(M))}^w=\overline{\Phi(L^q(M))}^{w*}
\ \ \mbox{(the weak* closure)}.
$$
Let $a\in L^p(M)$ and assume that $\Phi(b)(a)=\tau(ab)=0$ for all $b\in L^q(M)$. Then $a=0$
follows from \eqref{F-5.7}. This implies that $\overline{\Phi(L^q(M))}^{w*}=L^p(M)^*$, so
$\Phi(L^q(M))=L^p(M)^*$.

Next, assume that $p=1$ and $q=\infty$. Define $\Psi:L^1(M)\to M^*$ by
$$
\Psi(a)(x):=\tau(ax),\qquad a\in L^1(M),\ x\in M=L^\infty(M).
$$
Then $\Psi$ is a linear isometry since $\|\Psi(a)\|=\|a\|_1$ thanks to \eqref{F-5.7}. Let
$a\in L^1(M)$ with $a\ge0$. Let $a=\int_0^\infty\lambda\,de_\lambda$ and
$a_n:=\int_0^n\lambda\,de_\lambda$. For every net $\{x_\alpha\}$ in $M_+$ such that
$x_\alpha\nearrow x\in M_+$, note by Proposition \ref{P-5.7} that
$\tau(ax_\alpha)=\tau(a^{1/2}x_\alpha a^{1/2})\le\tau(a^{1/2}xa^{1/2})=\tau(ax)$. Moreover,
one has
\begin{align*}
\tau(ax)-\tau(ax_\alpha)
&\le\tau((a-a_n)x)+\tau(a_n(x-x_\alpha))+\tau((a_n-a)x_\alpha) \\
&\le2\|a_n-a\|_1\|x\|_\infty+\tau(a_n^{1/2}xa_n^{1/2})-\tau(a_n^{1/2}x_\alpha a_n^{1/2}).
\end{align*}
Since $a_n^{1/2}x_\alpha a_n^{1/2},a_n^{1/2}xa_n^{1/2}\in M_+$ and
$a_n^{1/2}x_\alpha a_n^{1/2}\nearrow a_n^{1/2}xa_n^{1/2}$ for any $n$ fixed, the normality of
$\tau$ gives
$$
\tau(ax)-\sup_\alpha\tau(ax_\alpha)\le2\|a_n-a\|_1\|x\|_\infty\ \longrightarrow\ 0
\quad\mbox{as $n\to\infty$}.
$$
Therefore, $\Psi(a)(x_\alpha)=\tau(ax_\alpha)\nearrow\tau(ax)=\Psi(a)(x)$, which implies that
$\Psi(a)\in M_*$. Hence $\Psi$ is a linear isometry from $L^1(M)$ to $M_*$, so
$\Psi(L^1(M))=\overline{\Psi(L^1(M))}^w$. Let $x\in M$ and assume that $\Psi(a)(x)=\tau(ax)=0$
for all $x\in L^1(M)$. For every $e\in\Proj(M)$ with $\tau(e)<+\infty$, since $x^*e\in L^1(M)$,
it follows that $\tau(x^*ex)=0$. We can let $e\nearrow1$ to have $x=0$. This implies that
$\overline{\Psi(L^1(M))}^w=M_*$, so $\Psi(L^1(M))=M_*$. By taking the dual map $\Phi:=\Psi^*$
we have an isometric isomorphism $\Phi:M=(M_*)^*\to L^1(M)^*$ and
$\Phi(x)(a)=\Psi(a)(x)=\tau(ax)$ for $a\in L^1(M)$ and $x\in M$.
\end{proof}

\begin{cor}\label{C-5.14}
We have $M_*=L^1(M)$ under the correspondence $\ffi\in M_*\leftrightarrow a\in L^1(M)$ given
by $\ffi(x)=\tau(ax)$, $x\in M$. Moreover, $\ffi\in M_*^+$ $\iff$ $a\in L^1(M)_+$.
\end{cor}

\begin{proof}
The first assertion was shown in the proof of the case $p=1$ of Theorem \ref{T-5.13}. For the
latter assertion, if $a\ge0$, then $\ffi(x)=\tau(a^{1/2}xa^{1/2})\ge0$ for all $x\in M_+$, so
$\ffi\ge0$. Conversely, if $\ffi\ge0$, then for every $x\in M_+$,
$$
\tau(ax)=\ffi(x)=\overline{\ffi(x)}=\tau(xa^*)=\tau(a^*x),
$$
which implies that $a=a^*$. For each $n\in\bN$ let $e_n$ be the spectral projection of $a$
corresponding to $[-n,-1/n]$. Since $ae_n\le-(1/n)e_n$, one has
$0\le\tau(ae_n)\le-(1/n)\tau(e_n)$, so $\tau(e_n)=0$ and hence $e_n=0$ for all $n$. Hence
$a\ge0$.
\end{proof}

The $a$ in $L^1(M,\tau)$ given in Corollary \ref{C-5.14} is often called the
\emph{Radon-Nikodym derivative} of $\ffi$ with respect to $\tau$ and denoted by $d\ffi/d\tau$.

\begin{remark}\label{R-5.15}\rm
For $x\in M$ consider the left multiplication $\pi(x)a:=xa$ for $a\in L^2(M)$. Since
$\|xa\|_2\le\|x\|\,\|a\|_2$, $\pi(x)\in B(L^2(M))$. Since
$\<a,xb\>_\tau=\tau(a^*xb)=\tau((x^*a)^*b)=\<x^*a,b\>_\tau$ for $a,b\in L^2(M)$,
$\pi(x^*)=\pi(x)^*$. If $\pi(x)=0$, then $\pi(x^*x)=\pi(x)^*\pi(x)=0$ and for every
$e\in\Proj(M)$ with $\tau(e)<+\infty$, $0=\<e,x^*xe\>_\tau=\|xe\|_2^2$ and hence $xe=0$.
Letting $e\nearrow1$ gives $x=0$. Thus, $\pi$ is a faithful representation of $M$ on $L^2(M)$.
We further note that $L^2(M)$ is the completion of $(\fN_\tau,\<\cdot,\cdot\>_\tau)$ and
$$
(\pi(M),L^2(M),J=\,^*,L^2(M)_+)
$$
is the standard form of $M$. For $x\in M$, $J\pi(x)J$ in $\pi(M)'=J\pi(M)J$ acts as the right
multiplications $\pi_r(x)a:=ax^*$, $a\in L^2(M)$ (in fact, $J\pi(x)Ja=Jxa^*ax^*$). By
Corollary \ref{C-5.14}, for every $\ffi\in M_*^+$ there is an $a\in L^1(M)_+$ such that
$\ffi(x)=\tau(ax)=\<a^{1/2},xa^{1/2}\>_\tau$, $x\in M$, so that $a^{1/2}\in L^2(M)_+$ is the
vector representative of $\ffi$.
\end{remark}

\section{Conditional expectations and generalized conditional expectations}

The notion of conditional expectations is essential in probability theory. Let $(X,\cX,\mu)$
be a probability space and $\cY$ be a sub-$\sigma$-algebra of $\cX$. For every
$f\in L^1(X,\cX,\mu)$ we have a unique $\widetilde f\in L^1(X,\cY,\mu)$ such that
$\int_B\widetilde f\,d\mu=\int_Bf\,d\mu$ for all $B\in\cY$, which is called the
\emph{conditional expectation} of $f$ with respect to $\cY$ and denoted by $E_\mu(f|\cY)$.
If $f\in L^1(X,\cX,\mu)$ and $g\in L^\infty(X,\cY,\mu)$, then $E_\mu(fg|\cY)=E_\mu(f|\cY)g$.
In this section we discuss some non-commutative versions of conditional expectations in the
von Neumann algebra setting.

\subsection{Conditional expectations}

Before entering the subject of the section let us recall different positivity notions of
linear maps between general $C^*$-algebras.

\begin{definition}\label{D-6.1}\rm
Let $\cA$ and $\cB$ be $C^*$-algebras and $\Phi:\cA\to\cB$ be a linear map. Define:
\begin{itemize}
\item $\Phi$ is \emph{positive} if $\Phi(a^*a)\ge0$ for all $a\in\cA$.
\item $\Phi$ is a \emph{Schwarz map} if $\Phi(a^*a)\ge\Phi(a)^*\Phi(a)$ for all $a\in\cA$.
\item For each $n\in\bN$, $\Phi$ is \emph{$n$-positive} if
$\Phi^{(n)}=\Phi\otimes\id_n:\bM_n(\cA)\to\bM_n(\cB)$ is positive, where
$\bM_n(\cA)=\cA\otimes\bM_n(\bC)$ is the $C^*$-algebra tensor product of $\cA$ with the
$n\times n$ matrix algebra $\bM_n(\bC)$, whose elements are represented as $n\times n$ matrices
$[a_{ij}]_{i,j=1}^n$ of $a_{ij}\in\cA$), and $\Phi^{(n)}$ is defined by
$\Phi^{(n)}([a_{ij}]):=[\Phi(a_{ij})]$. (Of course, $1$-positivity means positivity.)
\item $\Phi$ is \emph{completely positive} if it is $n$-positive for every $n\in\bN$.
\end{itemize}
\end{definition}

Obviously we have that completely positive\,$\implies$\,$n$-positive\,$\implies$\,positive. A
few basic properties on the notions are summarized in the next proposition.

\begin{prop}\label{P-6.2}
Let $\cA$ and $\cB$ be unital $C^*$-algebras and $\Phi:\cA\to\cB$ be a linear map. Then:
\begin{itemize}
\item[\rm(1)] If $\Phi$ is a Schwarz map, then it is positive. If $\Phi$ is unital and
$2$-positive, then it is a Schwarz map.
\item[\rm(2)] For each $n\in\bN$, $\Phi$ is $n$-positive if and only if
$\sum_{i,j=1}^nb_i^*\Phi(a_i^*a_j)b_j\ge0$ for all $a_i\in\cA$ and $b_i\in\cB$ ($i=1,\dots,n$).
\item[\rm(3)] If $\cA$ or $\cB$ is commutative, then any positive $\Phi$ is completely positive.
\item[\rm(4)] If $\Phi$ is positive, then it is bounded with $\|\Phi\|=\|\Phi(1)\|$. Hence, if
$\Phi$ is unital (i.e., $\Phi(1)=1$) and positive, then $\|\Phi\|=1$.
\end{itemize}
\end{prop}

The next lemma will be useful in proving the above proposition.

\begin{lemma}\label{L-6.3}
For any $a,b\in B(\cH)$ with $a\ge0$, $\begin{bmatrix}a&b^*\\b&1\end{bmatrix}\ge0$ in
$B(\cH\oplus\cH)$ if and only if $a\ge b^*b$.
\end{lemma}

\begin{proof}
If $a\ge b^*b$, then
$$
\begin{bmatrix}a&b^*\\b&1\end{bmatrix}\ge\begin{bmatrix}b^*b&b^*\\b&1\end{bmatrix}
=\begin{bmatrix}b&1\\0&0\end{bmatrix}^*\begin{bmatrix}b&1\\0&0\end{bmatrix}\ge0.
$$
Conversely, assume that $\begin{bmatrix}a&b^*\\b&1\end{bmatrix}\ge0$. If $a$ is invertible,
then
$$
0\le\begin{bmatrix}a^{-1/2}&0\\0&1\end{bmatrix}
\begin{bmatrix}a&b^*\\b&1\end{bmatrix}
\begin{bmatrix}a^{-1/2}&0\\0&1\end{bmatrix}
=\begin{bmatrix}1&a^{-1/2}b^*\\ba^{-1/2}&1\end{bmatrix}.
$$
It is easy to verify that this is equivalent to $w:=ba^{-1/2}$ is a contraction, which implies
that $b^*b=(wa^{1/2})^*(wa^{1/2})=a^{1/2}w^*wa^{1/2}\le a$. When $a$ is not invertible, since
$\begin{bmatrix}a+\eps1&b^*\\b&1\end{bmatrix}\ge0$ for every $\eps>0$, we have
$b^*b\le a+\eps 1$. Letting $\eps\searrow0$ gives $b^*b\le a$.
\end{proof}

\begin{proof}[Proof of Proposition \ref{P-6.2}]
We may assume that $\cA$ and $\cB$ are unital $C^*$-subalgebras of $B(\cH)$ and $B(\cK)$ on
Hilbert spaces $\cH,\cK$, respectively.

(1)\enspace
The first assertion is obvious. Assume that $\Phi$ is unital and $2$-positive. For any
$a\in\cA$, since $\begin{bmatrix}a^*a&a^*\\a&1\end{bmatrix}\ge0$ by Lemma \ref{L-6.3}, we have
$\begin{bmatrix}\Phi(a^*a)&\Phi(a)^*\\\Phi(a)&1\end{bmatrix}\ge0$ (here $\Phi(a^*)=\Phi(a)^*$
for positive $\Phi$ is standard), so that $\Phi(a^*a)\ge\Phi(a)^*\Phi(a)$ by Lemma \ref{L-6.3}
again.

(2)\enspace
Since $[a_{ij}]^*[a_{ij}]=\sum_{k=1}^n[a_{ki}^*a_{kj}]_{i,j=1}^n$ for
$[a_{ij}]_{i,j=1}^n\in\bM_n(\cA)$, we see that $\Phi$ is $n$-positive if and only if
$[\Phi(a_i^*a_j)]_{i,j=1}^n\ge0$ for all $a_i\in\cA$ ($i=1,\dots,n$). If $\Phi$ is $n$-positive,
then
$$
\sum_{i,j=1}^nb_i^*\Phi(a_ia_j^*)b_j=\begin{bmatrix}b_1\\\vdots\\b_n\end{bmatrix}^*
\bigl[\Phi(a_i^*a_j)\bigr]\begin{bmatrix}b_1\\\vdots\\b_n\end{bmatrix}\ge0
$$
for all $a_i\in\cA$ and $b_i\in\cB$. Conversely, assume that the above inequality holds for all
$a_i,b_i$. For any cyclic representation $\{\pi_0,\cK_0,\xi_0\}$ of $\cB$ one has
$$
\sum_{i,j=1}^n\<\pi_0(b_i)\xi_0,\pi_0(\Phi(a_i^*a_j))\pi(b_j)\xi_0\>
=\Bigl\<\xi_0,\pi_0\Biggl(\sum_{i,j=1}^nb_i\Phi(a_i^*a_j)b_j\Biggr)\xi_0\Bigr\>\ge0
$$
for all $b_i\in\cB$. This implies that
$(\pi_0\otimes\id_n)([\phi(a_i^*a_j)])=[\pi_0(\Phi(a_i^*a_j))]\ge0$. Note that the
representation $\cB$ in $B(\cK)$ is represented as the direct sum $\pi=\bigoplus_k\pi_k$ of
cyclic representations $\{\pi_k,\cK_k,\xi_k\}$ of $\cB$. Then it is immediate to see that
$\widetilde\pi=\bigoplus_k(\pi_k\otimes\id_n)$ is a faithful representation of $\bM_n(\cB)$.
From the above discussion it follows that $\widetilde\pi([\Phi(a_i^*a_j)])\ge0$, and hence
$[\Phi(a_i^*a_j)]\ge0$.

(3)\enspace
Assume that $\cA$ is commutative; then by the Gelfand-Naimark theorem we may write $\cA=C(X)$,
the complex continuous functions on a compact Hausdorff space $X$. As noted in the proof of (2)
it suffices to show that $\sum_{i,j=1}^n\<\xi_i,\Phi(\overline f_if_j)\xi_j\>\ge0$ for all
$f_i\in C(X)$ and $\xi\in\cK$ ($i=1,\dots,n$). By the Riesz-Markov theorem there are Radon
measures $\mu_{ij}$ ($i,j=1,\dots,n$) on the Borel space $(X,\cB_X)$ such that
$\<x_i,\Phi(f)\xi_j\>=\int_Xf\,d\mu_{ij}$ for all $f\in C(X)$. Choose a positive Radon measure
$\mu$ on $(X,\cB_X)$ such that $\mu_{ij}\ll\mu$ (absolutely continuous) for all $i,j$, and let
$\phi_{ij}=d\mu_{ij}/d\mu\in L^1(X,\mu)$ be the Radon-Nikodym derivatives. For every
$c_1,\dots,c_n\in\bC$, since
$$
\int_Xf\Biggl(\sum_{i,j=1}^n\overline c_ic_j\phi_{ij}\Biggr)\,d\mu
=\sum_{i,j=1}^n\overline c_ic_j\int_Xf\,d\mu_{ij}
=\Bigl\<\sum_{i=1}^nc_i\xi_i,\Phi(f)\Biggl(\sum_{j=1}^nc_j\xi_j\Biggr)\Bigr\>\ge0
$$
for all $f\in C(X)_+$, we have $\sum_{i,j=1}^n\overline c_ic_j\phi_{ij}(x)\ge0$ for
$\mu$-a.e.\ $x\in X$. This implies that $\sum_{i,j=1}^n\overline{f_i(x)}f_j(x)\phi_{ij}(x)\ge0$
for $\mu$-a.e.\ $x\in X$, and hence
$$
\sum_{i,j=1}^n\<\xi_i,\Phi(\overline f_if_j)\xi_j\>
=\int_X\sum_{i,j=1}^n\overline{f_i(x)}f_j(x)\phi_{ij}(x)\,d\mu(x)\ge0.
$$

Next, assume that $\cB$ is commutative, so we write $\cB=C(X)$ as above. For every $a_i\in\cA$
and $f_i\in C(X)$ ($i=1,\dots,n$) and every $x\in X$ one has
\begin{align*}
\Biggl(\sum_{i,j=1}^n\overline f_i\Phi(a_i^*a_J)f_j\Biggr)(x)
&=\sum_{i,j=1}^n\overline{f_i(x)}\Phi(a_i^*a_j)(x)f_j(x) \\
&=\sum_{i,j=1}^n\Phi\bigl(\overline{f_i(x)}f_j(x)a_i^*a_j\bigr)(x) \\
&=\Phi\Biggl(\Biggl(\sum_{i=1}^nf_i(x)a_i\Biggr)^*\Biggl(\sum_{j=1}^nf_j(x)a_j
\Biggr)\Biggr)(x)\ge0,
\end{align*}
and hence $\Phi$ is $n$-positive by (2).

(4)\enspace
Assume first that $\Phi$ is unital and positive. For each unitary $u\in\cA$ let $\cA_0$ be the
commutative $C^*$-subalgebra of $\cA$ generated by $u,1$. Since $\Phi|_{\cA_0}$ is completely
positive by (3) and $\begin{bmatrix}1&u^*\\u&1\end{bmatrix}\ge0$ by Lemma \ref{L-6.3}, one has
$\begin{bmatrix}1&\Phi(u)^*\\\Phi(u)&1\end{bmatrix}\ge0$. Hence by Lemma \ref{L-6.3} again
one has $\Phi(u)^*\Phi(u)\le1$ and so $\|\Phi(u)\|\le1$. The Russo-Dye theorem says that the
unit ball $\{a\in\cA:\|a\|\le\}$ is the norm-closed convex hull of the unitaries of $\cA$,
see \cite[Proposition 1.1.12]{Ped} for instance. It thus follows that $\|\Phi(a)\|\le1$ for
all $a\in\cA$, $\|a\|\le1$, so that $\|\Phi\|=1$. Next, let $\Phi$ be a positive map. When
$\Phi(1)$ is invertible, define a unital map $\hat\Phi(a):=\Phi(1)^{-1/2}\Phi(a)\Phi(1)^{-1/2}$,
$a\in\cA$. The previous case implies that $\|\hat\Phi\|=1$ so that
$$
\|\Phi(a)|=\|\Phi(1)^{1/2}\hat\Phi(a)\Phi(1)^{1/2}\|
\le\|\Phi(1)\|\,\|\hat\Phi(a)\|\le\|\Phi(1)\|\,\|a\|.
$$
Hence $\|\Phi\|=\|\Phi(1)\|$ follows. When $\Phi(1)$ is not invertible, define
$\Phi_\eps(a):=\Phi(a)+\eps\omega(a)1$, where $\eps>0$ and $\omega$ is a state of $\cA$. Since
$\Phi_\eps(1)$ is invertible, one has $\|\Phi_\eps(a)\|\le\|\Phi_\eps(1)\|\,\|a\|$. Letting
$\eps\searrow0$ gives $\|\Phi(a)\|\le\|\Phi(1)\|\,\|a\|$, as desired.
\end{proof}

Concerning complete positive maps the most significant is the \emph{Stinespring representation
theorem} saying that, for any completely positive map $\Phi:\cA\to B(\cH)$, there exist a
representation $\{\pi,\cK\}$ of $\cA$ and a bounded operator $V:\cH\to\cK$ such that
$\Phi(a)=V^*\pi(a)V$ for all $a\in\cA$. Under the minimality condition
$\overline{\pi(\cA)V\cH}=\cK$, the triplet $\{\pi,\cK,V\}$ is unique up to a unitary
conjugation. Moreover, $\Phi$ is unital $\iff$ $V$ is an isometry. See \cite[Chap.~4]{Pau} for
the proof of the Stinespring theorem and other dilation theorems.

Now we turn to the subject of the present section. In \cite{Um} Umegaki introduced the notion
of conditional expectations in the finite von Neumann algebra setting as follows. Let $M$ be a
von Neumann algebra with a faithful normal tracial state $\tau$. In this case, by H\"older's
inequality in \eqref{F-5.5} the following are easy to show (exercises):
\begin{align*}
&M\subset L^p(M,\tau)\subset L^1(M,\tau),\qquad1<p<\infty, \\
&\|a\|_\infty\ge\|a\|_p\ge\|a\|_1,\qquad a\in\widetilde M,\ 1<p<\infty.
\end{align*}

\begin{thm}[Umegaki]\label{T-6.4}
Let $(M,\tau)$ be as stated above, and $N$ be a von Neumann subalgebra of $M$.
\begin{itemize}
\item[\rm(1)] There exists a unique linear map $E_\tau:M\to N$ such that
\begin{itemize}
\item[\rm(i)] $E_\tau(y)=y$ for all $y\in N$,
\item[\rm(ii)] $E_\tau(y_1xy_2)=y_1E_\tau(x)y_2$ for all $x\in M$ and $y_1,y_2\in N$,
\item[\rm(iii)] $\tau(E_\tau(x))=\tau(x)$ for all $x\in M$.
\end{itemize}
Moreover, $E_\tau$ satisfies the following properties as well:
\begin{itemize}
\item[\rm(iv)] $\|E_\tau(x)\|\le\|x\|$ for all $x\in M$,
\item[\rm(v)] $E_\tau$ is completely positive, in particular,
$E_\tau(x^*x)\ge E_\tau(x)^*E_\tau(x)$ for all $x\in M$,
\item[\rm(vi)] $E_\tau(x^*x)=0$ $\implies$ $x=0$,
\item[\rm(vii)] $E_\tau$ is normal, i.e., $x_\alpha\nearrow x$ in $M_+$ $\implies$
$E_\tau(x_\alpha)\nearrow E_\tau(x)$.
\end{itemize}

\item[\rm(2)] The map $E_\tau:M\to N$ uniquely extends to a linear map
$E_\tau:L^1(M,\tau)\to L^1(N,\tau)=L^1(N,\tau|_N)$ satisfying
\begin{itemize}
\item[\rm(i)$'$] $E_\tau(b)=b$ for all $b\in L^1(N,\tau)$,
\item[\rm(ii)$'$] $E_\tau(y_1ay_2)=y_1E_\tau(a)y_2$ for all $a\in L^1(M,\tau)$ and $y_1,y_2\in N$,
\item[\rm(iii)$'$] $\tau(E_\tau(a))=\tau(a)$ for all $a\in L^1(M,\tau)$,
\item[\rm(iv)$'$] $\|E_\tau(a)\|_1\le\|a\|_1$ for all $a\in L^1(M,\tau)$.
\end{itemize}

\item[\rm(3)] When restricted to the Hilbert space $L^2(M,\tau)$, $E_\tau$ is the orthogonal
projection from $L^2(M,\tau)$ onto the closed subspace $L^2(N,\tau)=L^2(N,\tau|_N)$.
\end{itemize}
\end{thm}

\begin{proof}
(1)\enspace
For every $x\in M$ define $\ffi_x(b):=\tau(xb)$ for $b\in L^1(N,\tau)$. Since
$|\ffi_x(b)|\le\|xb\|_1\le\|x\|\,\|b\|_1$, it follows that $\ffi_x\in L^1(N,\tau)^*=N$ by
Theorem \ref{T-5.13} or Corollary \ref{C-5.14}, so there is a unique $y_x\in N$ with
$\|y_x\|\le\|x\|$ such that $\ffi_x(b)=\tau(y_xb)$ for all $y\in L^1(N,\tau)$. By letting
$E_\tau(x):=y_x$ we have a map $E_\tau:M\to N$, which is clearly linear and satisfies
(i)--(iii). Indeed, (i) and (iii) are obvious and (ii) follows since, for $x\in M$ and
$y_1,y_2\in N$,
$$
\tau(y_1xy_2b)=\tau(xy_2by_1)=\tau(E_\tau(x)y_2by_1)=\tau(y_1E_\tau(x)y_2b),
\qquad b\in L^1(N,\tau).
$$
To show the uniqueness, assume that $E:M\to N$ is a linear map satisfying (i)--(iii). Then
for every $x\in M$ and $b\in L^1(N,\tau)$ with $b=w|b|=w\int_0^\infty\lambda\,de_\lambda$,
we have
$$
\tau(xb)=\lim_{n\to\infty}\tau(xy_n)=\lim_{n\to\infty}\tau(E(xy_n))
=\lim_{n\to\infty}\tau(E(x)y_n)=\tau(E(x)b),
$$
where $y_n:=w\int_0^n\lambda\,de_\lambda\in N$. Hence $E=E_\tau$ follows.

We show (iv)--(vii). (iv) is $\|y_x\|\le\|x\|$ noted above. From Corollary \ref{C-5.14} it
follows that $E_\tau$ is positive. Furthermore, for every $x_i\in M$ and $y_i\in N$
($1\le i\le n$) we have
\begin{align}\label{F-6.1}
\sum_{i,j=1}^ny_i^*E_\tau(x_i^*x_j)y_j=\sum_{i,j=1}^nE_\tau(y_i^*x_i^*x_jy_j)
=E_\tau\Biggl(\Biggl(\sum_{i,j=1}^nx_iy_j\Biggr)^*\Biggl(\sum_{j=1}^nx_jy_j\Biggr)\Biggr)
\ge0,
\end{align}
which implies (v). If $E_\tau(x^*x)=0$, then $\tau(x^*x)=\tau(E_\tau(x^*x))=0$ so that $x=0$.
If $x_\alpha\nearrow x$ in $M_+$, then $E_\tau(x_\alpha)\nearrow y\le E_\tau(x)$ for some
$y\in N_+$. Since $\tau(x_\alpha)=\tau(E_\tau(x_\alpha))\nearrow\tau(y)$, we have
$\tau(y)=\tau(x)=\tau(E_\alpha(x))$, so $y=E_\tau(x)$. Hence $E_\tau(x_\alpha)\nearrow E(x)$
follows.

(2)\enspace
We first prove that $\|E_\tau(x)\|_1\le\|x\|_1$ for all $x\in M$. For $x\in M$ let $x=v|x|$
and $E_\tau(x)=w|E_\tau(x)|$ be the polar decompositions. We have
\begin{align*}
\|E_\tau(x)\|_1&=\tau(w^*E_\tau(x))=\tau(E_\tau(w^*x))=\tau(w^*x) \\
&=\tau(w^*v|x|)\le\|vw\|\,\|x\|_1\le\|x\|_1
\end{align*}
thanks to \eqref{F-5.5} and Lemma \ref{L-5.3}\,(2). Since $M$ is dense in $L^1(M)$ by Theorem
\ref{T-5.10}, we can uniquely extend $E_\tau$ to $E_\tau:L^1(M,\tau)\to L^1(N,\tau)$ by
continuity. Then (i)$'$—(iv)$'$ are verified by simple arguments taking limits, whose details
are omitted.

(3)\enspace
For every $x\in M$ and $y\in N$ note that
\begin{align*}
&\|E_\tau(x)\|_2=\tau(E_\tau(x)^*E_\tau(x))^{1/2}\le\tau(E_\tau(x^*x))^{1/2}
=\tau(x^*x)^{1/2}=\|x\|_2, \\
&\<y,E_\tau(x)\>_\tau=\tau(y^*E_\tau(x))=\tau(y^*x)=\<y,x\>_\tau.
\end{align*}
Since $M$ and $N$ are dense in $L^2(M,\tau)$ and $L^2(N,\tau)$ respectively, we see by
arguments taking limits that the $E_\tau$ extended to $L^1(M,\tau)$ in (2) maps $L^2(M,\tau)$
to $L^2(N,\tau)$, and the above inequality and equality are extended to all $x\in L^2(M,\tau)$
and $y\in L^2(N,\tau)$. Hence the assertion follows.
\end{proof}

The map $E_\tau$ given in Theorem \ref{T-6.4} is called the \emph{conditional expectation}
from $M$ onto $N$ with respect to $\tau$.

In \cite{To0,To1} Tomiyama characterized conditional expectations from a unital $C^*$-algebra
onto a $C^*$-subalgebra in terms of norm one projections.

\begin{thm}[Tomiyama]\label{T-6.5}
Let $\cA$ be a $C^*$-algebra and $\cB$ be a $C^*$-subalgebra of $\cA$. A norm one projection
$E:\cA\to\cB$ satisfies the properties (ii) and (v) of Theorem \ref{T-6.4}.
\end{thm}

\begin{proof}
For simplicity let us prove the case where $\cA$ is a von Neumann algebra and $\cB$ is a unital
von Neumann subalgebra of $\cA$. Consider the dual map $E^*:\cB^*\to\cA^*$. If $\psi\in\cB^*$
and $\psi\ge0$, then $(E^*\psi)(1)=\psi(E(1))=\psi(1)=\|\psi\|\ge\|E^*\psi\|$, which implies
that $E^*\psi\ge0$. Hence for every $x\in\cA_+$ we have $\psi(E(x))=(E^*\psi)(x)\ge0$ for all
positive $\psi\in\cB^*$, so $E(x)\ge0$ follows. To prove (ii), it suffices to show that
$E(px)=pE(x)$ and $E(xp)=E(x)p$ for all $x\in\cA_+$ and $p\in\Proj(\cB)$. We may and do assume
that $0\le x\le1$. Since $0\le E(pxp)\le E(p)=p$, one has
\begin{align}\label{F-6.2}
E(pxp)=pE(pxp)p,\qquad\mbox{similarly}\quad
E(p^\perp xp^\perp)=p^\perp E(p^\perp xp^\perp)p^\perp.
\end{align}
Let $y:=E(pxp^\perp)$. For any $\lambda\in\bC$ and any $\xi\in p\cH$ with $\|\xi\|=1$, one has
\begin{align*}
|\<\xi,y\xi\>+\lambda|^2&\le\|y+\lambda p\|^2=\|E(pxp^\perp+\lambda p)\|^2 \\
&\le\|pxp^\perp+\lambda p\|^2
=\|(pxp^\perp+\lambda p)(p^\perp xp+\overline\lambda p)\| \\
&=\|pxp^\perp xp+|\lambda|^2p\|\le1+|\lambda|^2
\end{align*}
so that
$$
|\<\xi,y\xi\>|^2+2\Re(\overline\lambda\<\xi,y\xi\>)+|\lambda|^2\le1+|\lambda|^2,
$$
which implies that $\<\xi,y\xi\>=0$ for all $\xi\in p\cH$. Therefore, $pyp=0$. By a similar
argument with $\xi\in p^\perp\cH$, one has $p^\perp yp^\perp=0$ as well, so
$y=pyp^\perp+p^\perp yp=py+yp$. Hence, for any $\lambda>0$ one has
\begin{align*}
(\lambda+1)\|p^\perp yp\|&\le\|pyp^\perp+(\lambda+1)p^\perp yp\|
=\|py+yp+\lambda p^\perp yp\| \\
&=\|y+\lambda p^\perp yp\|=\|E(pxp^\perp+\lambda p^\perp yp)\| \\
&\le\|pxp^\perp+\lambda p^\perp yp\|
=\max\{\|pxp^\perp\|,\lambda\|p^\perp yp\|\}
\end{align*}
which implies that $p^\perp yp=0$. Since $pyp=p^\perp yp^\perp=p^\perp yp=0$, we find that
$y=pyp^\perp$. Since
$$
E(x)=E(pxp)+y+y^*+E(p^\perp xp^\perp),
$$
we have by \eqref{F-6.2},
$$
pE(x)p^\perp=pyp^\perp+py^*p^\perp=y=E(pxp^\perp),\qquad pE(x)p=E(pxp).
$$
Therefore,
$$
E(px)=E(pxp^\perp)+E(pxp)=pE(x)p^\perp+pE(x)p=pE(x),
$$
and (ii) of Theorem \ref{T-6.4} has been shown. Then (v) is shown in the same way as in \eqref{F-6.1}.
\end{proof}

Takesaki \cite{Ta1} presented a necessary and sufficient condition of a von Neumann subalgebra
onto which the conditional expectation exists for a given faithful normal state (or weight).

\begin{thm}[Takesaki]\label{T-6.6}
Let $M$ be a von Neumann algebra and $N$ be a von Neumann subalgebra of $M$. Let $\omega$ be
a faithful normal state on $M$. Then the following conditions are equivalent:
\begin{itemize}
\item[\rm(a)] $N$ is globally invariant under the modular automorphism group $\sigma_t^\omega$,
i.e., $\sigma_t^\omega(N)=N$ for all $t\in\bR$;
\item[\rm(b)] there exists a conditional expectation (i.e., a normal norm one projection)
$E:M\to N$ such that $\omega=\omega\circ E$ on $M$.
\end{itemize}
\end{thm}

The equivalence of (a) and (b) was more generally proved in \cite{Ta1} when $\omega$ is a
faithful semifinite normal weight on $M$ which is also semifinite on $N$, where
$\omega=\omega\circ E$ in (b) holds on $\fM_\omega$ that is defined as in Lemma \ref{L-5.4}
with $\omega$ in place of $\tau$ (see Definition \ref{D-7.1} of Sec.~7.1).

\begin{proof}
By taking the GNS representation of $M$ with respect to $\omega$, we may assume that $M$ is
represented on a Hilbert space $\cH$ with a cyclic and separating vector $\Omega$ for $M$
such that $\omega(x)=\<\Omega,x\Omega\>$, $x\in M$. Consider the modular operator $\Delta$
and the modular conjugation $J$ with respect to $\omega$ so that $S=J\Delta^{1/2}$ (see
Sec.~2.1). Let $\cH_N:=\overline{N\Omega}$ and $P$ be the projection onto $\cH_N$; then
$P\in N'$ and $\Omega$ is a cyclic and separating vector for $NP\cong N|_{\cH_N}$ representing
$\omega|_N$. So the modular operator $\Delta_N$ and the modular conjugation $J_N$ with
respect to $\omega|_N$ (or $\Omega\in\cH_N$) are given on $\cH_N$.

Assume (a); then $\sigma_t^{\omega|_N}=\sigma_t^\omega|_N$ by the KMS condition (Theorem
\ref{T-2.14}). Hence for every $y\in N$ and $t\in\bR$,
$$
\Delta^{it}y\Omega=\Delta^{it}y\Delta^{-it}\Omega=\sigma_t^\omega(y)\Omega
=\sigma_t^{\omega|_N}(y)\Omega=\Delta_N^{it}y\Omega,
$$
which implies that $\Delta_N^{it}=\Delta^{it}|_{\cH_N}$, $t\in\bR$, so
$\Delta_N=\Delta|_{\cH_N}$. Moreover, for every $y\in N$,
$$
J_N\Delta_N^{1/2}y\Omega=y^*\Omega=J\Delta^{1/2}y\Omega=J\Delta_N^{1/2}y\Omega,
$$
which implies that $J_N=J|_{\cH_N}$ and $JP=PJ$. Now, let $x\in M$ and set
$\xi_x:=PJx\Omega\in\cH_N$. Since $P\in N'$ and $JxJ\in M'$, one has for any $y\in N$,
$$
\|y\xi_x\|=\|PyJxJ\Omega\|=\|PJxJy\Omega\|\le\|x\|\,\|y\Omega\|.
$$
Hence one can define an $x'\in B(\cH_N)$ with $\|x'\|\le\|x\|$ so that $x'(y\Omega)=y\xi_x$
for all $y\in N$. Then it is easy to verify that $x'\in(N|_{\cH_N})'$. Moreover, one has
$$
Px\Omega=J_NPJx\Omega=J_N\xi_x=J_Nx'\Omega=J_Nx'J_N\Omega.
$$
Since $J_Nx'J_N\in J_N(N|_{\cH_N})'J_N=N|_{\cH_N}$ by Tomita's theorem (Theorem \ref{T-2.2})
for $N|_{\cH_N}$, there exists an $E(x)\in N$ such that $J_Nx'J_N=E(x)|_{\cH_N}$, so
$\|E(x)\|=\|x'\|\le\|x\|$ and $Px\Omega=E(x)\Omega$. Since $\Omega$ is separating for
$N|_{\cH_N}$, $E(x)\in N$ is uniquely determined by the equality $E(x)\Omega=Px\Omega$, so
$E:M\to N$ is a linear map and $E(y)=y$ for all $y\in N$. Furthermore, for every $x\in M$,
$$
\omega(E(x))=\<\Omega,E(x)\Omega\>=\<\Omega,Px\Omega\>=\<\Omega,x\Omega\>=\omega(x).
$$
Therefore, $E$ is a norm one projection onto $N$ with $\omega=\omega\circ E$, so (b) follows.

Conversely, assume (b). For every $x\in M$ and $y\in N$,
$$
\<y\Omega,E(x)\Omega\>=\omega(y^*E(x))=\omega(y^*x)=\<y\Omega,x\Omega\>
=\<y\Omega,Px\Omega\>.
$$
Hence $E(x)\Omega=Px\Omega$ and moreover
$$
PSx\Omega=Px^*\Omega=E(x^*)\Omega=SE(x)\Omega=SPx\Omega,\qquad x\in M.
$$
This implies that $PS\subset SP$ and so $(1-2P)S\subset S(1-2P)$. Since $1-2P$ is a
self-adjoint unitary, one has $(1-2P)S=S(1-2P)$ and $S^*(1-2P)=(1-2P)S^*$. Since $\Delta=S^*S$,
it follows that $\Delta=(1-2P)\Delta(1-2P)$ so that $\Delta^{it}=(1-2P)\Delta^{it}(1-2P)$.
Therefore, one has $(1-2P)\Delta^{it}=\Delta^{it}(1-2P)$ so that $\Delta^{it}P=P\Delta^{it}$
for all $t\in\bR$. For every $y\in N$,
$$
\sigma_t^\omega(y)\Omega=\Delta^{it}Py\Omega=P\Delta^{it}y\Omega
=P\sigma_t^\omega(y)\Omega=E(\sigma_t^\omega(y))\Omega,
$$
implying $\sigma_t^\omega(y)=E(\sigma_t^\omega(y))\in N$, $t\in\bR$. Hence (a) follows.
\end{proof}

\subsection{Generalized conditional expectations}

Theorem \ref{T-6.6} says that the existence of the conditional expectation with respect to a
non-tracial normal state is rather restrictive. But there is a weaker and generalized notion
of conditional expectations introduced by Accardi and Cecchini \cite{AC}, which can be defined
for any von Neumann subalgebra and for any faithful normal state. The rest of the section is
a concise account on this generalized conditional expectation.

First, we consider the general situation that $M$ and $N$ are von Neumann algebras and
$\gamma:N\to M$ is a unital (i.e., $\gamma(1)=1$) positive linear map. Let a faithful
$\omega\in M_*^+$ be given, and assume that $\omega_0:=\omega\circ\gamma$ is normal and
faithful on $N$. In this case, $\gamma$ is automatically normal and faithful (i.e.,
$\gamma(y^*y)=0$ $\implies$ $y=0$). We may assume that $M$ and $N$ are represented on $\cH$
and $\cH_0$ with respective cyclic and separating vectors $\Omega$ and $\Omega_0$ satisfying
$$
\omega(x)=\<\Omega,x\Omega\>\quad(x\in M),\qquad
\omega_0(y)=\<\Omega_0,y\Omega_0\>\quad(y\in N).
$$

\begin{thm}[Accardi and Cecchini]\label{T-6.7}
In the above situation, there exists a unique positive linear map $\gamma':M'\to N'$ such that
\begin{align}\label{F-6.3}
\<\gamma(y)\Omega,x'\Omega\>=\<y\Omega_0,\gamma'(x')\Omega_0\>
\quad\mbox{for all $y\in N$ and $x'\in M'$}.
\end{align}
Moreover,
\begin{itemize}
\item[\rm(1)] $\gamma'$ is unital, normal and faithful.
\item[\rm(2)] $\gamma$ is completely positive if and only if so is $\gamma'$.
\end{itemize}
\end{thm}

\begin{proof}
For every $x'\in M_+'$ define $\psi_{x'}(y):=\<x'\Omega,\gamma(y)\Omega\>$ for $y\in N$. Then
$\psi_{x'}\in N_*^+$ and for every $y\in N_+$,
$$
\psi_{x'}(y)=\|x'^{1/2}\gamma(y)^{1/2}\Omega\|^2\le\|x'\|\,\|\gamma(y)^{1/2}\Omega\|^2
=\|x'\|\omega(\gamma(y))=\|x'\|\omega_0(y).
$$
Hence there is a unique $\gamma'(x')\in N_+'$ such that
$\psi_{x'}(y)=\<\Omega_0,\gamma'(x')y\Omega_0\>$ so that
\begin{align}\label{F-6.4}
\<\gamma(y)\Omega,x'\Omega\>=\<y\Omega_0,\gamma'(x')\Omega_0\>,\qquad y\in N.
\end{align}
By linearly extending $\gamma'$, for every $x'\in M'$ there is a unique $\gamma'(x')\in N'$
for which \eqref{F-6.4} holds. Then it is clear that $\gamma':M'\to N'$ is linear, positive
and unital. Let $\{x_\alpha'\}$ be a net in $M_+'$ with $x_\alpha'\nearrow x'\in M_+'$. Then
$\gamma'(x_\alpha')\nearrow y'\le\gamma'(x')$ for some $y'\in N_+'$. For every $y\in N$,
\begin{align*}
\<y\Omega_0,y'\Omega_0\>&=\lim_\alpha\<y\Omega_0,\gamma'(x_\alpha')\Omega_0\>
=\lim_\alpha\<\gamma(y)\Omega,x_\alpha'\Omega\> \\
&=\<\gamma(y)\Omega,x'\Omega\>=\<y\Omega_0,\gamma'(x')\Omega_0\>,
\end{align*}
which implies that $y'\Omega_0=\gamma'(x')\Omega_0$ and so $y'=\gamma'(x')$ since $\Omega_0$
is separating for $N'$. Hence $\gamma'$ is normal. If $x'\in M_+'$ and $\gamma'(x')=0$, then
$\<\Omega,x'\Omega\>=\<\gamma(1)\Omega,x'\Omega\>=\<\Omega_0,\gamma'(x')\Omega_0\>=0$, so
$x'=0$. Hence $\gamma'$ is faithful.

Finally, let $x_i'\in M'$ and $y_i\in N$ for $1\le i\le n$. Then
\begin{align*}
\sum_{i,j=1}^n\<y_i\Omega_0,\gamma'(x_i'^*x_j')y_j\Omega_0\>
&=\sum_{i,j=1}^n\<y_j^*y_i\Omega_0,\gamma'(x_i'^*x_j')\Omega_0\> \\
&=\sum_{i,j=1}^n\<\gamma(y_j^*y_i)\Omega,x_i'^*x_j'\Omega\>
=\sum_{i,j=1}^n\<\gamma(y_j^*y_i)x_i'\Omega,x_j'\Omega\>,
\end{align*}
which shows the assertion in (2).
\end{proof}

In the situation of Theorem \ref{T-6.7}, let $J$ and $J_0$ be the modular conjugations with
respect to $\Omega$ and $\Omega_0$, respectively. One can transform $\gamma':M'\to N'$ into
$\gamma^*:M\to N$ by defining $\gamma^*:=j_0\circ\gamma'\circ j$, where $j:=J\cdot J$ and
$j_0:=J_0\cdot J_0$, so Theorem \ref{T-6.7} is reformulated as follows:

\begin{cor}\label{C-6.8}
In the situation of Theorem \ref{T-6.7}, there exists a unique unital normal positive map
$\gamma^*:M\to N$ such that
\begin{align}\label{F-6.5}
\<Jx\Omega,\gamma(y)\Omega\>=\<J_0\gamma^*(x)\Omega_0,y\Omega_0\>
\quad\mbox{for all $x\in M$ and $y\in N$}.
\end{align}
We have $\omega_0\circ\gamma^*=\omega$, and $\gamma$ is completely positive if and only if so
is $\gamma^*$.
\end{cor}

In fact, \eqref{F-6.5} is a restatement of \eqref{F-6.3}, and $\omega_0\circ\gamma^*=\omega$
follows by letting $y=1$ in \eqref{F-6.5}.

\begin{definition}\label{D-6.9}\rm
The map $\gamma^*$ (also, more explicitly, denoted by $\gamma_\omega^*$) is called the
\emph{$\omega$-dual map} of $\gamma$, which is often called the \emph{Petz' recovery map} too
because Petz \cite{Pe1,Pe2,JP} successfully used the map $\gamma^*$ in the reversibility (or
recovery) theorem for quantum operations in von Neumann algebras. Note that the correspondence
between $\gamma:N\to M$ and $\gamma^*:M\to N$, determined by \eqref{F-6.5}, is completely dual
so that $\gamma$ is the $\omega_0$-dual of $\gamma^*$.
\end{definition}

Now, assume that $\gamma:N\to M$ is a Schwarz map, i.e., $\gamma(y^*y)\ge\gamma(y)^*\gamma(y)$
for all $y\in N$. Since
\begin{align}\label{F-6.6}
\|y\Omega_0\|^2=\omega_0(y^*y)=\omega(\gamma(y^*y))
\ge\omega(\gamma(y)^*\gamma(y))=\|\gamma(y)\Omega\|^2,\qquad y\in N,
\end{align}
we have a linear contraction $V:\cH_0\to\cH$ by extending the operator given by
\begin{align}\label{F-6.7}
Vy\Omega_0=\gamma(y)\Omega,\qquad y\in N.
\end{align}
Then we have:

\begin{lemma}\label{L-6.10}
Let $\gamma:N\to M$ be a Schwarz map and $V:\cH_0\to\cH$ be as given above. Let
$\gamma':M'\to N'$ be the $\omega$-dual of $\gamma$. Then the following conditions are
equivalent:
\begin{itemize}
\item[\rm(i)] $V$ is an isometry;
\item[\rm(ii)] $\gamma$ is a *-isomorphism from $N$ into $M$;
\item[\rm(iii)] $\gamma'(x')=V^*x'V$ for all $x'\in M'$.
\end{itemize}
\end{lemma}

\begin{proof}
(ii)\,$\iff$\,(iii).\enspace
Note that, for every $x'\in M'$ and $y_1,y_2\in N$,
\begin{align*}
&\<y_1\Omega_0,\gamma'(x')y_2\Omega_0\>=\<y_2^*y_1\Omega_0,\gamma'(x')\Omega_0\>
=\<\gamma(y_2^*y_1)\Omega,x'\Omega\>, \\
&\<y_1\Omega_0,V^*x'Vy_2\Omega\>=\<\gamma(y_1)\Omega,x'\gamma(y_2)\Omega\>
=\<\gamma(y_2)^*\gamma(y_1)\Omega,x'\Omega\>.
\end{align*}
Hence (iii) holds if and only if $\gamma(y_2^*y_1)=\gamma(y_2^*)\gamma(y_1)$ for all
$y_1,y_2\in N$, that is equivalent to (ii) since $\gamma$ is faithful.

(ii)\,$\implies$\,(i) is obvious since the inequality in \eqref{F-6.6} becomes equality.

(i)\,$\implies$\,(iii).\enspace
For every $x'\in M_+'$ and $y\in N$ one has
\begin{align*}
\<y\Omega_0,(\gamma'(x')-V^*x'V)y\Omega_0\>
&=\<y^*y\Omega_0,\gamma'(x')\Omega_0\>-\<Vy\Omega_0,x'Vy\Omega_0\> \\
&=\<\gamma(y^*y)\Omega,x'\Omega\>-\<\gamma(y)\Omega,x'\gamma(y)\Omega\> \\
&=\<(\gamma(y^*y)-\gamma(y)^*\gamma(y))\Omega,x'\Omega\>\ge0,
\end{align*}
which implies that $\Phi:M'\to B(\cH_0)$ defined by $\Phi(x'):=\gamma'(x')-V^*x'V$ is a
positive map. Assume (i); then $\Phi(1)=0$. Hence by Proposition \ref{P-6.2}\,(4) one has
$\Phi(x')=0$ for all $x'\in M'$, i.e., (iii) holds.
\end{proof}

In particular, assume that $N$ is a von Neumann subalgebra of $M$ and
$\gamma:N\hookrightarrow M$ is the injection. Then $\omega_0:=\omega|_N$ and we may take
$\Omega_0=\Omega$ and $\cH_0=\overline{N\Omega}\subset\cH$. In this case, the isometry
$V$ given by \eqref{F-6.7} is the injection $\cH_0\hookrightarrow\cH$ so that $V^*$ is
the projection from $\cH$ onto $\cH_0$. By specializing Corollary \ref{C-6.8} with Lemma
\ref{L-6.10} we have the following:

\begin{cor}\label{C-6.11}
Let $M$ be a von Neumann algebra on $\cH$ with a cyclic and separating vector $\Omega$, and
$\omega(x)=\<\Omega,x\Omega\>$, $x\in M$. For every von Neumann subalgebra $N$ of $M$, let
$P$ be the projection from $\cH$ onto $\overline{N\Omega}$. Let $J$ and $J_N$ be the respective
modular conjugations for $M$ and $N$ with respect to $\Omega$. Then the $\omega$-dual map
$E_\omega:=\gamma_\omega^*:M\to N$ of the injection $\gamma:N\hookrightarrow M$ is explicitly
given as
\begin{align}\label{F-6.8}
E_\omega(x)=J_NPJxJPJ_N=J_NPJxJJ_N,\qquad x\in M,
\end{align}
which satisfies $\omega\circ E_\omega=\omega$.
\end{cor}

\begin{definition}\label{D-6.12}\rm
The map $E_\omega:M\to N$ given in Corollary \ref{C-6.11} is called the
\emph{$\omega$-conditional expectation} or the \emph{generalized conditional expectation} for
$N$ with respect to $\omega$, which is a weaker notion of conditional expectations due to
Accardi and Cecchini \cite{AC}
\end{definition}

For further discussions we recall a few notions related to a Schwarz map between
$C^*$-algebras.

\begin{definition}\label{D-6.13}\rm
Let $\cA$ and $\cB$ be unital $C^*$-algebras.
\begin{itemize}
\item[\rm(1)] For a unital Schwarz map $\gamma:\cB\to\cA$, the \emph{multiplicative domain} of
$\gamma$ is defined as
$$
\cM_\gamma:=\{y\in\cB:\gamma(y^*y)
=\gamma(y)^*\gamma(y),\,\gamma(yy^*)=\gamma(y)\gamma(y)^*\}.
$$
\item[\rm(2)] For a unital Schwarz map $\gamma$ from $\cA$ into itself, the \emph{fixed-point}
set of $\gamma$ is defined as
$$
\cF_\gamma:=\{x\in\cA:\gamma(x)=x\}.
$$
\end{itemize}
\end{definition}

The next result was first shown by Choi \cite{Ch} when $\gamma$ is a unital 2-positive map.
The proof below is from \cite{AC}.

\begin{lemma}\label{L-6.14}
Let $\gamma$ be as in Definition \ref{D-6.13}\,(1). For any $y\in\cB$,
\begin{align}
&\gamma(y^*y)=\gamma(y)^*\gamma(y)\ \iff
\ \gamma(by)=\gamma(b)\gamma(y)\ \ \mbox{for all $b\in\cB$}, \label{F-6.9}\\
&\gamma(yy^*)=\gamma(y)\gamma(y^*)\ \iff
\ \gamma(yb)=\gamma(y)\gamma(b)\ \ \mbox{for all $b\in\cB$}. \label{F-6.10}
\end{align}
Consequently,
$$
\cM_\gamma=\{y\in\cB:\gamma(by)=\gamma(b)\gamma(y),\,\gamma(yb)=\gamma(y)\gamma(b)
\ \mbox{for all $b\in\cB$}\},
$$
and it is a unital $C^*$-subalgebra of $\cB$. If $\cA,\cB$ are von Neumann algebras and
$\gamma$ is normal, then $\cM_\gamma$ is a von Neumann subalgebra of $\cB$.
\end{lemma}

\begin{proof}
Let $\cR_\gamma:=\{y\in\cB:\gamma(y^*y)=\gamma(y)^*\gamma(y)\}$ and
$\cL_\gamma:=\{y\in\cB:\gamma(yy^*)=\gamma(y)\gamma(y)^*\}$; then $\cR_\gamma^*=\cL_\gamma$.
We may prove \eqref{F-6.9} only since \eqref{F-6.10} follows from \eqref{F-6.9} immediately.
Define $D(b_1,b_2):=\gamma(b_1^*b_2)-\gamma(b_1)^*\gamma(b_2)$ for $b_1,b_2\in\cB$, which is
sesquilinear on $\cB\times\cB$ and satisfies $D(b,b)=\gamma(b^*b)-\gamma(b)^*\gamma(b)\ge0$
since $\gamma$ is a Schwarz map. Assume that $y\in\cR_\gamma$, i.e., $D(y,y)=0$. For any
$\psi\in\cA_+^*$ and $b\in\cB$, since $\psi\circ D$ is a positive sesquilinear form, the
Schwarz inequality gives
$$
|\psi(D(b,y))|\le\psi(D(b,b))^{1/2}\psi(D(y,y))^{1/2}=0.
$$
Hence $\psi(D(b,y))=0$ for all $\psi\in\cA_+^*$ and so $D(b,y)=0$ for all $b\in\cB$, i.e.,
\eqref{F-6.9} holds. The converse follows by taking $b=y^*$ in \eqref{F-6.9}. It is easy to
see that $\cR_\gamma$ and $\cL_\gamma$ are unital algebras, so
$\cM_\gamma=\cR_\gamma\cap\cL_\gamma$ is a unital $C^*$-subalgebra of $\cB$.
\end{proof}

The set $\cF_\gamma$ is not generally a subalgebra of $\cA$, and there are no general
inclusion relations between $\cF_\gamma$ and $\cM_\gamma$, see \cite[Appendix B]{HM}. But we
have the next result, which seems first observed in \cite{KN} (see also \cite{AC},
\cite[Theorem 2.3]{AGG}).

\begin{lemma}\label{L-6.15}
Let $\gamma$ be as in Definition \ref{D-6.13}\,(2). Assume that there exists a faithful
$\omega\in\cA_+^*$ such that $\omega\circ\gamma=\omega$. Then
$$
\cF_\gamma=\{x\in\cA:\gamma(xa)=x\gamma(a),\,\gamma(ax)=\gamma(a)x
\ \mbox{for all $a\in\cA$}\}\subset\cM_\gamma,
$$
and hence $\cF_\gamma$ is a unital $C^*$-subalgebra of $\cA$. If $\cA,\cB$ are von Neumann
algebras and $\gamma$ is normal, then $\cF_\gamma$ is a von Neumann subalgebra of $\cA$.
\end{lemma}

\begin{proof}
It suffices to prove the first equality assertion. The inclusion $\supset$ is obvious.
Conversely, assume that $x\in\cF_\gamma$. Then $x^*x=\gamma(x)^*\gamma(x)\le\gamma(x^*x)$ and
$\omega(\gamma(x^*x)-x^*x)=\omega(x^*x)-\omega(x^*x)=0$, implying
$\gamma(x^*x)=x^*x=\gamma(x)^*\gamma(x)$. Similarly,
$\gamma(xx^*)=xx^*=\gamma(x)\gamma(x)^*$. Therefore, $x\in\cM_\gamma$, and by Lemma
\ref{L-6.14} we have $\gamma(xa)=\gamma(x)\gamma(a)$ and $\gamma(ax)=\gamma(a)\gamma(x)$ for
all $a\in\cA$.
\end{proof}

Now, we return to the situation of Corollary \ref{C-6.11} and state the result concerning the
fixed-points of the $\omega$-conditional expectation.

\begin{thm}[Accardi and Cecchini]\label{T-6.16}
Let $N\subset M$ and $\omega$ be as in Corollary \ref{C-6.11}, and $E_\omega:M\to N$ be the
$\omega$-conditional expectation. For any $y\in N$ the following conditions are equivalent:
\begin{itemize}
\item[\rm(i)]  $E_\omega(y)=y$;
\item[\rm(ii)] $E_\omega(yx)=yE_\omega(x)$ and $E_\omega(xy)=E_\omega(x)y$ for all $x\in M$;
\item[\rm(iii)] $\sigma_t^\omega(y)\in N$ for all $t\in\bR$;
\item[\rm(iv)] $\sigma_t^\omega(y)=\sigma_t^{\omega|_N}(y)$ for all $t\in\bR$.
\end{itemize}
\end{thm}

\begin{proof}
(i)\,$\iff$\,(ii) holds by Lemma \ref{L-6.15}.

(iii)\,$\implies$\,(iv).\enspace
Assume (iii). For each $n\in\bN$ let $f_n(z):=\sqrt{n\over\pi}\,e^{-nz^2}$ ($z\in\bC$), which
is an entire function. Define
$$
y_n(z):=\int_{-\infty}^\infty f_n(s-z)\sigma_s^\omega(y)\,ds,\qquad
b_n(z):=\int_{-\infty}^\infty f_n(t-z)\sigma_t^{\omega|_N}(b)\,dt
$$
for any $b\in N$. It is easy to verify that $y_n(z)$ (resp., $b_n(z)$) is an entire analytic
$M$-valued (resp., $N$-valued) function. For every $r\in\bR$ let $y_1:=\sigma_r^\omega(y)$ and
$b_1:=\sigma_r^{\omega|_N}(b)$. Note that
\begin{align*}
&y_n(r)=\int_{-\infty}^\infty f_n(s-r)\sigma_s^\omega(y)\,ds
=\int_{-\infty}^\infty f_n(s)\sigma_{s+r}^\omega(y)\,ds
=\int_{-\infty}^\infty f_n(s)\sigma_s^\omega(y_1)\,ds, \\
&y_n(r+i)=\int_{-\infty}^\infty f_n(s-r-i)\sigma_s^\omega(y)\,ds
=\int_{-\infty}^\infty f_n(s-i)\sigma_s^\omega(y_1)\,ds,
\end{align*}
and similarly for $b_n(r)$ and $b_n(r+i)$. From the KMS condition (see Definition \ref{D-2.11}
and \cite[Proposition 5.3.12]{BR2}) it follows that
\begin{align*}
\omega(b_n(r)y_n(r))
&=\int_{-\infty}^\infty f_n(s)\int_{-\infty}^\infty f_n(t)\omega(\sigma_t^{\omega|_N}(b_1)
\sigma_s^\omega(y_1))\,dt\,ds \\
&=\int_{-\infty}^\infty f_n(s)\int_{-\infty}^\infty f_n(t-i)\omega(\sigma_s^\omega(y_1)
\sigma_t^{\omega|_N}(b_1))\,dt\,ds \\
&=\int_{-\infty}^\infty f_n(t-i)\int_{-\infty}^\infty f_n(s)\omega(\sigma_s^\omega(y_1)
\sigma_t^{\omega|_N}(b_1))\,ds\,dt \\
&=\int_{-\infty}^\infty f_n(t-i)\int_{-\infty}^\infty f_n(s-i)\omega(\sigma_t^{\omega|_N}(b_1)
\sigma_s^\omega(y_1))\,ds\,dt \\
&=\omega(b_n(r+i)y_n(r+i)),\qquad r\in\bR.
\end{align*}
This implies that the entire function $\omega(b_n(z)y_n(z))$ has a period $i$, so it is
bounded on $\bC$. The Liouville theorem yields that
$\omega(b_n(z)y_n(z))\equiv\omega(b_n(0)y_n(0))$ so that
\begin{align}
&\int_{-\infty}^\infty\int_{-\infty}^\infty f_n(s)f_n(t)
\omega(\sigma_t^{\omega|_N}(\sigma_r^{\omega|_N}(b))
\sigma_s^\omega(\sigma_r^\omega(y)))\,ds\,dt \nonumber\\
&\qquad=\int_{-\infty}^\infty\int_{-\infty}^\infty f_n(s)f_n(t)
\omega(\sigma_t^{\omega|_N}(b)\sigma_s^\omega(y))\,ds\,dt,\qquad r\in\bR. \label{F-6.11}
\end{align}
Since $\lim_{n\to\infty}\int_{-\infty}^\infty f_n(s)\phi(s)\,ds=\phi(0)$ for any bounded
continuous function $\phi$ on $\bR$, it follows that
$$
\int_{-\infty}^\infty f_n(s)\sigma_s^\omega(\sigma_r^\omega(y))\,ds
\,\longrightarrow\,\sigma_r^\omega(y),\qquad
\int_{-\infty}^\infty f_n(t)\sigma_t^{\omega|_N}(\sigma_r^{\omega|_N}(b))\,dt
\,\longrightarrow\,\sigma_r^{\omega|_N}(b)
$$
strongly. Therefore, letting $n\to\infty$ for both sides of \eqref{F-6.11} we arrive at
$\omega(\sigma_r^{\omega|_N}(b)\sigma_r^\omega(y))=\omega(by)$, that is,
$\<\Delta_N^{ir}b^*\Omega,\Delta^{ir}y\Omega\>=\<b^*\Omega,y\Omega\>$ for all $r\in\bR$ and
$b\in N$, where $\Delta$ and $\Delta_N$ are the respective modular operators for $M$ and $N$
with respect to $\Omega$. Since (iii) implies that
$\Delta^{ir}y\Omega=\sigma_r^\omega(y)\Omega\in\overline{N\Omega}$, it follows that
$\<b^*\Omega,\Delta_N^{-ir}(\Delta^{ir}y\Omega)\>=\<b^*\Omega,y\Omega\>$ for all $b\in N$ so
that $\Delta_N^{-ir}(\Delta^{ir}y\Omega)=y\Omega$. Hence
$$
\sigma_r^\omega(y)\Omega=\Delta^{ir}y\Omega=\Delta_N^{ir}y\Omega
=\sigma_r^{\omega|_N}(y)\Omega,
$$
implying that $\sigma_r^\omega(y)=\sigma_r^{\omega|_N}(y)$ for all $r\in\bR$.

(iv)\,$\implies$\,(i).\enspace
From (iv) it follows that $\Delta^{it}y\Omega=\Delta_N^{it}y\Omega$ for all $t\in\bR$. Since
$y\Omega\in\cD(\Delta^{1/2})\cap\cD(\Delta_N^{1/2})$, the analytic continuation gives
$\Delta^{1/2}y\Omega=\Delta_N^{1/2}y\Omega$. Therefore,
$$
y^*\Omega=J_N\Delta_N^{1/2}y\Omega=J_NP\Delta^{1/2}y\Omega=J_NPJy^*\Omega
=E_\omega(y^*)\Omega
$$
thanks to \eqref{F-6.8}, implying that $E_\omega(y^*)=y^*$ so that $E_\omega(y)=y$.

(i)\,$\implies$\,(iii).\enspace
Choose an invariant mean $m$ on $\ell^\infty(\bN)$. For every $x\in M$, note that
$\psi\in M_*\mapsto m[\psi(E_\omega^n(x))]$ is a bounded linear functional on $M_*$ such that
$|m[\psi(E_\omega^n(x))]|\le\|x\|\,\|\psi\|$. Hence there is an $E(x)\in M$ such that
$\|E(x)\|\le\|x\|$ and $\psi(E(x))=m[\psi(E_\omega^n(x))]$ for all $\psi\in M_*$. Since
$$
\psi(E_\omega(E(x))=m[\psi\circ E_\omega(E_\omega^n(x))]
=m[\psi(E_\omega^{n+1}(x))]=\psi(E(x)),\qquad\psi\in M_*,
$$
we have $E_\omega(E(x))=E(x)$, i.e., $E(x)\in\cF_{E_\omega}$. If $x\in\cF_{E_\omega}$, then
$\psi(E_\omega(x))=\psi(x)$ for all $\psi\in M_*$, so $E(x)=x$. Therefore,
$E:M\to\cF_{E_\omega}$ is a norm one projection onto $\cF_{E_\omega}$. Moreover, since
$\omega(E(x))=m[\omega(E_\omega^n(x))]=\omega(x)$, it follows that $E$ is the conditional
expectation from $M$ onto $\cF_{E_\omega}$ with respect to $\omega$. By Theorem \ref{T-6.6} we
have $\sigma_t^\omega(\cF_{E_\omega})=\cF_{E_\omega}$ for all $t\in\bR$, which shows that
(i)\,$\implies$\,(iii).
\end{proof}

\begin{cor}\label{C-6.17}
In the situation of Theorem \ref{T-6.16}, $\cF_{E_\omega}$ is the largest von Neumann
subalgebra of $N$ onto which the conditional expectation from $M$ with respect to $\omega$
exists.
\end{cor}

\begin{proof}
In the above proof of (i)\,$\implies$\,(iii) we have shown the existence of the conditional
expectation from $M$ onto $\cF_{E_\omega}$ with respect to $\omega$. Let $N_1$ be a von Neumann
subalgebra of $N$ and assume that there is the conditional expectation from $M$ onto $N_1$ with
respect to $\omega$. For every $y\in N_1$, since Theorem \ref{T-6.6} implies that
$\sigma^\omega(y)\in N_1\subset N$ for all $t\in\bR$, we have $y\in\cF_{E_\omega}$ due to
Theorem \ref{T-6.16}. Hence $N_1\subset\cF_{E_\omega}$ follows.
\end{proof}

\section{Connes' cocycle derivatives}

Up to now, we have avoided using the notion of normal weights on von Neumann algebras, except
the semifinite normal trace in Secs.~4 and 5, to make the presentations simpler, which may
be rather harmless as far as von Neumann algebras are $\sigma$-finite. However, it seems that
the notion of Connes' cocycle derivatives for (faithful) semifinite normal weights will
play an essential role in the study of the structure theory of von Neumann algebras and that
of operator valued weights, presented in Secs.~8 and 9 below. Thus, in this section, we
will give a minimal requirement about Connes' cocycle derivatives for faithful semifinite
normal weights, and further discussions will be in Sec.~12.2.

\subsection{Basics of f.s.n.\ weights}

Let $M$ be a von Neumann algebra on a Hilbert space $\cH$. We start with the next definition.

\begin{definition}\label{D-7.1}\rm
A functional $\ffi:M_+\to[0,+\infty]$ satisfying the following properties is called a
\emph{weight} on $M$: for any $x,y\in M_+$,
\begin{itemize}
\item $\ffi(\lambda x)=\lambda\ffi(x)$, $\lambda\ge0$,
\item $\ffi(x+y)=\ffi(x)+\ffi(y)$.
\end{itemize}
Define
\begin{align*}
\fF_\ffi&:=\{x\in M_+:\ffi(x)<+\infty\}, \\
\fN_\ffi&:=\{x\in M:\ffi(x^*x)<+\infty\}, \\
\fM_\ffi&:=\fN_\ffi^*\fN_\ffi:=\lin\{y^*x:x,y\in\fN_\ffi\}.
\end{align*}
Then $\fN_\ffi$ is a left ideal of $M$, $\fM_\ffi=\lin\,\fF_\ffi$ and
$\fF_\ffi=M_+\cap\fM_\ffi$ (similarly to Lemma \ref{L-5.4}).
\begin{itemize}
\item $\ffi$ is said to be \emph{normal} if, for any increasing net $\{x_\alpha\}$ in
$M_+$ with $x_\alpha\nearrow x\in M_+$, $\ffi(x_\alpha)\nearrow\ffi(x)$.
\item $\ffi$ is said to be \emph{faithful} if $\ffi(x^*x)=0$ $\implies$ $x=0$ for any
$x\in M$.
\item $\ffi$ is said to be \emph{semifinite} if the following equivalent conditions hold:
\begin{itemize}
\item[(a)] $\fN_\ffi$ is weakly dense in $M$;
\item[(b)] $\sup\{e\in\Proj(M):\ffi(e)<+\infty\}=1$;
\item[(c)] there is an increasing net $\{u_\alpha\}$ in $\fF_\ffi$ such that
$u_\alpha\nearrow1$.
\end{itemize}
\end{itemize}
\end{definition}

Here we record Haagerup's theorem \cite{Haa} characterizing normal weights on von Neumann
algebras, where the implication (iv)\,$\implies$\,(v) is due to \cite[Theorem 7.2]{PT}.

\begin{thm}[Haagerup]\label{T-7.2}
Let $\ffi$ be a weight on $M$. Then the following properties are equivalent:
\begin{itemize}
\item[\rm(i)] $\ffi$ is normal;
\item[\rm(ii)] $\ffi$ is completely additive, i.e., $\ffi\bigl(\sum_ix_i\bigr)=\sum_i\ffi(x_i)$
for any set $\{x_i\}$ in $M_+$ with $\sum_ix_i\in M_+$;
\item[\rm(iii)] $\ffi$ is $\sigma$-weakly lower semicontinuous;
\item[\rm(iv)] $\ffi(x)=\sup\{\omega(x):\omega\in M_*^+,\,\omega\le\ffi\}$ for all $x\in M_+$;
\item[\rm(v)] $\ffi(x)=\sum_{i\in I}\omega_i(x)$ for all $x\in M_+$ with some
$(\omega_i)_{i\in I}\subset M_*^+$;
\item[\rm(vi)] $\ffi(x)=\sum_{i\in I}\<\zeta_i,x\zeta_i\>$ for all $x\in M_+$ with some
$(\zeta_i)_{i\in I}\subset\cH$.
\end{itemize}
\end{thm}

\begin{example}\label{E-7.3}\rm
(1)\enspace
Let $M=L^\infty(X,\mu)$ be a commutative von Neumann algebra over a localizable measure space
$(X,\cX,\mu)$. Define $\ffi(f):=\int_Xf\,d\mu$ for $f\in L^\infty(X,\mu)_+$. Then $\ffi$ is a
faithful semifinite normal weight on $M$. In this case,
$\fF_\ffi=L^\infty(X,\mu)_+\cap L^1(X,\mu)$, $\fN_\ffi=L^\infty(X,\mu)\cap L^2(X,\mu)$ and
$\fM_\ffi=L^\infty(X,\mu)\cap L^1(X,\mu)$.

(2)\enspace
Let $M$ be a semifinite von Neumann algebra with a faithful semifinite normal trace $\tau$
(discussed in Secs.~4 and 5), in particular, $M=B(\cH)$ with the usual trace $\Tr$. Then
$\tau$ is a special case of faithful semifinite normal weights, $\fF_\tau=M_+\cap L^1(M,\tau)$,
$\fN_\tau=M\cap L^2(M,\tau)$ and $\fM_\tau=M\cap L^1(M,\tau)$ as in Lemma \ref{L-5.4}.
\end{example}

In the rest of the section we write `f.s.n.' to mean `faithful semifinite normal.' Associated
with an f.s.n.\ weight $\ffi$ on $M$, we can perform the GNS construction in the following way.
A Hilbert space $\cH_\ffi$ is the completion of $\fN_\ffi$ with an inner product
$\<a,b\>_\ffi:=\ffi(b^*a)$, $a,b\in\fN_\ffi$. Let $a\in\fN_\ffi\mapsto a_\ffi\in\cH_\ffi$ be
the canonical injection of $\fN_\ffi$ into $\cH_\ffi$. For any $x\in M$, since $\fN_\ffi$ is a
left ideal of $M$ and $\|(xa)_\ffi\|=\ffi(a^*x^*xa)^{1/2}\le\|x\|\,\|a_\ffi\|$,
$\pi_\ffi(x)\in B(\cH_\ffi)$ is defined by $\pi_\ffi(x)a_\ffi:=(xa)_\ffi$. Then it is easy to
see that $\pi_\ffi$ is a faithful representation of $M$ on $\cH_\ffi$. So we may consider
$M$ as a von Neumann algebra on $\cH_\ffi$ by identifying $x$ with $\pi_\ffi(x)$. Note that
$\fN_\ffi\cap\fN_\ffi^*$ has a *-algebra structure. Although we don't enter into the details
here, the following fundamental results hold (see \cite{Ta3,SZ} for details):
\begin{itemize}
\item[(A)] The subspace $\fA_\ffi:=\{a_\ffi:a\in\fN_\ffi\cap\fN_\ffi^*\}$ of $\cH_\ffi$ becomes
a so-called \emph{left Hilbert algebra} with the product $a_\ffi b_\ffi=(ab)_\ffi$ and the
involution $(a_\ffi)^\sharp:=(a^*)_\ffi$, and the associated left von Neumann algebra
$\mathfrak{L}(\fA_\ffi):=\{L_\xi:\xi\in\fA_\ffi\}''$ is $\pi_\ffi(M)\cong M$, where $L_\xi$ is
the left multiplication (bounded operator) $L_\xi\eta:=\xi\eta$ for $\eta\in\fA_\ffi$.
Moreover, the weight $\ffi$ is recaptured from $\fA_\ffi$ in such a way that, for $x\in M_+$,
$$
\ffi(x)=\begin{cases}\|\xi\|^2 &
\text{if $\pi_\ffi(x)^{1/2}=L_\xi$ for some $\xi\in\fA_\ffi$}, \\
+\infty & \text{otherwise}.\end{cases}
$$
(Note that if $x\in\fF_\ffi$ and so $x^{1/2}\in\fN_\ffi$, then
$\pi_\ffi(x^{1/2})=L_{(x^{1/2})_\ffi}$ and $\ffi(x)=\|(x^{1/2})_\ffi\|^2$.)
\item[(B)] (Tomita's theorem)\enspace
Let $S_\ffi$ be the closure of the conjugate linear closable operator
$a_\ffi\in\fA_\ffi\mapsto(a^*)_\ffi\in\fA_\ffi$, and $S_\ffi=J_\ffi\Delta_\ffi^{1/2}$ be the
polar decomposition of $S_\ffi$. Then $J_\ffi$ and $\Delta_\ffi$ satisfy the properties
(i)--(iv) of Lemma \ref{L-2.1}, and Tomita's fundamental theorem holds as
$$
J_\ffi\pi_\ffi(M)J_\ffi=\pi_\ffi(M)',\qquad
\Delta_\ffi^{-it}\pi_\ffi(M)\Delta_\ffi^{-it}=\pi_\ffi(M),\quad t\in\bR.
$$
\item[(C)] Define the \emph{modular automorphism group} $\sigma_t^\ffi$ ($t\in\bR$) of $M$
associated with $\ffi$ by
$$
\pi_\ffi(\sigma_t^\ffi(x)):=\Delta_\ffi^{it}\pi_\ffi(x)\Delta_\ffi^{-it}.
$$
Then $\ffi\circ\sigma_t^\ffi=\ffi$, $t\in\bR$, and the KMS condition (i.e., Definition
\ref{D-2.11} with $\beta=-1$ and $x,y$ restricted to elements of $\fN_\ffi\cap\fN_\ffi^*$)
holds. Furthermore, $\sigma_t^\ffi$ is uniquely determined as a weakly continuous one-parameter
automorphism group satisfying these properties.
\item[(D)] More intrinsic behind the above results (A)--(C) are the following: The subspace
$J\fA_\ffi$ of $\cH_\ffi$ is the \emph{right Hilbert algebra}
$\fA_\ffi':=\{\eta\in\cD(S^*):R_\eta\ \mbox{is bounded}\}$, where $R_\eta$ is the right
multiplication $R_\eta x_\ffi:=\pi_\ffi(x)\eta$ for $x\in\fN_\ffi$, and
$$
R_{J\xi}=JL_\xi J\ \ (\xi\in\fA_\ffi),\qquad
J\pi_\ffi(x)Jy_\ffi=\pi_\ffi(y)Jx_\ffi\ \ (x,y\in\fN_\ffi)
$$
hold. Furthermore, for every $t\in\bR$, $\Delta^{it}\fA_\ffi=\fA_\ffi$ and
$$
L_{\Delta^{it}\xi}=\Delta^{it}L_\xi\Delta^{-it}\ \ (\xi\in\fA_\ffi),\qquad
(\sigma_t^\ffi(x))_\ffi=\Delta^{it}x_\ffi\ \ (x\in\fN_\ffi)
$$
hold.
\item[(E)] Although we don't give the explicit definition, there exists a subalgebra
$$
\fT_\ffi\subset\fA_\ffi\cap\fA_\ffi'\cap\bigcap_{\alpha\in\bC}\cD(\Delta_\ffi^z),
$$
which is called the \emph{Tomita algebra} and satisfies the nice properties
$$
J_\ffi\fT_\ffi=\fT_\ffi,\quad\Delta_\ffi^z\fT_\ffi=\fT_\ffi,\quad
\fT_\ffi\ \mbox{is a core of $\Delta_\ffi^z$}
$$
for all $z\in\bC$.
\end{itemize}

\subsection{Connes' cocycle derivatives}

Now, consider the tensor product $M\otimes\bM_2(\bC)=\bM_2(M)$ of $M$ and the $2\times2$
matrix algebra $\bM_2(\bC)$. Let $\ffi,\psi$ be f.s.n.\ weights on $M$. The
\emph{balanced weight} $\theta=\theta(\ffi,\psi)$ on $\bM_2(M)$ is given by
\begin{align}\label{F-7.1}
\theta\Biggl(\sum_{i,j=1}^2x_{ij}\otimes e_{ij}\Biggr)
:=\ffi(x_{11})+\psi(x_{22}),\qquad\sum_{i,j=1}^2x_{ij}\otimes e_{ij}\in\bM_2(M)_+,
\end{align}
where $e_{ij}$ ($i,j=1,2$) are the matrix units of $\bM_2(\bC)$. Then we have:

\begin{lemma}\label{L-7.4}
\begin{itemize}
\item[\rm(1)] $\theta$ is an f.s.n.\ weight on $\bM_2(M)$,
\item[\rm(2)] $\fN_\theta=\fN_\ffi\otimes e_{11}+\fN_\psi\otimes e_{12}
+\fN_\ffi\otimes e_{21}+\fN_\psi\otimes e_{22}$,
\item[\rm(3)] $\fM_\theta=\fM_\ffi\otimes e_{11}+\fN_\ffi^*\fN_\psi\otimes e_{12}
+\fN_\psi^*\fN_\ffi\otimes e_{21}+\fM_\psi\otimes e_{22}$, where
$\fN_\psi^*\fN_\ffi:=\lin\{y^*x:x\in\fN_\ffi,\,y\in\fN_\psi\}$,
\item[\rm(4)] $\fN_\theta\cap\fN_\theta^*=(\fN_\ffi\cap\fN_\ffi^*)\otimes e_{11}
+(\fN_\psi\cap\fN_\ffi^*)\otimes e_{12}+(\fN_\ffi\cap\fN_\psi^*)\otimes e_{21}
+(\fN_\psi\cap\fN_\psi^*)\otimes e_{22}$.
\end{itemize}
\end{lemma}

\begin{proof}
(2)\enspace
For $X=\begin{bmatrix}x_{11}&x_{12}\\x_{21}&x_{22}\end{bmatrix}\in\bM_2(M)$, since
\begin{align}\label{F-7.2}
\theta(X^*X)=\ffi(x_{11}^*x_{11}+x_{21}^*x_{21})+\psi(x_{12}^*x_{12}+x_{22}^*x_{22}),
\end{align}
it follows that $\theta(X^*X)<+\infty$ if and only if $x_{11},x_{21}\in\fN_\ffi$ and
$x_{12},x_{22}\in\fN_\psi$.

(1)\enspace
By \eqref{F-7.2}, $\theta(X^*X)=0$ $\iff$ $X=0$, so $\theta$ is faithful. From definition
\eqref{F-7.1} it is easy to verify that $\theta$ is normal. There are increasing nets
$\{u_\alpha\}\subset\cF_\ffi$ and $\{v_\alpha\}\subset\cF_\psi$ such that $u_\alpha\nearrow1$
and $v_\alpha\nearrow1$. Then $u_\alpha\otimes e_{11}+v_\alpha\otimes e_{22}\in\cF_\theta$ and
$u_\alpha\otimes e_{11}+v_\alpha\otimes e_{22}\nearrow1$, hence $\theta$ is semifinite.

(3) and (4) are immediate from (2).
\end{proof}

By Lemma \ref{L-7.4}\,(2) the GNS Hilbert space for $\theta$ is given as
$$
\cH_\theta=\cH_\ffi\oplus\cH_\psi\oplus\cH_\ffi\oplus\cH_\psi
$$
with the canonical injection
$$
\begin{bmatrix}a_{11}&a_{12}\\a_{21}&a_{22}\end{bmatrix}\in\fN_\theta
\ \longmapsto\ \begin{bmatrix}a_{11}&a_{12}\\a_{21}&a_{22}\end{bmatrix}_\theta
=\begin{bmatrix}(a_{11})_\ffi\\(a_{12})_\psi\\(a_{21})_\ffi\\(a_{22})_\psi\end{bmatrix},
$$
and the GNS representation of $\bM_2(M)$ associated with $\theta$ is
$$
\pi_\theta\biggl(\begin{bmatrix}x_{11}&x_{12}\\x_{21}&x_{22}\end{bmatrix}\biggr)
\begin{bmatrix}a_{11}&a_{12}\\a_{21}&a_{22}\end{bmatrix}_\theta
=\begin{bmatrix}x_{11}a_{11}+x_{12}a_{21}&x_{11}a_{12}+x_{12}a_{22}\\
x_{21}a_{11}+x_{22}a_{21}&x_{21}a_{12}+x_{21}a_{22}\end{bmatrix}_\theta,
$$
so that $\pi_\theta$ is given in the $4\times4$ form as
\begin{align}\label{F-7.3}
\pi_\theta\biggl(\begin{bmatrix}x_{11}&x_{12}\\x_{21}&x_{22}\end{bmatrix}\biggr)
=\begin{bmatrix}\pi_\ffi(x_{11})&0&\pi_\ffi(x_{12})&0\\
0&\pi_\psi(x_{11})&0&\pi_\psi(x_{12})\\
\pi_\ffi(x_{21})&0&\pi_\ffi(x_{22})&0\\
0&\pi_\psi(x_{21})&0&\pi_\psi(x_{22})\end{bmatrix}.
\end{align}
By Lemma \ref{L-7.4}\,(4) the closure $S_\theta$ of
$\begin{bmatrix}a_{11}&a_{12}\\a_{21}&a_{22}\end{bmatrix}_\theta\mapsto
\begin{bmatrix}a_{11}^*&a_{21}^*\\a_{12}^*&a_{22}^*\end{bmatrix}_\theta$
($\begin{bmatrix}a_{11}&a_{12}\\a_{21}&a_{22}\end{bmatrix}\in\fN_\theta\cap\fN_\theta^*$)
is given as
\begin{align}\label{F-7.4}
S_\theta=\begin{bmatrix}S_\ffi&0&0&0\\0&0&S_{\psi,\ffi}&0\\
0&S_{\ffi,\psi}&0&0\\0&0&0&S_\psi\end{bmatrix},
\end{align}
where $S_{\ffi,\psi}$ is the closure of
$a_{12}\in\fN_\psi\cap\fN_\ffi^*\mapsto a_{12}^*\in\fN_\ffi\cap\fN_\psi^*$ and $S_{\psi,\ffi}$
is the closure of $a_{21}\in\fN_\ffi\cap\fN_\psi^*\mapsto a_{21}^*\in\fN_\psi\cap\fN_\ffi^*$.
Thus, the polar decomposition $S_\theta=J_\theta\Delta_\theta^{1/2}$ is given as
\begin{align}\label{F-7.5}
J_\theta=\begin{bmatrix}J_\ffi&0&0&0\\0&0&J_{\psi,\ffi}&0\\
0&J_{\ffi,\psi}&0&0\\0&0&0&J_\psi\end{bmatrix},\qquad
\Delta_\theta=\begin{bmatrix}\Delta_\ffi&0&0&0\\0&\Delta_{\ffi,\psi}&0&0\\
0&0&\Delta_{\psi,\ffi}&0\\0&0&0&\Delta_\psi\end{bmatrix},
\end{align}
where $S_{\ffi,\psi}=J_{\ffi,\psi}\Delta_{\ffi,\psi}^{1/2}$ and
$S_{\psi,\ffi}=J_{\psi,\ffi}\Delta_{\psi,\ffi}^{1/2}$. Concerning the modular automorphism
group $\sigma_t^\theta$ of $\bM_2(M)$ we have:

\begin{lemma}\label{L-7.5}
The $\sigma_t^\theta$ is given as
\begin{align}\label{F-7.6}
\sigma_t^\theta\biggl(\begin{bmatrix}x_{11}&x_{12}\\x_{21}&x_{22}\end{bmatrix}\biggr)
=\begin{bmatrix}\sigma_t^\ffi(x_{11})&\sigma_t^{\ffi,\psi}(x_{12})\\
\sigma_t^{\psi,\ffi}(x_{21})&\sigma_t^\psi(x_{22})\end{bmatrix},\qquad
\begin{bmatrix}x_{11}&x_{12}\\x_{21}&x_{22}\end{bmatrix}\in\bM_2(M),
\end{align}
where $\sigma_t^\ffi,\sigma_t^\psi$ are the modular automorphism groups of $M$ associated
with $\ffi,\psi$, respectively, $\sigma_t^{\psi,\ffi}$ is a strongly continuous one-parameter
group of isometries on $M$, and $\sigma_t^{\ffi,\psi}(x)=\sigma_t^{\psi,\ffi}(x^*)^*$, $x\in M$.
Furthermore,
\begin{align}\label{F-7.7}
\pi_\ffi(\sigma_t^{\psi,\ffi}(x))=\Delta_{\psi,\ffi}^{it}\pi_\ffi(x)\Delta_\ffi^{-it},
\quad\pi_\psi(\sigma_t^{\psi,\ffi}(x))=\Delta_\psi^{it}\pi_\psi(x)\Delta_{\ffi,\psi}^{-it},
\qquad x\in M.
\end{align}
\end{lemma}

\begin{proof}
Note that $\pi_\theta(\sigma_t^\theta(X))=\Delta_\theta^{it}\pi_\theta(X)\Delta_\theta^{-it}$
for $X=\begin{bmatrix}x_{11}&x_{12}\\x_{21}&x_{22}\end{bmatrix}\in\bM_2(M)$. Hence the
expression in \eqref{F-7.6} with \eqref{F-7.7} is a direct computation based on \eqref{F-7.3}
and \eqref{F-7.5}. Since $\sigma_t^\theta$ is a strongly continuous one-parameter automorphism
group of $\bM_2(M)$, it is clear that $\sigma_t^{\psi,\ffi}$ is a strongly continuous
one-parameter group of isometries of $M$ and 
$\sigma_t^{\ffi,\psi}(x)=\sigma_t^{\psi,\ffi}(x^*)^*$.
\end{proof}

\begin{thm}[\cite{Co1}]\label{T-7.6}
Define $u_t:=\sigma_t^{\psi,\ffi}(1)\in M$ for $t\in\bR$. Then $t\in\bR\mapsto u_t$ is a
strongly* continuous map into the unitaries of $M$ satisfying
\begin{align}
\sigma_t^\psi(x)&=u_t\sigma_t^\ffi(x)u_t^*,\qquad t\in\bR,\ x\in M. \label{F-7.8}\\
u_{s+t}&=u_s\sigma_s^\ffi(u_t),\qquad s,t\in\bR
\qquad\mbox{(\emph{cocycle identity})}. \label{F-7.9}
\end{align}
\end{thm}

\begin{proof}
By Lemma \ref{L-7.5} note that $\begin{bmatrix}0&0\\u_t&0\end{bmatrix}
=\sigma_t^\theta\biggl(\begin{bmatrix}0&0\\1&0\end{bmatrix}\biggr)$ for $t\in\bR$, so
$t\mapsto u_t$ is strongly* continuous. Since
\begin{align*}
\begin{bmatrix}u_t^*u_t&0\\0&0\end{bmatrix}
&=\begin{bmatrix}0&0\\u_t&0\end{bmatrix}^*\begin{bmatrix}0&0\\u_t&0\end{bmatrix}
=\sigma_t^\theta\biggl(\begin{bmatrix}0&0\\1&0\end{bmatrix}^*
\begin{bmatrix}0&0\\1&0\end{bmatrix}\biggr) \\
&=\sigma_t^\theta\biggl(\begin{bmatrix}1&0\\0&0\end{bmatrix}\biggr)
=\begin{bmatrix}\sigma_t^\ffi(1)&0\\0&0\end{bmatrix}
=\begin{bmatrix}1&0\\0&0\end{bmatrix}
\end{align*}
and similarly $\begin{bmatrix}0&0\\0&u_tu_t^*\end{bmatrix}
=\sigma_t^\theta\biggl(\begin{bmatrix}0&0\\0&1\end{bmatrix}\biggr)
=\begin{bmatrix}0&0\\0&1\end{bmatrix}$, we see that $u_t$'s are unitaries. Moreover,
\begin{align*}
\begin{bmatrix}0&0\\0&\sigma_t^\psi(x)\end{bmatrix}
&=\sigma_t^\theta\biggl(\begin{bmatrix}0&0\\0&x\end{bmatrix}\biggr)
=\sigma_t^\theta\biggl(\begin{bmatrix}0&0\\1&0\end{bmatrix}
\begin{bmatrix}x&0\\0&0\end{bmatrix}\begin{bmatrix}0&0\\1&0\end{bmatrix}^*\biggr) \\
&=\begin{bmatrix}0&0\\u_t&0\end{bmatrix}\begin{bmatrix}\sigma_t^\ffi(x)&0\\0&0\end{bmatrix}
\begin{bmatrix}0&0\\u_t&0\end{bmatrix}^*
=\begin{bmatrix}0&0\\0&u_t\sigma_t^\ffi(x)u_t^*\end{bmatrix}
\end{align*}
so that \eqref{F-7.8} follows.

Since
\begin{align*}
\begin{bmatrix}0&0\\u_{s+t}&0\end{bmatrix}
&=\sigma_{s+t}^\theta\biggl(\begin{bmatrix}0&0\\1&0\end{bmatrix}\biggr)
=\sigma_s^\theta\biggl(\begin{bmatrix}0&0\\u_t&0\end{bmatrix}\biggr)
=\sigma_s^\theta\biggl(\begin{bmatrix}0&0\\1&0\end{bmatrix}
\begin{bmatrix}u_t&0\\0&0\end{bmatrix}\biggr) \\
&=\begin{bmatrix}0&0\\u_s&0\end{bmatrix}
\begin{bmatrix}\sigma_s^\ffi(u_t)&0\\0&0\end{bmatrix}
=\begin{bmatrix}0&0\\u_s\sigma_s^\ffi(u_t)&0\end{bmatrix},
\end{align*}
we have \eqref{F-7.9}.
\end{proof}

\begin{definition}\label{D-7.7}\rm
The map $t\mapsto u_t$ given in Theorem \ref{T-7.6} is called \emph{Connes' cocycle
(Radon-Nikodym) derivative} of $\psi$ with respect to $\ffi$, and denoted by $(D\psi:D\ffi)_t$,
i.e., $u_t=(D\psi:D\ffi)_t$, $t\in\bR$.
\end{definition}

By construction it is clear that $\sigma_t^{\ffi,\ffi}=\sigma_t^\ffi$ and so
$(D\ffi:D\ffi)_t=1$ for all $t\in\bR$. It is also clear that
\begin{align}\label{F-7.10}
(D\ffi:D\psi)_t=(D\psi:D\ffi)_t^*,\qquad t\in\bR.
\end{align}
More properties of Connes' cocycle derivatives $(D\psi:D\ffi)_t$ will be given in
Proposition \ref{P-9.6} and will be also discussed for (not necessarily faithful)
$\ffi,\psi\in M_*^+$ in Sec.~12.2.

\section{Operator valued weights}

This section is aimed at giving an essential part of operator valued weights in von Neumann
algebras developed by Haagerup \cite{Haa2,Haa3}. Let $M$ be a von Neumann algebra with the
predual $M_*$ as before.

\subsection{Generalized positive operators}

\begin{definition}\label{D-8.1}\rm
A functional $m:M_*^+\to[0,\infty]$ satisfying the following properties is called a
\emph{generalized positive operator} affiliated with $M$: for any $\ffi,\psi\in M_*^+$,
\begin{itemize}
\item $m(\lambda\ffi)=\lambda m(\ffi)$, $\lambda\ge0$,
\item $m(\lambda+\psi)=m(\ffi)+m(\psi)$,
\item $m$ is lower semicontinuous on $M_*^+$.
\end{itemize}
We denote by $\widehat M_+$ the set of generalized positive operators affiliated with $M$ and
call it the \emph{extended positive part} of $M$. Obviously, $M_+\subset\widehat M_+$ by
regarding $x\in M_+$ as $\ffi\mapsto\ffi(x)$, $\ffi\in M_*^+$.

For any $m,n\in\widehat M_+$, $a\in M$ and $\lambda\ge0$, define $\lambda m$, $m+n$,
$a^*ma\in\widehat M_+$ by
$$
(\lambda m)(\ffi):=\lambda m(\ffi),\quad(m+n)(\ffi):=m(\ffi)+n(\ffi),\quad
(a^*ma)(\ffi):=m(a\ffi a^*)
$$
for $\ffi\in M_*^+$, where $a\ffi a*:=\ffi(a^*\cdot a)$. Define $m\le n$ if
$m(\ffi)\le n(\ffi)$ for all $\ffi\in M_*^+$. For an increasing net $(m_\alpha)$ in
$\widehat M_+$, $m\in\widehat M_+$ is defined by $m(\ffi):=\sup_\alpha m_\alpha(\ffi)$ for
$\ffi\in M_*^+$. In this case, write $m_\alpha\nearrow m$. In particular, the sum
$m=\sum_{i\in I}m_i\in\widehat M_+$ is defined for any family
$(m_i)_{i\in I}\subset\widehat M_+$.
\end{definition}

\begin{example}\label{E-8.2}\rm
(1)\enspace
Let $A$ be a positive self-adjoint operator affiliated with $M$ with the spectral decomposition
$A=\int_0^\infty\lambda\,de_\lambda$. Define
$$
m_A(\ffi):=\int_0^\infty\lambda\,d\ffi(e_\lambda),\qquad\ffi\in M_*^+.
$$
Then $m_A$ is lower semicontinuous on $M_*^+$ since $m_A(\ffi)=\sup_n\ffi(x_n)$ where
$x_n:=\int_0^n\lambda\,de_\lambda\in M_+$, $n\in\bN$. Hence we have $m_A\in\widehat M_+$. For
each $\xi\in\cH$ write $\omega_\xi$ for the vector functional $x\in M\mapsto\<\xi,x\xi\>$. Note
that
$$
m_A(\omega_\xi)=\int_0^\infty\lambda\,d\|e_\lambda\xi\|^2
=\begin{cases}\|A^{1/2}\xi\|^2 & \text{if $\xi\in\cD(A^{1/2})$}, \\
\infty & \text{otherwise}.\end{cases}
$$
For positive self-adjoint operators $A,B$ affiliated with $M$, if $m_A=m_B$, then
$\cD(A^{1/2})=\cD(B^{1/2})$ and $\|A^{1/2}\xi\|^2=\|B^{1/2}\xi\|^2$ for all
$\xi\in\cD(A^{1/2})$, which means that $A=B$ (see Theorem \ref{T-A.10} of Appendix~A).

(2)\enspace
Let $M=L^\infty(X,\mu)$ be a commutative von Neumann algebra over a $\sigma$-finite measure
space $(X,\cX,\mu)$. Let $f:X\to[0,\infty]$ be a measurable function. Define
$$
m_f(\ffi):=\int_Xf\ffi\,d\mu,\qquad\ffi\in L^1(X,\mu)_+\cong M_*^+.
$$
When $\ffi,\ffi_n\in L^1(X,\mu)_+$ with $\|\ffi_n-\ffi\|_1\to0$, a subsequence $\ffi_{n_k}$
can be chosen so that $m_f(\ffi_{n_k})\to\liminf_nm_f(\ffi_n)$ and $\ffi_{n_k}\to\ffi$ a.e.
By Fatou's lemma we have $m_f(\ffi)\le\liminf_km_f(\ffi_{n_k})=\liminf_nm_f(\ffi_n)$. Hence
$m_f$ is lower semicontinuous on $L^1(X,\mu)_+$, so $m_f\in\widehat M_+$. It is clear that
$m_f=m_f$ $\iff$ $f=g$ $\mu$-a.e. Conversely, for any $m\in\widehat M_+$, by Theorem
\ref{T-8.3} below, there is an increasing sequence $f_n\in L^\infty(X,\mu)_+$ such that
$\int_Xf_n\ffi\,d\mu\nearrow m(\ffi)$ for all $\ffi\in L^1(X,\mu)_+$. By letting
$f:=\sup_nf_n$ we have $m=m_f$. Thus, $\widehat M_+$ is identified with
$\{f:X\to[0,\infty]\ \mbox{measurable}\}$, where $f=g$ means $f=g$ $\mu$-a.e.
\end{example}

The next theorem gives an explicit description of $m\in\widehat M_+$ in terms of a certain
spectral resolution in $M$.

\begin{thm}\label{T-8.3}
Each $m\in\widehat M_+$ has a unique spectral resolution of the form
\begin{align}\label{F-8.1}
m(\ffi)=\int_0^\infty\lambda\,d\ffi(e_\lambda)+\infty\ffi(p),\qquad\ffi\in M_*^+,
\end{align}
where $(e_\lambda)_{\lambda\ge0}$ is an increasing family of projections in $M$ such that
$\lambda\mapsto e_\lambda$ is strongly right-continuous and
$p=1-\lim_{\lambda\to\infty}e_\lambda$. Moreover,
\begin{align}
e_0=0\ &\iff\ m(\ffi)>0\ \mbox{for any $\ffi\in M_*^+\setminus\{0\}$}, \label{F-8.2}\\
p=0\ &\iff\ \{\ffi\in M_*^+:m(\ffi)<\infty\}\ \mbox{is dense in $M_*^+$}. \label{F-8.3}
\end{align}
Consequently, there is an increasing sequence $x_n\in M_+$ such that
$\ffi(x_n)\nearrow m(\ffi)$ for all $\ffi\in M_*^+$.
\end{thm}

\begin{proof}(Sketch).\enspace
Define the function $q:\cH\to[0,\infty]$ by $q(\xi):=m(\omega_\xi)$, which is a positive form
form (see Definition \ref{D-A.13}) satisfying
\begin{itemize}
\item[(1)] $q(\lambda\xi)=|\lambda|^2q(\xi)$ for all $\lambda\in\bC$,
\item[(2)] $q(\xi+\eta)+q(\xi-\eta)=2q(\xi)+2q(\eta)$,
\item[(3)] $q$ is lower semicontinuous,
\item[(4)] $q(u'\xi)=q(\xi)$ for any unitary $u'\in M'$.
\end{itemize}
Then the closure $\cK$ of $\{\xi\in\cH:q(\xi)<\infty\}$ is a closed subspace of $\cH$. There
exists a positive self-adjoint operator $A$ on $\cK$ such that
$$
m(\omega_\xi)=q(\xi)=\begin{cases}
\|A^{1/2}\xi\|^2 & \text{if $\xi\in\cD(A^{1/2})$}, \\
\infty & \text{otherwise}.\end{cases}
$$
Furthermore, $A$ is affiliated with $M$ and the projection onto $\cK$ is in $M$. Take the
spectral decomposition $A=\int_0^\infty\lambda\,de_\lambda$. Then $e_\lambda\in M$ and
$e_\lambda\nearrow1-p$ as $\lambda\to\infty$, where $p$ is the projection onto $\cK^\perp$.
We then have
$$
m(\omega_\xi)=\int_0^\infty\lambda\,d\|e_\lambda\xi\|^2+\infty\|p\xi\|^2
=\int_0^\infty\lambda\,de\omega_\xi(e_\lambda)+\infty\omega_\xi(p),\qquad\xi\in\cH,
$$
which extends to \eqref{F-8.1} for all $\ffi\in M_*^+$. The uniqueness of
$(e_\lambda)_{\lambda\ge0}$ follows from that of $(\cK,A)$ as above. The assertions in
\eqref{F-8.2} and \eqref{F-8.3} are easy to verify. The last assertion follows by letting
$x_n:=\int_0^n\lambda\,de_\lambda+np$ for $n\in\bN$.
\end{proof}

\begin{prop}\label{P-8.4}
Any normal weight $\ffi$ on $M$ has a unique extension (denoted by the same $\ffi)$ to
$\widehat M_+$ satisfying
\begin{itemize}
\item[\rm(1)] $\ffi(\lambda m)=\lambda\ffi(m)$ for all $m\in\widehat M_+$, $\lambda\ge0$,
\item[\rm(2)] $\ffi(m+n)=\ffi(m)+\ffi(n)$ for all $m,n\in\widehat M_+$,
\item[\rm(3)] $m,m_\alpha\in\widehat M_+$, $m_\alpha\nearrow m$ $\implies$
$\ffi(m_\alpha)\nearrow\ffi(m)$.
\end{itemize}
\end{prop}

\begin{proof}
For $m\in\widehat M_+$ with the spectral resolution
$m=\int_0^\infty\lambda\,de_\lambda+\infty p$, let $x_n:=\int_0^n\lambda\,de_\lambda+np$ and
define $\ffi(m):=\lim_{n\to\infty}\ffi(x_n)$. On the other hand, it is known
\cite{PT} (see Theorem \ref{T-7.2}) that $\ffi$ is written as $\ffi=\sum_{i\in I}\omega_i$
with some $(\omega_i)_{i\in I}\subset M_*^+$. Then
$$
\ffi(m)=\lim_{n\to\infty}\sum_{i\in I}\omega_i(x_n)=\sum_{i\in I}m(\omega_i),
$$
from which it is easy to verify (1)--(3). The uniqueness of the extension is immediate by
applying (3) to $x_n\nearrow m$.
\end{proof}

\subsection{Operator valued weights}

Next, we state the definition of operator valued weights, which is considered as the unbounded
version (like weights) of conditional expectations discussed in Sec.~6.1. In the rest of the
section we assume that $N$ is a von Neumann subalgebra of $M$.

\begin{definition}\label{D-8.5}\rm
A map $T:M_+\to\widehat N_+$ satisfying the following properties is called an \emph{operator
valued weight} from $M$ to $N$: for any $x,y\in M_+$,
\begin{itemize}
\item $T(\lambda x)=\lambda T(x)$, $\lambda\ge0$,
\item $T(x+y)=T(x)+T(y)$,
\item $T(b^*xb)=b^*T(x)b$ for all $b\in N$.
\end{itemize}
For such $T$ define
\begin{align*}
\fF_T&:=\{x\in M_+:T(x)\in N_+\}, \\
\fN_T&:=\{x\in M:T(x^*x)\in N_+\}, \\
\fM_T&:=\fN_T^*\fN_T:=\lin\{y^*x:x,y\in\fN_T\}.
\end{align*}
It is easy to see that $\fN_T$ and $\fM_T$ are bimodules over $N$. We have $\fM_T=\lin\,\fF_T$
and $\fF_T=M_+\cap\fM_T$ similarly to the weight case in Definition \ref{D-7.1}.
\begin{itemize}
\item $T$ is said to be \emph{normal} if, for any increasing net $\{x_\alpha\}$ in $M_+$ with
$x_\alpha\nearrow x\in M_+$, $T(x_\alpha)\nearrow T(x)$.
\item $T$ is said to be \emph{faithful} if $T(x^*x)=0$ $\implies$ $x=0$ for any $x\in M$.
\item $T$ is said to be \emph{semifinite} if $\fN_T$ is weakly dense in $M$, or equivalently,
there is an increasing net $\{u_\alpha\}$ in $\fF_T$ such that $u_\alpha\nearrow1$.
\end{itemize}
\end{definition}

For an operator valued weight $T:M_+\to\widehat N_+$ we can define a linear map
$\dot T:\fM_T\to N$ in such a way that for any $x=x_1-x_2+i(x_3-x_4)$ with $x_k\in\fF_T$
($k=1,\dots4$),
$$
\dot T(x):=T(x_1)-T(x_2)+i(T(x_3)-T(x_4)).
$$
Then we have $\dot T(b_1xb_2)=a\dot T(x)b$ for all $x\in\fM_T$ and $b_1,b_2\in N$. Hence, if
$T(1)=1$, then $\dot T:M\to N$ is a conditional expectation.

We write
\begin{align*}
P(M)&:=\{\ffi:\mbox{f.s.n.\ weight on $M$}\}, \\
P(N)&:=\{\ffi:\mbox{f.s.n.\ weight on $N$}\}, \\
P(M,N)&:=\{T:\mbox{f.s.n.\ operator valued weight from $M$ to $N$}\}.
\end{align*}

\begin{prop}\label{P-8.6}
Let $T$ be a normal operator valued weight from $M$ to $N$ and $\ffi$ be a normal weight on
$N$. Then $\ffi\circ T$ is a normal weight on $M$. Furthermore, if $T\in P(M,N)$ and
$\ffi\in P(N)$, then $\ffi\circ T\in P(M)$.
\end{prop}

\begin{proof}
From Proposition \ref{P-8.4} it is obvious that $\ffi\circ T$ is a well-defined normal weight
on $M$. Let $T\in P(M,N)$ and $\ffi\in P(N)$. Assume that $x\in M$ and $\ffi(T(x^*x))=0$.
By Theorem \ref{T-8.3} choose a sequence $y_n\in N_+$ such that $y_n\nearrow T(x^*x)$. Then
$\ffi(y_n)=0$ and so $y_n=0$ for all $n$, implying $T(x^*x)=0$, so $x=0$ follows. Hence
$\ffi\circ T$ is faithful. Since $\ffi$ is semifinite, one can choose a net $\{b_\alpha\}$ in
$\fN_\ffi$ such that $b_\alpha\to1$ strongly. For any $x\in\fN_T$, since
$$
(\ffi\circ T)(b_\alpha^*x^*xb_\alpha)=\ffi(b_\alpha^*T(x^*x)b_\alpha)
\le\|T(x^*x)\|\ffi(b_\alpha^*b_\alpha)<+\infty,
$$
it follows that $xb_\alpha\in\fN_{\ffi\circ T}$. Since $xb_\alpha\to x$ strongly,
$\fN_{\ffi\circ T}$ is strongly dense in $\fN_T$. Since $T$ is semifinite, $\fN_{\ffi\circ T}$
is strongly dense in $N$, so $\ffi\circ T$ is semifinite.
\end{proof}

The rest of the section is devoted to the proof of the following theorem:

\begin{thm}[Haagerup]\label{T-8.7}
Let $T\in P(M,N)$ and $\ffi,\psi\in P(N)$. Then:
\begin{itemize}
\item[\rm(a)] $\sigma_t^{\ffi\circ T}(y)=\sigma_t^\ffi(y)$ for all $y\in N$ and $t\in\bR$,
\item[\rm(b)] $(D\ffi\circ T:D\psi\circ T)_t=(D\ffi:D\psi)_t$ for all $t\in\bR$.
\end{itemize}
\end{thm}

In particular, we have:

\begin{cor}\label{C-8.8}
Let $E:M\to N$ be a normal conditional expectation and $\ffi,\psi\in N_*^+$ be faithful. Then
$\sigma_t^{\ffi\circ E}(y)=\sigma_t^\ffi(y)$ and $(D\ffi\circ E:\psi\circ E)_t=(D\ffi:D\psi)_t$
for all $y\in N$ and $t\in\bR$.
\end{cor}

Since $\sigma_t^\ffi=\sigma_t^{\ffi,\ffi}$ and $(D\ffi:D\psi)_t=\sigma_t^{\psi,\ffi}(1)$ (see
Definition \ref{D-7.7} and a remark after that), the statements (a) and (b) of Theorem
\ref{T-8.7} are unified into
\begin{itemize}
\item[$(\star)$] $\sigma_t^{\psi\circ T,\ffi\circ T}(y)=\sigma_t^{\psi,\ffi}(y)$ for all
$y\in N$ and $t\in\bR$.
\end{itemize}

Toward the proof of $(\star)$, we start with recalling results of Cior\v anescu and Zsid\'o
\cite{CZ} on analytic generators for one-parameter groups of linear operators. Let $\sigma_t$
($t\in\bR$) be a $\sigma$-weakly (equivalently, weakly) continuous one-parameter group of
isometries on $M$ (though \cite{CZ} dealt with more general one-parameter groups of linear
operators on a Banach space). For $\alpha\in\bC$ with $\Im\alpha>0$, define
\begin{align*}
\cD(\sigma_\alpha)
&:=\{x\in M:\mbox{there exists a $\sigma$-weakly continuous $M$-valued function} \\ 
&\hskip2.3cm\mbox{$f_x(z)$ on $0\le\Im z\le\Im\alpha$, analytic in $0<\Im z<\Im\alpha$,} \\
&\hskip2.3cm\mbox{such that $f_x(t)=\sigma_t(x)$, $t\in\bR$}\},
\end{align*}
and similarly for $\alpha\in\bC$ with $\Im\alpha<0$. Define $\sigma_\alpha(x):=f_x(\alpha)$
for $x\in\cD(\sigma_\alpha)$. We further write $M(\sigma)$ for the set of $\sigma_t$-analytic
elements, i.e., the set of $x\in M$ such that there is a $\sigma$-weakly entire $M$-valued
function $f_x(z)$ such that $f_x(t)=\sigma_t(x)$, $t\in\bR$. We recall the next theorem without
proof from \cite[Theorems 2.4 and 4.4]{CZ}. The theorem explains the methodological idea of
the following proof of Theorem \ref{T-8.7} while it will not be directly utilized and its
extension to the inclusion $N\subset M$ setting will be given in Lemma \ref{L-8.14}.

\begin{thm}[Cior\v anescu and Zsid\'o]\label{T-8.9}
{\rm(1)}\enspace
In the above situation, for any $\alpha\in\bC$, $\sigma_\alpha$ is a closed densely-defined
operator on $M$ (with the $\sigma$-weak topology). In fact,
$M(\sigma)=\bigcap_{\alpha\in\bC}\cD(\sigma_\alpha)$ is $\sigma$-weakly dense in $M$. Moreover,
\begin{align*}
\sigma_\alpha\sigma_\beta&=\sigma_{\alpha+\beta},\qquad
\alpha,\beta\in\bC,\ (\Im\alpha)(\Im\beta)\ge0, \\
\sigma_{-\alpha}&=\sigma_\alpha^{-1},\qquad\ \,\alpha\in\bC.
\end{align*}

{\rm(2)}\enspace
Let $\widetilde\sigma_t$ ($t\in\bR$) be another $\sigma$-weakly continuous one-parameter group
of isometries on $M$. Then $\sigma_t=\widetilde\sigma_t$ for all $t\in\bR$ if and only if
$\sigma_{-i}=\widetilde\sigma_{-i}$.
\end{thm}

The operator $\sigma_{-i}$ is called the \emph{analytic generator} of $\sigma_t$. A familiar
version of Theorem \ref{T-8.9}\,(2) is Stone's representation theorem saying that a
continuous one-parameter unitary group $U_t$ on a Hilbert space is uniquely determined by its
generator (a positive self-adjoint operator) $A$ as $U_t=A^{it}$ ($t\in\bR$).

To prove the above $(\star)$, we need to analyze the analytic generator
$\sigma_{-i}^{\psi,\ffi}$, for which the next theorem is the most essential.

\begin{thm}\label{T-8.10}
Let $\ffi,\psi\in P(M)$ and $a,b\in M$. Then the following conditions are equivalent:
\begin{itemize}
\item[\rm(i)] $a\in\cD(\sigma_{-i}^{\psi,\ffi})$ and $b=\sigma_{-i}^{\psi,\ffi}(a)$;
\item[\rm(ii)] $a\fN_\ffi^*\subset\fN_\psi^*$, $\fN_\psi b\subset\fN_\ffi$ and
$\psi(ax)=\ffi(xb)$ for all $x\in\fN_\ffi^*\fN_\psi$.
\end{itemize}
\end{thm}

\begin{cor}\label{C-8.11}
Let $\ffi,\psi\in M_*^+$ be faithful. Then $a\in\cD(\sigma_{-i}^{\psi,\ffi})$ and
$b=\sigma_{-i}^{\psi,\ffi}(a)$ if and only if $\psi(ax)=\ffi(xb)$ for all $x\in M$.
\end{cor}

To prove Theorem \ref{T-8.10}, we first give the following lemma.

\begin{lemma}\label{L-8.12}
Let $\ffi,\psi\in P(M)$, $a\in M$ and $k>0$. Then the following conditions are equivalent:
\begin{itemize}
\item[\rm(a)] $a^*\psi a\ (=\psi(a\cdot a^*))\,\le k^2\ffi$;
\item[\rm(b)] $\fN_\ffi a^*\subset\fN_\psi$ and $\|(xa^*)_\psi\|\le k\|x_\ffi\|$ for all
$x\in\fN_\ffi$;
\item[\rm(c)] $a\in\cD\bigl(\sigma_{-i/2}^{\psi,\ffi}\bigr)$ and
$\big\|\sigma_{-i/2}^{\psi,\ffi}(a)\big\|\le k$.
\end{itemize}
If $\ffi=\psi$ and (a)--(c) hold, then
$$
(xa^*)_\ffi=J_\ffi\pi_\ffi\bigl(\sigma_{-i/2}^\ffi(a)\bigr)J_\psi x_\ffi,
\qquad x\in\fN_\ffi.
$$
\end{lemma}

\begin{proof}
(a)\,$\iff$\,(b).\enspace
(a) means that $\psi((xa^*)^*(xa^*))\le k^2\ffi(x^*x)$ for all $x\in M$, which is equivalent to
(b).

(c)\,$\implies$\,(a)\enspace[in the case $\ffi=\psi$].\enspace
Assume that $a\in\cD(\sigma_{-i/2}^\ffi)$ and $\|\sigma_{-i/2}^\ffi(a)\|\le k$. Since
$$
\Delta_\ffi^{it}\pi_\ffi(a)(x^*)_\ffi
=\pi_\ffi(\sigma_t^\ffi(a))\Delta_\ffi^{it}(x^*)_\ffi,\qquad x\in\fN_\ffi^*,
$$
we have, by analytic continuation,
$$
\Delta_\ffi^{1/2}\pi_\ffi(a)(x^*)_\ffi
=\pi_\ffi(\sigma_{-i/2}^\ffi(a))\Delta_\ffi^{1/2}(x^*)_\ffi,\qquad
x\in\fN_\ffi\cap\fN_\ffi^*,
$$
so that
$$
S_\ffi\pi_\ffi(a)(x^*)_\ffi
=J_\ffi\pi_\ffi(\sigma_{-i/2}^\ffi(a))J_\ffi S_\ffi(x^*)_\ffi,\qquad
x\in\fN_\ffi\cap\fN_\ffi^*.
$$
That is,
\begin{align}\label{F-8.4}
(xa^*)_\ffi=J_\ffi\pi_\ffi(\sigma_{-i/2}^\ffi(a))J_\ffi x_\ffi,\qquad
x\in\fN_\ffi\cap\fN_\ffi^*.
\end{align}
If $x\in M_+$ and $\ffi(x)<+\infty$, then $x^{1/2}\in\fN_\ffi\cap\fN_\ffi^*$ and so
$$
\ffi(axa^*)=\|(x^{1/2}a^*)_\ffi\|^2
\le\|\pi_\ffi(\sigma_{-i/2}^\ffi(a))\|^2\|(x^{1/2})_\ffi\|^2\le k^2\ffi(x).
$$
Hence (a) holds.

Moreover, since $\{x_\ffi:x\in\fN_\ffi\cap\fN_\ffi^*\}$ is dense in $\cH_\ffi$ and
$x_\ffi\mapsto(xa^*)_\ffi$ ($x\in\fN_\ffi$) is bounded thanks to (a)\,$\implies$\,(b), the
latter assertion of the lemma follows from \eqref{F-8.4}.

(b)\,$\implies$\,(c)\enspace[in the case $\ffi=\psi$].\enspace
By (b) there is a $T\in B(\cH_\ffi)$ with $\|T\|\le k$ such that $Tx_\ffi=(xa^*)_\ffi$ for all
$x\in\fN_\ffi$. For every $x\in\fN_\ffi\cap\fN_\ffi^*$, since
$$
\pi_\ffi(a)\Delta_\ffi^{-1/2}J_\ffi x_\ffi=\pi_\ffi(a)S_\ffi x_\ffi
=\pi_\ffi(a)(x^*)_\ffi=(ax^*)_\ffi,
$$
one has
$$
\Delta_\ffi^{1/2}\pi_\ffi(a)\Delta_\ffi^{-1/2}J_\ffi x_\ffi
=J_\ffi S_\ffi(ax^*)_\ffi=J_\ffi(xa^*)_\ffi=J_\ffi Tx_\ffi.
$$
Let $\xi\in\cD(\Delta_\ffi^{1/2})$ and $\eta\in\cD(\Delta_\ffi^{-1/2})$. Then the function
$z\mapsto\bigl\<\Delta_\ffi^{\overline{iz}}\xi,\pi_\ffi(a)\Delta_\ffi^{-iz}\eta\bigr\>$ is
analytic in $-1/2<\Im z<0$. When $\Im z=0$ so that $z=t$, one has
$$
\<\Delta_\ffi^{-it}\xi,\pi_\ffi(a)\Delta_\ffi^{-it}\eta\>
=\<\xi,\Delta_\ffi^{it}\pi_\ffi(a)\Delta_\ffi^{-it}\eta\>
=\<\xi,\pi_\ffi(\sigma_t^\ffi(a))\eta\>,
$$
$$
|\<\Delta_\ffi^{-it}\xi,\pi_\ffi(a)\Delta_\ffi^{-it}\eta\>|
\le\|a\|\,\|\xi\|\,\|\eta\|.
$$
When $\Im z=-1/2$ so that $z=t-i/2$, one has
\begin{align*}
\big|\big\<\Delta_\ffi^{\overline{iz}}\xi,\pi_\ffi(a)\Delta_\ffi^{-iz}\eta\bigr\>\big|
&=|\<\Delta_\ffi^{1/2}\Delta_\ffi^{-it}\xi,
\pi_\ffi(a)\Delta_\ffi^{-1/2}\Delta_\ffi^{-it}\eta\>| \\
&=|\<\Delta_\ffi^{-it}\xi,
\Delta_\ffi^{1/2}\pi_\ffi(a)\Delta_\ffi^{-1/2}\Delta_\ffi^{-it}\eta\>| \\
&=|\<\Delta_\ffi^{-it}\xi,J_\ffi TJ_\ffi\Delta_\ffi^{-it}\eta\>|
\le k\|\xi\|\,\|\eta\|.
\end{align*}
Therefore, from the three-lines theorem it follows that
$$
\big|\big\<\Delta_\ffi^{\overline{iz}}\xi,\pi_\ffi(a)\Delta_\ffi^{-iz}\eta\bigr\>\big|
\le(\|a\|\vee k)\|\xi\|\,\|\eta\|,\qquad-1/2\le\Im z\le0,
$$
so that there exists an $f(z)\in B(\cH_\ffi)$ for $-1/2\le\Im z\le0$ such that
$$
\<\xi,f(z)\eta\>
=\bigl\<\Delta_\ffi^{\overline{iz}}\xi,\pi_\ffi(a)\Delta_\ffi^{-iz}\eta\bigr\>,
\qquad\|f(z)\|\le\|a\|\vee k.
$$
Then it is easy to see that $f(z)$ is $\sigma$-weakly continuous on $-1/2\le\Im x\le0$ and
analytic in $-1/2<\Im z<0$. Moreover, we have $f(t)=\pi_\ffi(\sigma_t^\ffi(a))$, $t\in\bR$. It
is easily seen that $\pi_\ffi^{-1}(f(z))\in M$ for $-1/2\le\Im z\le0$ so that $\pi_\ffi(f(z))$
is the analytic continuation of $\sigma_t^\ffi(a)$ to $-1/2\le\Im z\le0$. Hence we find that
$a\in\cD\bigl(\sigma_{-i/2}^\ffi\bigr)$ and $\big\|\sigma_{-i/2}^\ffi(a)\big\|\le k$.

(a)\,$\iff$\,(c)\enspace[in the general case].\enspace
Consider the balanced weight $\theta:=\theta(\ffi,\psi)$ on $\bM_2(M)$. Let
$\widetilde a:=\begin{bmatrix}0&0\\a&0\end{bmatrix}$ and
$\widetilde x:=\begin{bmatrix}x_{11}&x_{12}\\x_{21}&x_{22}\end{bmatrix}\in\bM_2(M)_+$. Since
$\widetilde a\widetilde x\widetilde a^*=\begin{bmatrix}0&0\\0&ax_{11}a^*\end{bmatrix}$ so that
$\theta(\widetilde a\widetilde x\widetilde a^*)=\psi(ax_{11}a^*)$, one has
\begin{align*}
(a)&\ \iff\ \theta(\widetilde a\widetilde x\widetilde a^*)\le k^2\theta(\widetilde x)
\ \mbox{for all $\widetilde x\in\bM_2(M)_+$} \\
&\ \iff\ \widetilde a\in\cD\bigl(\sigma_{-i/2}^\theta\bigr)\ \mbox{and}
\ \big\|\sigma_{-i/2}^\theta(\widetilde a)\big\|\le k \\
&\ \iff\ (c)
\end{align*}
thanks to $\sigma_t^\theta(\widetilde a)=\begin{bmatrix}0&0\\
\sigma_t^{\psi,\ffi}(a)&0\end{bmatrix}$, $t\in\bR$, by Lemma \ref{L-7.5}.
\end{proof}

\begin{proof}[Proof of Theorem \ref{T-8.10}]
(i)\,$\implies$\,(ii)\enspace[in the case $\ffi=\psi$].\enspace
Assume that $a\in\cD(\sigma_{-i}^\ffi)$ and $b=\sigma_{-i}^\ffi(a)$. Then
$a\in\cD(\sigma_{-i/2}^\ffi)$, $b\in\cD(\sigma_{i/2}^\ffi)$ and
$\sigma_{-i/2}^\ffi(a)=\sigma_{i/2}^\ffi(b)$. For $x\in M$, since
$\sigma_t^\ffi(x^*)=\sigma_t^\ffi(x)^*$ for $t\in\bR$, note that
$x^*\in\cD(\sigma_{\overline z}^\ffi)\,\iff\,x\in\cD(\sigma_z^\ffi)$ for any $z\in\bC$, so
$b^*\in\cD(\sigma_{-i/2}^\ffi)$. Hence by Lemma \ref{L-8.12} we have
$\fN_\ffi a^*\subset\fN_\ffi$, so $a\fN_\ffi^*\subset\fN_\ffi^*$, and
$\fN_\ffi b\subset\fN_\ffi$. Therefore, $a\fM_\ffi\subset\fM_\ffi$ and
$\fM_\ffi b\subset\fM_\ffi$ so that $\ffi(ax)$ and $\ffi(xb)$ are well defined for any
$x\in\fM_\ffi$. If $x\in\fF_\ffi$, then, using Lemma \ref{L-8.12} twice, we have
\begin{align*}
\ffi(ax)&=\<(x^{1/2}a^*)_\ffi,(x^{1/2})_\ffi\> \\
&=\<J_\ffi\pi_\ffi(\sigma_{-i/2}^\ffi(a))J_\ffi(x^{1/2})_\ffi,(x^{1/2})_\ffi\> \\
&=\<J_\ffi\pi_\ffi(\sigma_{i/2}^\ffi(b))J_\ffi(x^{1/2})_\ffi,(x^{1/2})_\ffi\> \\
&=\<(x^{1/2}_\ffi,J_\ffi\pi_\ffi(\sigma_{i/2}^\ffi(b)^*)J_\ffi(x^{1/2})_\ffi\> \\
&=\<(x^{1/2}_\ffi,J_\ffi\pi_\ffi(\sigma_{-i/2}^\ffi(b^*))J_\ffi(x^{1/2})_\ffi\> \\
&=\<(x^{1/2})_\ffi,(x^{1/2}b)_\ffi\>=\ffi(xb).
\end{align*}
Since $\fM_\ffi=\lin\,\fF_\ffi$, it follows that $\ffi(ax)=\ffi(xb)$ for all $x\in\fM_\ffi$.

(ii)\,$\implies$\,(i)\enspace[in the case $\ffi=\psi$].\enspace
Assume (ii). We here utilize the Tomita algebra $\fT_\ffi$ associated with $\ffi$, see (E) of
Sec.~7.1. Let $\xi,\eta\in\fT_\ffi$, and let $\xi_1:=S_\ffi\xi$ and $\eta_1:=F_\ffi\eta$
($F_\ffi:=S_\ffi^*$). Since $\xi_1,\eta_1\in\fT_\ffi$, we write $\xi_1=x_\ffi$ and
$\eta_1=y_\ffi$ with $x,y\in\fN_\ffi\cap\fN_\ffi^*$. Since (ii) implies that $ay^*\in\fN_\ffi^*$
and $xb\in\fN_\ffi$, note that $ya^*,xb\in\fN_\ffi\cap\fN_\ffi^*$. Since
\begin{align*}
&\Delta_\ffi\xi=F_\ffi S_\ffi\xi=F_\ffi x_\ffi,\qquad\ \,\xi=S_\ffi\xi_1=S_\ffi x_\ffi, \\
&\Delta_\ffi^{-1}\eta=S_\ffi F_\ffi\eta=S_\ffi y_\ffi,\qquad\eta=F_\ffi\eta_1=F_\eta y_\ffi,
\end{align*}
we have
\begin{align}
\<\Delta_\ffi\xi,\pi_\ffi(a)\Delta_\ffi^{-1}\eta\>
&=\<F_\ffi x_\ffi,\pi_\ffi(a)S_\ffi y_\ffi\>
=\<F_\ffi x_\ffi,(ay^*)_\ffi\>=\<F_\ffi x_\ffi,S_\ffi(ya^*)_\ffi\> \nonumber\\
&=\<(ya^*)_\ffi,x_\ffi\>=\ffi(ay^*x)=\ffi(y^*xb)=\<y_\ffi,(xb)_\ffi\> \nonumber\\
&=\<S_\ffi(xb)_\ffi,F_\ffi y_\ffi\>=\<(b^*x^*)_\ffi,F_\ffi y_\ffi\>
=\<\pi_\ffi(b^*)S_\ffi x_\ffi,F_\ffi y_\ffi\> \nonumber\\
&=\<\xi,\pi_\ffi(b)\eta\>. \label{F-8.5}
\end{align}
For every $\xi,\eta\in\fT_\ffi$ consider the entire function
$$
z\in\bC\ \longmapsto\ \bigl\<\Delta_\ffi^{\overline{iz}}\xi,
\pi_\ffi(a)\Delta_\ffi^{-iz}\eta\bigr\>.
$$
When $\Im z=0$, i.e., $z=t$, one has
$$
\<\Delta_\ffi^{-it}\xi,\pi_\ffi(a)\Delta_\ffi^{-it}\eta\>
=\<\xi,\pi_\ffi(\sigma_t^\ffi(a))\eta\>,
$$
$$
|\<\Delta_\ffi^{-it}\xi,\pi_\ffi(a)\Delta_\ffi^{-it}\eta\>|\le\|a\|\,\|\xi\|\,\|\eta\|.
$$
When $\Im z=-1$, i.e., $z=t-i$, one has, by applying \eqref{F-8.5} to
$\Delta_\ffi^{-it}\xi,\Delta_\ffi^{-it}\eta\in\fT_\ffi$,
$$
\<\Delta_\ffi\Delta_\ffi^{-it}\xi,\pi_\ffi(a)\Delta_\ffi^{-1}\Delta_\ffi^{-it}\eta\>
=\<\Delta_\ffi^{-it}\xi,\pi_\ffi(b)\Delta_\ffi^{-it}\eta\>
=\<\xi,\pi_\ffi(\sigma_t^\ffi(b))\eta\>,
$$
$$
|\<\Delta_\ffi\Delta_\ffi^{-it}\xi,\pi_\ffi(a)\Delta_\ffi^{-1}\Delta_\ffi^{-it}\eta\>|
\le\|b\|\,\|\xi\|\,\|\eta\|.
$$
Therefore, as in the proof of Lemma \ref{L-8.12}, there exists a $B(\cH_\ffi)$-valued
$\sigma$-weakly continuous function $f(z)$ on $-1\le\Im z\le0$, analytic in $-1<\Im z<0$, such
that $f(t)=\pi_\ffi(\sigma_t^\ffi(a))$ and $f(t-i)=\pi_\ffi(\sigma_t^\ffi(b))$ for all
$t\in\bR$. Since $\pi_\ffi^{-1}(f(z))\in M$ for $-1\le\Im z\le0$, we have
$a\in\cD(\sigma_{t-i}^\ffi)$ and $\sigma_{t-i}^\ffi(a)=\sigma_t^\ffi(b)$ for all $t\in\bR$.
In particular, (i) holds.

(i)\,$\iff$\,(ii)\enspace[in the general case].\enspace
Consider the balanced weight $\theta:=\theta(\ffi,\psi)$ on $\bM_2(M)$. Let
$\widetilde a:=\begin{bmatrix}0&0\\a&0\end{bmatrix}$,
$\widetilde b:=\begin{bmatrix}0&0\\b&0\end{bmatrix}\in\bM_2(M)$. Since
$\sigma_t^\theta(\widetilde a)=\begin{bmatrix}0&0\\\sigma_t^{\psi,\ffi}(a)&0\end{bmatrix}$ for
$t\in\bR$,
$$
a\in\cD(\sigma_{-i}^{\psi,\ffi})\ \,\mbox{and}\ \,b=\sigma_{-i}^{\psi,\ffi}(a)
\ \iff \ \widetilde a\in\cD(\sigma_{-i}^\theta)\ \,\mbox{and}
\ \,\widetilde b=\sigma_{-i}^\theta(\widetilde a).
$$
Therefore, from the case $\ffi=\psi$,
$$
{\rm(i)}\ \iff\ \widetilde a\fN_\theta^*\subset\fN_\theta^*,
\ \,\fN_\theta\widetilde b\subset\fN_\theta\ \,\mbox{and}
\ \,\theta(\widetilde a\widetilde x)=\theta(\widetilde x\widetilde b)
\ \,\mbox{for all}\ \,\widetilde x\in\fM_\theta.
$$
By Lemma \ref{L-7.4} note that
$$
\widetilde a\fN_\theta^*\subset\fN_\theta^*\ \iff\ a\fN_\ffi^*\subset\fN_\psi^*,\qquad
\fN_\theta\widetilde b\subset\fN_\theta\ \iff\ \fN_\psi b\subset\fN_\ffi,
$$
and for $\widetilde x=\begin{bmatrix}x_{11}&x_{12}\\x_{21}&x_{22}\end{bmatrix}\in\fM_\theta
=\begin{bmatrix}\fM_\ffi&\fN_\ffi^*\fN_\psi\\\fN_\psi^*\fN_\ffi&\fM_\psi\end{bmatrix}$,
$\theta(\widetilde a\widetilde x)=\psi(ax_{12})$ and
$\theta(\widetilde x\widetilde b)=\ffi(x_{12}b)$. Combining the above facts altogether,
(i)\,$\iff$\,(ii) has been shown.
\end{proof}

To prove Theorem \ref{T-8.7}, we need several lemmas. Let $\sigma_t$ ($t\in\bR$) be a
$\sigma$-weakly continuous one-parameter group of isometries on $M$. An element $x\in M$ is
said to be of \emph{$\sigma$-exponential type} if $x$ is $\sigma$-analytic and there is a
$c>0$ such that
$$
\sup_{z\in\bC}\|\sigma_z(x)\|e^{-c|\Im z|}<+\infty.
$$

\begin{lemma}\label{L-8.13}
Let $M_{\exp}(\sigma)$ be the set of $x\in M$ of $\sigma$-exponential type. Then
$M_{\exp}(\sigma)$ is $\sigma$-weakly dense in $M$.
\end{lemma}

\begin{proof}
Let
$$
p_R(x):=\begin{cases}R^{-1/2}, & |x|\le R/2, \\ 0, & |x|>R/2. \end{cases}
$$
The Fourier transform of $p_R$ is
\begin{align*}
\widehat p_R(z)&=(2\pi)^{-1/2}\int_{-\infty}^\infty p_R(x)e^{-izx}\,dx
=(2\pi R)^{-1/2}\int_{-R/2}^{R/2}e^{-izx}\,dx \\
&=(2\pi R)^{-1/2}\biggl[{e^{-izx}\over-iz}\biggr]_{-R/2}^{R/2}
=\biggl({2\over\pi R}\biggr)^{1/2}{\sin(Rz/2)\over z},\qquad z\in\bC.
\end{align*}
Define
$$
q_R(z):=\widehat p_R(z)^2={2\sin^2(Rz/2)\over\pi Rz^2}={1-\cos Rz\over\pi Rz^2}.
$$
We then notice that $q_R(s)\ge0$ ($s\in\bR$) and
\begin{itemize}
\item[(1)] $\int_{-\infty}^\infty q_R(s)\,ds=1$,
\item[(2)] $\lim_{R\to\infty}\int_{-\infty}^\infty q_R(s)f(s)\,ds=f(0)$ for all continuous
bounded functions $f$ on $\bR$,
\item[(3)] $\int_{-\infty}^\infty|q_R(s+it)|\,ds\le e^{R|t|}$ for all $t\in\bR$.
\end{itemize}
Indeed, (1) follows from
$$
\int_{-\infty}^\infty q_R(s)\,ds=\|\widehat p_R\|_2^2=\|p_R\|_2^2=1.
$$
(2) is easy to check directly. (3) follows since
\begin{align*}
\int_{-\infty}^\infty|q_R(s+it)|\,ds&=\|\widehat p_R(\cdot+it)\|_2^2
=\|p_R(x)e^{tx}\|_2^2 \\
&={1\over R}\int_{-R/2}^{R/2}e^{2tx}\,dx={e^{Rt}-e^{-Rt}\over2Rt}\le e^{R|t|}.
\end{align*}

Now, for any $x\in M$ and $R>0$ let $x_R:=\int_{-\infty}^\infty q_R(s)\sigma_s(x)\,ds$. Define
$$
f(z):=\int_{-\infty}^\infty q_R(s-z)\sigma_s(x)\,ds,\qquad z\in\bC.
$$
Then, by Lebesgue's convergence theorem, one can see that $f(z)$ is an entire function. Since
$$
f(t)=\int_{-\infty}^\infty q_R(s)\sigma_{s+t}(x)\,ds=\sigma_t(x_R),\qquad t\in\bR,
$$
it follows that $x_R$ is $\sigma$-analytic. Furthermore, one has by (3)
$$
\|f(z)\|\le\|x\|\int_{-\infty}^\infty q_R(s-z)\,ds\le\|x\|e^{R|\Im z|},\qquad z\in\bC,
$$
so that $x_R\in M_{\exp}(\sigma)$ for any $R>0$. For every $\ffi\in M_*$, not by (1) and (2)
that
$$
\ffi(x_R-x)=\int_{-\infty}^\infty q_R(s)\ffi(\sigma_s(x)-x)\,ds
\ \longrightarrow\ 0\quad\mbox{as $R\to\infty$}.
$$
Hence, $x_R\to x$ $\sigma$-weakly as $R\to\infty$.
\end{proof}

The next lemma extends Theorem \ref{T-8.9}\,(2) to the case of a von Neumann subalgebra
$N\subset M$. When $N=M$, the lemma says that
$\widetilde\sigma_{-i}\subset\sigma_{-i}$\,$\implies$\,$\widetilde\sigma_{-i}=\sigma_{-i}$,
similarly to the fact that for self-adjoint operators $A,B$ on a Hilbert space,
$A\subset B$\,$\implies$\,$A=B$.

\begin{lemma}\label{L-8.14}
Let $N\subset M$ be a von Neumann subalgebra. Let $\sigma_t$ and $\widetilde\sigma_t$
($t\in\bR$) be $\sigma$-weakly continuous one-parameter groups of isometries on $M$ and $N$,
respectively. If $\widetilde\sigma_{-i}\subset\sigma_{-i}$, then
$\sigma_t(y)=\widetilde\sigma_t(y)$ for all $y\in N$ and all $t\in\bR$.
\end{lemma}

\begin{proof}
Assume that $\widetilde\sigma_{-i}\subset\sigma_{-i}$. First, we prove that
$N_{\exp}(\widetilde\sigma)\subset M_{\exp}(\sigma)$. Let
$y\in N_{\exp}(\widetilde\sigma)$, i.e., $y\in N$ is $\widetilde\sigma$-analytic and
$\|\widetilde\sigma_z(y)\|\le Ke^{c|\Im z|}$ for all $z\in\bC$ with some $K,c>0$. Hence, in
particular, $\|y\|\le K$. For every $n\in\bZ$, since
$\widetilde\sigma_{-in}=(\widetilde\sigma_{-i})^n\subset(\sigma_{-i})^n=\sigma_{-in}$, one has
$y\in\cD(\sigma_{-ni})$ and $\sigma_{-in}(y)=\widetilde\sigma_{-in}(y)$. Hence it is clear
that $y$ is $\sigma$-analytic. For every $z\in\bC$ write $z=s+it$ and $n-1\le t<n$ (or
$-(n-1)\ge t>-n$) for some $n\in\bN$. By the three-lines theorem, for $\Im z\ge0$ we have
\begin{align*}
\|\sigma_z(y)\|&\le\bigl(\sup_{s\in\bR}\|\sigma_s(y)\|\Bigr)\vee
\Bigl(\sup_{s\in\bR}\|\sigma_{s+in}(y)\|\Bigr)
=\|y\|\vee\sup_{s\in\bR}\|\sigma_s(\sigma_{in}(y))\| \\
&=\|y\|\vee\|\sigma_{in}(y)\|=\|y\|\vee\|\widetilde\sigma_{in}(y)\|\le Ke^{cn}
\le(Ke^c)e^{c|t|},
\end{align*}
and the same holds also for $\Im z\le0$. Hence $y\in M_{\exp}(\sigma)$.

Now, for any $N_{\exp}(\widetilde\sigma)$ ($\subset M_{\exp}(\sigma)$) set
$f(z):=\sigma_z(y)-\widetilde\sigma_z(y)$, $z\in\bC$, which is an entire function. Then the
following hold:
\begin{itemize}
\item $\|f(z)\|\le\|\sigma_z(y)\|+\|\widetilde\sigma_z(y)\|\le Ke^{c|\Im z|}$, $z\in\bC$,
for some $K,c>0$,
\item $\|f(t)\|\le2\|y\|$, $t\in\bR$,
\item $f(in)=(\sigma_{-i})^n(y)-(\widetilde\sigma_{-i})^n(y)=0$ for all $n\in\bZ$.
\end{itemize}
Carlson's theorem\footnote{
Carlson's theorem: Assume that $f$ is a continuous function on $\Im z\ge0$, analytic in
$\Im z>0$, $|f(z)|\le Ke^{c|z|}$, $\Im z>0$, with some $K,c>0$, $|f(t)|\le K'e^{c'|t|}$,
$t\in\bR$, with some $K',c'>0$, and $f(in)=0$ for all non-negative integers $n$. Then
$f(z)\equiv0$.}
implies that $f(z)\equiv0$, so $\sigma_t(y)=\widetilde\sigma_t(y)$ for all $t\in\bR$. Since
$N_{\exp}(\widetilde\sigma)$ is $\sigma$-weakly dense in $N$ by Lemma \ref{L-8.13}, the
result follows.
\end{proof}

\begin{lemma}\label{L-8.15}
Let $T\in P(M,N)$ and $\dot T:\fM_T\to N$ be as defined after Definition \ref{D-8.5}. Define
$R:\bM_2(\fM_T)\to\bM_2(N)$ by
$$
R\biggl(\begin{bmatrix}x_{11}&x_{12}\\x_{21}&x_{22}\end{bmatrix}\biggr)
:=\begin{bmatrix}\dot T(x_{11})&\dot T(x_{12})\\\dot T(x_{21})&\dot T(x_{22})\end{bmatrix},
\qquad x_{ij}\in\fM_T.
$$
Then:
\begin{itemize}
\item[\rm(1)] $\bM_2(\fM_T)$ is a bimodule over $\bM_2(N)$ and $R(axb)=aR(x)b$ for all
$x\in\bM_2(\fM_T)$ and $a,b\in\bM_2(N)$.
\item[\rm(2)] $\bM_2(\fM_T)=\lin\,\bM_2(\fM_T)_+$, and
$x\in\bM_2(\fM_T)_+$\,$\implies$\,$R(x)\ge0$.
\end{itemize}
\end{lemma}

\begin{proof}
(1)\enspace
Since $\fM_T$ is a bimodule over $N$, it is clear that so is $\bM_2(\fM_T)$ over $\bM_2(N)$.
If $x=[x_{ij}]\in\bM_1(\fM_T)$ and $a=[a_{ij}]$, $b=[b_{ij}]\in\bM_2(N)$, then
$axb=\bigl[\sum_{k,l=1}^2a_{ik}x_{kl}b_{lj}\bigr]_{i,j=1}^2$ and
$$
R(axb)=\Biggl[\sum_{k,l=1}^2a_{ik}\dot T(x_{kl})b_{lj}\Biggr]_{i,j=1}^2=aR(x)b.
$$

(2)\enspace
Since $\fM_T=\fN_T^*\fN_T$, we have
$$
\bM_2(\fM_T)=\lin\biggl\{\begin{bmatrix}x_{11}^*y_{11}&x_{12}^*y_{12}\\
x_{21}^*y_{21}&x_{22}^*y_{22}\end{bmatrix}:x_{ij},y_{ij}\in\fN_T\biggr\}.
$$
Furthermore, since
\begin{align*}
\begin{bmatrix}x_{11}^*y_{11}&x_{12}^*y_{12}\\x_{21}^*y_{21}&x_{22}^*y_{22}\end{bmatrix}
&=\begin{bmatrix}x_{11}&0\\0&0\end{bmatrix}^*\begin{bmatrix}y_{11}&0\\0&0\end{bmatrix}
+\begin{bmatrix}0&0\\x_{12}&0\end{bmatrix}^*\begin{bmatrix}0&0\\y_{12}&0\end{bmatrix} \\
&\qquad+\begin{bmatrix}0&x_{21}\\0&0\end{bmatrix}^*\begin{bmatrix}0&y_{21}\\0&0\end{bmatrix}
+\begin{bmatrix}0&0\\0&x_{22}\end{bmatrix}^*\begin{bmatrix}0&0\\0&y_{22}\end{bmatrix},
\end{align*}
we find that $\bM_2(\fM_T)=\bM_2(\fN_T)^*\bM_2(\fN_T)$. Hence as in the proof of Lemma
\ref{L-5.4}\,(2) with use of polarization, $\bM_2(\fM_T)=\lin\,\bM_2(\fM_T)_+$. For any
$\ffi\in N_*^+$ let $\theta:=\ffi\otimes\Tr$ and $\widetilde\theta:=(\ffi\circ T)\otimes\Tr$,
where $\Tr$ is the usual trace on $\bM_2$, i.e.,
$\theta\biggl(\begin{bmatrix}a_{11}&a_{12}\\a_{21}&a_{22}\end{bmatrix}\biggr)
=\ffi(a_{11})+\ffi(a_{22})$, $a_{ij}\in N$. Since $x\in\bM_2(\fM_T)_+$ $\implies$
$\widetilde\theta(x)<+\infty$, we have
$\fM_{\widetilde\theta}\supset\lin\,\bM_2(\fM_T)_+=\bM_2(\fM_T)$. Note that, for any
$x=[x_{ij}]\in\bM_2(\fM_T)$,
$$
\widetilde\theta(x)=\ffi\circ\dot T(x_{11})+\ffi\circ\dot T(x_{22})=\theta\circ R(x).
$$
With the cyclic representation $(\pi_\theta,\cH_\theta,\xi_\theta)$ of $\bM_2(N)$ associated
with $\theta$, for any $x\in\bM_2(\fM_T)$ and $a\in\bM_2(N)$,
\begin{align*}
\<\pi_\theta(a)\xi_\theta,\pi_\theta(R(x))\pi_\theta(a)\xi_\theta\>
&=\theta(a^*R(x)a)=\theta\circ R(a^*xa)\quad\mbox{(by (1))} \\
&=\widetilde\theta(a^*xa)\ge0
\end{align*}
so that $\pi_\theta(R(x))\ge0$, i.e., $s(\pi_\theta)R(x)\ge0$, where $s(\pi_\theta)$ is the
support projection of $\pi_\theta$. It is easy to see that
$\{\theta=\ffi\otimes\Tr:\ffi\in N_*^+\}$ is a separating family of $\bM_2(N)$, which implies
that $\bigvee_\theta s(\pi_\theta)=1$. Hence $R(x)\ge0$.
\end{proof}

\begin{lemma}\label{L-8.16}
Let $T\in P(M,N)$, $\ffi,\psi\in P(N)$ and $\widetilde\ffi:=\ffi\circ T$,
$\widetilde\psi:=\psi\circ T$. Then
$x\in(\fN_{\widetilde\ffi}\cap\fN_T)^*(\fN_{\widetilde\psi}\cap\fN_T)$ $\implies$
$\dot Tx\in\fN_\ffi^*\fN_\psi$.
\end{lemma}

\begin{proof}
First, recall that $\dot Tx\in N$ is well-defined for $x\in\fN_T^*\fN_T=\fM_T$. We may assume
that $x=y^*z$ with $y\in\fN_{\widetilde\ffi}\cap\fN_T$ and $z\in\fN_{\widetilde\psi}\cap\fN_T$.
Set $x_{11}:=y^*y$, $x_{12}:=y^*z=x$, $x_{21}:=z^*y$ and $x_{22}:=z^*z$; then
$\widetilde x:=\begin{bmatrix}x_{11}&x_{12}\\x_{21}&x_{22}\end{bmatrix}\in\bM_2(\fM_T)$ and
$\widetilde x=\begin{bmatrix}y&z\\0&0\end{bmatrix}^*\begin{bmatrix}y&z\\0&0\end{bmatrix}\ge0$.
Hence $R(\widetilde x)=\begin{bmatrix}\dot T(x_{11})&\dot T(x_{12})\\
\dot T(x_{21})&\dot T(x_{22})\end{bmatrix}\ge0$ by Lemma \ref{L-8.15}\,(2). Consider
$\theta:=\theta(\ffi,\psi)\in P(\bM_2(N))$ and note that
$$
\theta(R(\widetilde x))=\ffi(Tx_{11})+\psi(Tx_{22})
=\widetilde\ffi(y^*y)+\widetilde\psi(z^*z)<+\infty.
$$
Therefore, we have $R(\widetilde x)\in\fM_\theta$ so that
$\dot Tx=\dot Tx_{12}\in\fN_\ffi^*\fN_\psi$ by Lemma \ref{L-7.4}\,(2).
\end{proof}

We are now in a position to prove Theorem \ref{T-8.7}.

\begin{proof}[Proof of Theorem \ref{T-8.7}]
Let $T\in P(M,N)$ and $\ffi,\psi\in P(N)$. We prove the claim $(\star)$. By Lemma \ref{L-8.14}
it suffices to prove that $\sigma_{-i}^{\psi,\ffi}\subset\sigma_{-i}^{\psi\circ T,\ffi\circ T}$.
Let $\widetilde\ffi:=\ffi\circ T$ and $\widetilde\psi:=\psi\circ T$. Assume that
\begin{align}\label{F-8.6}
a\in\cD\bigl(\sigma_{-i}^{\psi,\ffi}\bigr)\qquad\mbox{and}
\qquad b=\sigma_{-i}^{\psi,\ffi}(a),
\end{align}
and prove that
\begin{align}\label{F-8.7}
a\in\cD\bigl(\sigma_{-i}^{\widetilde\psi,\widetilde\ffi}\bigr)\qquad\mbox{and}
\qquad b=\sigma_{-i}^{\widetilde\psi,\widetilde\ffi}(a).
\end{align}
One has $a\in\cD\bigl(\sigma_{-i/2}^{\psi,\ffi}\bigr)$,
$b\in\cD\bigl(\sigma_{i/2}^{\psi,\ffi}\bigr)$ and
$\sigma_{-i/2}^{\psi,\ffi}(a)=\sigma_{i/2}^{\psi,\ffi}(b)$, since
$b=\sigma_{-i/2}^{\psi,\ffi}\bigl(\sigma_{-i/2}^{\psi,\ffi}(a)\bigr)$. Since
$\sigma_t^{\ffi,\psi}(b^*)=\sigma_t^{\psi,\ffi}(b)^*$ by Lemma \ref{L-7.5}, one has
$b^*\in\cD\bigl(\sigma_{-i/2}^{\ffi,\psi}\bigr)$. Hence by Lemma \ref{L-8.12}, there is a
$k>0$ such that $a^*\psi a\le k^2\ffi$ and $b\ffi b^*\le k^2\psi$ on $N_+$. For any
$m\in\widehat N_+$, by Theorem \ref{T-8.3} there is a sequence $y_n\in N_+$ such that
$y_n\nearrow m$. Then
$$
\psi(ama^*)=(a^*\psi a)(m)=\lim_{n\to\infty}(a^*\psi a)(y_n)
\le k^2\lim_{n\to\infty}\ffi(y_n)=k^2\ffi(m)
$$
and similarly $\ffi(b^*mb)\le k^2\psi(m)$. Therefore, for every $x\in M_+$ one has
$$
\widetilde\psi(axa^*)=\psi(aT(x)a^*)\le k^2\ffi(T(x))=k^2\widetilde\ffi(x)
$$
and similarly $\widetilde\ffi(b^*xb)\le k^2\widetilde\psi(x)$. By Lemma \ref{L-8.12} again,
one has
\begin{align}
&\fN_{\widetilde\ffi}a^*\subset\fN_{\widetilde\psi},\ \ \mbox{i.e.},
\ \ a\fN_{\widetilde\ffi}^*\subset\fN_{\widetilde\psi}^*, \label{F-8.8}\\
&\fN_{\widetilde\psi}b\subset\fN_{\widetilde\ffi}, \label{F-8.9}\\
&\|(ya^*)_{\widetilde\psi}\|\le k\|y_{\widetilde\ffi}\|
\ \ \mbox{for all $y\in\fN_{\widetilde\ffi}$}, \label{F-8.10}\\
&\|(zb)_{\widetilde\ffi}\|\le k\|z_{\widetilde\psi}\|
\ \ \mbox{for all $z\in\fN_{\widetilde\psi}$}. \label{F-8.11}
\end{align}

By Theorem \ref{T-8.10} with \eqref{F-8.8} and \eqref{F-8.9}, to show \eqref{F-8.7}, it
remains to prove that
\begin{align}\label{F-8.12}
\widetilde\psi(ax)=\widetilde\ffi(xb)\qquad
\mbox{for all $x\in\fN_{\widetilde\ffi}^*\fN_{\widetilde\psi}$}.
\end{align}
First, assume that $x_0=y_0^*z_0$ with $y_0\in\fN_{\widetilde\ffi}\cap\fN_T$ and
$z_0\in\fN_{\widetilde\psi}\cap\fN_T$. Since $x_0\in\fM_T$ and $\dot Tx_0\in\fN_\ffi^*\fN_\psi$
by Lemma \ref{L-8.16}, we have by \eqref{F-8.6} and Theorem \ref{T-8.10}
$$
\psi(\dot T(ax_0))=\psi(a(\dot Tx_0))=\ffi((\dot Tx_0)b)=\ffi(\dot T(x_0b)).
$$
Moreover, note from \eqref{F-8.8} and \eqref{F-8.9} that
\begin{align*}
ax_0=(y_0a^*)^*z_0&\in(\fN_{\widetilde\psi}\cap\fN_T)^*(\fN_{\widetilde\psi}\cap\fN_T)
\subset\lin(\fF_{\widetilde\psi}\cap\fF_T), \\
x_0b=y_0^*(z_0b)&\in(\fN_{\widetilde\ffi}\cap\fN_T)^*(\fN_{\widetilde\ffi}\cap\fN_T)
\subset\lin(\fF_{\widetilde\ffi}\cap\fF_T),
\end{align*}
so that $\psi(\dot T(ax_0))=(\psi\circ T)(ax_0)=\widetilde\psi(ax_0)$ and
$\ffi(\dot T(x_0b))=(\ffi\circ T)(x_0b)=\widetilde\ffi(x_0b)$. Therefore,
\begin{align}\label{F-8.13}
\widetilde\psi(ax_0)=\widetilde\ffi(x_0b).
\end{align}

Next, assume that $x=y^*z$ with $y\in\fN_{\widetilde\ffi}$ and $z\in\fN_{\widetilde\psi}$.
Since $\ffi(T(y^*y))<+\infty$, $T(y^*y)$ has the spectral resolution
$T(y^*y)=\int_0^\infty\lambda\,de_\lambda$ (see Theorem \ref{T-8.3}). For any $s>0$, since
$$
T(e_sy^*ye_s)=e_sT(y^*y)e_s=\int_0^s\lambda\,de_\lambda,
$$
we have $ye_s\in\fN_T$ and
$$
\ffi\circ T(e_sy^*ye_s)=\ffi\biggl(\int_0^s\lambda\,de_\lambda\biggr)
\le\ffi\biggl(\int_0^\infty\lambda\,de_\lambda\biggr)<+\infty
$$
so that $ye_s\in\fN_{\widetilde\ffi}$. Hence $ye_s\in\fN_{\widetilde\ffi}\cap\fN_T$. Moreover,
\begin{align*}
\|(y-ye_s)_{\widetilde\ffi}\|^2
&=\ffi\circ T((y-ye_s)^*(y-ye_s)) \\
&=\ffi\circ T(y^*y-y^*ye_s-e_sy^*y+e_sy^*ye_s) \\
&=\ffi(T(y^*y)-T(y^*y)e_s-e_sT(y^*y)+e_sT(y^*y)e_s) \\
&=\ffi((1-e_s)T(y^*y)(1-e_s)) \\
&=\ffi\biggl(\int_s^\infty\lambda\,de_\lambda\biggr)
\ \longrightarrow\ 0\quad\mbox{as $s\to\infty$}.
\end{align*}
Similarly, $T(z^*z)$ has the spectral resolution $T(z^*z)=\int_0^\infty\lambda\,df_\lambda$,
and for any $s>0$, $zf_s\in\fN_{\widetilde\psi}\cap\fN_T$ and
$\|(z-zf_s)_{\widetilde\psi}\|^2\to0$ as $s\to\infty$. Fron \eqref{F-8.10} and \eqref{F-8.11}
it follows that
$$
\|(ya^*)_{\widetilde\psi}-(ye_sa^*)_{\widetilde\psi}\|\ \longrightarrow\ 0,\quad
\|(zb)_{\widetilde\psi}-(zf_sb)_{\widetilde\psi}\|\ \longrightarrow\ 0\quad
\mbox{as $s\to\infty$}.
$$
Therefore, we have
\begin{align*}
\widetilde\psi(ax)&=\widetilde\psi((ya^*)^*z)
=\<(ya^*)_{\widetilde\psi},z_{\widetilde\psi}\> \\
&=\lim_{s\to\infty}\<(ye_sa^*)_{\widetilde\psi},(zf_s)_{\widetilde\psi}\>
=\lim_{s\to\infty}\widetilde\psi(ae_sy^*zf_s)
=\lim_{s\to\infty}\widetilde\psi(a(e_sxf_s)), \\
\widetilde\ffi(xb)&=\widetilde\ffi(y^*zb)
=\<y_{\widetilde\ffi},(zb)_{\widetilde\ffi}\> \\
&=\lim_{s\to\infty}\<(ye_s)_{\widetilde\ffi},(zf_sb)_{\widetilde\ffi}\>
=\lim_{s\to\infty}\widetilde\ffi((e_sxf_s)b).
\end{align*}
From \eqref{F-8.13} for $x_0:=e_sxf_s=(ye_s)^*(zf_s)$ it follows that
$$
\widetilde\psi(a(e_sxf_s))=\widetilde\ffi((e_sxf_s)b),\qquad s>0.
$$
Letting $s\to\infty$ gives \eqref{F-8.12}.
\end{proof}

\section{Pedersen-Takesaki's construction}

In this section we present a concise description of Pedersen-Takesaki's construction \cite{PT}
and related results, based on the results in Sec.~8.2. Define the \emph{centralizer}
$M_\ffi$ for $\ffi\in P(M)$ as the fixed-point algebra of $\sigma_t^\ffi$ (see Proposition
\ref{P-2.16} in the case of $\ffi\in M_*^+$ faithful), i.e.,
$$
M_\ffi:=\{x\in M:\sigma_t^\ffi(x)=x,\,t\in\bR\}.
$$
Of course, $M_\ffi$ is a von Neumann subalgebra of $M$.

\begin{lemma}\label{L-9.1}
Let $\ffi\in P(M)$.
\begin{itemize}
\item[\rm(1)] For $a\in M$, $a\in M_\ffi$ $\iff$ $a\fN_\ffi^*\subset\fN_\ffi^*$,
$\fN_\ffi a\subset\fN_\ffi$ and $\ffi(ax)=\ffi(xa)$ for all $x\in\fM_\ffi$.
\item[\rm(2)] For a unitary $u\in M$, $u\in M_\ffi$ $\iff$ $\ffi(u\cdot u^*)=\ffi$.
\item[\rm(3)] If $a\in(M_\ffi)_+$ ($:=M_\ffi\cap M_+$), then
\begin{align}\label{F-9.1}
\ffi_a(x):=\ffi(a^{1/2}xa^{1/2}),\qquad x\in M_+
\end{align}
is a semifinite normal weight on $M$ and $\ffi_a(x)=\ffi(ax)=\ffi(xa)$ for all $x\in\fM_\ffi$.
Moreover,
\begin{align}\label{F-9.2}
\ffi_a(x^*x)=\|\pi_\ffi(a)^{1/2}J_\ffi x_\ffi\|^2
=\<J_\ffi x_\ffi,\pi_\ffi(a)J_\ffi x_\ffi\>,\qquad x\in\fN_\ffi,
\end{align}
\item[\rm(4)] If $a,b\in(M_\ffi)_+$, then $\ffi_{a+b}=\ffi_a+\ffi_b$. Hence, if $a,b\in M_\ffi$
and $0\le a\le b$, then $\ffi_a\le\ffi_b$.
\item[\rm(5)] If $a_\alpha$ is an increasing net in $(M_\ffi)_+$ and 
$a_\alpha\nearrow a\in(M_\ffi)_+$, then $\ffi_{a_\alpha}\nearrow\ffi_a$.
\end{itemize}
\end{lemma}

\begin{proof}
(1)\enspace
Assume that $a\in M_\ffi$; then it is obvious that $a\in\cD(\sigma_{-i}^\ffi)$ and
$\sigma_{-i}^\ffi(a)=a$. Hence the stated condition holds by Theorem \ref{T-8.10} for
$\ffi=\psi$. Conversely, assume the stated condition. Then the same holds for $a^*$ too
since for $y,z\in\fN_\ffi$,
$$
\ffi(a^*y^*z)=\ffi((ya)^*z)=\overline{\ffi(z^*ya)}=\overline{\ffi(az^*y)}
=\overline{\ffi((za^*)^*y)}=\ffi(y^*za^*).
$$
Hence we may assume that $a=a^*$. By Theorem \ref{T-8.10} for $\ffi=\psi$ again, we have
$a\in\cD(\sigma_{-i}^\ffi)$ and $\sigma_{-i}^\ffi(a)=a$ so that
$\sigma_{t-i}^\ffi(a)=\sigma_t^\ffi(\sigma_{-i}^\ffi(a))=\sigma_t^\ffi(a)$ for all $t\in\bR$.
Therefore, from the Schwarz reflection principle it follows that, for any $\xi\in\cH$,
$f_\xi(t):=\<\xi,\sigma_t^\ffi(a)\xi\>$ extends to an entire function
$f_\xi(z)$ with a period $i$. From the Liouville theorem, $f_\xi(z)\equiv f_\xi(0)$ for all
$\xi\in\cH$, so $\sigma_t^\ffi(a)=a$ for all $t\in\bR$. Hence $a\in M_\ffi$ holds.

(2)\enspace
Assume that $u\in M_\ffi$, so $u^*\in M_\ffi$ as well. Then by (1),
$u\fM_\ffi=\fM_\ffi u^*=\fM_\ffi$ and $\ffi(uxu^*)=\ffi(xu^*u)=\ffi(x)$ for all $x\in\fM_\ffi$.
Hence $\ffi(u\cdot u^*)=\ffi$ holds. Conversely, assume that $\ffi(u\cdot u^*)=\ffi$. Then
$\ffi(u^*\cdot u)=\ffi$ holds as well. So $\fN_\ffi u^*=\fN_\ffi u=\fN_\ffi$ and
$\ffi(ux)=\ffi(u^*uxu)=\ffi(xu)$ for all $x\in\fM_\ffi$. Hence $u\in M_\ffi$ follows from (1).

(3)\enspace
Assume that $a\in(M_\ffi)_+$. It is clear that $\ffi_a$ is a normal weight on $M$. If
$x\in\fN_\ffi$, then $xa\in\fN_\ffi$ by (1), so $\ffi_a(x^*x)=\ffi((xa)^*(xa))<+\infty$. This
means that $\fN_\ffi\subset\fN_{\ffi_a}$, so $\ffi_a$ is semifinite. For every $x\in\fM_\ffi$,
since $a^{1/2}\in M_\ffi$, by (1) one has $a^{1/2}x,xa^{1/2}\in\fM_\ffi$ and
$\ffi(a^{1/2}xa^{1/2})=\ffi(ax)=\ffi(xa)$. Moreover, for every $x\in\fN_\ffi$, by Lemma
\ref{L-8.12} for $\ffi=\psi$, one has $(xa^{1/2})_\ffi=J_\ffi\pi_\ffi(a^{1/2})J_\psi x_\ffi$,
which gives \eqref{F-9.2}.

(4)\enspace
We have unique contractions $v,w\in M$ such that $a^{1/2}=v(a+b)^{1/2}$, $b^{1/2}=w(a+b)^{1/2}$
and $v(1-s(a+b))=w(1-s(a+b))=0$, where $s(a+b)$ is the support projection of $a+b$. It is
immediate to see that $v,w\in M_\ffi$ and $v^*v+w^*w=s(a+b)$. Let $x\in M_+$. First, assume
that $\ffi_{a+b}(x)=+\infty$. If $\ffi_a(x)<+\infty$ and $\ffi_b(x)<+\infty$, then
$x^{1/2}a^{1/2},x^{1/2}b^{1/2}\in\fN_\ffi$ so that
$x^{1/2}(a+b)^{1/2}v^*v=x^{1/2}a^{1/2}v\in\fN_\ffi$ and
$x^{1/2}(a+b)^{1/2}w^*w=x^{1/2}b^{1/2}w\in\fN_\ffi$ thanks to (1). Hence
$x^{1/2}(a+b)^{1/2}\in\fN_\ffi$, contradicting $f_{a+b}(x)=+\infty$. Therefore,
$\ffi_a(x)+\ffi_v(x)=+\infty=\ffi_{a+b}(x)$ in this case.

Next, assume that $\ffi_{a+b}(x)<+\infty$ and so $(a+b)^{1/2}x(a+b)^{1/2}\in\fM_\ffi$. By (1)
one has $(a+b)^{1/2}x(a+b)^{1/2}v^*\in\fM_\ffi$ and
$$
\ffi_a(x)=\ffi(v(a+b)^{1/2}x(a+b)^{1/2}v^*)=\ffi((a+b)^{1/2}x(a+b)^{1/2}v^*v),
$$
and similarly $\ffi_b(x)=\ffi((a+b)^{1/2}x(a+b)^{1/2}w^*w)$. Hence
$\ffi_a(x)+\ffi_b(x)=\ffi_{a+b}(x)$ follows. The latter assertion of (4) is now obvious.

(5)\enspace
From (4) it is clear that $\ffi_{a_\alpha}\nearrow$ and $\ffi_{a_\alpha}\le\ffi_a$. For any
$x\in M_+$, since $a_n^{1/2}xa_n^{1/2}\to a^{1/2}xa^{1/2}$ strongly, by the lower
semicontinuity of $\ffi$ (see Theorem \ref{T-7.2}\,(iii)) one has
$$
\ffi_a(x)=\ffi(a^{1/2}xa^{1/2})\le\liminf_\alpha\ffi(a_\alpha^{1/2}xa_\alpha^{1/2})
\le\sup_\alpha\ffi_{a_\alpha}(x).
$$
Hence $\ffi_{a_\alpha}\nearrow\ffi_a$ follows.
\end{proof}

\begin{lemma}\label{L-9.2}
Let $\ffi\in P(M)$.
\begin{itemize}
\item[\rm(1)] If $\alpha$ is an automorphism of $M$, then
$$
\sigma_t^{\ffi\circ\alpha}=\alpha^{-1}\circ\sigma_t^\ffi\circ\alpha,\qquad t\in\bR.
$$
\item[\rm(2)] If $a\in M_\ffi$ is positive invertible, then $\ffi_a\in P(M)$ and
$$
\sigma_t^{\ffi_a}(x)=a^{it}\sigma_t^\ffi(x)a^{-it},\qquad x\in M,\ t\in\bR.
$$
\item[\rm(3)] If $e\in M_\ffi$ is a projection, then $\ffi_{|e}:=\ffi|_{eMe}\in P(eMe)$ and
$$
\sigma_t^{\ffi_{|e}}(x)=\sigma_t^\ffi(x),\qquad x\in eMe,\ t\in\bR.
$$
(Note that $\ffi_{|e}$ is essentially the same as $\ffi_e$ though $\ffi_e$ is defined on $M$.)
\end{itemize}
\end{lemma}

\begin{proof}
(1)\enspace
It is immediate to see that $\fN_{\ffi\circ\alpha}=\alpha^{-1}(\fN_\ffi)$ and so
$\fM_{\ffi\circ\alpha}=\alpha^{-1}(\fM_\ffi)$. Define a $\sigma$-weakly continuous
one-parameter automorphism group $\sigma_t$ by
$\sigma_t:=\alpha^{-1}\circ\sigma_t^\ffi\circ\alpha$ ($t\in\bR$). From Lemma \ref{L-8.14} in
the case of $N=M$ (or Theorem \ref{T-8.9}) it suffices to prove that
$\sigma_{-i}^{\ffi\circ\alpha}=\sigma_{-i}$. Note that
$\cD(\sigma_{-i})=\alpha^{-1}(\cD(\sigma_{-i}^\ffi))$ and
$\sigma_{-i}(a)=\alpha^{-1}\circ\sigma_{-i}^\ffi\circ\alpha(a)$ for all $a\in\cD(\sigma_{-i})$.
From Theorem \ref{T-8.10} we find that
\begin{align*}
&\mbox{$a\in\cD(\sigma_{-i}^{\ffi\circ\alpha})$ and $b=\sigma_{-i}^{\ffi\circ\alpha}(a)$} \\
&\quad\iff\,\mbox{$a\fN_{\ffi\circ\alpha}^*\subset\fN_{\ffi\circ\alpha}^*$,
$\fN_{\ffi\circ\alpha}b\subset\fN_{\ffi\circ\alpha}$
and $\ffi\circ\alpha(ax)=\ffi\circ\alpha(xb)$ for all $x\in\fM_{\ffi\circ\alpha}$} \\
&\quad\iff\,\mbox{$\alpha(a)\in\cD(\sigma_{-i}^\ffi)$
and $\alpha(b)=\sigma_{-i}^\ffi(\alpha(a))$} \\
&\quad\iff\,\mbox{$a\in\cD(\sigma_{-i})$ and $b=\sigma_{-i}(a)$}.
\end{align*}
Therefore, $\sigma_{-i}^{\ffi\circ\alpha}=\sigma_{-i}$ holds.

(2)\enspace
Assume that $a\in M_\ffi$ is positive invertible. It is clear that $\ffi_a\in P(M)$. Since
$\|a^{-1}\|^{-1}\le a\le \|a\|$, it follows from Lemma \ref{L-9.1}\,(4) that
$\|a^{-1}\|^{-1}\ffi\le \ffi_a\le\|a\|\ffi$ and so $\fN_\ffi=\fN_{\ffi_a}$ and
$\fM_\ffi=\fM_{\ffi_a}$. Define a $\sigma$-weakly continuous one-parameter automorphism group
$\sigma_t$ by $\sigma_t(x):=a^{it}\sigma_t^\ffi(x)a^{-it}$ ($x\in M$, $t\in\bR$). Note that
$\cD(\sigma_{-i})=\cD(\sigma_{-i}^\ffi)$ and $\sigma_{-i}(c)=a\sigma_{-i}^\ffi(c)a^{-1}$ for
all $c\in\cD(\sigma_{-i})$. From Theorem \ref{T-8.10} and Lemma \ref{L-9.1}\,(3), we find that
\begin{align*}
&\mbox{$c\in\cD(\sigma_{-i}^{\ffi_a})$ and $b=\sigma_{-i}^{\ffi_a}(c)$} \\
&\quad\iff\,\mbox{$c\fN_{\ffi_a}^*\subset\fN_{\ffi_a}^*$, $\fN_{\ffi_a}b\subset\fN_{\ffi_a}$
and $\ffi_a(cx)=\ffi_a(xb)$ for all $x\in\fM_{\ffi_a}$} \\
&\quad\iff\,\mbox{$c\fN_\ffi^*\subset\fN_\ffi^*$, $\fN_\ffi b\subset\fN_\ffi$
and $\ffi(cxa)=\ffi(xba)$ for all $x\in\fM_\ffi$} \\
&\quad\iff\,\mbox{$c\fN_\ffi^*\subset\fN_\ffi^*$, $\fN_\ffi(a^{-1}ba)\subset\fN_\ffi$
and $\ffi(cx)=\ffi(xa^{-1}ba)$ for all $x\in\fM_\ffi$} \\
&\hskip6cm\mbox{(since $\fN_\ffi a^{-1}=\fN_\ffi$ and $\fM_\ffi a^{-1}=\fM_\ffi$)} \\
&\quad\iff\,\mbox{$c\in\cD(\sigma_{-i}^\ffi)$ and $a^{-1}ba=\sigma_{-i}^\ffi(c)$} \\
&\quad\iff\,\mbox{$c\in\cD(\sigma_{-i})$ and $b=\sigma_{-i}(c)$}.
\end{align*}
Therefore, $\sigma_{-i}^{\ffi_a}=\sigma_{-i}$ so that $\sigma_t^{\ffi_a}=\sigma_t$ follows from
Lemma \ref{L-8.14} in the case of $N=M$.

(3)\enspace
It is clear that $\ffi_{|e}$ is faithful and normal. Since $\fN_{\ffi_{|e}}=e\fN_\ffi e$,
note that $\ffi_{|e}$ is also semifinite. Moreover, $\fM_{\ffi_{|e}}\subset e\fM_\ffi e$ is
clear. If $x\in e\fF_\ffi e$, then $x^{1/2}\in e\fN_\ffi e$ so that $x\in\fM_{\ffi_{|e}}$.
Hence $e\fM_\ffi e\subset\fM_{\ffi_{|e}}$, so one has $\fM_{\ffi_{|e}}=e\fM_\ffi e$. For every
$x\in eMe$, one has $\sigma_t^\ffi(x)=\sigma_t^\ffi(exe)=e\sigma_t^\ffi(x)e$. Hence one can
define a $\sigma$-weakly continuous one-parameter automorphism group $\sigma_t$ on $eMe$ by
$\sigma_t(x)=\sigma_t^\ffi(x)=e\sigma_t^\ffi(x)e$ ($s\in eMe$, $t\in\bR$). For $a,b\in eMe$,
from Theorem \ref{T-8.10} we have
\begin{align*}
&\mbox{$a\in\cD(\sigma_{-i})$ and $b=\sigma_{-i}(a)$} \\
&\quad\iff\,\mbox{$a\in\cD(\sigma_{-i}^\ffi)$ and $b=\sigma_{-i}^\ffi(a)$} \\
&\quad\iff\,\mbox{$a\fN_\ffi^*\subset\fN_\ffi^*$, $\fN_\ffi b\subset\fN_\ffi$ and
$\ffi(ax)=\ffi(xb)$ for all $x\in\fM_\ffi$} \\
&\quad\,\implies\ \mbox{$a\fN_{\ffi_{|e}}^*\subset\fN_{\ffi_{|e}}^*$,
$\fN_{\ffi_{|e}}b\subset\fN_{\ffi_{|e}}$ and
$\ffi(ax)=\ffi(xb)$ for all $x\in\fM_{\ffi_{|e}}$} \\
&\quad\iff\,\mbox{$a\in\cD(\sigma_{-i}^{\ffi_{|e}})$ and $b=\sigma_{-i}^{\ffi_{|e}}(a)$}.
\end{align*}
Therefore, $\sigma_{-i}\subset\sigma_{-i}^{\ffi_{|e}}$ so that $\sigma_t^{\ffi_{|e}}=\sigma_t$
follows from Lemma \ref{L-8.14}.
\end{proof}

Note that another standard way to prove Lemma \ref{L-9.2} is to use the KMS condition
characterizing the modular automorphism group (see (C) in Sec.~7.1).

Let $\ffi\in P(M)$ and $A$ be a positive self-adjoint operator affiliated with $M_\ffi$. We
extend $\ffi_a$ (defined by \eqref{F-9.1} for $a\in(M_\ffi)_+$) to $\ffi_A$. To do so, set
$$
A_\eps:=A(1+\eps A)^{-1}\in(M_\ffi)_+,\qquad\eps>0.
$$
Then Lemma \ref{L-9.1}\,(4) implies that $\ffi_{A_\eps}\nearrow$ as $\eps\searrow$, so we
define
\begin{align}\label{F-9.3}
\ffi_A(x):=\sup_{\eps>0}\ffi_{A_\eps}(x)=\lim_{\eps\searrow0}\ffi_{A_\eps}(x),
\qquad x\in M_+.
\end{align}

\begin{prop}\label{P-9.3}
Let $\ffi$ and $A$ be as above.
\begin{itemize}
\item[\rm(1)] $\ffi_A$ is a semifinite normal weight on $M$, and
$$
\ffi_A(x^*x)=\|\pi_\ffi(A)^{1/2}J_\ffi x_\ffi\|^2,\qquad x\in\fN_\ffi,
$$
where the quadratic form in the right-hand side $=\infty$ unless
$J_\ffi x_\ffi\in\cD(\pi_\ffi(A)^{1/2})$.
\item[\rm(2)] $\ffi_A$ is faithful if and only if $A$ is non-singular.
\item[\rm(3)] Let $B$ be another positive self-adjoint operator affiliated with $M_\ffi$. Then
$A\le B$ (in the sense of Definition \ref{D-A.2} of Appendix A) $\iff$ $\ffi_A\le\ffi_B$.
\item[\rm(4)] Let $A_\alpha$ be a net of positive self-adjoint operators affiliated with
$M_\ffi$. Then $A_\alpha\nearrow A$ (in the sense of Definitions \ref{D-A.2} and \ref{D-A.6})
$\iff$ $\ffi_{A_\alpha}\nearrow\ffi_A$.
\end{itemize}
\end{prop}

\begin{proof}
(1)\enspace
That $\ffi_A$ is normal is clear by definition \eqref{F-9.3}. Let $E_n$ be the spectral
projection of $A$ corresponding to $[0,n]$, so $E_n\in M_\ffi$. For every $x\in\fN_\ffi$
one has
$$
\ffi_{A_\eps}(E_nx^*xE_n)=\ffi(A_\eps^{1/2}E_nx^*xE_nA_\eps^{1/2})
=\ffi(x^*xE_nA_\eps)\le n\ffi(x^*x)<+\infty
$$
thanks to Lemma \ref{L-9.1}\,(1) and (4). Hence $\ffi_A(E_nx^*xE_n)<+\infty$ so that
$\fN_\ffi E_n\subset\fN_{\ffi_A}$. Since $E_n\nearrow1$, $\ffi_A$ is semifinite. Moreover,
for every $x\in\fN_\ffi$, by \eqref{F-9.2}
$$
\ffi_A(x^*x)=\sup_{\eps>0}\|\pi_\ffi(A_\eps)^{1/2}J_\ffi x_\ffi\|^2
=\|\pi_\ffi(A)^{1/2}J_\ffi x_\ffi\|^2.
$$

(2) is easy to verify.

(3)\enspace
If $A\le B$, then $A_\eps\le B_\eps$ for all $\eps>0$ by Lemma \ref{L-A.1}. By Lemma
\ref{L-9.1}\,(4), for any $x\in M_+$ one has $\ffi_{A_\eps}(x)\le\ffi_{B_\eps}(x)\le\ffi_B(x)$.
Hence $\ffi_A(x)\le\ffi_B(x)$. Conversely, assume that $\ffi_A\le\ffi_B$. Let $F_n$ be the
spectral projection of $B$ corresponding to $[0,n]$. Then, for every $x\in\fN_\ffi$ one has
\begin{align*}
\<J_\ffi x_\ffi,\pi_\ffi(F_nA_\eps F_n)J_\ffi x_\ffi\>
&=\<\pi_\ffi(F_n)J_\ffi x_\ffi,\pi_\ffi(A_\eps)\pi_\ffi(F_n)J_\ffi x_\ffi\> \\
&=\<J_\ffi(xF_n)_\ffi,\pi_\ffi(A_\eps)J_\ffi(xF_n)_\ffi\>\quad\mbox{(by Lemma \ref{L-8.12})} \\
&=\ffi_{A_\eps}(F_nx^*xF_n)\quad\mbox{(by \eqref{F-9.2})} \\
&\le\ffi_B(F_nx^*xF_n)=\sup_{\eps>0}\ffi(B_\eps^{1/2}F_nx^*xF_nB_\eps^{1/2}) \\
&\le\ffi_{BF_n}(x^*x)=\<J_\ffi x_\ffi,\pi_\ffi(BF_n)J_\ffi x_\ffi\>,
\end{align*}
which implies that $F_nA_\eps F_n\le BF_n\le B$. Since $F_n\nearrow1$, $A_\eps\le B$ for any
$\eps>0$, implying that $A\le B$.

(4)\enspace
For any $\eps>0$, since $(A_\alpha)_\eps\nearrow A_\eps$ by Lemmas \ref{L-A.1} and
\ref{L-A.4}, for every $x\in M_+$ one has $\ffi_{(A_\alpha)_\eps}(x)\nearrow\ffi_{A_\eps}(x)$
by Lemma \ref{L-9.1}\,(5). This implies that $\ffi_{A_\alpha}\nearrow$ and
$$
\ffi_A(x)=\sup_{\eps>0}\ffi_{A_\eps}(x)=\sup_{\eps>0}\sup_\alpha\ffi_{(A_\alpha)_\eps}(x)
=\sup_\alpha\ffi_{A_\alpha}(x)
$$
so that $\ffi_{A_\alpha}\nearrow\ffi_A$. Conversely, assume that
$\ffi_{A_\alpha}\nearrow\ffi_A$. It follows from (3) that $A_\alpha$ is an increasing net and
$A_\alpha\le A$. Hence there is a positive self-adjoint operator $B$ such that
$A_\alpha\nearrow B$. Noting that $B\,\eta M_\ffi$, one has $\ffi_{A_\alpha}\nearrow\ffi_B$
by the first part of the proof. Hence $\ffi_A=\ffi_B$ so that $A=B$ holds by (3).
\end{proof}

\begin{thm}\label{T-9.4}
Let $\ffi\in P(M)$ and $A$ be a non-singular positive self-adjoint operator affiliated with
$M_\ffi$. Then:
\begin{itemize}
\item[\rm(1)] $\sigma_t^{\ffi_A}(x)=A^{it}\sigma_t^\ffi(x)A^{-it}$ for all $x\in M$ and $t\in\bR$.
\item[\rm(2)] $(D\ffi_A:D\ffi)_t=A^{it}$ for all $t\in\bR$.
\end{itemize}
\end{thm}

For the proof we need one more lemma.

\begin{lemma}\label{L-9.5}
Let $\ffi$ and $A$ be as in Theorem \ref{T-9.4}. Let $e_n$ be the spectral projection of $A$
corresponding to $[1/n,n]$ for $n\in\bN$. Then $e_n\in M_{\ffi_A}$.
\end{lemma}

\begin{proof}
For every $\eps>0$, since $\ffi_{A_\eps}\le\ffi_A$ by Proposition
\ref{P-9.3}\,(3), one has $\fN_{\ffi_A}\subset\fN_{\ffi_{A_\eps}}$. For any $x\in M$ one has
$$
\ffi_{A_\eps}(e_nx^*xe_n)=\ffi(A_\eps^{1/2}e_nx^*xe_nA_\eps^{1/2})
=\ffi_{A_\eps e_n}(x^*x)\le\ffi_A(x^*x).
$$
Letting $\eps\searrow0$ yields that
$\ffi_{A_\eps}(e_nx^*xe_n)\le\ffi_A(e_nx^*xe_n)\le\ffi_A(x^*x)$. Therefore, $\fN_{\ffi_A}e_n\subset\fN_{\ffi_A}$ and
$\fN_{\ffi_A}e_n\subset\fN_{\ffi_{A_\eps}}$, so it follows that
\begin{align*}
\fM_{\ffi_A}e_n&=\fN_{\ffi_A}^*\fN_{\ffi_A}e_n\subset
\fN_{\ffi_{A_\eps}}^*\fN_{\ffi_{A_\eps}}=\fM_{\ffi_{A_\eps}}, \\
e_n\fM_{\ffi_A}&=e_n\fN_{\ffi_A}^*\fN_{\ffi_A}\subset
\fN_{\ffi_{A_\eps}}^*\fN_{\ffi_{A_\eps}}=\fM_{\ffi_{A_\eps}}.
\end{align*}
Furthermore, since $\ffi_A(x^*x)=\sup_{\eps>0}\ffi(A_\eps^{1/2}x^*xA_\eps^{1/2})$ for all
$x\in M$, we have
\begin{align}\label{F-9.4}
\fN_{\ffi_A}A_\eps^{1/2}\subset\fN_\ffi,\qquad
\mbox{hence}\quad A_\eps^{1/2}\fM_{\ffi_A}A_\eps^{1/2}\subset\fM_\ffi
\end{align}
for all $\eps>0$. For every $x\in\fM_{\ffi_A}$, since $e_nx,xe_n\in\fM_{\ffi_{A_\eps}}$, we
find that
\begin{align*}
\ffi_{A_\eps}(e_nx)&=\ffi(A_\eps^{1/2}e_nxA_\eps^{1/2})
=\ffi(e_nA_\eps^{1/2}xA_\eps^{1/2}) \\
&=\ffi(A_\eps^{1/2}xA_\eps^{1/2}e_n)\quad
\mbox{(by \eqref{F-9.4} and Lemma \ref{L-9.1}\,(1))} \\
&=\ffi(A_\eps^{1/2}xe_nA_\eps^{1/2})=\ffi_{A_\eps}(xe_n).
\end{align*}
Letting $\eps\searrow0$ yields that $\ffi_A(e_nx)=\ffi_A(xe_n)$. This, with
$\fN_{\ffi_A}e_n\subset\fN_{\ffi_A}$ shown above, implies by Lemma \ref{L-9.1}\,(1) that
$e_n\in M_{\ffi_A}$.
\end{proof}

\begin{proof}[Proof of Theorem \ref{T-9.4}]
(1)\enspace
For each $n\in\bN$ let $e_n$ be as given in Lemma \ref{L-9.5}. From Lemma \ref{L-9.2}\,(3)
it follows that $\sigma_t^{\ffi_{|e_n}}(Ae_n)=\sigma_t^\ffi(Ae_n)=Ae_n$ for all $t\in\bR$,
so that $Ae_n\in M_{\ffi_{|e_n}}$. On the other hand, $e_n\in M_{\ffi_A}$ by Lemma
\ref{L-9.5}. For every $x\in(e_nMe_n)_+$ note that
$$
(\ffi_A)_{|e_n}(x)=\ffi_A(x)=\sup_{\eps>0}\ffi(A_\eps^{1/2}e_nxe_nA_\eps^{1/2})
=\sup_{\eps>0}(\ffi_{|e_n})_{A_\eps e_n}(x)=(\ffi_{|e_n})_{Ae_n}(x),
$$
where the last equality follows from Proposition \ref{P-9.3}\,(4) since
$A_\eps e_n\nearrow Ae_n$. Therefore,
$$
(\ffi_A)_{|e_n}=(\ffi_{|e_n})_{Ae_n}.
$$
Since $Ae_n$ is invertible in $e_nMe_n$, for any $x\in e_nMe_n$ we have
\begin{align*}
\sigma_t^{\ffi_A}(x)&=\sigma_t^{(\ffi_A)_{|e_n}}(x)=\sigma_t^{(\ffi_{|e_n})_{Ae_n}}(x)
=(Ae_n)^{it}\sigma_t^{\ffi_{|e_n}}(x)(Ae_n)^{-it} \\
&=A^{it}e_n\sigma_t^\ffi(x)e_nA^{-it}=A^{it}\sigma_t^\ffi(x)A^{-it},
\end{align*}
where the first and the fourth equalities follow from Lemma \ref{L-9.2}\,(3) and the third
equality follows from Lemma \ref{L-9.2}\,(2). Since $\bigcup_{n\in\bN}e_nMe_n$ is strongly
dense in $M$ due to $e_n\nearrow1$, the result follows.

(2)\enspace
Consider the balanced weights $\theta:=\theta(\ffi,\ffi)$ and
$\widetilde\theta:=\theta(\ffi,\ffi_A)$. If $a,b\in M_\ffi$, then
$\sigma^\theta_t\biggl(\begin{bmatrix}a&0\\0&b\end{bmatrix}\biggr)
=\begin{bmatrix}\sigma_t^\ffi(a)&0\\0&\sigma_t^\ffi(b)\end{bmatrix}
=\begin{bmatrix}a&0\\0&b\end{bmatrix}$ by Lemma \ref{L-7.5}, so
$\begin{bmatrix}a&0\\0&b\end{bmatrix}\in M_\theta$. Hence we have
$\widetilde A:=\begin{bmatrix}1&0\\0&A\end{bmatrix}\,\eta M_\theta$. Then it is immediate to
see that $\widetilde\theta=\theta_{\widetilde A}$. Therefore, from (1) it follows that
\begin{align*}
\begin{bmatrix}0&0\\(D\ffi_A:D\ffi)_t&0\end{bmatrix}
&=\sigma_t^{\widetilde\theta}\biggl(\begin{bmatrix}0&0\\1&0\end{bmatrix}\biggr)
=\begin{bmatrix}1&0\\0&A\end{bmatrix}^{it}
\sigma_t^\theta\biggl(\begin{bmatrix}0&0\\1&0\end{bmatrix}\biggr)
\begin{bmatrix}1&0\\0&A\end{bmatrix}^{-it} \\
&=\begin{bmatrix}1&0\\0&A^{it}\end{bmatrix}\begin{bmatrix}0&0\\1&0\end{bmatrix}
\begin{bmatrix}1&0\\0&A^{-it}\end{bmatrix}
=\begin{bmatrix}0&0\\A^{it}&0\end{bmatrix}
\end{align*}
so that $(D\ffi_A:D\ffi)_t=A^{it}$ holds.
\end{proof}

While (2) of Theorem \ref{T-9.4} has been proved from (1), we note that (1) conversely follows
from (2) in view of Theorem \ref{T-7.6}.

Here, we give some basic properties of Connes' cocycle derivatives to use in the proof of
Theorem \ref{T-9.7}. 

\begin{prop}\label{P-9.6}
Let $\ffi$, $\psi$ and $\ffi_i$ be f.s.n.\ weights on $M$.
\begin{itemize}
\item[\rm(1)] $(D\ffi:D\psi)_t=(D\psi:D\ffi)_t^*$, $t\in\bR$.
\item[\rm(2)] $\psi=\ffi$ $\iff$ $(D\psi:D\ffi)_t=1$, $t\in\bR$.
\item[\rm(3)] $(D\ffi_1:D\ffi_3)_t=(D\ffi_1:D\ffi_2)_t(D\ffi_2:D\ffi_3)_t$, $t\in\bR$
\ \ (\emph{chain rule}).
\item[\rm(4)] $\ffi_1=\ffi_2$ $\iff$ $(D\ffi_1:D\ffi)_t=(D\ffi_2:D\ffi)_t$, $t\in\bR$.
\end{itemize}
\end{prop}

\begin{proof}
(1)\enspace
Let $\theta:=\theta(\ffi,\psi)$ and $\widehat\theta:=\theta(\psi,\ffi)$. Define an
automorphism $\gamma$ of $\bM_2(M)$ by
$\gamma(x):=\begin{bmatrix}0&1\\1&0\end{bmatrix}x\begin{bmatrix}0&1\\1&0\end{bmatrix}$ for
$x=\begin{bmatrix}x_{11}&x_{12}\\x_{21}&x_{22}\end{bmatrix}\in\bM_2(M)$. Since
$\widehat\theta(x)=\theta(\gamma(x))$ for all $x\in\bM_2(M)$, we find by Lemma \ref{L-9.2}\,(1)
that
$$
\sigma_t^{\widehat\theta}(x)=\gamma^{-1}\circ\sigma_t^\theta\circ\gamma(x)
=\begin{bmatrix}\sigma^\psi(x_{11})&\sigma^{\psi,\ffi}(x_{12})\\
\sigma^{\ffi,\psi}(x_{21})&\sigma^\ffi(x_{22})\end{bmatrix},
$$
and hence $(D\ffi:D\psi)_t=\sigma_t^{\ffi,\psi}(1)
=\sigma_t^{\psi,\ffi}(1)^*=(D\psi:D\ffi)_t^*$ by Lemma \ref{L-7.5}.

(2)\enspace
That $(D\ffi:D\ffi)_t=1$ is clear as seen at the end of Sec.~7.2. Conversely, assume that
$(D\psi:D\ffi)=1$ for all $t\in\bR$. Let $\theta:=\theta(\ffi,\psi)$. Then we have
$\sigma_t^\theta\biggl(\begin{bmatrix}0&0\\1&0\end{bmatrix}\biggr)
=\begin{bmatrix}0&0\\1&0\end{bmatrix}$ and so
$\sigma_t^\theta\biggl(\begin{bmatrix}0&1\\1&0\end{bmatrix}\biggr)
=\begin{bmatrix}0&1\\1&0\end{bmatrix}$ for all $t\in\bR$. Hence
$\begin{bmatrix}0&1\\1&0\end{bmatrix}\in(\bM_2(M))_\theta$ follows. By Lemma \ref{L-9.1}\,(1)
this implies that
$$
\ffi(x)=\theta\biggl(\begin{bmatrix}x&0\\0&0\end{bmatrix}\biggr)
=\theta\biggl(\begin{bmatrix}0&1\\1&0\end{bmatrix}\begin{bmatrix}x&0\\0&0\end{bmatrix}
\begin{bmatrix}0&1\\1&0\end{bmatrix}\biggr)
=\theta\biggl(\begin{bmatrix}0&0\\0&x\end{bmatrix}\biggr)=\psi(x)
$$
for all $x\in M_+$.

(3)\enspace
It is convenient to use the triple balanced weight on $\bM_3(M)$:
$$
\theta\Biggl(\sum_{i,j=1}^3x_{ij}\otimes e_{ij}\Biggr)
:=\sum_{i=1}^3\ffi_i(x_{ii}),\qquad\sum_{i,j=1}^3x_{ij}\otimes e_{ij}\in\bM_3(M)_+.
$$
As in the case of $\theta(\ffi,\psi)$ treated in Sec.~7.2, though the details are omitted here,
we have
$$
\sigma_t^\theta(x\otimes e_{ij})=\sigma_t^{\ffi_i,\ffi_j}(x)\otimes e_{ij},
\qquad x\in M,\ i,j=1,2,3.
$$
Therefore,
\begin{align*}
(D\ffi_1:D\ffi_3)_t\otimes e_{13}
&=\sigma_t^{\ffi_1,\ffi_3}(1)\otimes e_{13}=\sigma_t^\theta(1\otimes e_{13})
=\sigma_t^\theta(1\otimes e_{12})\sigma_t^\theta(1\otimes e_{23}) \\
&=(D\ffi_1:D\ffi_2)_t(D\ffi_2:D\ffi_3)_t\otimes e_{13}.
\end{align*}

(4)\enspace
By (1), (3) and (2) we find that
$$
(D\ffi_1:D\ffi)_t=(D\ffi_2:D\ffi)_t,\ t\in\bR
\ \iff\ (D\ffi_1:D\ffi_2)_t=1,\ t\in\bR\ \iff\ \ffi_1=\ffi_2.
$$
\end{proof}

The main result of Pedersen and Takesaki \cite{PT} is the following:

\begin{thm}[Pedersen and Takesaki]\label{T-9.7}
Let $\ffi,\psi\in P(M)$. Then the following conditions are equivalent:
\begin{itemize}
\item[\rm(i)] $\psi$ is $\sigma^\ffi$-invariant, i.e., $\psi\circ\sigma_t^\ffi=\psi$ for all
$t\in\bR$;
\item[\rm(ii)] $\ffi$ is $\sigma^\psi$-invariant;
\item[\rm(iii)] $(D\psi:D\ffi)_t\in M_\psi$ for all $t\in\bR$;
\item[\rm(iv)] $(D\psi:D\ffi)_t\in M_\ffi$ for all $t\in\bR$;
\item[\rm(v)] $(D\psi:D\ffi)_t$ ($t\in\bR)$ is a strongly continuous one-parameter unitary
group;
\item[\rm(vi)] there exists a (unique) non-singular positive self-adjoint operator
$A\,\eta M_\ffi$ such that $\psi=\ffi_A$.
\end{itemize}
\end{thm}

\begin{proof}
Let $u_t:=(D\psi:D\ffi)_t$, $t\in\bR$.

(i)\,$\implies$\,(iii).\enspace
By Lemma \ref{L-9.2}\,(1) one has
$\sigma_s^\psi=\sigma_{-t}^\ffi\circ\sigma_s^\psi\circ\sigma_t^\ffi$ so that
$$
\sigma_t^\ffi\circ\sigma_s^\psi=\sigma_s^\psi\circ\sigma_t^\ffi,\qquad s,t\in\bR.
$$
If $x\in M_\psi$, then
$\sigma_t^\ffi(x)=\sigma_t^\ffi(\sigma_s^\psi(x))=\sigma_s^\psi(\sigma_t^\ffi(x))$ for all
$s,t\in\bR$, so $\sigma_t^\ffi(x)\in M_\psi$. Hence $\sigma_t^\ffi(M_\psi)=M_\psi$, $t\in\bR$.
For every $x\in M_+$ and $t\in\bR$,
\begin{align*}
\psi(x)&=\psi(\sigma_t^\psi(x))=\psi(u_t\sigma_t^\ffi(x)u_t^*)
\qquad\mbox{(by \eqref{F-7.8})} \\
&=\psi(\sigma_{-t}^\ffi(u_t\sigma_t^\ffi(x)u_t^*))
=\psi(\sigma_{-t}^\ffi(u_t)x\sigma_{-t}^\ffi(u_t)^*).
\end{align*}
Hence by Lemma \ref{L-9.1}\,(2) one has $\sigma_{-t}^\ffi(u_t)\in M_\psi$, so
$u_t\in\sigma_t^\ffi(M_\psi)=M_\psi$ for all $t\in\bR$.

(iii)\,$\iff$\,(v).\enspace
Since $\sigma_s^\psi(u_t)=u_s\sigma_s^\ffi(u_t)u_s^*=u_{s+t}u_s^*$ by \eqref{F-7.8} and
\eqref{F-7.9}, we have
$$
u_t\in M_\psi,\ t\in\bR\ \iff\ \sigma_s^\psi(u_t)=u_t,\ s,t\in\bR
\ \iff\ u_{t+s}=u_tu_s,\ s,t\in\bR.
$$

(iv)\,$\iff$\,(v).\enspace
This is immediate from (iii)\,$\iff$\,(v) since $u_t^*=(D\ffi:D\psi)_t$.

(v)\,$\implies$\,(vi).\enspace
From (v) (also (iv)), Stone's theorem says that there exists a non-singular positive
self-adjoint operator $A\,\eta M_\ffi$ such that $u_t=A^{it}$ for all $t\in\bR$. Hence one has
$(D\psi:D\ffi)_t=A^{it}=(D\ffi_A:D\ffi)_t$, $t\in\bR$, by Theorem \ref{T-9.4}\,(2), so
Proposition \ref{P-9.6}\,(4) implies that $\psi=\ffi_A$. The uniqueness of $A$ in (vi) is
immediate from Theorem \ref{T-9.4}\,(2).

(vi)\,$\implies$\,(i).\enspace
Assume that $\psi=\ffi_A$ as stated in (vi). Then, for every $x\in M_+$ one has
\begin{align*}
\psi(\sigma_t^\ffi(x))&=\sup_{\eps>0}\ffi(A_\eps^{1/2}\sigma_t^\ffi(x)A_\eps^{1/2})
=\sup_{\eps>0}\ffi(\sigma_t^\ffi(A_\eps^{1/2}xA_\eps^{1/2})) \\
&=\sup_{\eps>0}\ffi(A_\eps^{1/2}xA_\eps^{1/2})=\ffi_A(x)=\psi(x),
\end{align*}
showing (i).
\end{proof}

In this way, $A\mapsto\ffi_A$ is a bijective correspondence between the set of non-singular
positive self-adjoint operators $A\,\eta M_\ffi$ and
$\{\psi\in P(M):\mbox{$\sigma^\ffi$-invariant}\}$. When $\psi=\ffi_A$, $A$ is called the
\emph{Radon-Nikodym derivative} of $\psi$ with respect to $\ffi$, and $\ffi_A$ is often
written as $\ffi(A\,\cdot)$. In particular, assume that $\ffi=\tau$ is an f.s.n.\ trace on $M$;
then $\sigma_t^\tau=\id$ ($t\in\bR$) so that $M_\tau=M$. So Theorem \ref{T-9.7} shows that any
$\psi\in P(M)$ is represented as $\psi=\tau(A\,\cdot)$ with a positive self-adjoint $A\,\eta M$.
This extends Corollary \ref{C-5.14}. In fact, Connes' cocycle derivative $(D\psi:D\ffi)_t$
can be extended to the case where $\psi$ is a (not necessarily faithful) semifinite normal
weight on $M$, and the conditions (except (ii)) of Theorem \ref{T-9.7} are still equivalent
in this case.

Since $\sigma_t^\ffi(M_\ffi)=M_\ffi$ ($t\in\bR$) trivially, Takesaki's theorem \cite{Ta1} (see
Theorem \ref{T-6.6} for $\ffi$ bounded) implies that if $\ffi$ is semifinite on $M_\ffi$, then
there exists a a normal faithful conditional expectation $E_\ffi:M\to M_\ffi$ such that
$\ffi=\ffi\circ E_\ffi$. In fact, Combes \cite{Com} proved the next theorem characterizing the
situation, which we record here without proof.

\begin{thm}\label{T-9.8}
Let $\ffi\in P(M)$. Then the following conditions are equivalent:
\begin{itemize}
\item[\rm(a)] $\ffi|_{M_\ffi}$ is an f.s.n.\ trace on $M_\ffi$;
\item[\rm(b)] there exists a normal faithful conditional expectation $E_\ffi:M\to M_\ffi$ with
$\ffi=\ffi\circ E_\ffi$ (then $E_\ffi$ is automatically $\sigma^\ffi$-invariant);
\item[\rm(c)] there exists a family $\{\ffi_i\}\subset M_*^+$ such that $\sum_is(\ffi_i)=1$ and
$\ffi=\sum_i\ffi_i$;
\item[\rm(d)] $M$ is $\sigma^\ffi$-finite in the sense that for any $x\in M$, $x\ne0$, there
exists a $\sigma^\ffi$-invariant $f\in M_*^+$ such that $f(x)\ne0$.
\end{itemize}
\end{thm}

When the conditions of Theorem \ref{T-9.8} hold, $\ffi$ is said to be
\emph{strictly semifinite}. In this case, let $\tau:=\ffi|_{M_\ffi}$. Then one can easily see
that $\psi\in P(M)$ is $\sigma^\ffi$-invariant if and only if $\psi$ is $E_\ffi$-invariant,
and if this is the case, then $\psi=\ffi_A=\tau_A\circ E_\ffi$ with $A\,\eta M_\ffi$ as in (vi).
Thus, by Theorem \ref{T-8.7}\,(b),
$$
(D\ffi_A:D\ffi)_t=(D\tau_A\circ E_\ffi:D\tau\circ E_\ffi)_t=(D\tau_A:D\tau)=A^{it}
$$
so that Theorem \ref{T-9.4}\,(2) reduces to $(D\tau_A:D\tau)=A^{it}$. Furthermore, we note
a result in \cite{Haa3} that $M_\ffi$ is semifinite (a bit weaker condition than (a)) if and
only if there exists an $\sigma^\ffi$-invariant $T\in P(M,M_\ffi)$ (weaker than (b)). 

We end the section with Takesaki's theorem first proved in \cite{Ta}, saying that $M$ is
semifinite if and only if the modular automorphism group is inner.

\begin{thm}\label{T-9.9}
The following conditions are equivalent:
\begin{itemize}
\item[\rm(i)] $M$ is semifinite;
\item[\rm(ii)] $\sigma^\ffi$ is inner for any (equivalently, some) $\ffi\in P(M)$.
\end{itemize}
\end{thm}

\begin{proof}
Assume that $M$ is semifinite with an f.s.n.\ trace $\tau$, so $M_\tau=M$. For any
$\ffi\in P(M)$, it follows from Theorem \ref{T-9.7} that $\ffi=\tau_A$ for some non-singular
positive self-adjoint operator $A\,\eta M$. Hence by Theorem \ref{T-9.4}\,(1),
$$
\sigma_t^\ffi(x)=A^{it}\sigma_t^\tau(x)A^{-it}=A^{it}xA^{-it},\qquad x\in M,\ t\in\bR.
$$
Conversely, assume that $\sigma^\ffi$ is inner for some $\ffi\in P(M)$, that is, there is a
non-singular positive self-adjoint operator $A\,\eta M$ such that
$\sigma_t^\ffi=A^{it}\cdot A^{-it}$, $t\in\bR$. Since $\sigma_t^\ffi(A^{is})=A^{is}$, we have
$A^{is}\in M_\ffi$ for all $s\in\bR$. Hence $A\,\eta M_\ffi$, so we can define
$\tau:=\ffi_{A^{-1}}\in P(M)$. Then by Theorem \ref{T-9.4}\,(1),
$$
\sigma_t^\tau(x)=A^{-it}\sigma_t^\ffi(x)A^{it}=x,\qquad x\in M,\ t\in\bR,
$$
that is, $M_\tau=M$, which means that $\tau$ is an f.s.n.\ trace.
\end{proof}

\section{Takesaki duality and structure theory}

The subject of this section is Takesaki's duality theorem for crossed products by locally
compact abelian group actions and the crossed product decomposition theorem based on Takesaki's
duality.

\subsection{Takesaki's duality theorem}

We begin with the definition of crossed products of von Neumann algebras. Let $M$ be a von
Neumann algebra on a Hilbert space $\cH$. Let $G$ be a locally compact group and $\alpha$ be a
continuous action of $G$ on $M$, i.e., $t\in G\mapsto\alpha_t\in\Aut(M)$ (= the automorphism
group of $M$) is a weakly (equivalently, strongly) continuous homomorphism. The triplet
$(M,G,\alpha)$ is called a \emph{$W^*$-dynamical system}. We write $ds$ for a left invariant
Haar measure on $G$. Let $L^2(G)$ be the Hilbert space of square integral functions on $G$
with respect to $ds$. Further, let $L^2(G,\cH)$ be the $\cH$-valued integral functions on $G$
with respect to $ds$, which becomes a Hilbert space with the inner product
$\<\xi,\eta\>:=\int_G\<\xi(s),\eta(s)\>\,ds$ for $\xi,\eta\in L^2(G,\cH)$. Note that
$L^2(G,\cH)\cong\cH\otimes L^2(\cH)$, the tensor product Hilbert space of $\cH$ and
$L^2(G,\cH)$, so we always identify $L^2(G,\cH)$ and $\cH\otimes L^2(G)$. For every $x\in M$
and $t\in G$ define
$$
(\pi_\alpha(x)\xi)(s):=\alpha_{s^{-1}}(x)\xi(s),\qquad
(\lambda(t)\xi)(s):=\xi(t^{-1}s),\qquad s\in G,\ \xi\in L^2(G,\cH).
$$

\begin{lemma}\label{L-10.1}
Let $\pi_\alpha$ and $\lambda$ be as above. Then:
\begin{itemize}
\item[\rm(1)] $\pi_\alpha$ is a faithful normal representation of $M$ on $L^2(G,\cH)$.
\item[\rm(2)] $\lambda$ is a strongly continuous unitary representation of $G$ on $L^2(G,\cH)$.
\item[\rm(3)] $\pi_\alpha(\alpha_t(x))=\lambda(t)\pi_\alpha(x)\lambda(t)^*$ for all $x\in M$ and
$t\in G$.
\end{itemize}
\end{lemma}

\begin{proof}
(1) and (2) are shown by direct computations (an exercise). As for (3) let $x\in M$ and
$t\in G$. For every $\xi\in L^2(G,\cH)$ we have
\begin{align*}
(\lambda(t)\pi_\alpha(x)\lambda(t)^*\xi)(s)
&=(\pi_\alpha(x)\lambda(t)^*\xi)(t^{-1}s)
=\alpha_{s^{-1}t}(x)(\lambda(t)^*\xi)(t^{-1}s) \\
&=\alpha_{s^{-1}}(\alpha_t(x))\xi(s)=(\pi_\alpha(\alpha_t(x))\xi)(s),
\qquad s\in G.
\end{align*}
\end{proof}

The above pair $\{\pi_\alpha,\lambda\}$ is called a \emph{covariant representation} of
$(M,G,\alpha)$. Obviously, $\lambda(t)$ is written as $\lambda(t)=1\otimes\lambda_t$, where
$\lambda_t\in B(L^2(G))$ is a unitary defined by $(\lambda_tf)(s):=f(t^{-1}s)$, $s\in G$, for
$f\in L^2(G)$.

\begin{definition}\label{D-10.2}\rm
The \emph{crossed product} $M\rtimes_\alpha G$ (or $M\otimes_\alpha G$) of $M$ by the action
$\alpha$ is the von Neumann algebra generated by $\pi_\alpha(M)$ and $\lambda(G)$, i.e.,
$$
M\rtimes_\alpha G:=\bigl(\{\pi_\alpha(x):x\in M\}\cup\{\lambda(t):t\in G\}\bigr)''.
$$
\end{definition}

By Lemma \ref{L-10.1}\,(3) note that $\lin\{\pi_\alpha(x)\lambda(t):x\in M,\,t\in G\}$ is a
subalgebra of $M\rtimes_\alpha G$ and its strong closure is equal to $M\times_\alpha G$.

\begin{example}\label{E-10.3}\rm
(1)\enspace
When $M=\bC$ is trivial and so $G$ acts trivially on $\bC$, the crossed product $\bC\rtimes G$
is nothing but the \emph{group von Neumann algebra} $\cL(G):=\lambda(G)''$ generated by the
left regular representation $\lambda(t)$, $t\in G$, on $L^2(G)$.

(2)\enspace
Let $n\in\bN$. Let $M=\bC^n$, an abelian von Neumann algebra on $\cH=\bC^n$, and $G=\bZ_n$
be the cyclic group of order $n$. Define the action $\alpha$ of $\bZ_n$ on $\bC^n$ by cyclic
coordinate permutations. Then the crossed product $\bC^n\rtimes_\alpha\bZ_n$ is *-isomorphic
to $\bM_n(\bC)=B(\bC^n)$. This might be the simplest non-trivial example of crossed products.
The proof is easy as follows: Let $e_i$ ($1\le i\le n$) be the natural basis of $\bC^n$,
and $w:=\lambda(1)$, where $1$ is the generator of $\bZ_n=\{0,1,\dots,n-1\}$. Since
$\pi_\alpha(e_i)=w^{i-1}\pi_\alpha(e_1)w^{*(i-1)}$, $\bC^n\rtimes\bZ_n$ is generated by
$E_{ij}:=w^{i-1}\pi_\alpha(e_1)w^{*(j-1)}$ ($i,j=1,\dots,n$). We find that $E_{ij}^*=E_{ji}$
and
$$
E_{ij}E_{kl}=w^{i-1}w^{*(j-k)}\pi_\alpha(e_1)w^{*(k-1)}
=w^{i-j+k-1}\pi_\alpha(e_{j-k+1}e_1)w^{*(k-1)}
=\delta_{jk}E_{il},
$$
where $j-k+1$ is given in mod $n$. Hence $\{E_{ij}\}_{i,j=1}^n$ constitutes a system of
$n\times n$ matrix units. So $\bC^n\rtimes_\alpha\bZ_n\cong\bM_n(\bC)$.

(3)\enspace
Let $M=\ell^\infty(\bZ)$, an abelian von Neumann algebra on $\cH=\ell^2(\bZ)$, and $G=\bZ$
acts on $\ell^\infty(\bZ)$ by shift, i.e., $(\alpha_n(a))(i)=a(i+n)$ for $a\in\ell^\infty(\bZ)$
and $n,i\in\bZ$. Then $\ell^\infty(\bZ)\rtimes_\alpha\bZ\cong B(\ell^2(\bZ))$. The proof is
similar to that in (2). The crossed products in (2) and (3) are factors (of type I). This is
due to the fact that the action $\alpha$ of $\bZ$ (or $\bZ_n$) on the space $\bZ$ (or $\bZ_n$)
is \emph{ergodic}.
\end{example}

In the rest of the section we assume that $G$ is a locally compact abelian group and
$\widehat G$ is the (Pontryagin) dual group of $G$. We write $\<t,p\>:=p(t)$ for $t\in G$ and
a character $p\in\widehat G$. For every $p\in\widehat G$ define a unitary
$v(p)\in B(L^2(G,\cH))$ by
$$
(v(p)\xi)(s):=\overline{\<s,p\>}\xi(s),\qquad s\in G,\ \xi\in L^2(G,\cH).
$$
Then $v$ is a strongly continuous unitary representation of $\widehat G$ on $L^2(G,\cH)$ and
it is written as $v(p)=1\otimes v_p$, where $v_p$ is defined by
$(v_pf)(s):=\overline{\<s,p\>}f(s)$, $s\in G$, for $f\in L^2(G)$.

\begin{lemma}\label{L-10.4}
For every $x\in M$, $t\in G$ and $p\in\widehat G$,
$$
v(p)\pi_\alpha(x)v(p)^*=\pi_\alpha(x),\qquad
v(p)\lambda(t)v(p)^*=\overline{\<t,p\>}\lambda(t).
$$
\end{lemma}

The proof of the lemma is a straightforward computation (an exercise).

\begin{definition}\label{D-10.5}\rm
The continuous action $\widehat\alpha$ of $\widehat G$ on $M\rtimes_\alpha G$ can be defined by
$$
\widehat\alpha_p(y):=v(p)yv(p)^*,\qquad p\in\widehat G,\ y\in M\rtimes_\alpha G.
$$
The action $\widehat\alpha$ of $\widehat G$ is called the \emph{dual action}.
\end{definition}

The next proposition shows that $\widehat\alpha$ is independent, up to conjugacy, of the
particular representation of $\alpha$ on $M$; in particular, the crossed product
$M\rtimes_\alpha G$ is independent, up to isomorphism, of the particular representation
$\{M,\cH\}$ of $M$.

\begin{prop}\label{P-10.6}
Let $(M,G,\alpha)$ and $(N,G,\beta)$ be covariant representations of von Neumann algebras
$M$ and $N$. If $\alpha$ and $\beta$ is conjugate, i.e., there is a *-isomorphism
$\gamma:M\to N$ such that $\gamma\circ\alpha_t=\beta_t\circ\gamma$ ($t\in G$). Then there
exists a *-isomorphism $\widehat\gamma:M\rtimes_\alpha G\to N\rtimes_\beta G$ such that
$\widehat\gamma\circ\pi_\alpha=\pi_\beta\circ\gamma$ and
$\widehat\gamma\circ\widehat\alpha_p=\widehat\beta_p\circ\widehat\gamma$ ($p\in\widehat G$).
\end{prop}

\begin{proof}
Note that $\widetilde\gamma:=\gamma\otimes\id:M\otimes B(L^2(G))\to N\otimes B(L^2(G))$ is a
*-isomorphism. Consider a directed set consisting of finite Borel partitions of $G$, where
for finite Borel partition $\nu,\nu'$ the order $\nu\le\nu'$ is defined if $\nu'$ is a
refinement of $\nu$. For every $x\in M$ and any Borel partition $\nu=\{A_1,\dots,A_n\}$
define $\widetilde x_\nu:=\sum_{i=1}^n\alpha_{s_i^{-1}}(x)\otimes M_{A_i}$ with some
$s_i\in A_i$, where $M_{A_i}$ is the multiplication operator by the indicator function of $A_i$.
Then $\{\widetilde x_\nu\}$ is a net in $M\otimes B(L^2(G))$. For $\xi=\eta\otimes f$ with
$\eta\in\cH$ and $f\in L^2(G)$,
$$
\|\pi_\alpha(x)\xi-\widetilde x_\nu\xi\|^2
=\sum_{i=1}^n\int_{A_i}\big\|\alpha_{s^{-1}}(x)\eta
-\alpha_{s_i^{-1}}(x)\eta\big\|^2|f(s)|^2\,ds.
$$
For any $\eps>0$ choose a compact $K\subset G$ such that
$\int_{G\setminus K}|f(s)|^2\,ds<\eps^2$. Moreover, choose a Borel partition
$\{B_1,\dots,B_m\}$ of $K$ such that $\|\alpha_{s^{-1}}(x)\eta-\alpha_{s'^{-1}}(x)\eta\|<\eps$
for all $s,s'\in B_i$, $1\le i\le m$. If a Borel partition $\nu$ refines
$\{B_1,\dots,B_m,G\setminus K\}$, then
$$
\|\pi_\alpha(x)\xi-\widetilde x_\nu\xi\|^2\le(2\|x\|\,\|\eta\|)^2\eps^2+\|f\|^2\eps^2.
$$
Hence $\widetilde x_\nu\to\pi_\alpha(x)$ strongly. Also, since
$$
\widetilde\gamma(\widetilde x_\nu)=\sum_{i=1}^n\gamma(\alpha_{s_i^{-1}}(x))\otimes M_{A_i}
=\sum_{i=1}^m\beta_{s_i^{-1}}(\gamma(x))\otimes M_{A_i},
$$
one similarly has $\widetilde\gamma(\widetilde x_\nu)\to\pi_\beta(\gamma(x))$ strongly, so that
$\widetilde\gamma(\pi_\alpha(x))=\pi_\beta(\gamma(x))$ for all $x\in M$. On the other hand,
one has $\widetilde\gamma(1_M\otimes\lambda_t)=1_N\otimes\lambda_t$ for all $t\in G$.
Therefore, it follows that $\widetilde\gamma$ maps $M\rtimes_\alpha G$
($\subset M\otimes B(L^2(G))$) onto $N\rtimes_\beta G$ ($\subset N\otimes B(L^2(G))$).

Furthermore, for every $p\in\widehat G$ and $X\in M\rtimes_\alpha G$ one has
$$
\widetilde\gamma(\widehat\alpha_p(X))
=\widetilde\gamma((1_M\otimes v_p)X(1_M\otimes v_p^*))
=(1_N\otimes v_p)\widetilde\gamma(X)(1_N\otimes v_p^*)
=\widehat\beta_p(\widetilde\gamma(X)).
$$
\end{proof}

We can introduce the \emph{second crossed product}
$$
(M\rtimes_\alpha G)\rtimes_{\widehat\alpha}\widehat G,
$$
the crossed product of $M\rtimes_\alpha G$ by the dual action $\widehat\alpha$, which is a
von Neumann algebra on
$\cH\otimes L^2(G)\otimes L^2(\widehat G)=\cH\otimes L^2(G\times\widehat G)$. Moreover,
we have the second dual action $\widehat{\widehat\alpha}$ of $G=\widehat{\widehat G}$ that is
dual to $\widehat\alpha$ of $\widehat G$. \emph{Tekesaki's duality theorem} \cite{Ta4} is
stated as follows:

\begin{thm}[Takesaki]\label{T-10.7}
We have the *-isomorphism
$$
(M\rtimes_\alpha G)\rtimes_{\widehat\alpha}\widehat G\ \cong
\ M\otimes B(L^2(G)),
$$
and the action $\widehat{\widehat\alpha}$ is transformed to the action $\widetilde\alpha$ on
$M\otimes B(L^2(G))$ defined by $\widetilde\alpha_t:=\alpha_t\otimes\Ad(\lambda_t^*)$,
$t\in G$, where $\Ad(\lambda_t^*):=\lambda_t^*\cdot\lambda_t$.
\end{thm}

From Lemma \ref{L-10.1} there is a faithful representation $\pi$ of $M$ on a Hilbert space
$\cH_1$ and a continuous unitary representation $V$ of $G$ on $\cH_1$ such that
$\pi(\alpha_t(x))=V_t\pi(x)V_t^*$ for all $x\in M$ and $t\in G$ (in fact, we may take
$\pi=\pi_\alpha$ and $V=\lambda$ on $\cH_1=L^2(G,\cH)$). So by Proposition \ref{P-10.6}, to
prove the theorem, we may assume that $\alpha$ is given as $\alpha_t=\Ad(V_t)$
($=V_t\cdot V_t^*$) with a continuous unitary representation $V$ of $G$ on $\cH$. We define
a unitary $U\in B(L^2(G,\cH))$ by
$$
(U\xi)(s):=V_s\xi(s),\qquad s\in G,\ \xi\in L^2(G,\cH).
$$

\begin{lemma}\label{L-10.8}
By $\Ad(U)$ we have
$$
M\rtimes_\alpha G\ \cong\ \{x\otimes1,V_t\otimes\lambda_t:x\in M,\,t\in G\}'',
$$
and $\widehat\alpha_p=\Ad(1\otimes v_p)$ is unchanged under this *-isomorphism.
\end{lemma}

\begin{proof}
Since $(U^*\xi)(s)=V_s^*\xi(s)$ and $((x\otimes1)\xi)(s)=x\xi(s)$ for all $\xi\in L^2(G,\cH)$,
one has
\begin{align}\label{F-10.1}
(U^*(x\otimes1)U\xi)(s)=V_s^*xV_s\xi(s)=\alpha_{s^{-1}}(x)\xi(s)=(\pi_\alpha(x)\xi)(s),
\quad s\in G,
\end{align}
so that $U\pi_\alpha(x)U^*=x\otimes1$ for all $x\in M$. Since
$((V_t\otimes\lambda_t)\xi)(s)=V_t\xi(t^{-1}s)$, one has
\begin{align}\label{F-10.2}
(U^*(V_t\otimes\lambda_t)U\xi)(s)=V_s^*V_t(U\xi)(t^{-1}s)
=V_s^*V_tV_{t^{-1}s}\xi(t^{-1}s)=(\lambda(t)\xi)(s),\quad s\in G,
\end{align}
so that $U\lambda(t)U^*=V_t\otimes\lambda_t$ for all $t\in G$. Moreover, since
$U(1\otimes v_p)U^*=1\otimes v_p$ obviously, one has
$\Ad(U)\circ\widehat\alpha\circ\Ad(U^*)=\widehat\alpha_p$ for all $p\in\widehat G$.
\end{proof}

\begin{lemma}\label{L-10.9}
{\rm(1)}\enspace
$\bigl\{v_p:p\in\widehat G\bigr\}''$ is the maximal abelian von Neumann algebra on $L^2(G)$
consisting of all multiplication operators.

{\rm(2)}\enspace
$\bigl\{\lambda_t,v_p:t\in G,\,p\in\widehat G\bigr\}''=B(L^2(G))$.
\end{lemma}

\begin{proof}
(1)\enspace
Let $\fA$ be the abelian von Neumann algebra consisting of all multiplication operators
$m_f$ on $L^2(G)$ with $f\in L^\infty(G)$. Since $L^1(G)^*=L^\infty(G)$ (i.e., the Haar measure
on $G$ is localizable), it is well-known that $\fA$ is maximal abelian. Obviously,
$\bigl\{v_p:p\in\widehat G\bigr\}''\subset\fA$. To prove equality, assume that $\ffi\in\fA_*$
and $\ffi(v_p)=0$ for all $p\in\widehat G$. There are sequences $\{u_n\},\{v_n\}\subset L^2(G)$
with $\sum_{n=1}^\infty\|u_n\|^2<+\infty$, $\sum_{n=1}^\infty\|v_n\|^2<+\infty$ such that
$\ffi(a)=\sum_{n=1}^\infty\<u_n,av_n\>$ for all $a\in\fA$. Set
$w(s):=\sum_{n=1}^\infty\overline{u_n(s)}v_n(s)$; then $w\in L^1(G)$ and
$\int_G\overline{\<s,p\>}w(s)\,da=\ffi(v_p)=0$ for all $p\in\widehat G$. The injectivity of the
Fourier transform yields $w=0$. Hence $\ffi(m_f)=\int_Gf(s)w(s)\,ds=0$ for all
$f\in L^\infty(G)$, so $\ffi=0$. Thus, $\fA=\bigl\{v_p:p\in\widehat G\bigr\}''$ follows.

(2)\enspace
Let $x\in\bigl\{\lambda_t,v_p:t\in G,\,p\in\widehat G\bigr\}'$; then $x\in\fA'=\fA$ by (1).
Hence $x=m_f$ for some $f\in L^\infty(G)$. For any $t\in G$, since
$m_f=\lambda_tm_f\lambda_t^*$, one has $f(s)=f(t^{-1}s)$ a.e. This implies that $f$ is constant
(a.e.), so $x\in\bC1$ follows.
\end{proof}

Now, we give a sketchy proof of Theorem \ref{T-10.7}. Apart from the original paper \cite{Ta4},
a detailed exposition is found in \cite{vD3}.

\begin{proof}[Proof of Theorem \ref{T-10.7}](Sketch)\enspace
The proof is divided into several steps, which we sketch as follows:

{\it Step 1.}\enspace
$(M\rtimes_\alpha G)\rtimes_{\widehat\alpha}\widehat G$ is *-isomorphic to
$$
N_1:=\bigl\{x\otimes1\otimes1,V_t\otimes\lambda_t\otimes1,1\otimes v_p\otimes\lambda_t:
x\in M,\,t\in G,\,p\in\widehat G\bigr\}''
$$
on $\cH\otimes L^2(G)\otimes L^2(\widehat G)$, and
$\widehat{\widehat\alpha}_t=\Ad(1\otimes1\otimes v_t)$ is unchanged under this *-isomorphism.
In fact, applying Lemma \ref{L-10.8} twice to $\widehat\alpha$ and then to $\alpha$, we find
that
\begin{align*}
(M\rtimes_\alpha G)\rtimes_{\widehat\alpha}\widehat G
&\cong\bigl\{X\otimes1,1\otimes v_p\otimes\lambda_p:
X\in M\rtimes_\alpha G,\,p\in\widehat G\bigr\}'' \\
&\cong\bigl\{x\otimes1\otimes1,V_t\otimes\lambda_t\otimes1,1\otimes v_p\otimes v_p:
x\in M,\,t\in G,\,p\in\widehat G\bigr\}''.
\end{align*}
Since $1\otimes1\otimes v_t$ commutes with $U\otimes1$, we see that $\widehat{\widehat\alpha}$
is unchanged under the above *-isomorphism.

{\it Step 2.}\enspace
$N_1$ is *-isomorphic to
$$
N_2:=\bigl\{x\otimes1\otimes1,V_t\otimes\lambda_t\otimes1,1\otimes v_p\otimes v_p:
x\in M,\,t\in G,\,p\in\widehat G\bigr\}''
$$
on $\cH\otimes L^2(G)\otimes L^2(G)$, and $\Ad(1\otimes1\otimes v_t)$ is transformed to
$\Ad(1\otimes1\otimes\lambda_t^*)$ under this *-isomorphism. To see this, consider the Fourier
transform $\cF:L^2(\widehat G)\to L^2(G)$, that is a unitary operator. Then, since
$\cF\lambda_p\cF^*=v_p$ for all $p\in\widehat G$, we have
$(1\otimes1\otimes\cF)N_1(1\otimes1\otimes\cF^*)=N_2$. Moreover, since
$\cF v_t\cF^*=\lambda_t^*$ for all $t\in G$, we see that $\Ad(1\otimes1\otimes v_t)$ is
transformed by $\Ad(1\otimes1\otimes\cF)$ to $\Ad(1\otimes1\otimes\lambda_t^*)$.

{\it Step 3.}\enspace
$N_2$ is *-isomorphic to
$$
N_3:=\bigl\{x\otimes1,V_t\otimes\lambda_t,1\otimes v_p:
x\in M,\,t\in G,\,p\in\widehat G\bigr\}''
$$
on $\cH\otimes L^2(G)$, and $\Ad(1\otimes1\otimes\lambda_t^*)$ is transformed to
$\Ad(1\otimes\lambda_t^*)$ under this *-isomorphism. To see this, consider a unitary operator
$W$ on $L^2(G)\otimes L^2(G)=L^2(G\times G)$ defined by $(Wf)(s,t):=f(st,t)$ for
$f\in L^2(G\times G)$. Then, since
$$
W^*(\lambda_t\otimes1)W=\lambda_t\otimes1,\qquad
W^*(v_p\otimes v_p)W=v_p\otimes1,
$$
we have $(1\otimes W^*)N_2(1\otimes W)=N_3\otimes\bC1$. Furthermore, since
$W^*(1\otimes\lambda_t^*)W=\lambda_t^*\otimes\lambda_t^*$, we see that
$\Ad(1\otimes1\otimes\lambda_t^*)$ is transformed by $\Ad(W^*)$ to
$\Ad(1\otimes\lambda_t^*\otimes\lambda_t^*)$ that is $\Ad(1\otimes\lambda_t^*\otimes1)$ on
$N_3\otimes\bC1$.

{\it Step 4.}\enspace
By $\Ad(U^*)$ we have
$$
N_3\ \cong\ M\otimes B(L^2(G)),
$$
and $\Ad(1\otimes\lambda_t^*)$ is transformed to $\widetilde\alpha_t$ under this *-isomorphism.
To see this, let $\fA:=\bigl\{v_p:p\in\widehat G\bigr\}''$. We have
\begin{align*}
U^*(M\otimes\bC1)U&=\pi_\alpha(M)\subset(\bC1\otimes\fA)'\cap(M\otimes B(L^2(G)) \\
&=(\bC1\otimes\fA)'\cap(M'\cap\bC1)'=(M'\otimes\fA)'=M\otimes\fA,
\end{align*}
where we have used \eqref{F-10.1} for the above first equality, Lemma \ref{L-10.4} for the
inclusion (also see the proof of Proposition \ref{P-10.6}), and Lemma \ref{L-10.9}\,(1) for
the last equality. Since $1\otimes v_p$ commutes with $U$, we have
$U^*(\bC1\otimes\fA)U=\bC1\otimes\fA$, so $U^*(M\otimes\fA)U\subset M\otimes\fA$. Furthermore,
for every $x\in M$, $x'\in M'$ and $\xi\in L^2(G,\cH)$ one has
\begin{align*}
(U(x\otimes1)U^*(x'\otimes1)\xi)(s)
&=V_sxV_s^*x'\xi(s)=\alpha_s(x)x'\xi(s) \\
&=x'\alpha_s(x)\xi(s)=((x'\otimes1)U(x\otimes1)U^*\xi)(s),\qquad s\in G, \\
(U(x\otimes1)U^*(1\otimes v_p)\xi)(s)
&=\overline{\<s,p\>}V_sxV_s^*\xi(s) \\
&=((1\otimes v_p)U(x\otimes1)U^*\xi)(s),\qquad p\in\widehat G,\ s\in G.
\end{align*}
Therefore,
$$
U(M\otimes\bC1)U^*\subset(M'\otimes\bC1)'\cap(\bC1\otimes\fA)'
=(M'\otimes\fA)'=M\otimes\fA
$$
so that $U(M\otimes\fA)U^*\subset M\otimes\fA$ as well. Hence $U^*(M\otimes\fA)U=M\otimes\fA$,
from which we obtain
\begin{align*}
U^*N_3U&=U^*((M\otimes\fA)\cup\{V_t\otimes\lambda_t:t\in G\})''U \\
&=((M\otimes\fA)\cup\{1\otimes\lambda_t:t\in G\})'' \\
&=M\otimes(\fA\cup\{\lambda_t:t\in G\})''=M\otimes B(L^2(G)),
\end{align*}
where we have used \eqref{F-10.2} for the above second equality and Lemma \ref{L-10.9}\,(2)
for the last equality. Furthermore, since $U^*(1\otimes\lambda_t^*)U=V_t\otimes\lambda_t^*$,
we see that $\Ad(1\otimes\lambda_t^*)$ is transformed by $\Ad(U^*)$ to
$\Ad(V_t\otimes\lambda_t^*)=\widetilde\lambda_t$.
\end{proof}

\begin{remark}\label{R-10.10}\rm
When $M$ is properly infinite (i.e., any non-zero central projection in $M$ is infinite) and
$G$ satisfies the second axiom of countability (so $L^2(G)$ is separable), since
$M\otimes B(L^2(G))\cong M$, it follows that
$(M\rtimes_\alpha G)\rtimes_{\widehat\alpha}\widehat G$ is *-isomorphic to the original $M$.
\end{remark}

The next theorem says that the original $M$ ($\cong\pi_\alpha(M))$) is captured as the
fixed-point algebra of the dual action $\widehat\alpha$.

\begin{thm}\label{T-10.11}
We have
\begin{itemize}
\item[\rm(1)] $\pi_\alpha(M)=\bigl\{y\in M\rtimes_\alpha G:
\widehat\alpha_p(y)=y,\,p\in\widehat G\bigr\}$,
\item[\rm(2)] $M\rtimes_\alpha G=\bigl\{x\in M\otimes B(L^2(G)):
\widetilde\alpha_t(x)=x,\,t\in G\bigr\}$.
\end{itemize}
\end{thm}

\begin{proof}
As in the proof of Theorem \ref{T-10.7}, we may assume that $\alpha$ is given as
$\alpha_t=\Ad(V_t)$ with a continuous unitary representation $V$ of $G$ on $\cH$.

(1)\enspace
That $\widehat\alpha_p(\pi_\alpha(x))=\pi_\alpha(x)$ for all $x\in M$ is in Lemma
\ref{L-10.4}. Since
\begin{align*}
(\pi_\alpha(x)(V_t\otimes\lambda_t^*)\xi)(s)&=\alpha_{s^{-1}}(x)V_t\xi(st)
=V_t\alpha_{(st)^{-1}}(x)\xi(st) \\
&=((V_t\otimes\lambda_t^*)\pi_\alpha(x)\xi)(s),\qquad\xi\in L^2(G,\cH),
\end{align*}
it follows that $\pi_\alpha(x)$ commutes with $V_t\otimes\lambda_t^*$. It is also clear that
$1\otimes\lambda_s$ commutes with $V_t\otimes\lambda_t^*$. Hence
$M\rtimes_\alpha G\subset\{V_t\otimes\lambda_t^*:t\in G\}'$. Now, assume that
$y\in M\rtimes_\alpha G$ and $\widehat\alpha_p(y)=y$ for all $p\in\widehat G$. Then
$$
y\in(M\otimes B(L^2(G)))\cap\bigl\{V_t\otimes\lambda_t^*,1\otimes v_p:
t\in G,\,p\in\widehat G\bigr\}'.
$$
Since $V_t\otimes\lambda_t^*=U^*(1\otimes\lambda_t^*)U$ (as already mentioned in Step 4 of
the proof of Theorem \ref{T-10.7}) and $1\otimes v_p=U^*(1\otimes v_p)U$, we have
\begin{align*}
\bigl\{V_t\otimes\lambda_t^*,1\otimes v_p:t\in G,\,p\in\widehat G\bigr\}''
&=U^*\Bigl(\bC1\otimes\bigl\{\lambda_t^*,v_p:t\in G,\,p\in\widehat G\bigr\}''\Bigr)U \\
&=U^*(\bC1\otimes B(L^2(G)))U.
\end{align*}
Therefore,
$$
y\in(M\otimes B(L^2(G))\cap U^*(B(\cH)\otimes\bC1)U
$$
so that $y=U^*(x\otimes1)U$ for some $x\in B(\cH)$. For every $x'\in M'$ and any continuous
$\xi\in L^2(G,\cH)$, we have
$$
[x,x']\xi(e)=\bigl([U^*(x\otimes1)U,x'\otimes1]\xi\bigr)(e)=0,
$$
where $[x,x']:=xx'-x'x$ and $e$ is the identity of $G$. Hence $[x,x']=0$ for any $x'\in M'$,
so that we have $x\in M$ and $y=U^*(x\otimes1)U=\pi_\alpha(x)\in\pi_\alpha(M)$.

(2)\enspace
We can apply the above proof of (1) to $(M\rtimes_\alpha G,\widehat G,\widehat\alpha)$ in place
of $(M,G,\alpha)$. Then we obtain
$$
\pi_{\widehat\alpha}(M\rtimes_\alpha G)
=\bigl\{X\in(M\rtimes_\alpha G)\rtimes_{\widehat\alpha}\widehat G:
\widehat{\widehat\alpha}_t(X)=X,\,t\in G\bigr\}.
$$
It is easy to check that $\pi_{\widehat\alpha}(M\rtimes_\alpha G)$ is mapped to
$M\rtimes_\alpha G$ by the *-isomorphism from
$(M\rtimes_\alpha G)\rtimes_{\widehat\alpha}\widehat G$ onto $M\otimes B(L^2(G))$ given in the
proof of Theorem \ref{T-10.7}.
\end{proof}

In the rest of the section we give a short survey on the dual weights, whose notion is quite
important in theory of crossed products. The notion was introduced in Takesaki's paper
\cite{Ta4} for a special class of weights, and then developed by Digernes \cite{Di}
and Haagerup \cite{Haa4,Haa5}.

Let $(M,G,\alpha)$ be a $W^*$-dynamical system, where $G$ is a general locally compact group,
and $M\rtimes_\alpha G$ be the crossed product. The basic idea here is to construct a map
$$
\ffi\in P(M)\ \mbox{(= the set of f.s.n.\ weights on $M$)}
\ \longmapsto\ \widetilde\ffi\in P(M\rtimes_\alpha G)
$$
in such a way that the modular automorphism groups $\sigma^\ffi$ and $\sigma^{\widetilde\ffi}$
have a natural close relation. This is done by using a suitable construction of left Hilbert
algebra whose left von Neumann algebra is $N\rtimes_\alpha G$. The approach to do this is to
consider the set $K(G,M)$ of $\sigma$-strongly* continuous functions $x:G\to M$ with compact
support. The set $K(G,M)$ becomes a *-algebra with the product
$$
(a\star b)(s):=\int_G\alpha_t(a(st))b(t^{-1})\,dt
$$
and the involution
$$
a^\sharp(s):=\Delta_G(s)^{-1}\alpha_{s^{-1}}(a(s^{-1})^*)
$$
for $a,b\in K(G,M)$, where $\Delta_G$ is the modular function of $G$. With the covariant
representation $\{\pi_\alpha,\lambda\}$ of $(M,G,\alpha)$, define
$$
\mu(a):=\int_G\lambda(s)\pi_\alpha(a(s))\,ds,\qquad a\in K(G,M).
$$
Then it is not difficult to see that $\mu$ is a *-representation of the *-algebra $K(G,M)$
on $L^2(G,\cH)$, whose range is $\sigma$-weakly dense in $M\rtimes_\alpha G$. Now, we may
assume that $M$ is represented in a standard form $(M,\cH,J,\cP)$, so by the uniqueness of the
standard form, for any $\ffi\in P(M)$, we may identify the GNS Hilbert space $\cH_\ffi$ with
$\cH$. For a given $\ffi\in P(M)$ define
$$
B_\ffi:=K(G,M)\cdot\fN_\ffi=\{a(\cdot)x:a\in K(M,G),\,x\in\fN_\ffi\},
\ \,\mbox{a left ideal in $K(G,M)$},
$$
and
$$
\Lambda_\ffi:B_\ffi\,\longrightarrow\,L^2(G,H),\quad(\Lambda_\ffi a)(s):=(a(s))_\ffi.
$$
Then it is shown that $\fA_\ffi:=\Lambda_\ffi(B_\ffi\cap B_\ffi^\sharp)$ has a left Hilbert
algebra structure and its left von Neumann algebra is $M\rtimes_\alpha G$. By taking the
corresponding weight $\widetilde\ffi$, called the \emph{dual weight}, on $M\rtimes_\alpha G$,
the following theorem is proved (see \cite{Haa4} for details). The proof of this is omitted
here, for Theorem \ref{T-10.13} below is sufficient for us.

\begin{thm}\label{T-10.12}
For any $\ffi\in P(M)$ there corresponds a $\widetilde\ffi\in P(M\rtimes_\alpha G)$ having the
following properties:
\begin{itemize}
\item[\rm(1)] $\widetilde\ffi(\mu(a^\sharp\star a))=\ffi((a^\sharp\star a)(e))$ for any 
$a\in B_\ffi$.
\item[\rm(2)] The modular automorphism $\sigma_t^{\widetilde\ffi}$ is given by
$$
\begin{cases}
\sigma_t^{\widetilde\ffi}(\pi_\alpha(x))=\pi_\alpha(\sigma_t^\ffi(x)),
&x\in M,\ t\in\bR, \\
\sigma_t^{\widetilde\ffi}(\lambda(s))
=\Delta_G(s)^{it}\lambda(s)\pi_\alpha((D\ffi\circ\alpha_s:D\ffi)_t),
&\,s\in G,\ t\in\bR.\end{cases}
$$
\item[\rm(3)] For any $\ffi,\psi\in P(M)$,
$$
(D\widetilde\psi:D\widetilde\ffi)_t=\pi_\alpha((D\psi:D\ffi)_t),\qquad t\in\bR.
$$
\end{itemize}

Moreover, $\widetilde\ffi$ is determined as a unique element of $P(M\rtimes_\alpha G)$
satisfying {\rm(1)} and {\rm(2)}.
\end{thm}

When $G$ is a locally compact \emph{abelian} group, Haagerup \cite{Haa5} gave an alternative
construction of the dual weights by using an operator valued weight. In the following let
$dp$ denote the dual Haar measure on $\widehat G$ taken under the normalization that the
Fourier and the inverse Fourier transforms are formally given as
$$
\widehat f(p)=\int_Gf(s)\overline{\<s,p\>}\,ds\quad(p\in\widehat G),\qquad
f(s)=\int_{\widehat G}\widehat f(p)\<s,p\>\,dp\quad(s\in G).
$$

\begin{thm}[\cite{Haa5}]\label{T-10.13}
Let $M\rtimes_\alpha G$ be the crossed product of $M$ by an action $\alpha$ of a locally
compact abelian group $G$.
\begin{itemize}
\item[\rm(a)] The expression
$$
Tx:=\int_{\widehat G}\widehat\alpha_p(x)\,dp,\qquad x\in(M\rtimes_\alpha G)_+
$$
defines an f.s.n.\ operator valued weight from $M\rtimes_\alpha G$ to $\pi_\alpha(M)$, where
$\widehat\alpha$ be the dual action and $dp$ is the dual Haar measure.
\item[\rm(b)] $T$ satisfies
\begin{align}
T(\mu(a^\sharp\star a))&=\pi_\alpha((a^\sharp\star a)(e)),\qquad a\in K(G,M), \label{F-10.3}\\
T(\lambda(s)x\lambda(s)^*)&=\lambda(s)T(x)\lambda(s)^*,
\qquad x\in(M\rtimes_\alpha G)_+,\ s\in G. \label{F-10.4}
\end{align}
\item[\rm(c)] For any $\ffi\in P(M)$ the dual weight $\widetilde\ffi$ on $M\rtimes_\alpha G$
is given by
\begin{align}\label{F-10.5}
\widetilde\ffi=(\ffi\circ\pi_\alpha^{-1})\circ T.
\end{align}
\end{itemize}
\end{thm}

\begin{proof}
We write $N$ for $M\rtimes_\alpha G$ for brevity.

(a)\enspace
For any $x\in N_+$, a generalized positive operator $Tx\in\widehat N_+$ is defined by
$$
\<\omega,Tx\>:=\int_G\omega(\widehat\alpha_p(x))\,dp,\qquad\omega\in N_*^+.
$$One can naturally extend $\widehat\alpha_p$ to $\widehat N_+$ by
$\<\omega,\widehat\alpha_p(m)\>:=\<\omega\circ\widehat\alpha_p,m\>$ for $m\in\widehat N_+$,
$\omega\in N_*^+$. Since
$$
\<\omega,\widehat\alpha_p(Tx)\>=\<\omega\circ\widehat\alpha_p,Tx\>
=\int_G\omega(\widehat\alpha_{pq}(x))\,dq=\<\omega,Tx\>,\qquad\omega\in N_*^+,
$$
one has $\widehat\alpha_p(Tx)=Tx$ for all $p\in\widehat G$. Let
$Tx=\int_0^\infty\lambda\,de_\lambda+\infty e_\infty$ be the spectral resolution of $Tx$ (see
Theorem \ref{T-8.3}). Since the spectral resolution of $\widehat\alpha_p(Tx)=Tx$ is
$$
\widehat\alpha_p(Tx)=\int_0^\infty\lambda\,d\widehat\alpha_p(e_\lambda)
+\infty\widehat\alpha_p(e_\infty),
$$
one has $\widehat\alpha_p(e_\lambda)=e_\lambda$ and $\widehat\alpha_p(e_\infty)=e_\infty$ for
all $p\in\widehat G$ by Theorem \ref{T-8.3}. Hence Theorem \ref{T-10.11}\,(1) implies that
$e_\lambda,e_\infty\in\pi_\alpha(M)$ for all $\lambda$, which means that
$Tx\in\widehat{\pi_\alpha(M)}_+$. Hence $T:N_+\to\widehat{\pi_\alpha(M)}_+$ is defined. For
any $x\in N_+$ and $a\in\pi_\alpha(M)$ one has
\begin{align*}
\<\omega,T(a^*xa)\>&=\int_{\widehat G}\omega(\widehat\alpha_p(a^*xa))\,dp
=\int_{\widehat G}\omega(a^*\widehat\alpha_p(x)a)\,dp \\
&=\<a\omega a^*,T(x)\>=\<\omega,a^*T(x)a\>,\qquad\omega\in N_*^+,
\end{align*}
so that $T(a^*xa)=a^*T(x)a$. Therefore, $T$ is an operator valued weight from $N$ to
$\pi_\alpha(M)$. It is straightforward to see that $T$ is normal and faithful. The
semifiniteness of $T$ will be shown in the proof of (b).

(b)\enspace
To show this, let $K(G)$ be the set of continuous functions on $G$ with compact support, and
$P(G)$ be the set of continuous, positive definite functions on $G$. We show the fact that
if $f\in K(G)$ and $\widehat f(p)=\int_Gf(s)\overline{\<s,p\>}\,ds\ge0$ for all
$p\in\widehat G$, then $\widehat f\in L^1(\widehat G)$ and
\begin{align}\label{F-10.6}
\int_{\widehat G}\widehat f(p)\,dp=f(e).
\end{align}
Indeed, for any $\phi\in P(G)$, since $\Phi$ is given as
$\phi(s)=\int_{\widehat G}\overline{\<s,p\>}\,d\mu(p)$ for some finite positive measure $\mu$
on $\widehat G$ (Bochner's theorem, see \cite[1.4.3]{Ru}), we have
$$
\int_Gf(s)\phi(s)\,ds=\int_{\widehat G}\int_Gf(s)\overline{\<s,p\>}\,ds\,d\mu(p)\ge0.
$$
For any $k\in K(G)$, since $k^**k\in P(G)$, one has $\int_Gf(s)(k^**k)(s)\,ds\ge0$, so
$f\in P(G)$. Hence it follows from \cite[1.5.1]{Ru} that $\widehat f\in L^1(\widehat G)$ and
$f(s)=\int_{\widehat G}\widehat f(p)\overline{\<s,p\>}\,dp$ for all $s\in G$; in particular,
\eqref{F-10.6} holds.

Now, let $a\in K(M,G)$ and $\omega\in N_*^+$. Since
\begin{align*}
0\le\<\omega,\widehat\alpha_p(\mu(a^\sharp\star a))\>
&=\int_G\<\omega,\widehat\alpha_p(\lambda(s))\pi_\alpha((a^\sharp\star a)(s))\>\,ds \\
&=\int_G\overline{\<s,p\>}\<\omega,\lambda(s)\pi_\alpha((a^\sharp\star a)(s))\>\,ds
\quad\mbox{(by Lemma \ref{L-10.4})},
\end{align*}
one can apply \eqref{F-10.6} to $f(s):=\<\omega,\lambda(s)\pi_\alpha((a^\sharp\star a)(s))\>$
and $\widehat f(p):=\<\omega,\widehat\alpha_p(\mu(a^\sharp\star a))\>$ so that
$$
\int_{\widehat G}\<\omega,\widehat\alpha_p(\mu(a^\sharp\star a))\>\,dp
=f(e)=\<\omega,\pi_\alpha((a^\sharp\star a)(e))\>,\qquad\omega\in N_*^+,
$$
which means \eqref{F-10.3}. In particular, one has $T(\mu(a)^*\mu(a))\in\pi_\alpha(M)$, i.e.,
$\mu(a)\in\fN_T$ for all $a\in K(G,M)$, so $T$ is semifinite. For any $x\in N$ and $s\in G$,
since $\widehat\alpha_p(\lambda(s)x\lambda(s)^*)=\lambda(s)\widehat\alpha_p(x)\lambda(s)$ for
all $p\in\widehat G$, \eqref{F-10.4} is immediate.

(c)\enspace
For each $\ffi\in P(M)$ we write $\widetilde\ffi:=(\ffi\circ\pi_\alpha^{-1})\circ T$; then
$\widetilde\ffi\in P(N)$ by Proposition \ref{P-8.6}. By Theorem \ref{T-8.7}, for every
$\ffi,\psi\in P(M)$ we have
\begin{align}\label{F-10.7}
\sigma_t^{\widetilde\ffi}(\pi_\alpha(x))
=\sigma_t^{\ffi\circ\pi_\alpha^{-1}}(\pi_\alpha(x))
=\pi_\alpha(\sigma_t^\ffi(x)),\qquad x\in M,\ t\in\bR,
\end{align}
\begin{align}\label{F-10.8}
(D\widetilde\psi:D\widetilde\ffi)_t=(D\psi\circ\pi_\alpha^{-1}:D\ffi\circ\pi_\alpha^{-1})_t
=\pi_\alpha((D\psi:D\ffi)_t),\qquad t\in\bR,
\end{align}
where Lemma \ref{L-10.14}\,(1) below  have been used. From Lemma \ref{L-10.1}\,(3) and
\eqref{F-10.4} it follows that
\begin{align*}
\widetilde{\ffi\circ\alpha_s}(x)&=(\ffi\circ\alpha_s\circ\pi_\alpha^{-1})(Tx)
=\ffi\circ\pi_\alpha^{-1}(\lambda(s)(Tx)\lambda(s)^*) \\
&=(\ffi\circ\pi_\alpha^{-1})\circ T(\lambda(s)x\lambda(s)^*)
=\widetilde\ffi(\lambda(s)x\lambda(s)^*),\qquad x\in N_+,\ s\in G.
\end{align*}
Hence by \eqref{F-10.8} we have
\begin{align*}
\pi_\alpha((D\ffi\circ\alpha_s,D\ffi)_t)
&=(D\widetilde{\ffi\circ\alpha_s}:D\widetilde\ffi)_t
=(D\widetilde\ffi\circ\Ad(\lambda(s)):D\widetilde\ffi)_t \\
&=\lambda(s)^*\sigma_t^{\widetilde\ffi}(\lambda(s)),\qquad s\in G,\ t\in\bR,
\end{align*}
where Lemma \ref{L-10.14}\,(2) below has been used. Therefore,
\begin{align}\label{F-10.9}
\sigma_t^{\widetilde\ffi}(\lambda(s))=\lambda(s)\pi_\alpha((D\ffi\circ\alpha_s:D\ffi)_t),
\qquad s\in G,\ t\in\bR.
\end{align}
Furthermore, for any $a\in K(G,M)$ it follows from \eqref{F-10.3} that
\begin{align}\label{F-10.10}
\widetilde\ffi(\mu(a^\sharp\star a))
=(\ffi\circ\pi_\alpha^{-1})\circ T(\mu(a^\sharp\star a))=\ffi((a^\sharp\star a)(e)).
\end{align}
Thus, it follows from \eqref{F-10.7}--\eqref{F-10.10} that the map
$\ffi\in P(M)\mapsto\widetilde\ffi\in P(M\rtimes_\alpha G)$ defined by \eqref{F-10.5}
satisfies the properties in (1) and (2) (also (3)) of Theorem \ref{T-10.12}. So the uniqueness
assertion in Theorem \ref{T-10.12} shows that the present $\widetilde\ffi$ coincides with the
dual weight.
\end{proof}

Note that \eqref{F-10.10} gives a slight extension of (1) of Theorem \ref{T-10.12} to all
$a\in K(G,M)$. Also, the new definition \eqref{F-10.5} is applicable to all normal weights
$\ffi$ on $M$ (see Proposition \ref{P-8.6}), for which
\begin{align}\label{F-10.11}
\widetilde{\ffi+\psi}=\widetilde\ffi+\widetilde\psi\quad\mbox{for every normal
weights $\ffi,\psi$ on $M$}.
\end{align}
From Theorem \ref{T-10.13} we may conveniently consider \eqref{F-10.5} as the definition of
the dual weight $\widetilde\ffi$.

\begin{lemma}\label{L-10.14}
Let $\ffi,\psi\in P(M)$.
\begin{itemize}
\item[\rm(1)] Let $\alpha:M_0\to M$ be a *-isomorphism between von Neumann algebras. Then:
$$
\sigma_t^{\ffi\circ\alpha}=\alpha^{-1}\circ\sigma_t^\ffi\circ\alpha,\qquad
(D\psi\circ\alpha:D\ffi\circ\alpha)_t=\alpha^{-1}((D\psi:D\ffi)_t),\qquad t\in\bR.
$$
\item[\rm(2)] For every unitary $u\in M$,
$$
D(\ffi\circ(Ad(u)):D\ffi)_t=u^*\sigma_t^\ffi(u),\qquad t\in\bR.
$$
\end{itemize}
\end{lemma}

\begin{proof}
(1)\enspace
The first formula is essentially the same as Lemma \ref{L-9.2}\,(1). For the second, apply
the first to $\alpha\otimes\id_2:\bM_2(M_0)\to\bM_2(M)$ and the balanced weight
$\theta:=\theta(\ffi,\psi)$. Since
$\theta\circ(\alpha\otimes\id_2)=\theta(\ffi\circ\alpha,\psi\circ\alpha)$, we have
$$
\sigma_t^{\theta(\ffi\circ\alpha,\psi\circ\alpha)}
=(\alpha^{-1}\otimes\id_2)\circ\sigma_t^\theta\circ(\alpha\otimes\id_2)
$$
so that
\begin{align*}
\begin{bmatrix}0&0\\(D\psi\circ\alpha:D\ffi\circ\alpha)_t&0\end{bmatrix}
&=\sigma_t^{\theta(\ffi\circ\alpha,\psi\circ\alpha)}
\biggl(\begin{bmatrix}0&0\\1&0\end{bmatrix}\biggr)
=(\alpha^{-1}\otimes\id_2)\circ\sigma_t^\theta\biggl(\begin{bmatrix}0&0\\1&0
\end{bmatrix}\biggr) \\
&=(\alpha^{-1}\otimes\id_2)\biggl(\begin{bmatrix}0&0\\(D\psi:D\ffi)_t&0\end{bmatrix}\biggr)
=\begin{bmatrix}0&0\\\alpha^{-1}((D\psi:D\ffi)_t)&0\end{bmatrix}.
\end{align*}

(2)\enspace
Note that $\theta(\ffi,\ffi\circ\Ad(u))
=\theta(\ffi,\ffi)\circ\Ad\biggl(\begin{bmatrix}1&0\\0&u\end{bmatrix}\biggr)$. From (1)
we have
$$
\sigma_t^{\theta(\ffi,\ffi\circ\Ad(u))}
=\Ad\biggl(\begin{bmatrix}1&0\\0&u\end{bmatrix}^*\biggr)\circ\sigma_t^{\theta(\ffi,\ffi)}
\circ\Ad\biggl(\begin{bmatrix}1&0\\0&u\end{bmatrix}\biggr)
$$
so that
\begin{align*}
\begin{bmatrix}0&0\\(D\ffi\circ\Ad(u):D\ffi)_t&0\end{bmatrix}
&=\Ad\biggl(\begin{bmatrix}1&0\\0&u^*\end{bmatrix}\biggr)\circ\sigma_t^{\theta(\ffi,\ffi)}
\circ\Ad\biggl(\begin{bmatrix}1&0\\0&u\end{bmatrix}\biggr)
\biggl(\begin{bmatrix}0&0\\1&0\end{bmatrix}\biggr) \\
&=\Ad\biggl(\begin{bmatrix}1&0\\0&u^*\end{bmatrix}\biggr)\circ\sigma_t^{\theta(\ffi,\ffi)}
\biggl(\begin{bmatrix}0&0\\u&0\end{bmatrix}\biggr) \\
&=\Ad\biggl(\begin{bmatrix}1&0\\0&u^*\end{bmatrix}\biggr)
\biggl(\begin{bmatrix}0&0\\\sigma_t^\ffi(u)&0\end{bmatrix}\biggr)
=\begin{bmatrix}0&0\\u^*\sigma_t^\ffi(u)&0\end{bmatrix}.
\end{align*}
\end{proof}

\subsection{Structure of von Neumann algebras of type III}

Let $M$ be a general von Neumann algebra on $\cH$. Choose a $\ffi_0\in P(M)$, and construct
the pair
\begin{align}\label{F-10.12}
\bigl(N:=M\rtimes_{\sigma^{\ffi_0}}\bR,\ \theta:=\widehat{\sigma^{\ffi_0}}\bigr)
\end{align}
of the crossed product $M\rtimes_{\sigma^{\ffi_0}}\bR$ by the modular automorphism group
$\sigma^{\ffi_0}$ and the dual action $\widehat{\sigma^{\ffi_0}}$. Now, choose another
$\ffi_1\in P(M)$, and let $u_t:=(D\ffi_1:D\ffi_0)_t$ be Connes' cocycle derivative and
define a unitary $U$ on $L^2(\bR,\cH)$ by $(U\xi)(s):=u_{-s}\xi(s)$ for $\xi\in L^2(\bR,\cH)$.
For any $\xi\in L^2(\bR,\cH)$, $x\in M$ and $t,s\in\bR$, we compute
\begin{align*}
(U\pi_{\sigma^{\ffi_0}}(x)U^*\xi)(s)&=u_{-s}\sigma_{-s}^{\ffi_0}(x)u_{-s}^*\xi(s)
=\sigma_{-s}^{\ffi_1}(x)\xi(s)\quad\mbox{(by \eqref{F-7.8})} \\
&=(\pi_{\sigma^{\ffi_1}}(x)\xi)(s)
\end{align*}
and
\begin{align*}
(U\lambda(t)U^*\xi)(s)&=u_{-s}u_{-s+t}^*\xi(s-t)
=u_{-s}(u_{-s}\sigma_{-s}^{\ffi_0}(u_t))^*\xi(s-t)\quad\mbox{(by \eqref{F-7.9})} \\
&=u_{-s}\sigma_{-s}^{\ffi_0}(u_t^*)u_{-s}^*\xi(s-t)
=(\sigma_{-s}^{\ffi_1}(u_t^*)\lambda(t)\xi)(s), \\
&=(\pi_{\sigma^{\ffi_1}}(u_t^*)\lambda(t)\xi)(s),
\end{align*}
so that $U\pi_{\sigma^{\ffi_0}}(x)U^*=\pi_{\sigma^{\ffi_1}}(x)$ and
$U\lambda(t)U^*=\pi_{\sigma^{\ffi_1}}(u_t^*)\lambda(t)$. Hence
$U(M\rtimes_{\sigma^{\ffi_0}}\bR)U^*\subset M\rtimes_{\sigma^{\ffi_1}}\bR$. Since
$u_s^*=(D\ffi_0:D\ffi_1)_s^*$ by \eqref{F-7.10}, the argument can be reversed, so that
$U^*(M\rtimes_{\sigma^{\ffi_1}}\bR)U\subset M\rtimes_{\sigma^{\ffi_0}}\bR$. Therefore,
\begin{align}\label{F-10.13}
U(M\rtimes_{\sigma^{\ffi_0}}\bR)U^*=M\rtimes_{\sigma^{\ffi_1}}\bR.
\end{align}
Moreover, since $(Uv(t)U^*\xi)(s)=e^{-ist}\xi(s)=(v(t)\xi)(s)$ for any $\xi\in L^2(\bR,\cH)$
and $t,s\in\bR$, we have
\begin{align}\label{F-10.14}
\Ad(U)\circ\widehat{\sigma_t^{\ffi_0}}=\widehat{\sigma_t^{\ffi_1}}\circ\Ad(U),
\qquad t\in\bR.
\end{align}
From \eqref{F-10.13} and \eqref{F-10.14} we see that the construction \eqref{F-10.12} is
canonical in the sense that it is independent of the choice of ${\ffi_0}\in P(M)$.

\begin{thm}\label{T-10.15}
Let $(N,\theta)$ be defined in \eqref{F-10.12} and $\widetilde\ffi_0$ be the dual weight
defined in \eqref{F-10.5}. Then:
\begin{itemize}
\item[\rm(1)] $\sigma_t^{\widetilde\ffi_0}=\Ad(\lambda(t))$ for all $t\in\bR$.
\item[\rm(2)] $N$ is a semifinite von Neumann algebra with an f.s.n.\ trace $\tau$ (called
the \emph{canonical trace}) satisfying the trace scaling property
\begin{align}\label{F-10.15}
\tau\circ\theta_s=e^{-s}\tau,\qquad s\in\bR.
\end{align}
\item[\rm(3)] $M\otimes B(L^2(\bR))\cong N\rtimes_\theta\bR$ and $M\cong N^\theta$
(the fixed-point algebra of $\theta$).
\end{itemize}
\end{thm}

\begin{proof}
(1)\enspace
By \eqref{F-10.7} we have
$$
\sigma_t^{\widetilde\ffi_0}(\pi_{\sigma^{\ffi_0}}(x))
=\pi_{\sigma^{\ffi_0}}(\sigma_t^{\ffi_0}(x))
=\lambda(t)\pi_{\sigma^{\ffi_0}}(x)\lambda(t)^*,\qquad x\in M,\ t\in\bR.
$$
Also, since $\ffi_0\circ\sigma^{\ffi_0}=\ffi_0$, by \eqref{F-10.9} we have
$$
\sigma_t^{\widetilde\ffi_0}(\lambda(s))=\lambda(s)=\lambda(t)\lambda(s)\lambda(t)^*,
\qquad s,t\in\bR.
$$
These shows that $\sigma_t^{\widetilde\ffi_0}=\Ad(\lambda(t))$ for all $t\in\bR$.

(2)\enspace
By (1) and Theorem \ref{T-9.9}, $N$ is semifinite. Furthermore, there is a non-singular
positive self-adjoint operator $A$ such that $\lambda(t)=A^{it}$ for all $t\in\bR$. Then as
in the proof of Theorem \ref{T-9.9}, it follows that $A\,\eta N_{\widetilde\ffi_0}$ and
$\tau:=(\widetilde\ffi_0)_{A^{-1}}\in P(N)$ is an f.s.n.\ trace. Since
\begin{align*}
\theta_s(A^{-1})^{it}&=\theta_s(A^{-it})=\theta_s(\lambda(-t))=e^{ist}\lambda(-t)
\quad\mbox{(by Lemma \ref{L-10.4} and Definition \ref{D-10.5})} \\
&=e^{ist}A^{-it}=(e^sA^{-1})^{it},\qquad t\in\bR,
\end{align*}
one has $\theta_s(A^{-1})=e^{s}A^{-1}$ for all $s\in\bR$. Set $B:=A^{-1}$; then
$\tau:=(\widetilde\ffi_0)_B$. From definition \eqref{F-9.3}, for every $x\in N_+$ it
follows that
\begin{align*}
\tau\circ\theta_s(x)
&=\lim_{\eps\searrow0}\widetilde\ffi_0(B_\eps^{1/2}\theta_s(x)B_\eps^{1/2})
=\lim_{\eps\searrow0}\widetilde\ffi_0(\theta_{-s}(B_\eps^{1/2}\theta_s(x)B_\eps^{1/2})) \\
&=\lim_{\eps\searrow0}\widetilde\ffi_0(\theta_{-s}(B_\eps)^{1/2}x\theta_{-s}(B_\eps)^{1/2})
=\lim_{\eps\searrow0}\widetilde\ffi_0(e^{-s}B_{e^{-s}\eps}^{1/2}xB_{e^{-s}\eps}^{1/2}) \\
&=e^{-s}\tau(x),\qquad s\in\bR,
\end{align*}
where we have used $\widetilde\ffi_0\circ\theta_{-s}=\widetilde\ffi_0$ by definition
\eqref{F-10.5} and
$$
\theta_{-s}(B_\eps)=\theta_{-s}(B)(1+\eps\theta_{-s}(B))^{-1}
=e^{-s}B(1+e^{-s}\eps B)^{-1}=e^{-s}B_{e^{-s}\eps}.
$$
Hence \eqref{F-10.15} holds.

(3) is seen from Theorems \ref{T-10.7} and \ref{T-10.11}\,(1).
\end{proof}

\begin{remark}\label{R-10.16}\rm
In particular, when $M$ is semifinite with an f.s.n.\ trace $\tau_0$, we may let
$\ffi_0=\tau_0$; then $\sigma^{\tau_0}=\id$ and
$(N,\theta_s)=(M\otimes\cL(\bR),\id\otimes\Ad(v_s))$. Via the Fourier transform $\cF$ (a
unitary on $L^2(\bR)$), note that $\cF\cL(\bR)\cF^*=L^\infty(\bR)$ (the multiplication
operators), $\cF\lambda_s\cF^*:\xi\in L^2(\bR)\mapsto e^{-ist}\xi(t)$ and
$\cF v_s\cF^*:\xi\mapsto\xi(\,\cdot+s)$, so that $\Ad(v_s)$ is transformed to the
shift $f\in L^\infty(\bR)\mapsto f(\,\cdot+s)$. Furthermore, $\widetilde\ffi_0$ and $\tau$
($=(\widetilde\ffi_0)_{A^{-1}}$, see the proof of Theorem \ref{T-10.15}\,(2)) are transformed
to $\tau_0\otimes\int_\bR\cdot\,dt$ and $\tau_0\otimes\int_\bR\cdot\,e^t\,dt$, respectively.
Thus we see that $(N,\theta_s,\tau)$ is identified with
\begin{align}\label{F-10.16}
\biggl(M\otimes L^\infty(\bR),\ \id\otimes(f\mapsto f(\cdot+s)),
\ \tau_0\otimes\int_\bR\cdot\,e^t\,dt\biggr).
\end{align}
\end{remark}

When $M$ is of type III, we state the \emph{continuous crossed product decomposition} or the
\emph{structure theorem} for von Neumann algebras of type III.

\begin{thm}[\cite{Ta4}]\label{T-10.17}
Let $M$ be a von Neumann algebra of type III. Then there exist a von Neumann algebra $N$ of
type II$_\infty$, an f.s.n.\ trace $\tau$ on $N$ and a continuous action
$\theta:\bR\to\Aut(N)$ such that the trace scaling property \eqref{F-10.15} holds and
$$
M=N\rtimes_\theta\bR.
$$
Furthermore, such $(N,\theta)$ as above is unique up to conjugacy.
\end{thm}

After Theorem \ref{T-10.15}, what we need to prove is type II$_\infty$ of $N$ and the last
assertion of uniqueness, whose proofs are omitted here. For the details see \cite{Ta4} or
\cite{Ta3} (a proof of type II$_\infty$ of $N$ is also in \cite{vD3}).

\begin{remark}\label{R-10.18}\rm
When $M$ is a factor of type III, Connes and Takesaki \cite{CT} introduced an important
concept called the flow of weights on $M$. Under the above continuous decomposition of $M$,
the (\emph{smooth}) \emph{flow of weights} on $M$ may be defined by $(X,F_t^M)$ where
$X:=N\cap N'$, the center of $N$, and $F_t^M:=\theta_t|_X$. Then $(X,F_t^M)$ is an ergodic
flow, and the type classification of factors of type III can be given in terms of $(X,F_t^M)$
as follows:
\begin{itemize}
\item $M$ is of type III$_1$ if $X$ is a single point (i.e., $N$ is a factor),
\item $M$ is of type III$_\lambda$, $0<\lambda<1$, if $(X,F_t^M)$ has a period
$-\log\lambda$,
\item $M$ is of type III$_0$ otherwise (i.e., $(X,F_t^M)$ is aperiodic).
\end{itemize}
Furthermore, the flow of weights is a complete invariant for the isomorphism classes of
\emph{injective} factors of type III.
\end{remark}

Finally, we record the \emph{discrete crossed product decompositions} of factors of
type III$_\lambda$ ($0\le\lambda<1$) due to Connes \cite{Co2}.

\begin{thm}[\cite{Co2}]\label{T-10.19}
Let $M$ be a factor of type III$_\lambda$ where $0<\lambda<1$. Then there exist a factor
$N$ of type II$_\infty$ and a $\theta\in\Aut(N)$ such that $\tau\circ\theta=\lambda\tau$
(where $\tau$ is an f.s.n.\ trace on $N$) and $M=N\rtimes_\theta\bZ$.
Furthermore, such $(N,\theta)$ as above is unique up to conjugacy.
\end{thm}

\begin{thm}[\cite{Co2}]\label{T-10.20}
Let $M$ be a factor of type III$_0$. Then there exist a von Neumann algebra $N$ of
type II$_\infty$ whose center is non-atomic, $\lambda\in(0,1)$ and a $\theta\in\Aut(N)$
such that $\tau\circ\theta\le\lambda\tau$, $\theta$ is ergodic on the center of $N$, and
$M=N\rtimes_\theta\bZ$.
\end{thm}

\section{Haagerup's $L^p$-spaces}

In this section, based on Terp's thesis \cite{Te},\footnote{Unfortunately, the thesis
\cite{Te} was not published as a paper, though it is well distributed. There is Haagerup's
paper \cite{Haa6} about this, but just a brief survey.}
we give a concise exposition on Haagerup's $L^p$-spaces associated to general von Neumann
algebras, extending conventional non-commutative $L^p$-spaces $L^p(M,\tau)$ on semifinite
von Neumann algebras in Sec.~5.

\subsection{Description of $L^1(M)$}

Let $M$ be a general von Neumann algebra, and choose a $\ffi_0\in P(M)$. Consider the triple
$(N,\theta,\tau)$ given in Theorem \ref{T-10.15} associated with $(M,\ffi_0)$, that is,
$N:=M\rtimes_{\sigma^{\ffi_0}}\bR$, $\theta:=\widehat{\sigma^{\ffi_0}}$ and $\tau$ is an
f.s.n.\ trace on $N$ satisfying \eqref{F-10.15}. Here we may assume for simplicity that
$M\subset N$ by identifying $x$ and $\pi_{\sigma^{\ffi_0}}(x)$ for $x\in M$. Consider also the
f.s.n.\ operator valued weight $T:=\int_\bR\theta_s\,ds$ from $N$ to $M$ (see Theorem
\ref{T-10.13}\,(a)).

We first give the following fact. This is rather obvious while a proof is given for
completeness.

\begin{lemma}\label{L-11.1}
The dual action $\theta_s$ ($s\in\bR$) on $N$ uniquely extends to $\widetilde N$ (the
$\tau$-measurable operators affiliated with $N$, see Sec.~4.1) as a one-parameter group of homeomorphic *-isomorphisms with
respect to the measure topology.
\end{lemma}

\begin{proof}
Recall Theorem \ref{T-4.12} saying that $\widetilde N$ is a complete metrizable Hausdorff
*-algebra and $N$ is dense in $\widetilde N$. Since $\theta_s$'s are *-isomorphism on $N$, we
may only show that if $\{y_n\}\subset N$ is Cauchy in the measure topology, then so is
$\{\theta_s(y_n)\}$ for any $s\in\bR$. The Cauchyness of $\{y_n\}$ means that for any
$\eps,\delta>0$ there is an $n_0\in\bN$ such that $y_m-y_n\in\cN(\eps,\delta)$ for all
$n\ge n_0$. So it suffices to prove that
$$
\theta_s(\cN(\eps,\delta)\cap N)=\cN(\eps,e^{-s}\delta)\cap N
$$
for any $\eps,\delta>0$. Assume that $a\in\cN(\eps,\delta)\cap N$, so $\|ae\|\le\eps$ and
$\tau(e^\perp)\le\delta$ for some $e\in\Proj(N)$. Then $\|\theta(a)\theta_s(e)\|=\|ae\|\le\eps$
and $\tau(\theta_s(e)^\perp)=\tau(\theta_s(e^\perp))=e^{-s}\tau(e^\perp)\le e^{-s}\delta$.
Hence $\theta_s(\cN(\eps,\delta)\cap N)\subset\cN(\eps,e^{-s}\delta)\cap N)$ follows. The
converse inclusion is similar. The uniqueness of the extension is clear.
\end{proof}

More generally, let $a$ be a positive self-adjoint operator affiliated with $N$, and let
$a=\int_0^\infty\lambda\,de_\lambda$ be the spectral decomposition of $a$. Then one can define
a positive self-adjoint operator $\theta_s(a)\,\eta N$ for any $s\in\bR$ as
$\theta_s(a):=\int_0^\infty\lambda\,d\theta_s(e_\lambda)$. It is immediate to see that this
definition $\theta_s(a)$ for $a\in\widetilde N_+$ coincides with that given in Lemma
\ref{L-11.1}.

In the following we will deal with general (not necessarily faithful) semifinite normal
weights on $M$ and $N$, so we write, instead of $P(M)$ and $P(N)$,
\begin{align*}
\overline P(M)&:=\{\ffi:\ \mbox{semifinite normal weight on $M$}\}, \\
\overline P(N)&:=\{\psi:\ \mbox{semifinite normal weight on $N$}\}, \\
\overline P_\theta(N)&:=\{\psi\in\overline P(N):\psi\circ\theta_s=\psi,\,s\in\bR\}, \\
\overline N_+&:=\{a:\ \mbox{positive self-adjoint operator},\,a\,\eta N\}, \\
\overline N_{+,\theta}&:=\{a\in\overline N_+:\theta_s(a)=e^{-s}a,\,s\in\bR\}, \\
\widetilde N_{+,\theta}&:=\{a\in\widetilde N_+:\theta_s(a)=e^{-s}a,\,s\in\bR\},
\end{align*}
where $\widetilde N$ is the space of $\tau$-measurable operators affiliated with $N$ (see
Sec.~4.1). Obviously,
$$
\overline P_\theta(N)\subset\overline P(N),\qquad
\widetilde N_{+,\theta}\subset\overline N_{+,\theta}\subset\overline N_+.
$$
We first show the following correspondences:

\begin{lemma}\label{L-11.2}
\begin{itemize}
\item[\rm(1)] $a\mapsto\tau(a\,\cdot)$ is a bijection between $\overline N_+$ and
$\overline P(N)$, where $\tau(a\,\cdot)$ is defined as \eqref{F-9.3} for $\ffi=\tau$, i.e.,
$\tau(ax):=\lim_{\eps\searrow0}\tau(a_\eps x)$, $x\in N_+$. Moreover, we have
$s(\tau(a\,\cdot))=s(a)$ for support projections $s(\cdot)$.
\item[\rm(2)] $h\mapsto\tau(h\,\cdot)$ is a bijection between $\overline N_{+,\theta}$ and
$\overline P_\theta(N)$. Moreover, we have $s(\tau(h\,\cdot))=s(h)$.
\item[\rm(3)] $\ffi\mapsto\widetilde\ffi$ is a bijection between $\overline P(M)$ and
$\overline P_\theta(N)$, where $\widetilde\ffi$ is defined by $\widetilde\ffi:=\ffi\circ T$,
extending \eqref{F-10.5} to $\overline P(M)$ (with omitting $\pi_{\sigma^{\ffi_0}}$).
Moreover, we have $s(\widetilde\ffi)=s(\ffi)$.
\end{itemize}
\end{lemma}

\begin{proof}
(1)\enspace
We mentioned this assertion after Theorem \ref{T-9.7}. But we give a brief proof for
completeness. It is clear that $\tau(a\,\cdot)\in\overline P(N)$ for any $a\in\overline N_+$.
The bijection between $\{a\in\overline N_+: \mbox{non-singular}\}$ and $P(N)$ is really a
special case of (i)\,$\iff$\,(vi) of Theorem \ref{T-9.7} when $\ffi=\tau$. For any
$\psi\in\overline P(N)$ we can apply the faithful case to
$(s(\psi)Ns(\psi),\tau|_{s(\psi)Ns(\psi)})$. Hence $a\mapsto\tau(a\,\cdot)$ is surjective. To
see injectivity, assume that $\tau(a\,\cdot)=\tau(b\,\cdot)$ for $a,b\in\overline N_+$. Then
$s(a)=s(b)$ follows immediately. Hence the question reduces to the faithful case as above.

(2)\enspace
From (1) it suffices to prove that for $a\in\overline N_+$,
$\tau(a\,\cdot)\in\overline P_\theta(N)$ if and only if $a\in\overline N_{+,\theta}$. If
$\tau(a\,\cdot)\in\overline P_\theta(N)$, then
$$
\tau(ax)=\tau(a\theta_{-s}(x))=e^s\tau(\theta_s(a\theta_{-s}(x))
=e^s\tau(\theta_s(a)x),\qquad x\in N,
$$
so that $a=e^s\theta_s(a)$ for all $s\in\bR$, i.e., $a\in\overline N_{+,\theta}$. The argument
can be reversed to see the converse implication.

(3)\enspace
We see from the proof of Proposition \ref{P-8.6} with definition \eqref{F-10.5} that
$\widetilde\ffi$ is a semifinite weight on $N$ if $\ffi\in\overline P(M)$. First, prove
equality for supports. For any $\ffi\in\overline P(M)$ let $p_0:=1-s(\ffi)\in M$ and
$q_0:=1-s(\widetilde\ffi)\in N$. Note that $Mp_0=\{x\in M:\ffi(x^*x)=0\}$ and
$Np_0=\{y\in N:\widetilde\ffi(y^*y)=0\}$. For any $y\in\fN_T$ and any $x\in M$,
$$
\widetilde\ffi(p_0x^*y^*yxp_0)=\ffi(T(p_0x^*y^*yxp_0))=\ffi(p_0x^*T(y^*y)xp_0)=0,
$$
so that $\fN_TMp_0\subset Nq_0$. Since $\fN_T$ is $\sigma$-weakly dense in $M$, one has
$Mp_0\subset Nq_0$, which implies that $p_0\le q_0$. Since $\widetilde\ffi$ is
$\theta$-invariant, so is $q_0$ and hence $q_0\in M$. So, to prove that $p_0=q_0$, it suffices
to show that $\ffi(q_0)=0$. There is an increasing net $\{u_\alpha\}\subset\fF_T$ such that
$u_\alpha\nearrow1$. For any $u_\alpha$ and any $n\in\bN$ one has
$$
\ffi(q_0(T(nu_\alpha)\wedge1)q_0)\le n\ffi(T(q_0u_\alpha q_0)
=n\widetilde\ffi(q_0u_\alpha q_0)=0.
$$
Letting $n\to\infty$ gives $\ffi(q_0s(T(u_\alpha))q_0)=0$. Since $T(u_\alpha)$ is increasing,
$s(T(u_\alpha))\nearrow e_0$ for some projection $e_0\in M$. Let $f_0:=1-e_0$. Then
$$
0=f_0T(u_\alpha)f_0=T(f_0u_\alpha f_0)\ \nearrow\ T(f_0)
$$
so that one has $T(f_0)=0$. Hence $f_0=0$ and $e_0=1$, so that $\ffi(q_0)=0$ follows.

To see injectivity, assume first that $\ffi_1,\ffi_2\in P(M)$ and
$\widetilde\ffi_1=\widetilde\ffi_2$. By \eqref{F-10.8} one has
$(D\ffi_1:D\ffi_2)_t=(D\widetilde\ffi_1:D\widetilde\ffi_2)_t=1$ for all $t\in\bR$. Hence
$\ffi_1=\ffi_2$ by Proposition \ref{P-9.6}\,(2). Now, assume that
$\ffi_1,\ffi_2\in\overline P(M)$ and $\widetilde\ffi_1=\widetilde\ffi_2$. Then
$s(\ffi_1)=s(\ffi_2)$ follows from the fact already proved. Choose an $\omega\in\overline P(M)$
such that $s(\omega)=1-s(\ffi_1)$. Then one has
$\widetilde{\ffi_1+\omega}=\widetilde{\ffi_2+\omega}$ by \eqref{F-10.11}, which implies that
$\ffi_1+\omega=\ffi_2+\omega$. Hence $\ffi_1=\ffi_2$ follows.

Next, to prove surjectivity, we first assume that $\psi\in\overline P_\theta(N)$ is faithful.
Choose a $\ffi_1\in P(M)$. Since both $\psi$ and $\widetilde\ffi_1$ are $\theta$-invariant,
by Lemma \ref{L-10.14}\,(1) we have $(D\psi:D\widetilde\ffi_1)_t$'s are $\theta$-invariant,
so that $u_t:=(D\psi:D\widetilde\ffi_1)_t\in M$. Then,
$$
u_{s+t}=u_s\sigma_s^{\widetilde\ffi_1}(u_t)=u_s\sigma_s^{\ffi_1}(u_t),\qquad s,t\in\bR,
$$
thanks to \eqref{F-10.7}. Hence, by Theorem \ref{T-7.8} there exists a $\ffi\in P(M)$ such
that $(D\ffi:D\ffi_1)_t=u_t$ for all $t\in\bR$. By \eqref{F-10.8},
$$
(D\widetilde\ffi:D\widetilde\ffi_1)_t=(D\ffi:D\ffi_1)_t=u_t=(D\psi:D\widetilde\ffi_1)_t,
\qquad t\in\bR.
$$
Hence $\widetilde\ffi_1=\psi$ holds by Proposition \ref{P-9.6}\,(4). For general
$\psi\in\overline P_\theta(N)$ set $e_0:=1-s(\psi)$. Since $\psi$ is $\theta$-invariant,
one has $e_0\in M$. Choose a semifinite normal weight $\omega$ on $M$ such that $s(\omega)=e_0$.
Then $\widetilde\omega\in\overline P_\theta(N)$ and $s(\widetilde\omega)=s(\omega)=e_0$ as
already proved. Since $\psi+\widetilde\omega$ is faithful, it follows from the faithful case
that $\psi+\widetilde\omega=\widehat\ffi$ for some $\ffi\in P(M)$. Then,
$$
\psi=s(\psi)\cdot(\psi+\widetilde\omega)\cdot s(\psi)
=s(\psi)\cdot\widetilde\ffi\cdot s(\psi)
=\widetilde{(s(\psi)\cdot\ffi\cdot s(\psi))},
$$
where the last equality is easily verified from definition \eqref{F-10.5} (extended to
$\ffi\in\overline P(M)$).
\end{proof}

Combining the correspondences given in (2) and (3) of Lemma \ref{L-11.2}, we immediately see
that a bijective map
\begin{align}\label{F-11.1}
\ffi\in\overline P(M)\ \longmapsto\ h_\ffi\in\overline N_{+,\theta}
\end{align}
is determined by equality $\widetilde\ffi=\tau(h_\ffi\,\cdot)$ for all $\ffi\in\overline P(M)$,
satisfying $s(h_\ffi)=s(\ffi)$ for the support projections. The next lemma plays a key role
in defining Haagerup's $L^p$-spaces.

\begin{lemma}\label{L-11.3}
In the correspondence \eqref{F-11.1} between $\overline P(M)$ and $\overline N_{+,\theta}$,
$$
\ffi\in M_*^+\ \iff\ h_\ffi\in\widetilde N_{+,\theta}
\ \ \mbox{i.e., $h_\ffi$ is $\tau$-measurable}.
$$
Hence we have the bijection $\ffi\in M_*^+\mapsto h_\ffi\in\widetilde N_{+,\theta}$ by
restricting \eqref{F-11.1} to $M_*^+$.
\end{lemma}

\begin{proof}
Let $\ffi\in\overline P(M)$ and $h_\ffi\in\overline N_{+,\theta}$ be as given in \eqref{F-11.1}.
Let $h_\ffi=\int_0^\infty\lambda\,de_\lambda$ be the spectral decomposition. What is essential
to prove is the formula
$$
\tau(e_1^\perp)=\ffi(1).
$$
Since
\begin{align}\label{F-11.2}
\int_0^\infty\lambda\,d\theta_s(e_\lambda)=\theta_s(h_\ffi)=e^{-s}h_\ffi
=\int_0^\infty e^{-s}\lambda\,de_\lambda=\int_0^\infty\lambda\,de_{e^s\lambda},
\end{align}
we note that $\theta_s(e_\lambda)=e_{e^s\lambda}$ for all $\lambda\ge0$ and $s\in\bR$. Hence
it follows that
\begin{align*}
\int_\bR\theta_s(h_\ffi^{-1}e_1^\perp)\,ds
&=\int_\bR\theta_s\biggl(\int_{(1,\infty)}\lambda^{-1}\,de_\lambda\biggr)ds
=\int_\bR\biggl(\int_{(1,\infty)}\lambda^{-1}\,de_{e^s\lambda}\biggr)ds \\
&=\int_\bR\biggl(\int_{(e^s,\infty)}e^s\lambda^{-1}\,de_\lambda\biggr)ds
=\int_{(0,\infty)}\biggl(\int_{(-\infty,\log\lambda)}e^s\,ds\biggr)\lambda^{-1}\,de_\lambda \\
&\hskip5.5cm\mbox{(note $\lambda>e^s$ $\iff$ $s<\log\lambda$)} \\
&=\int_{(0,\infty)}\lambda\lambda^{-1}\,de_\lambda=s(h_\ffi)=s(\ffi).
\end{align*}
Therefore, recalling that $\tau(h_\ffi\,\cdot)=\widetilde\ffi=\ffi\circ T$, we find that
$$
\tau(e_1^\perp)=\widetilde\ffi(h_\ffi^{-1}e_1^\perp)
=\ffi\biggl(\int_\bR\theta_s(h_\ffi^{-1}e_1^\perp)\,ds\biggr)
=\ffi(s(\ffi))=\ffi(1).
$$
Furthermore, for every $\lambda>0$,
\begin{align}\label{F-11.3}
\tau(e_\lambda^\perp)=\tau((\theta_{\log\lambda}(e_1)^\perp)
=\tau(\theta_{\log\lambda}(e_1^\perp))
=e^{-\log\lambda}\tau(e_1^\perp)={1\over\lambda}\,\ffi(1).
\end{align}
From this and Proposition \ref{P-4.7}, the assertion follows.
\end{proof}

\begin{lemma}\label{L-11.4}
For every $\ffi,\psi\in M_*^+$ and $x\in M$,
\begin{itemize}
\item[\rm(1)] $h_{\ffi+\psi}=h_\ffi+h_\psi$,
\item[\rm(2)] $h_{x\ffi x^*}=xh_\ffi x^*$, where $(x\ffi x^*)(y):=\ffi(x^*yx)$, $y\in M$.
\end{itemize}
\end{lemma}

\begin{proof}
(1)\enspace
One has
$$
\widetilde{\ffi+\psi}=(\ffi+\psi)\circ T=\ffi\circ T+\psi\circ T
=\widetilde\ffi+\widetilde\psi=\tau(h_\ffi\,\cdot)+\tau(h_\psi\,\cdot).
$$
Since $h_\ffi,h_\psi\in\widetilde N_{+,\theta}$, note that $h_\ffi+h_\psi$
($:=\overline{h_\ffi+h_\psi}$) $\in\widetilde N_+$ is well-defined and by Lemma \ref{L-11.1},
$$
\theta_s(h_\ffi+h_\psi)=\theta_s(h_\ffi)+\theta_s(h_\psi)=e^{-s}(h_\ffi+h_\psi).
$$
Hence $h_\ffi+h_\psi\in\widetilde N_{+,\theta}$, and so it suffices to show that
$$
\tau((h_\ffi+h_\psi)\,\cdot)=\tau(h_\ffi\,\cdot)+\tau(h_\psi\,\cdot).
$$
For this, choose sequences $\{a_n\},\{b_n\}\subset N_+$ such that $a_n\nearrow h_\ffi$ and
$b_n\nearrow h_\psi$. By Proposition \ref{P-9.3}\,(4) one has
$\tau(a_n\,\cdot)\nearrow\tau(h_\ffi\,\cdot)$, $\tau(b_n\,\cdot)\nearrow\tau(h_\psi\,\cdot)$
and $\tau((a_n+b_n)\,\cdot)\nearrow\tau((h_\ffi+h_\psi)\,\cdot)$. Since
$\tau((a_n+b_n)\,\cdot)=\tau(a_n\,\cdot)+\tau(b_n\,\cdot)$, the desired equality follows.

(2)\enspace
For every $y\in N_+$ one has
\begin{align*}
\widetilde{(x\ffi x^*)}(y)&=(x\ffi x^*)\circ T(y)=\ffi(x^*T(y)x)
=\ffi(T(x^*yx))\quad\mbox{(see Definition \ref{D-8.5})} \\
&=\widetilde\ffi(x^*yx)=\tau(h_\ffi\,\cdot x^*yx)=\tau(xh_\ffi x^*\,\cdot y).
\end{align*}
Note that $xh_\ffi x^*\in\widetilde N_+$ and by Lemma \ref{L-11.1},
$$
\theta_s(xh_\ffi x^*)=\theta_s(x)\theta_s(h_\ffi)\theta_s(x)=e^{-s}xh_\ffi x^*.
$$
Hence one has $xh_\ffi x^*\in\widetilde N_{+,\theta}$, so that $h_{x\ffi x^*}=xh_\ffi x^*$
follows.
\end{proof}

We are now in a position to prove the following:

\begin{thm}\label{T-11.5}
\begin{itemize}
\item[\rm(a)] The bijection $\ffi\in M_*^+\mapsto h_\ffi\in\widetilde N_{+,\theta}$ extends
to a linear bijection, still denoted by $\ffi\mapsto h_\ffi$, from $M_*$ onto the subspace
\begin{align}\label{F-11.4}
\widetilde N_\theta:=\{a\in\widetilde N:\theta_s(a)=e^{-s}a,\,s\in\bR\}
\end{align}
of $\widetilde N$.
\item[\rm(b)] For every $\ffi\in M_*$ and $x,y\in M$,
$$
h_{x\ffi y^*}=xh_\ffi y^*,\qquad h_{\ffi^*}=h_\ffi^*.
$$
\item[\rm(c)] If $\ffi=u|\ffi|$ is the polar decomposition of $\ffi\in M_*$, then
$h_\ffi=uh_{|\ffi|}$ is the polar decomposition of $h_\ffi$. Hence
$$
|h_\ffi|=h_{|\ffi|}
$$
and the partial isometry part of $h_\ffi$ is $u\in M$.
\end{itemize}
\end{thm}

\begin{proof}
First, note that $\widetilde N_\theta$ is an $M$-bimodule and invariant under $a\mapsto a^*$
and $a\mapsto|a|$. For the invariance under $a\mapsto|a|$, we may just note that
$|\theta_s(a)|=\theta_s(|a|)$.

(a)\enspace
Each $\ffi\in M_*$ is written as $\ffi=\ffi_1-\ffi_2+i(\ffi_3-\ffi_4)$ with $\ffi_k\in M_*^+$.
So define
$$
h_\ffi:=h_{\ffi_1}-h_{\ffi_2}+i(h_{\ffi_3}-h_{\ffi_4}).
$$
By Lemma \ref{L-11.4}\,(1) it is easy to check that this definition is well defined independent
of the expression of $\ffi$ as above and $\ffi\mapsto h_\ffi$ is a linear map from $M_*$ to
$\widetilde N_\theta$. On the other hand, for each $a\in\widetilde N_\theta$ let
${a+a^*\over2}=a_1-a_2$ and ${a-a^*\over2i}=a_3-a_4$ be the Jordan decompositions. Since
$$
{a+a^*\over2}=e^s\theta_s\biggl({a+a^*\over2}\biggr)=e^s(\theta_s(a_1)-\theta_s(a_2)),
$$
the uniqueness of the Jordan decomposition implies that $e^s\theta_s(a_1)=a_1$ and
$e^s\theta_s(a_2)=a_2$ for all $s\in\bR$, so $a_1,a_2\in\widetilde N_{+,\theta}$. Similarly,
$a_3,a_4\in\widetilde N_{+,\theta}$. Hence there are $\ffi_i\in M_*^+$ such that
$h_{\ffi_k}=a_k$ for $1\le k\le4$. Putting $\ffi:=\ffi_1-\ffi_2+i(\ffi_3-\ffi_4)\in M_*$ we
have $h_\ffi=a$. Hence $\ffi\mapsto h_\ffi$ is a linear surjection from $M_*$ onto
$\widetilde N_\theta$. The injectivity of the map will be shown in (c).

(b)\enspace
The property $h_{\ffi^*}=h_\ffi^*$ is obvious from the definition of $\ffi\mapsto h_\ffi$ in
(a). Consider the polarization
$$
x\ffi y^*={1\over4}\sum_{k=0}^3i^k(x+i^ky)\ffi(x+i^ky)^*
$$
and that for $xh_\ffi y^*$. Then the property $h_{x\ffi y^*}=xh_\ffi y^*$ follows from
Lemma \ref{L-11.4}\,(2).

(c)\enspace
Let $\ffi=u|\ffi|$ be the polar decomposition of $\ffi\in M_*$, so $u$ is a partial isometry
in $M$ with $u^*u=s(|\ffi|)$. By (b) we have $h_\ffi=uh_{|\ffi|}$. Hence it suffices to show
that $h_\ffi=uh_{|\ffi|}$ is indeed the polar decomposition of $h_\ffi$. But this is clear
since $u^*u=s(|\ffi|)=s(h_{|\ffi|})$. In particular, $|h_\ffi|=h_{|\ffi|}$. Finally,
if $\ffi\in M_*$ and $h_\ffi=0$, then $h_{|\ffi|}=0$ so that $|\ffi|=0$, i.e., $\ffi=0$
follows. Hence $\ffi\mapsto h_\ffi$ is injective.
\end{proof}

\begin{definition}\label{D-11.6}\rm
We rewrite \eqref{F-11.4} as \emph{Haagerup's $L^1$-space}
\begin{align}\label{F-11.5}
L^1(M):=\{a\in\widetilde N:\theta_s(a)=e^{-s}a,\,s\in\bR\}.
\end{align}
Due to the linear bijection given in Theorem \ref{T-11.5}\,(a), define a linear functional
$\tr$ on $L^1(M)$ by
$$
\tr(h_\ffi):=\ffi(1),\qquad\ffi\in M_*.
$$
Then
$$
\tr(|h_\ffi|)=\tr(h_{|\ffi|})=|\ffi|(1)=\|\ffi\|,\qquad\ffi\in M_*.
$$
This means that $\|a\|_1:=\tr(|a|)$ for $a\in L^1(M)$ is the norm on $L^1(M)$ copied from the
norm on $M_*$ by the linear bijection.
\end{definition}

In this way, $(L^1(M),\|\cdot\|_1)$ becomes a Banach space identified with $M_*$. Since
$|\ffi(1)|\le|\ffi|(1)$ in the above, note that
\begin{align}\label{F-11.6}
|\tr(a)|\le\|a\|_1,\qquad a\in L^1(M).
\end{align}

The following rewriting of \eqref{F-11.3} is
useful:
\begin{align}\label{F-11.7}
\tau(e_{(\lambda,\infty)}(|a|))={1\over\lambda}\,\tr(|a|)
={1\over\lambda}\,\|a\|_1,\qquad a\in L^1(M),\ \lambda>0,
\end{align}
where $e_{(\lambda,\infty)}(|a|)$ is the spectral projection of $|a|$ corresponding to the
interval $(\lambda,\infty)$.

\subsection{Haagerup's $L^p$-spaces}

\begin{definition}\label{D-11.7}\rm
Including $L^1(M)$ in \eqref{F-11.5} we define, for each $p\in(0,\infty]$,
\emph{Haagerup's $L^p$-space}
$$
L^p(M):=\{a\in\widetilde N:\theta_s(a)=e^{-s/p}a,\,s\in\bR\}.
$$
\end{definition}

Clearly, $L^p(M)$'s are (closed) linear subspaces of $\widetilde N$, which are $M$-bimodules
and invariant under $a\mapsto a^*$ and $a\mapsto|a|$. Moreover, similarly to the proof of
Theorem \ref{T-11.5}\,(a), they are linearly spanned by their positive part
$L^p(M)_+:=L^p(M)\cap\widetilde N_+$. By Theorem \ref{T-11.5} with Definition \ref{D-11.6} we
have
$$
(L^1(M),\|\cdot\|_1)\cong M_*\quad\mbox{(isometric)}.
$$
Note that $L^p(M)$'s are disjointly realized in $\widetilde N$ for different $p$, i.e.,
$L^{p_1}(M)\cap L^{p_2}(M)=\{0\}$ if $p_1\ne p_2$. This situation is quite different from that
of $L^p(M,\tau)$ in Sec.~5.

\begin{prop}\label{P-11.8}
We have $L^\infty(M)=M$.
\end{prop}

\begin{proof}
In view of Theorem \ref{T-10.15}\,(3) it suffices to show that every $a\in L^\infty(M)$ is
bounded. Assume that $a\in L^\infty(M)$ with $|a|=\int_0^\infty\lambda\,de_\lambda$. Since
$|a|\in L^\infty(M)$ so that $\theta_s(e_\lambda)=e_\lambda$ for all $\lambda\ge0$ and
$s\in\bR$. Since
$$
\tau(e_\lambda^\perp)=\tau(\theta_s(e_\lambda^\perp))=e^{-s}\tau(e_\lambda^\perp),
\qquad s\in\bR,
$$
it follows that $\tau(e_\lambda^\perp)=0$ or $+\infty$. But
$\tau(e_\lambda^\perp)<+\infty$ for some $\lambda>0$ by Proposition \ref{P-4.7}. Hence one has
$\tau(e_\lambda^\perp)=0$ and so $e_\lambda^\perp=0$ for some $\lambda>0$. This means that
$|a|$ is bounded.
\end{proof}

\begin{remark}\label{R-11.9}\rm
In contrast to Proposition \ref{P-11.8}, any non-zero element of $L^p(M)$ for $0<p<\infty$
is unbounded. Indeed, assume that $0\ne a\in L^p(M)$, $0<p<\infty$, with
$|a|=\int_0^\infty\lambda\,de_\lambda$. Then $\tau(e_\lambda^\perp)>0$ for some $\lambda>0$.
Similarly to the argument in \eqref{F-11.2} one has $\theta_s(e_\lambda)=e_{e^{-s/p}\lambda}$
for any $\lambda\ge0$, so that
$$
\tau(e_{e^{-s/p}\lambda}^\perp)=\tau(\theta_s(e_\lambda^\perp))
=e^{-s}\tau(e_\lambda^\perp),\qquad s\in\bR,
$$
which implies that $|a|$ is unbounded.
\end{remark}

\begin{remark}\label{R-11.10}\rm
To construct Haagerup's $L^p$-spaces $L^p(M)$, we have started with $(M,\ffi_0)$ where
$\ffi_0\in P(M)$. But we remark that the construction is canonical independently of choice of
$\ffi_0$. For this choose another $\ffi_1\in P(M)$. By taking a unitary $U$ on $L^2(G,\cH)$
defined in the beginning of Sec.~10.2 and setting $\kappa:=\Ad(U)$ we have
$$
\kappa:N_0:=M\rtimes_{\sigma^{\ffi_0}}\bR\ \longrightarrow
\ N_1:=M\rtimes_{\sigma^{\ffi_1}}\bR,
$$
$$
\kappa\circ\theta_s^0\circ\kappa^{-1}=\theta_s^1,\qquad s\in\bR,
$$
where $\theta^0,\theta^1$ are the corresponding dual actions on $N_0,N_1$, see \eqref{F-10.13}
and \eqref{F-10.14}. Hence
\begin{align}\label{F-11.8}
\kappa\circ T_0\circ\kappa^{-1}=T_1
\end{align}
for the corresponding operator valued weights $T_0,T_1$. Under identification
$M=\pi_{\sigma^{\ffi_0}}(M)=\pi_{\sigma^{\ffi_1}}(M)$ (we have used in this section), since
$\kappa|_M=\id_M$, \eqref{F-11.8} may be written as
$$
T_0\circ\kappa^{-1}=T_1,
$$
so that for any $\ffi\in\overline P(M)$ we furthermore have
$$
\widetilde\ffi^{(0)}\circ\kappa^{-1}=\widetilde\ffi^{(1)}
\qquad\mbox{for\ \ $\widetilde\ffi^{(0)}:=\ffi\circ T_0$,
\ \ $\widetilde\ffi^{(1)}:=\ffi\circ T_1$}.
$$
From this and the definition of $\tau$ in the proof of Theorem \ref{T-10.15}\,(2) it follows
that
$$
\tau_0\circ\kappa^{-1}=\tau_1
$$
for the corresponding canonical traces $\tau_0,\tau_1$ on $N_0,N_1$. Extending $\kappa$ to
$\widetilde\kappa:\widetilde N_0\to\widetilde N_1$, we find that $\widetilde\kappa$ transforms
$L^p(M)$ with respect to $\ffi_0$ to $L^p(M)$ with respect to $\ffi_1$.
\end{remark}

\begin{example}\label{E-11.11}\rm
Assume that $M$ is semifinite with an f.s.n.\ trace $\tau_0$. Then $(N,\theta_s,\tau)$ is
identified with \eqref{F-10.16} in Remark \ref{R-10.16}. In this case, for each $\ffi\in M_*^+$,
we find that
$$
\widetilde\ffi=\ffi\circ T=\ffi\otimes\int_\bR\cdot\,dt,\qquad
h_\ffi={d\ffi\over d\tau}\otimes e^{-t},
$$
where $d\ffi/d\tau$ is the Radon-Nikodym derivative of $\ffi$ with respect to $\tau$, see
Corollary \ref{C-5.14}. Hence, for each $p\in(0,\infty]$,
$$
L^p(M)=L^p(M,\tau_0)\otimes e^{-t/p}
$$
and $\|a\otimes e^{-t/p}\|_{L^p(M)}=\|a\|_{L^p(M,\tau_0)}$ for all $a\in L^p(M,\tau_0)$.
With neglecting the superfluous tensor factor $e^{-t/p}$ we may identify $L^p(M)$ with
$L^p(M,\tau_0)$.
\end{example}

\begin{lemma}\label{L-11.12}
Let $a\in\widetilde N$ with the polar decomposition $a=u|a|$. Then for every $p\in[1,\infty)$,
$$
a\in L^p(M)\ \iff\ u\in M\ \ \mbox{and}\ \ |a|^p\in L^1(M).
$$
\end{lemma}

\begin{proof}
Assume that $a\in L^p(M)$. Since
$$
a=e^{s/p}\theta_s(a)=e^{s/p}\theta_s(u)\theta_s(|a|)=e^{s/p}\theta_s(u)|\theta_s(a)|,
$$
it follows from the uniqueness of the polar decomposition that
$$
u=\theta_s(u),\quad|a|=e^{s/p}\theta_s(|a|),\qquad s\in\bR.
$$
The latter equality above is equivalent to $|a|^p=e^s\theta_s(|a|^p)$, so that $u\in M$ and
$|a|^p\in L^1(M)$. Conversely, assume that $u\in M$ and $|a|^p\in L^1(M)$. Then
$\theta(a)=u\theta_s(|a|)=e^{-s/p}a$ for all $s\in\bR$. Hence $a\in L^p(M)$ follows.
\end{proof}

\begin{definition}\label{D-11.13}\rm
In view of Lemma \ref{L-11.12}, for every $a\in L^p(M)$ define $\|a\|_p\in[0,+\infty)$ by
\begin{align*}
&\|a\|_p:=\tr(|a|^p)^{1/p}\quad\mbox{if $0<p<\infty$}, \\
&\|a\|_\infty:=\|a\|\qquad\quad\ \mbox{if $p=\infty$}.
\end{align*}
In the case $p=1$, $\|\cdot\|_1$ is the same as that given in Definition \ref{D-11.6}.
\end{definition}

\begin{lemma}\label{L-11.14}
For every $p\in(0,\infty)$ and $\eps,\delta>0$,
$$
\cN(\eps,\delta)\cap L^p(M)=\{a\in L^p(M):\|a\|_p\le\eps\delta^{1/p}\}.
$$
Furthermore,
\begin{align}\label{F-11.9}
\mu_t(a)=t^{-1/p}\|a\|_p,\qquad a\in L^p(M),\ 0<p<\infty,
\end{align}
where $\mu_t(a)$ is the ($t$th) generalized $s$-number of $a$ (as an element in $\widetilde N$
with respect to $\tau$), see Definition \ref{D-4.14}.
\end{lemma}

\begin{proof}
Let $a\in L^p(M)$; then $|a|^p\in L^1(M)_+$ and so $|a|^p=h_\ffi$ for some $\ffi\in M_*^+$.
Since, by \eqref{F-11.7},
$$
\tau(e_{(\eps,\infty)}(|a|))=\tau(e_{(\eps^p,\infty)}(|a|^p)
={1\over\eps^p}\,\tr(|a|^p)={1\over\eps^p}\,\|a\|_p^p,
$$
we have
\begin{align}
a\in\cN(\eps,\delta)\ &\iff\ |a|\in\cN(\eps,\delta)
\ \iff\ \tau(e_{(\eps,\infty)}(|a|))\le\delta \nonumber\\
&\iff\ {1\over\eps^p}\,\|a\|_p^p\le\delta
\ \iff\ \|a\|_p\le\eps\delta^{1/p}, \label{F-11.10}
\end{align}
implying the first assertion. For every $t>0$ it follows from \eqref{F-11.10} that
$$
\mu_t(a)=\inf\{s>0:\tau(e_{(s,\infty)}(|a|)\le t\}
=\inf\{s>0:s^{-p}\|a\|_p^p\le t\},
$$
which implies \eqref{F-11.9}.
\end{proof}

The formula given in \eqref{F-11.9} was first explicitly pointed out in \cite{Ko}, which is
quite useful in treating Haagerup's $L^p$-norm.

Lemma \ref{L-11.14}, in particular, implies the following:

\begin{cor}\label{C-11.15}
The norm topology on $L^1(M)$ coincides with the relative topology induced from
$\widetilde N$ with the measure topology.
\end{cor}

\begin{lemma}\label{L-11.16}
Let $0<p\le\infty$. For every $a\in L^p(M)$ and $x,y\in M$,
\begin{align}\label{F-11.11}
\|xay\|_p\le\|x\|\,\|y\|\,\|a\|_p.
\end{align}
\end{lemma}

\begin{proof}
The case $p=\infty$ is obvious. Let $0<p<\infty$, $a\in L^p(M)$ and $x,y\in M$. Note that
$xay\in L^p(M)$, for $L_p(M)$ is an $M$-bimodule as mentioned right after Definition
\ref{D-11.7}. By \eqref{F-11.9} and Proposition \ref{P-4.17}\,(6) one has
$$
\|xay\|_p=\mu_1(xay)\le\|x\|\,\|y\|\mu_1(a)=\|x\|\,\|y\|\,\|a\|_p,
$$
as asserted.
\end{proof}

\begin{lemma}\label{L-11.17}
Let $a\in\widetilde N_+$. The function
$$
z\in\{z\in\bC:\Re z>0\}\ \longmapsto\ a^z\in\widetilde N
$$
is differentiable in the measure topology and ${d\over dz}\,a^z=a^z\log a$.
\end{lemma}

\begin{proof}
Let $a=\int_0^\infty\,de_\lambda\in\widetilde N_+$. Then, for any $z\in\bC$, $\Re z>0$, note
that $a^z=\int_0^\infty\lambda^z\,de_\lambda$ and $|a^z|=a^{\Re z}$, so $a^z\in\widetilde N$
follows.

First, assume that $a$ is bounded, and prove that $z\mapsto a^z$ is differentiable in the norm
on $\Re z>0$. Note that for any $\lambda>0$, ${d\over dz}\,\lambda^z=\lambda^z\log\lambda$ on
$\Re z>0$ and $z\mapsto\lambda^z\log\lambda$ is continuous on $\Re z\ge0$. From this one can
easily show that for any $z_0\in\bC$, $\Re z_0>0$ and $\eps>0$, there exists an $r>0$ such
that
$$
\Re z>0,\ \ |z-z_0|\le r\ \implies\ \sup_{0\le\lambda\le\|a\|}
\bigg|{\lambda^z-\lambda^{z_0}\over z-z_0}-\lambda^{z_0}\log\lambda\bigg|\le\eps.
$$
This implies that if $\Re z>0$ and $|z-z_0|\le r$, then
$$
\bigg\|{a^z-a^{z_0}\over z-z_0}-a^{z_0}\log a\bigg\|\le\sup_{0\le\lambda\le\|a\|}
\bigg|{\lambda^z-\lambda^{z_0}\over z-z_0}-\lambda^{z_0}\log\lambda\bigg|\le\eps.
$$
Hence the assertion follows.

Next, let $a\in\widetilde N_+$ and $\eps,\delta>0$ be arbitrary. Choose a $\lambda>0$ such
that $\tau(e_\lambda^\perp)<\delta$. Let $p:=e_\lambda$. Since $ap$ is bounded, the first case
implies that one can choose an $r>0$ such that
$$
\bigg\|\biggl({a^z-a^{z_0}\over z-z_0}-a^{z_0}\log a\biggr)p\bigg\|
=\biggl\|{(ap)^z-(ap)^{z_0}\over z-z_0}-(ap)^{z_0}\log(ap)\bigg\|\le\eps
$$
if $\Re z>0$ and $|z-z_0|\le r$. Hence
$$
{a^z-a^{z_0}\over z-z_0}-a^{z_0}\log a\ \in\ \cN(\eps,\delta)\quad
\mbox{if\ \ $\Re z>0$ and $|z-z_0|\le r$}.
$$
This shows the result.
\end{proof}

\begin{lemma}\label{L-11.18}
Let $a,b\in L^1(N)_+$. Then for any $z\in\bC$ with $0<\Re z<1$, we have $a^zb^{1-z}\in L^1(N)$
and the function
$$
z\in\{0<\Re z<1\}\ \longmapsto\ a^zb^{1-z}\in L^1(N)
$$
is analytic in the norm $\|\cdot\|_1$.
\end{lemma}

\begin{proof}
Since $a^z,b^{1-z}\in\widetilde N$, by Lemma \ref{L-11.1} we have
\begin{align*}
\theta_s(a^zb^{1-z})&=\theta_s(a^z)\theta_s(b^{1-z})=\theta_s(a)^z\theta_s(b)^{1-z} \\
&=e^{-zs}a^z\cdot e^{-(1-z)s}b^{1-z}=e^{-s}a^zb^{1-z},\qquad s\in\bR,
\end{align*}
so that $a^zb^{1-z}\in L^1(M)$. By Corollary \ref{C-11.15} we may prove the differentiability
of $z\mapsto a^zb^{1-z}$ on $0<\Re z<1$ as a function to $\widetilde N$. For any $z_0$,
$0<\Re z_0<1$, we have
\begin{align*}
{a^zb^{1-z}-a^{z_0}b^{1-z_0}\over z-z_0}
&={a^z-a^{z_0}\over z-z_0}\,b^{1-z}+a^{z_0}\,{b^{1-z}-b^{1-z_0}\over z-z_0} \\
&\longrightarrow\ a^{z_0}\log a\cdot b^{1-z_0}-a^{z_0}\cdot b^{1-z_0}\log b
\end{align*}
thanks to Lemma \ref{L-11.17}.
\end{proof}

\begin{lemma}\label{L-11.19}
For any $t\in\bR$ put
$$
\widetilde N\bigl(\textstyle{1\over2}+it\bigr)
:=\bigl\{a\in\widetilde N:\theta_s(a)=e^{-s\left({1\over2}+it\right)}a,\,s\in\bR\bigr\}.
$$
If $a,b\in\widetilde N\bigl(\textstyle{1\over2}+it\bigr)$, then $ab^*,b^*a\in L^1(M)$ and
$\tr(ab^*)=\tr(b^*a)$.
\end{lemma}

\begin{proof}
That $b^*a,ab^*\in L^1(M)$ is shown similarly to the proof of Lemma \ref{L-11.18}. To prove
that $\tr(b^*a)=\tr(ab^*)$, assume first that $a=b$. From \eqref{F-11.7} it follows that
$$
\tr(a^*a)=\tau(e_{(1,\infty)}(a^*a))=\tau(e_{(1,\infty)}(aa^*))=\tr(aa^*).
$$
For the general case, since $a+ib\in\widetilde N\bigl(\textstyle{1\over2}+it\bigr)$, the result
immediately follows from the polarizations
$$
b^*a={1\over4}\sum_{k=0}^3i^k(a+i^kb)^*(a+i^kb),\qquad
ab^*={1\over4}\sum_{k=0}^3i^k(a+i^kb)(a+i^kb)^*
$$
and the linearity of $\tr$.
\end{proof}

\begin{prop}\label{P-11.20}
Let $1\le p,q\le\infty$ with $1/p+1/q=1$. If $a\in L^p(M)$ and $b\in L^q(M)$, then
$ab,ba\in L^1(M)$ and
$$
\tr(ab)=\tr(ba).
$$
\end{prop}

\begin{proof}
If $p=1$ (so $q=\infty$), then we have $a=h_\ffi$ for some $\ffi\in M_*$. Hence by Theorem
\ref{T-11.5} one has
$$
\tr(h_\ffi b)=\tr(h_{\ffi b})=(\ffi b)(1)=\ffi(b)=(b\ffi)(1)=\tr(h_{b\ffi})=\tr(bh_\ffi),
$$
and the case $q=1$ is similar.

Now, assume that $1<p,q<\infty$. Note that $a$ can be written as $a=a_1-a_2+i(a_3-a_4)$
with $a_k\in L^p(M)_+$ and similarly for $b$. By the linearity of $\tr$ we may assume that
$a\in L^p(M)_+$ and $b\in L^q(M)_+$, so $a^p,b^q\in L^1(M)_+$. Then we have $ab,ba\in L^1(M)$
by Lemma \ref{L-11.18}. From Lemma \ref{L-11.18} it follows that the functions
$$
F(z):=\tr(a^{pz}b^{q(1-z)}),\qquad G(z):=\tr(b^{q(1-z)}a^{pz})
$$
are analytic in $0<\Re z<1$. For any $t\in\bR$, since
$a^{p\left({1\over2}+it\right)},b^{q\left({1\over2}+it\right)}\in
\widetilde N\bigl(\textstyle{1\over2}+it\bigr)$, by Lemma \ref{L-11.19} we have
\begin{align*}
F\bigl(\textstyle{1\over2}+it\bigr)&=\tr\bigl(a^{p({1\over2}+it)}b^{({1\over2}-it)}\bigr)
=\tr\bigl(a^{p({1\over2}+it)}\bigl(b^{q({1\over2}+it)}\bigr)^*\bigr) \\
&=\tr\bigl(\bigl(b^{q({1\over2}+it)}\bigr)^*a^{p({1\over2}+it)}\bigr)
=\tr\bigl(b^{q({1\over2}-it)}a^{p({1\over2}+it)}\bigr)=G\bigl(\textstyle{1\over2}+it\bigr).
\end{align*}
This implies that $F=G$ on $0<\Re z<1$; in particular,
$$
\tr(ab)=F(1/p)=G(1/p)=\tr(ba).
$$
\end{proof}

\begin{lemma}\label{L-11.21}
Let $a,b\in L^1(M)_+$ with $\|a\|_1=\|b\|_1=1$. Then $\|a^zb^{1-z}\|_1\le1$ for all
$z\in\bC$ with $0<\Re z<1$.
\end{lemma}

\begin{proof}
Write $z=s+it$ ($0<s<1$, $t\in\bR$). Then $a^s\in L^{1/s}(M)$ with $\|a^s\|_{1/s}=1=s^{-s}s^s$,
Lemma \ref{L-11.14} gives $a^s\in\cN(s^{-s},s)$, and similarly
$b^{1-s}\in\cN((1-s)^{-(1-s)},1-s)$. Hence by Lemma \ref{L-4.11}\,(6) we have
$a^s b^{1-s}\in\cN(s^{-s}(1-s)^{-(1-s)},1)$ so that
$$
a^zb^{1-z}=a^{it}a^s b^{1-s}b^{-it}\in\cN(s^{-s}(1-s)^{-(1-s)},1).
$$
By Lemma \ref{L-11.14} again we have $\|a^zb^{1-z}\|_1\le s^{-s}(1-s)^{-(1-s)}$.
Since $s\in[0,1]\mapsto s^{-s}(1-s)^{-(1-s)}$ is bounded, it follows from Lemma \ref{L-11.18}
that $z\in\{0<\Re z<1\}\mapsto a^zb^{1-z}\in L^1(M)$ is a bounded analytic function. Thanks to
the three-lines theorem, for any $\eps\in(0,1/2)$ it follows that
$$
\sup_{\eps\le\Re z\le1-\eps}\|a^zb^{1-z}\|_1\le\eps^{-\eps}(1-\eps)^{-(1-\eps)}.
$$
Since
$$
\eps^{-\eps}(1-\eps)^{-(1-\eps)}=\exp(-\eps\log\eps-(1-\eps)\log(1-\eps))
\ \longrightarrow\ 1\quad\mbox{as $\eps\to0^+$},
$$
the result follows.
\end{proof}

\begin{thm}[H\"older's inequality]\label{T-11.22}
Let $1\le p,q\le\infty$ with $1/p+1/q=1$. If $a\in L^p(M)$ and $b\in L^q(M)$, then
$ab\in L^1(M)$ and
$$
|\tr(ab)|\le\|ab\|_1\le\|a\|_p\|b\|_q.
$$
\end{thm}

\begin{proof}
Let $a\in L^p(M)$ and $b\in L^q(M)$. Then $ab\in L^1(M)$ by Proposition \ref{P-11.20} and
$|\tr(ab)|\le\|ab\|_1$ holds by \eqref{F-11.6}. When $p=1$ (so $q=\infty$), the inequality
$\|ab\|_1\le\|a\|_p\|b\|_q$ is a special case of \eqref{F-11.11}. The case $q=1$ is similar.

Now, assume that $1<p,q<\infty$. We may assume that $\|a\|_p=\|b\|_q=1$. Let $a=u|a|$ be the
polar decomposition and $b=|b^*|v$ is the right polar decomposition. Then
$|a|^p,|b^*|^q\in L^1(M)$ and $\|\,|a|^p\|_1=\|\,|b^*|^q\|_1=1$. Lemma \ref{L-11.21} yields
$$
\|ab\|_1=\|u|a|\,|b^*|v\|_1\le\|\,|a|\,|b^*|\,\|_1
=\|(|a|^p)^{1/p}(|b^*|^q)^{1/q}\|_1\le1,
$$
where the above first inequality follows from \eqref{F-11.11}.
\end{proof}

\begin{remark}\label{R-11.23}\rm
H\"older's inequality in Theorem \ref{T-11.22} (also in Proposition \ref{P-5.6} in the
semifinite case) holds in a bit more general form as follows: For every $p,q,r\in(0,\infty]$
with $1/p+1/q=1/r$, if $a\in L^p(M)$ and $b\in L^q(M)$, then $ab\in L^r(M)$ and
$\|ab\|_r\le\|a\|_p\|b\|_q$. A nice real analytic proof of this is found in \cite{FK}, while
the above proof from \cite{Te} is based on a complex analytic method using the three-lines
theorem. The idea in \cite{FK} is to apply the majorization in \eqref{F-5.8} of Remark
\ref{R-5.9} (proved in \cite{FK} in a real analytic way) to the expressions
$\mu_t(a)=t^{1/p}\|a\|_p$, $\mu_t(b)=t^{1/q}\|b\|_q$ and $\mu_t(ab)=t^{1/r}\|ab\|_r$ in
\eqref{F-11.9}. Then one has
$$
\log\|ab\|_r+{1\over r}\int_0^1\log t\,dt
\le\log(\|a\|_p\|b\|_q)+{1\over p}\int_0^1\log t\,dt+\int_0^1\log t\,dt,
$$
which implies that $\|ab\|_r\le\|a\|_p\|b\|_q$.
\end{remark}

\begin{prop}\label{P-11.24}
Let $1\le p,q\le\infty$ with $1/p+1/q=1$. Then for every $a\in L^p(M)$,
\begin{align}\label{F-11.12}
\|a\|_p=\sup\{|\tr(ab)|:b\in L^q(M),\,\|b\|_q\le1\}.
\end{align}
\end{prop}

\begin{proof}
If $p=1$, then $a=h_\ffi$ for some $\ffi\in M_*$ and
\begin{align*}
\|h_\ffi\|_1=\|\ffi\|&=\sup\{|\ffi(b)|:b\in M,\,\|b\|_\infty\le1\} \\
&=\sup\{|\tr(h_\ffi b)|:b\in M,\,\|b\|_\infty\le1\}.
\end{align*}
If $p=\infty$, then $a\in M_*$ and
\begin{align*}
\|a\|_\infty&=\sup\{|\ffi(a)|:\ffi\in M_*,\,\|\ffi\|\le1\} \\
&=\sup\{|\tr(ah_\ffi)|:h_\ffi\in L^1(M),\,\|h_\ffi\|_1\le1\}.
\end{align*}

Now, assume that $1<p,q<\infty$. We may assume that $\|a\|_p=1$. Let $a=u|a|$ be the polar
decomposition. Put $b:=|a|^{p/q}u^*$; then $b\in L^q(M)$ and $|b|^q=u|a|u^*$. Hence one has
$\|b\|_q^q=\tr(u|a|^pu^*)=\tr(|a|^p)=1$ by Proposition \ref{P-11.20}, so $\|b\|_q=1$. By
Proposition \ref{P-11.20} again one has
$$
\tr(ab)=\tr(ba)=\tr(|a|^{p/q}|a|)=\tr(|a|^p)=1,
$$
so that $\|a\|_p\le\mbox{the RHS of \eqref{F-11.12}}$. Since
$|\tr(ab)|\le\|ab\|_1\le\|a\|_p\|b\|_q$ by Theorem \ref{T-11.22}, the reverse inequality holds
as well.
\end{proof}

\begin{thm}\label{T-11.25}
\begin{itemize}
\item[\rm(a)] For every $p\in[1,\infty]$, $(L^p(M),\|\cdot\|_p)$ is a Banach space.
\item[\rm(b)] In particular, $L^2(M)$ is a Hilbert space with respect to the inner product
$\<a,b\>:=\tr(a^*b)$ ($=\tr(ba^*)$) for $a,b\in L^2(M)$.
\item[\rm(c)] For any $p\in[1,\infty)$, the norm topology on $L^p(M)$ coincides with the
relative topology induced from $\widetilde N$ with the measure topology. More precisely,
the uniform structure on $L^p(M)$ by $\|\cdot\|_p$ coincides with that induced from
$\widetilde N$.
\end{itemize}
\end{thm}

\begin{proof}
(a)\enspace
Let $a,b\in L^p(M)$. The triangle inequality, i.e., \emph{Minkowski's inequality}
$$
\|a+b\|_p\le\|a\|_p+\|b\|_p
$$
is immediate from the variational expression \eqref{F-11.12}. It is obvious that $\|a\|_p=0$
$\implies$ $a=0$. The completeness will be shown after proving (c).

(b)\enspace
It is straightforward to check that $(a,b)\mapsto\<a,b\>$ is an inner product on $L^2(M)$ and
$\|a\|_2=\<a,a\>^{1/2}$.

(c)\enspace
For every $p\in[1,\infty)$, since $\|\cdot\|_p$ is a norm on $L^p(M)$, the result is immediate
from Lemma \ref{L-11.14}.

Finally, we show the completeness of  $(L^p(M),\|\cdot\|_p)$. The case $p=\infty$ is obvious.
When $1\le p<\infty$, it suffices by (c) to show that $L^p(M)$ is complete with respect to
the uniform topology on $\widetilde N$. But this is immediate since $\widetilde N$ is a
complete metrizable space (see Theorem \ref{T-4.12}) and $L^p(M)$ is a closed subspace of
$\widetilde N$.
\end{proof}

\begin{prop}[Clarkson's inequality]\label{P-11.26}
Let $2\le p<\infty$. The for every $a,b\in L^p(M)$,
$$
\|a+b\|_p^p+\|a-b\|_p^p\le2^{p-1}(\|a\|_p^p+\|b\|_p^p).
$$
\end{prop}

\begin{proof}
An elegant complex analytic proof of this based on the three-lines theorem is in \cite{Te}.
We here give a real analytic proof from \cite{FK}. Set $p':=p/2$, so $1\le p'<\infty$. For
any $a,b\in L^p(M)$, from Lemma \ref{L-11.27} below we have
\begin{align*}
\|a+b\|_p^p+\|a-b\|_p^p
&=\|\,|a+b|^2\|_{p'}^{p'}+\|\,|a-b|^2\|_{p'}^{p'} \\
&\le\|\,|a+b|^2+|a-b|^2\|_{p'}^{p'} \\
&=\|2(|a|^2+|b|^2)\|_{p'}^{p'}=2^{p'}\|\,|a|^2+|b|^2\|_{p'}^{p'} \\
&\le2^{p'}2^{p'-1}\bigl(\|\,|a|^2\|_{p'}^{p'}+\|\,|b|^2\|_{p'}^{p'}\bigr) \\
&=2^{p-1}(\|a\|_p^p+\|b\|_p^p).
\end{align*}
\end{proof}


\begin{lemma}\label{L-11.27}
Let $1\le p<\infty$. For every $a,b\in L^p(M)_+$,
$$
2^{1-p}\|a+b\|_p^p\le\|a\|_p^p+\|b\|_p^p\le\|a+b\|_p^p.
$$
\end{lemma}

\begin{proof}
The first inequality follows from
$$
2^{-p}\|a+b\|_p^p\le\biggl({\|a\|_p+\|b\|_p\over2}\biggr)^p
\le{\|a\|_p^p+\|b\|_p^p\over2}.
$$
For the second, there are unique contractions $u$ and $v$ in $N$ such that
$a^{1/2}=u(a+b)^{1/2}$ and $b^{1/2}=v(a+b)^{1/2}$ with $us(a+b)^\perp=vs(a+b)^\perp=0$.
Since $a+b=(a+b)^{1/2}(u^*u+v^*v)(a+b)^{1/2}$, it immediately follows that $u^*u+v^*v=s(a+b)$.
Moreover, applying $\theta_s$ to $a^{1/2}=u(a+b)^{1/2}$ gives
$e^{-s/2p}a^{1/2}=e^{-s/2p}\theta_s(u)(a+b)^{1/2}$ so that $\theta_s(u)=u$ for all $s\in\bR$.
Hence $u\in M$ and similarly $v\in M$. We now have
\begin{align*}
\|a\|_p^p+\|b\|_p^p&=\|u(a+b)u^*\|_p^p+\|v(a+b)v^*\|_p^p \\
&=\mu_1((u(a+b)u^*)^p)+\mu_1((v(a+b)v^*)^p) \\
&\hskip3.5cm\mbox{(by \eqref{F-11.9} and Proposition \ref{P-4.17}\,(9))} \\
&\le\mu_1(u(a+b)^pu^*)+\mu_1(v(a+b)^pv^*)\quad\mbox{(by Lemma \ref{L-4.22})} \\
&=\tr(u(a+b)^pu^*)+\tr(v(a+b)^pv^*)\qquad\mbox{(by \eqref{F-11.9})} \\
&=\tr((a+b)^{p/2}u^*u(a+b)^{p/2})+\tr((a+b)^{p/2}v^*v(a+b)^{p/2}) \\
&\hskip6cm\mbox{(by Proposition \ref{P-11.20})} \\
&=\tr((a+b)^{p/2}s(a+b)(a+b)^{p/2})=\tr((a+b)^p)=\|a+b\|_p^p.
\end{align*}
\end{proof}

\begin{thm}\label{L-11.28}
Let $1\le p<\infty$ and $1/p+1/q=1$. Then the dual Banach space of $L^p(M)$ is $L^q(M)$ under
the duality pairing $(a,b)\in L^p(M)\times L^q(M)\mapsto\tr(ab)\in\bC$.
\end{thm}

\begin{proof}
Let $1\le p<\infty$ and write $\Phi(b)(a):=\tr(ab)$ for $a\in L^p(M)$ and $b\in L^q(M)$. From
Proposition \ref{P-11.24} it follows that $\Phi(b)\in L^p(M)^*$ and $\|\Phi(b)\|=\|b\|_q$.
Hence $\Phi:L^q(M)\to L^p(M)^*$ is a linear isometry so that $\Phi(L^q(M))$ is a norm-closed
(hence w-closed) subspace of $L^p(M)^*$, for $L^q(M)$ is complete.

If $2\le p<\infty$, then $L^p(M)$ is uniformly convex by Clarkson's inequality (Proposition
\ref{P-11.26}) and so $L^p(M)$ is reflexive. Hence $L^p(M)^*$ is also reflexive so that
$\Phi(L^q(M))$ is w*-closed. But it follows from Proposition \ref{P-11.24} that $\Phi(L^q(M))$
is w*-dense in $L^p(M)^*$. Hence $\Phi(L^q(M))=L^p(M)^*$ follows.

If $1<p<2$, then $2<q<\infty$. Hence from the above case it follows that $L^p(M)=L^q(M)^*$
under the duality $(b,a)\mapsto\tr(ba)=\tr(ab)$ thanks to Proposition \ref{P-11.20}. Hence
$L^p(M)^*=L^q(M)^{**}=L^q(M)$ under the duality $(a,b)\mapsto\tr(ab)$. Finally, the result
holds for $p=1$ since $L^1(M)^*=(M_*)^*=M$.
\end{proof}

\begin{prop}\label{P-11.29}
Let $1\le p,q\le\infty$ with $1/p+1/q=1$. Let $a\in L^q(M)$. Then
$$
a\ge0\ \iff\ \tr(ab)\ge0\ \ \mbox{for all $b\in L^p(M)_+$}.
$$
\end{prop}

\begin{proof}
When $p=\infty$, the result is trivial. When $p=1$, the result is well-known.

Now, assume that $1<p,q<\infty$. If $a\in L^q(M)_+$, then $a^{1/2}ba^{1/2}\in L^1(M)_+$ and
hence $\tr(ab)=\tr(a^{1/2}ba^{1/2})\ge0$. Conversely, assume that $a\in L^q(M)$ satisfies
$\tr(ab)\ge0$ for all $b\in L^q(M)_+$. Since
$$
\tr(ab)=\overline{\tr(ab)}=\tr((ab)^*)=\tr(ba^*)=\tr(a^*b)
$$
for all $b\in L^2(M)_+$, we have $a=a^*$. Let $a=a_+-a_-$ be the Jordan decomposition, so
$a_+a_-=0$. For $b:=a_-^{q/p}\in L^p(M)_+$ we have
$$
0\le\tr(ab)=\tr(a_+b)-\tr(a_-b)=-\tr(a_-b)=-\tr(a_-^q),
$$
which implies that $a_-=0$. Therefore, we have $a=a_+\in L^q(M)_+$.
\end{proof}

The last result in the section is the following:

\begin{thm}\label{T-11.30}
For each $x\in M$ we define the left action $\lambda(x)$ and the right action $\rho(x)$ on the
Hilbert space $L^2(M)$ by
$$
\lambda(x)a:=xa,\qquad\rho(x)a:=ax,\qquad a\in L^2(M),
$$
and the involution $J$ on $L^2(M)$ by $Ja:=a^*$. Then:
\begin{itemize}
\item[\rm(1)] $\lambda$ (resp., $\rho$) is a normal faithful representation (resp.,
anti-representation) of $M$ on $L^2(M)$.
\item[\rm(2)] The von Neumann algebras $\lambda(M)$ and $\rho(M)$ are the commutants of each
other and
$$
\rho(M)=J\lambda(M)J.
$$
\item[\rm(3)] $(\lambda(M),L^2(M),J,L^2(M)_+)$ is a standard form of $M$.
\end{itemize}
\end{thm}

\begin{proof}
(1)\enspace
Since $\<a^*,b\>=\tr(ab)=\tr(ba)=\<b^*,a\>$ for $a,b\in L^2(M)$ by Proposition \ref{P-11.20},
it is clear that $J$ is an involution on $L^2(M)$. Since $(ax)^*=x^*a^*$ for $x\in M$ and
$a\in L^2(M)$, we have
\begin{align}\label{F-11.13}
\rho(x)=J\lambda(x^*)J,\qquad x\in M.
\end{align}
Note that
$$
\<a,\lambda(x)b\>=\tr(a^*xb)=\tr((x^*a)^*b)=\<\lambda(x^*)a,b\>,
\qquad x\in M,\ a,b\in L^2(M),
$$
and if $x_\alpha\nearrow x\in M_+$, then
$$
\<a,\lambda(x_\alpha)a\>=\tr(a^*x_\alpha a)=\tr(x_\alpha aa^*)
\ \nearrow\ \tr(xaa^*)=\tr(a^*xa)=\<a,\lambda(x)a\>,\quad a\in L^2(M).
$$
Assume that $x\in M$ and $\lambda(x)=0$; then
$$
0=\<a,\lambda(x)b\>=\tr(a^*xb)=\tr(xba^*),\qquad a,b\in L^2(M).
$$
Since $L^1(M)=\{ba^*:a,b\in L^2(M)\}$, we have $x=0$. Therefore, $\lambda$ is a normal
faithful representation of $M$ on $L^2(M)$. The assertion for $\rho$ is immediate from
\eqref{F-11.13}.

(2)\enspace
By (1), $\lambda(M)$ and $\rho(M)$ are von Neumann algebras, and $\rho(M)=J\lambda(M)J$
follows from \eqref{F-11.13}. Moreover, $\rho(M)\subset\lambda(M)'$ is obvious by definition.
To prove the converse inclusion, let $T\in\lambda(M)'$, and let us prove that there is a
bounded linear map $Q:L^1(M)\to L^1(M)$ such that
$$
Q\biggl(\sum_{i=1}^nb_ia_i\biggr)=\sum_{i=1}^nb_iT(a_i)
$$
for any $a_i,b_i\in L^2(M)$. To show that $Q$ is well defined, assume that
$\sum_{i=1}^nb_ia_i=0$. Put $a:=\bigl(\sum_{i=1}^na_i^*a_i\bigr)^{1/2}\in L^2(M)_+$. For
$1\le i\le n$, since $a_i^*a_i\le a^2$, one can choose a unique $x_i\in N$ such that
$x_i(1-s(a))=0$ and $a_i=x_ia$. Applying $\theta_s$ gives
$e^{-s/2}a_i=\theta_s(x_i)(e^{-s/2}a)$, implying that $\theta_s(x_i)=x_i$ for all $s\in\bR$.
Hence we have $x_i\in M$. Since
$$
\Biggl(\sum_{i=1}^nb_ix_i\Biggr)a=\sum_{i=1}^nb_ia_i=0\qquad\mbox{and}\qquad
\Biggl(\sum_{i=1}^nb_ix_i\Biggr)(1-s(a))=0,
$$
we have $\sum_{i=1}^nb_ix_i=0$ so that
$$
\sum_{i=1}^nb_iT(a_i)=\sum_{i=1}^nb_iT(x_ia)=\sum_{i=1}^nb_iT(\lambda(x_i)a)
=\Biggl(\sum_{i=1}^nb_ix_i\Biggr)T(a)=0,
$$
as desired. For any $c\in L^1(M)$ one can choose $a,b\in L^2(M)$ such that $c=ab$ and
$\|c\|_1=\|a\|_2\|b\|_2$. Since
$$
\|Q(c)\|_1=\|bT(a)\|_1\le\|b\|_2\|T(a)\|_2\le\|b\|_2\|T\|\,\|a\|_2=\|T\|\,\|c\|_1,
$$
it follows that $Q:L^1(M)\to L^1(M)$ is a bounded linear map. Now, set $x:=Q^*(1)\in M$,
where $Q^*:M\to M$ is the dual map of $Q$. For every $a,b\in L^2(M)$ we find that
\begin{align*}
\<b^*,T(a)\>&=\tr(bT(a))=\tr(Q(ba))=\tr(Q^*(1)ba) \\
&=\tr(xba)=\tr(bax)=\<b^*,ax\>=\<b^*,\rho(x)a\>,
\end{align*}
which implies that $T=\rho(x)\in\rho(M)$. Therefore, $\lambda(M)'\subset\rho(M)$ has been
obtained, so that $\rho(M)=\lambda(M)'$.

(3)\enspace
That $L^2(M)_+$ is a self-dual cone follows from Proposition \ref{P-11.29}. The conditions of
Definitions \ref{D-3.5} are verified as follows:

(a)\enspace
$J\lambda(M)J=\rho(M)=\lambda(M)'$.

(b)\enspace
Any element in the center of $\lambda(M)$ is given as $\lambda(z)$ with $z\in M\cap M'$.
Since $J\lambda(z)J=\rho(z^*)=\rho(z)^*$, so we need to show that $\rho(z)=\lambda(z)$. For
every $a\in L^1(M)$ and $y\in M$ we have $\tr(zay)=\tr(ayz)=\tr(azy)$. This implies that
$za=az$ for all $a\in L^1(M)_+$. By considering the spectral decomposition of $a$, we find
that $za^{1/2}=a^{1/2}z$ for all $a\in L^1(M)_+$, which means that $zb=bz$ for all
$b\in L^2(M)_+$. Hence $\lambda(z)=\rho(z)$ holds.

(c)\enspace
$Ja=a^*=a$ for all $a\in L^2(M)_+$.

(d)\enspace
For every $x\in M$ and $a\in L^2(M)$ we have
$\lambda(x)j(\lambda(x))a=\lambda(x)\rho(x^*)a=xax^*\in L^2(M)_+$.
\end{proof}

\begin{remark}\label{R-11.31}\rm
For any projection $e\in M$, Haagerup's $L^p$-space $L^p(eMe)$ is identified with $eL^p(M)e$.
Furthermore, since $ej(e)L^2(M)=eL^2(M)e$, we see from Proposition \ref{P-3.8} that
$$
(eMe, eL^2(M)e, J=\,^*,eL^2(M)_+e)
$$
is a standard form of $eMe$, where $eMe$ acts on $eL^2(M)e$ as the left multiplication.
\end{remark}

\section{Relative modular operators and Connes' cocycle derivatives (cont.)}

The notion of relative modular operators was introduced by Araki \cite{Ar2} to extend the
relative entropy, the most important quantum divergence in quantum information, to the general
von Neumann algebra setting. The notion is a kind of Radon-Nikodym derivative for two
functionals $\psi,\ffi\in M_*^+$, in a similar vein to Connes' cocycle derivatives in Sec.~7.2.
In this section we give a concise account of relative modular operators and then give a more
detailed description of Connes' cocycle derivatives for functionals in $M_*^+$ (not weights,
while not necessarily faithful).

\subsection{Relative modular operators}

Let $M$ be a von Neumann algebra represented in a standard form $(M,\cH,J,\cP)$, see Sec.~3.1. 
Let $\psi,\ffi\in M_*^+$ be given with their vector representatives $\xi_\psi,\xi_\ffi\in\cP$,
so that $\psi(x)=\<\xi_\psi,x\xi_\psi\>$ and $\ffi(x)=\<\xi_\ffi,x\xi_\ffi\>$ for all $x\in M$.
The support projection $s(\ffi)$ ($\in M$) of $\ffi$ is the orthogonal projection onto
$\overline{M'\xi_\ffi}$, which is more specifically written as $s_M(\ffi)$, the
\emph{$M$-support} of $\ffi$. Also, let $s_{M'}(\ffi)$ ($\in M'$) be the orthogonal projection
onto $\overline{M\xi_\ffi}$, the \emph{$M'$-support} of $\ffi$. Note that
$Js_M(\ffi)J=s_{M'}(\ffi)$, for
$$
Js_M(\ffi)J\cH=Js_M(\ffi)\cH=\overline{JM'\xi_\ffi}=\overline{JM'J\xi_\ffi}
=\overline{M\xi_\ffi}=s_{M'}(\ffi)\cH
$$
thanks to $J\xi_\ffi=\xi_\ffi$ and $JM'J=M$.

\begin{definition}\label{D-12.1}\rm
For every $\psi,\ffi\in M_*^+$ define the operators $S_{\psi,\ffi}^0$ and $F_{\psi,\ffi}^0$ by
\begin{align*}
S_{\psi,\ffi}^0(x\xi_\ffi+\eta)&:=s_M(\ffi)x^*\xi_\psi,
\qquad\,x\in M,\ \eta\in(1-s_{M'}(\ffi))\cH, \\
F_{\psi,\ffi}^0(x'\xi_\ffi+\zeta)&:=s_{M'}(\ffi)x'^*\xi_\psi,
\quad\ \ x'\in M',\ \zeta\in(1-s_M(\ffi))\cH,
\end{align*}
which are closable conjugate linear operators, as shown in the next lemma. Let
$S_{\psi,\ffi},F_{\psi,\ffi}$ be the closures of $S_{\psi,\ffi}^0,F_{\psi,\ffi}^0$,
respectively. Define
$$
\Delta_{\psi,\ffi}:=S_{\psi,\ffi}^*S_{\psi,\ffi}
$$
and call it the \emph{relative modular operator} with respect to $\psi,\ffi$. When $\psi=\ffi$,
we simply write $S_\ffi$, $F_\ffi$ and $\Delta_\ffi$ for $S_{\ffi,\ffi}$, $F_{\ffi,\ffi}$ and
$\Delta_{\ffi,\ffi}$, respectively, and call $\Delta_\ffi$ the \emph{modular operator} with
respect to $\ffi$. Note that when $\ffi$ is faithful, $\Delta_\ffi$ here coincides with the
modular operator in Sec.~2.1.
\end{definition}

\begin{lemma}\label{L-12.2}
$S_{\psi,\ffi}^0$ and $F_{\psi,\ffi}^0$ are well defined and they are densely-defined and
closable conjugate linear operators.
\end{lemma}

\begin{proof}
Assume that $x_1\xi_\ffi+\eta_1=x_2\xi_\ffi+\eta_2$ for $x_i\in M$ and
$\eta_i\in(1-s_{M'}(\ffi))\cH$. Since $(x_1-x_2)\xi_\ffi=\eta_2-\eta_1=0$ is in
$s_{M'}(\ffi)\cH\cap(1-s_{M'}(\ffi))\cH=\{0\}$, one has $(x_1-x_2)\xi_\ffi=0$, which implies
that $x_1s_M(\ffi)=x_2s_M(\ffi)$ so that $s_M(\ffi)x_1^*\xi_\psi=s_M(\ffi)x_2^*\xi_\psi$. Hence
$S_{\psi,\ffi}^0$ is well defined and similarly for $F_{\psi,\ffi}^0$. It is clear that
$S_{\psi,\ffi}^0$ and $F_{\psi,\ffi}^0$ are conjugate linear and densely-defined. For every
$x\in M$, $\eta\in(1-s_{M'}(\ffi))\cH$ and $x'\in M'$, $\zeta\in(1-s_M(\ffi))\cH$, one has
\begin{align*}
\<S_{\psi,\ffi}^0(x\xi_\ffi+\eta),x'\xi_\ffi+\zeta\>
&=\<s_M(\ffi)x^*\xi_\psi,x'\xi_\ffi+\zeta\>
=\<x^*\xi_\psi,x'\xi_\ffi\>=\<x'^*\xi_\psi,x\xi_\ffi\> \\
&=\<s_{M'}(\ffi)x'^*\xi_\psi,x\xi_\ffi+\eta\>
=\<F_{\psi,\ffi}^0(x'\xi_\ffi+\zeta),x\xi_\ffi+\eta\>.
\end{align*}
Since $S_{\psi,\ffi}^0$ and $F_{\psi,\ffi}^0$ are densely-defined, the above equality implies
that so are $S_{\psi,\ffi}^{0*}$ and $F_{\psi,\ffi}^{0*}$ and hence $S_{\psi,\ffi}^0$ and
$F_{\psi,\ffi}^0$ are closable.
\end{proof}

\begin{prop}\label{P-12.3}
For every $\psi,\ffi\in M_*^+$ the following hold:
\begin{itemize}
\item[\rm(1)] The support projection of $\Delta_{\psi,\ffi}:=S_{\psi,\ffi}^*S_{\psi,\ffi}$ is
$s_M(\psi)s_{M'}(\ffi)$.
\item[\rm(2)] $S_{\psi,\ffi}=F_{\psi,\ffi}^*$ and $F_{\psi,\ffi}=S_{\psi,\ffi}^*$.
\item[\rm(3)] $S_{\psi,\ffi}=J\Delta_{\psi,\ffi}^{1/2}$ (the polar decomposition).
\item[\rm(4)] $\Delta_{\ffi,\psi}^{-1}=J\Delta_{\psi,\ffi}J$, where $\Delta_{\psi,\ffi}^{-1}$
is defined with restriction to the support.
\end{itemize}
\end{prop}

\begin{proof}
First, assume that $\psi=\ffi$. Let $e:=s_M(\ffi)$, $e':=s_{M'}(\ffi)$ and $q:=ee'=eJeJ$. By
Proposition \ref{P-3.8}, $(qMq,q\cH,qJq,q\cP)$ is a standard form of $qMq\cong eMe$. Define
$\overline\ffi\in(qMq)_*^+$ by $\overline\ffi(qxq):=\ffi(exe)$ for $x\in M$, whose vector
representative is $\xi_\ffi=q\xi_\ffi\in q\cP$. Note that $\xi_\ffi$ is cyclic and separating
for $qMq$ on $q\cH$. Let $\Delta_{\overline\ffi}$ and $J_{\overline\ffi}$ be the modular
operator and the modular conjugation with respect to $\overline\ffi$. Since
$S_\ffi((1-e)x\xi_\ffi)=ex^*(1-e)\xi_\ffi=0$ for all $x\in M$, note that
$S_\ffi|_{(1-e)e'\cH}=0$ as well as $S_\ffi|_{(1-e')\cH}=0$. Since $(1-e')+(1-e)e'=1-q$,
we find that $S_\ffi|_{(1-q)\cH}=0$, and similarly $F_\ffi|_{(1-q)\cH}=0$. For every $x\in M$,
$qxq\xi_\ffi=ex\xi_\ffi$ and
$$
S_{\overline\ffi}(qxq\xi_\ffi)=qx^*q\xi_\ffi=ex^*\xi_\ffi
=S_\ffi(x\xi_\ffi)=S_\ffi(qxq\xi_\ffi).
$$
Also, for every $x'\in M'$, $qx'q\xi_\ffi=e'x'\xi_\ffi$ and
$$
F_{\overline\ffi}(qx'q\xi\ffi)=qx'^*q\xi_\ffi=e'x'^*\xi_\ffi
=F_\ffi(x'\xi_\ffi)=F_\ffi(qx'q\xi_\ffi).
$$
Therefore, we find that
$$
S_\ffi=S_{\overline\ffi}\oplus0,\quad F_\ffi=F_{\overline\ffi}\oplus0,\qquad
\mbox{so}\quad\Delta_\ffi=\Delta_{\overline\ffi}\oplus0
$$
on the decomposition $\cH=q\cH\oplus(1-q)\cH$. Hence the support projection of $\Delta_\ffi$ is
$q$. Since $J_{\overline\ffi}=qJq$ by Proposition \ref{P-3.10}, we have
$$
S_\ffi=J_{\overline\ffi}\Delta_{\overline\ffi}^{1/2}\oplus0
=(Jq\oplus J(1-q))(\Delta_{\overline\ffi}^{1/2}\oplus0)=J\Delta_\ffi^{1/2}.
$$
So (1) and (3) in this case hold. Furthermore, (2) and (4) follow from those properties of
$S_{\overline\ffi}$, $F_{\overline\ffi}$ and $\Delta_{\overline\ffi}$, see Lemma \ref{L-2.1}
and Remark \ref{R-2.9}.

Next, prove the case of general $\psi,\ffi$. Let $M^{(2)}:=M\otimes\bM_2$, whose standard
form $(M^{(2)},\cH^{(2)},J^{(2)},\cP^{(2)})$ was described in Example \ref{E-3.11}. Define
$\theta\in(M^{(2)})_*^+$ to be the balanced functional of $\ffi,\psi$ similar to \eqref{F-7.1},
i.e., $\theta\Biggl(\begin{bmatrix}x_{11}&x_{12}\\x_{21}&x_{22}\end{bmatrix}\Biggr)
:=\ffi(x_{11})+\psi(x_{22})$, whose vector representative in $\cP^{(2)}$ is
$$
\xi_\theta=\begin{bmatrix}\xi_\ffi&0\\0&\xi_\psi\end{bmatrix}
=\begin{bmatrix}\xi_\ffi\\0\\0\\\xi_\psi\end{bmatrix}.
$$
It is clear that $s_{M^{(2)}}(\theta)=\begin{bmatrix}s_M(\ffi)&0\\0&s_M(\psi)\end{bmatrix}$ or
in the $4\times4$ form,
\begin{align}\label{F-12.1}
s_{M^{(2)}}(\theta)=\begin{bmatrix}s_M(\ffi)&0&0&0\\0&s_M(\ffi)&0&0\\
0&0&s_M(\psi)&0\\0&0&0&s_M(\psi)\end{bmatrix},
\end{align}
and also by \eqref{F-3.6},
\begin{align}\label{F-12.2}
s_{{(M^{(2)})}'}(\theta)=J^{(2)}s_{M^{(2)}}(\theta)J^{(2)}
=\begin{bmatrix}s_{M'}(\ffi)&0&0&0\\0&s_{M'}(\psi)&0&0\\
0&0&s_{M'}(\ffi)&0\\0&0&0&s_{M'}(\psi)\end{bmatrix},
\end{align}
Furthermore, in view of \eqref{F-3.4}, $S_\theta^0$ is defined as
$$
S_\theta^0\left(\begin{bmatrix}x_{11}&0&x_{12}&0\\0&x_{11}&0&x_{12}\\
x_{21}&0&x_{22}&0\\0&x_{22}&0&x_{22}\end{bmatrix}
\begin{bmatrix}\xi_\ffi\\0\\0\\\xi_\psi\end{bmatrix}
+\begin{bmatrix}\eta_{11}\\\eta_{12}\\\eta_{21}\\\eta_{22}\end{bmatrix}\right)
=s_{M^{(2)}}(\theta)\begin{bmatrix}x_{11}^*&0&x_{21}^*&0\\0&x_{11}^*&0&x_{21}^*\\
x_{12}^*&0&x_{22}^*&0\\0&x_{12}^*&0&x_{22}^*\end{bmatrix}
\begin{bmatrix}\xi_\ffi\\0\\0\\\xi_\psi\end{bmatrix},
$$
that is,
$$
S_\theta^0\begin{bmatrix}x_{11}\xi_\ffi+\eta_{11}\\x_{12}\xi_\psi+\eta_{12}\\
x_{21}\xi_\ffi+\eta_{21}\\x_{22}\xi_\psi+\eta_{22}\end{bmatrix}
=\begin{bmatrix}s_M(\ffi)x_{11}^*\xi_\ffi\\s_M(\ffi)x_{21}^*\xi_\psi\\
s_M(\psi)x_{12}^*\xi_\ffi\\s_M(\psi)x_{22}^*\xi_\psi\end{bmatrix},\qquad
x_{ij}\in M,\ \begin{bmatrix}\eta_{11}\\\eta_{12}\\\eta_{21}\\\eta_{22}\end{bmatrix}\in
\bigl(1-s_{{(M^{(2)})}'}(\theta)\bigr)\cH^{(2)}.
$$
Therefore, extending \eqref{F-3.5} (also \eqref{F-7.4}) we find that the closure of
$S_\theta^0$ is
\begin{align}\label{F-12.3}
S_\theta=\begin{bmatrix}S_\ffi&0&0&0\\0&0&S_{\psi,\ffi}&0
\\0&S_{\ffi,\psi}&0&0\\0&0&0&S_\psi\end{bmatrix}
\end{align}
and so $\Delta_\theta$ is written as
\begin{align}\label{F-12.4}
\Delta_\theta
=\begin{bmatrix}S_\ffi^*S_\ffi&0&0&0\\0&S_{\ffi,\psi}^*S_{\ffi,\psi}&0&0
\\0&0&S_{\psi,\ffi}^*S_{\psi,\ffi}&0\\0&0&0&S_\psi^*S_\psi\end{bmatrix}
=\begin{bmatrix}\Delta_\ffi&0&0&0\\0&\Delta_{\ffi,\psi}&0&0
\\0&0&\Delta_{\psi,\ffi}&0\\0&0&0&\Delta_\psi\end{bmatrix}.
\end{align}
On the other hand, in view of \eqref{F-3.4} and \eqref{F-3.6} note that $J^{(2)}XJ^{(2)}$ for
$X=\begin{bmatrix}x_{11}&x_{12}\\x_{21}&x_{22}\end{bmatrix}\in M^{(2)}$ is represented in the
$4\times4$ form as
$$
J^{(2)}XJ^{(2)}=\begin{bmatrix}Jx_{11}J&Jx_{12}J&0&0\\Jx_{21}J&Jx_{22}J&0&0\\
0&0&Jx_{11}J&Jx_{12}J\\0&0&Jx_{21}J&Jx_{22}J\end{bmatrix}.
$$
Hence, $F_\theta^0$ is defined by
$$
F_\theta^0\left(\begin{bmatrix}x_{11}'&x_{12}'&0&0\\x_{21}'&x_{22}'&0&0\\
0&0&x_{11}'&x_{12}'\\0&0&x_{21}'&x_{22}'\end{bmatrix}
\begin{bmatrix}\xi_\ffi\\0\\0\\\xi_\psi\end{bmatrix}
+\begin{bmatrix}\zeta_{11}\\\zeta_{12}\\\zeta_{21}\\\zeta_{22}\end{bmatrix}\right)
=s_{{(M^{(2)})}'}(\theta)\begin{bmatrix}x_{11}'^*&x_{21}'^*&0&0\\x_{12}'^*&x_{22}'^*&0&0\\
0&0&x_{11}'^*&x_{21}'^*\\0&0&x_{12}'^*&x_{22}'^*\end{bmatrix}
\begin{bmatrix}\xi_\ffi\\0\\0\\\xi_\psi\end{bmatrix}
$$
for $x_{ij}'^*\in M'$ and $\begin{bmatrix}\zeta_{11}\\\zeta_{12}\\\zeta_{21}\\\zeta_{22}
\end{bmatrix}\in\bigl(1-s_{M^{(2)}}(\theta)\bigr)\cH^{(2)}$, from which we find that the
closure of $F_\theta^0$ is
\begin{align}\label{F-12.5}
F_\theta=\begin{bmatrix}F_\ffi&0&0&0\\0&0&F_{\ffi,\psi}&0\\
0&F_{\psi,\ffi}&0&0\\0&0&0&F_\psi\end{bmatrix}.
\end{align}
Now, the assertions (1)--(4) of the proposition follow from those (applied to $\theta$) in the
case $\psi=\ffi$ proved previously. That is, (1) is seen by \eqref{F-12.4}, \eqref{F-12.1}
and \eqref{F-12.2}; (2) is seen by \eqref{F-12.3} and \eqref{F-12.5}; (3) follows from
\eqref{F-12.3}, \eqref{F-3.6} and \eqref{F-12.4}; and (4) follows from \eqref{F-12.4} and
\eqref{F-3.6}.
\end{proof}

\begin{example}\label{E-12.4}\rm
(1)\enspace
Let $M=L^\infty(X,\mu)$ as in Example \ref{E-3.6}\,(1). For every
$\psi,\ffi\in L^1(X,\mu)_+\cong M_*^+$, it is easy to verify that $\Delta_{\psi,\ffi}$ is the
multiplication of $1_{\{\ffi>0\}}(\psi/\ffi)$, which is the Radon-Nikodym derivative of $\psi\,d\mu$
with respect to $\ffi\,d\mu$ (restricted on the support of $\ffi$) in the classical sense.

(2)\enspace
Let $M=B(\cH)$ as in Example \ref{E-3.6}\,(2). For every $\psi,\ffi\in B(\cH)_*^+$ we
have the density operators (positive trace-class operators) $D_\psi,D_\ffi$ such that
$\ffi(x)=\Tr D_\ffi x$ for $x\in B(\cH)$ and similarly for $D_\psi$. Let
$D_\psi=\sum_{a>0}aP_a$ and $D_\ffi=\sum_{b>0}bQ_b$ be the spectral decompositions
of $D_\psi,D_\ffi$, where $P_a$ and $Q_b$ are finite-dimensional orthogonal projections.
Then the relative modular operators $\Delta_{\psi,\ffi}$ on $\cC_2(\cH)$ is given as
\begin{align}\label{F-12.6}
\Delta_{\psi,\ffi}=L_{D_\psi}R_{D_\ffi^{-1}}
=\sum_{a>0,\,b>0}ab^{-1}L_{P_a}R_{Q_b},
\end{align}
where $L_{[-]}$ and $R_{[-]}$ denote the left and the right multiplications and $D_\ffi^{-1}$
is the generalized inverse of $D_\ffi$.
\end{example}

In the rest of the subsection we discuss a bit more about relative modular operators in the
standard form on Haagerup's $L^2(M)$. The next lemma (due to Kosaki \cite{Ko1}) gives the
description of the modular automorphism group $\sigma_t^\ffi$ in terms of the element $h_\ffi$
in $L^1(M)$ corresponding to $\ffi\in M_*^+$.

\begin{lemma}\label{L-12.5}
For every $\ffi\in M_*^+$ let $\sigma_t^\ffi$ be the modular automorphism group with respect
to $\ffi|_{s(\ffi)Ms(\ffi)}$, where $s(\ffi)$ us the $M$-support of $\ffi$. Then
\begin{align}\label{F-12.7}
\sigma_t^\ffi(x)=h_\ffi^{it}xh_\ffi^{-it},\qquad x\in s(\ffi)Ms(\ffi),\ t\in\bR,
\end{align}
where $h_\ffi^{it}$ is defined with restriction to the support of $h_\ffi$ (note that
$s(h_\ffi)=s(\ffi)$).
\end{lemma}

\begin{proof}
First, assume that $\ffi$ is faithful. Define $\alpha_t(x):=h_\ffi^{it}xh_\ffi^{-it}$ for
$x\in M$ and $t\in\bR$. Since
$$
\theta_s(\alpha_t(x))=\theta_s(h_\ffi^{it})\theta_s(x)\theta_s(h_\ffi^{-it})
=(e^{-ist}h_\ffi)x(e^{ist}h_\ffi)=\alpha_t(x),\qquad s\in\bR,
$$
it follows that $\alpha_t(x)\in M$ and hence $\alpha_t$ is a strongly continuous one-parameter
automorphism group of $M$. Let $x,y\in M$ and assume that $x$ is entire $\alpha$-analytic with
the analytic extension $\alpha_z(x)$ ($z\in\bC$), see, e.g., \cite[\S2.5.3]{BR}. By analytic
continuation one can see that $h_\ffi^{is}\alpha_z(x)=\alpha_{z+s}(x)h_\ffi^{is}$ for all
$s\in\bR$ and $z\in\bC$, which implies further that $h_\ffi\alpha_z(x)=\alpha_{z-i}(x)h_\ffi$.
Since
$$
\ffi(\alpha_t(x)y)=\tr(h_\ffi\alpha_t(x)y)=\tr(\alpha_{t-i}(x)h_\ffi y)
=\ffi(y\alpha_{t-i}(x)),\qquad t\in\bR,
$$
it follows that $\ffi$ satisfies the $(\sigma_t^\ffi,-1)$-KMS condition for $x,y$. By a
convergence argument based on Lemma \ref{L-2.13} one can show that the KMS condition holds for
all $x,y\in M$ (for further details see, e.g., \cite[p.~82]{BR2}). Hence Theorem \ref{T-2.14}
implies that $\sigma_t^\ffi=\alpha_t$. For general $\ffi\in M_*^+$ let $e:=s(\ffi)$. Since
$h_\ffi\in eL^1(M)e$ corresponds to $\ffi|_{eMe}$ (see Remark \ref{R-11.31}), the result
follows from the above case.
\end{proof}

Before further discussions we here examine Haagerup's $L^p$-spaces for $M^{(2)}=M\otimes\bM_2$.
Take the tensor product $\psi_0\otimes\Tr$ of a faithful semifinite normal weight $\psi_0$ on
$M$ and the trace functional $\Tr$ on $\bM_2$. Then it is immediate to see that
$\sigma_t^{\psi_0\otimes\Tr}=\sigma_t^{\psi_0}\otimes\id_2$, where $\id_2$ is the identity map
on $\bM_2$. By looking at the construction of the crossed products
$N:=M\rtimes_{\sigma^{\psi_0}}\bR$ and $\fN:=M^{(2)}\rtimes_{\sigma^{\psi_0}\otimes\id_2}\bR$
(see Sec.~10.1), the following are easily seen:
\begin{itemize}
\item[\rm(a)] $\fN=N\otimes\bM_2$ (so we write $\fN=N^{(2)}$).
\item[\rm(b)] The canonical trace on $N^{(2)}$ ($=M^{(2)}\rtimes_{\sigma^{\ffi_0}}\bR
=N\otimes\bM_2$) is $\tau\otimes\Tr$, where $\tau$ is the canonical trace on $N$.
\item[\rm(c)] The dual action on $N^{(2)}$ is $\theta_s\otimes\id_2$ ($s\in\bR$), where
$\theta_s$ is the dual action on $N$.
\end{itemize}
Based on these facts we see that $\widetilde{N^{(2)}}=\widetilde N\otimes\bM_2$ for the spaces
$\widetilde N$ and $\widetilde{N^{(2)}}$ of $\tau$-measurable and $\tau\otimes\Tr$-measurable
operators affiliated with $N$ and $N^{(2)}$, respectively. Therefore, for $0<p\le\infty$,
Haagerup's $L^p$-space
$$
L^p(M^{(2)}):=\{a\in\widetilde{N^{(2)}}=\widetilde N\otimes\bM_2:
(\theta_s\otimes\id_2)(a)=e^{-s/p}a,\ s\in\bR\}
$$
is written as
$$
L^p(M^{(2)})=L^p(M)\otimes\bM_2
=\biggl\{a=\begin{bmatrix}a_{11}&a_{12}\\a_{21}&a_{22}\end{bmatrix}:
a_{ij}\in L^p(M),\ i,j=1,2\biggr\},
$$
and its positive part is $(L^p(M)\otimes\bM_2)\cap\widetilde{N^{(2)}}_+$. In particular,
$L^2(M^{(2)})=L^2(M)\otimes\bM_2$ is viewed as the Hilbert space tensor product of $L^2(M)$ and
$\bM_2$ with the Hilbert-Schmidt inner product. Moreover, by closely looking at the
construction of the functional $\tr$, we notice the following:
\begin{itemize}
\item[\rm(d)] The $\tr$-functional on $L^1(M^{(2)})=L^1(M)\otimes\bM_2$ is $\tr\otimes\Tr$,
where $\tr$ is the $\tr$-functional on $L^1(M)$.
\end{itemize}
In this way, the standard form of $M^{(2)}=M\otimes\bM_2$ is given in terms of Haagerup's
$L^2$-space (more specifically than Example \ref{E-3.11}) as
$$
(M\otimes\bM_2,L^2(M)\otimes\bM_2,J=\,^*,(L^2(M)\otimes\bM_2)_+),
$$
where $[x_{ij}]_{i,j=1}^2\in M\otimes\bM_2$ acts on $L^2(M)\otimes\bM_2$ as the left
multiplication as $2\times2$ matrices $\begin{bmatrix}x_{11}&x_{12}\\x_{21}&x_{22}\end{bmatrix}
\begin{bmatrix}\xi_{11}&\xi_{12}\\\xi_{21}&\xi_{22}\end{bmatrix}$ and $J=\,^*$ is the matrix
$*$-operation $\begin{bmatrix}\xi_{11}&\xi_{12}\\\xi_{21}&\xi_{22}\end{bmatrix}^*
=\begin{bmatrix}\xi_{11}^*&\xi_{21}^*\\\xi_{12}^*&\xi_{22}^*\end{bmatrix}$ for
$[\xi_{ij}]_{i,j=1}^2\in L^2(M)\otimes\bM_2$.

\begin{prop}\label{P-12.6}
Let $\psi,\ffi\in M_*^+$ with corresponding $h_\psi,h_\ffi\in L^1(M)$. Then:
\begin{align}
\Delta_{\psi,\ffi}^{it}\xi&=h_\psi^{it}\xi h_\ffi^{-it},
\qquad\ \ \xi\in L^2(M),\ t\in\bR, \label{F-12.8}\\
\Delta_{\psi,\ffi}^pxh_\ffi^{1/2}&=h_\psi^pxh_\ffi^{{1\over2}-p},
\qquad x\in M,\ 0\le p\le1/2, \label{F-12.9}
\end{align}
with the convention that $h_\psi^0=s(\psi)$, $h_\ffi^0=s(\ffi)$ and
$\Delta_{\psi,\ffi}^0=s(\psi)Js(\ffi)J$.
\end{prop}

\begin{proof}
Since $s(\psi)j(s(\ffi))L^2(M)=s(\psi)L^2(M)s(\ffi)$ (see Proposition \ref{P-12.3}\,(1)),
we may assume that $\xi\in s(\psi)L^2(M)s(\ffi)$ and $x\in s(\psi)Ms(\ffi)$. First, assume
that $\psi=\ffi$. Since $\Delta_\ffi h_\ffi^{1/2}=h_\ffi^{1/2}$ and so
$\Delta_\ffi^{it}h_\ffi^{1/2}=h_\ffi^{1/2}$, for every $x\in s(\ffi)Ms(\ffi)$ and $t\in\bR$ one
has
\begin{align}\label{F-12.10}
\Delta_\ffi^{it}(xh_\ffi^{1/2})=\Delta_\ffi^{it}x\Delta_\ffi^{-it}h_\ffi^{1/2}
=\sigma_t^\ffi(x)h_\ffi^{1/2}=h_\ffi^{it}xh_\ffi^{-it}h_\ffi^{1/2}
=h_\ffi^{it}(xh_\ffi^{1/2})h_\ffi^{-it}
\end{align}
thanks to Lemma \ref{L-12.5}. Since $s(\ffi)Mh_\ffi^{1/2}$ is dense in $s(\ffi)L^2(M)s(\ffi)$,
the above implies \eqref{F-12.8} in the case $\psi=\ffi$. Since
$xh_\ffi^{1/2}\in\cD(\Delta_\ffi^{1/2})$, Theorem \ref{T-A.7} of Appendix A implies that there
is an $s(\ffi)L^2(\cH)s(\ffi)$-valued strongly continuous function $f$ on $0\le\Im z\le1/2$,
analytic in $0<\Re z<1/2$, such that $f(it)=\Delta_\ffi^{it}(xh_\ffi^{1/2})$, $t\in\bR$. On the
other hand, consider the function $g(z):=h_\ffi^{{1\over2}+z}xh_\ffi^{{1\over2}-z}$ for
$-1/2<\Re z<1/2$. For $z=r+it$ with $-1/2<r<1/2$ we write
$$
g(r+it)=h_\ffi^{{1\over2}+r}h_\ffi^{it}xh_\ffi^{-it}h_\ffi^{{1\over2}-r}
=h_\ffi^{{1\over2}+r}\sigma_t^\ffi(x)h_\ffi^{{1\over2}-r}
$$
thanks to Lemma \ref{L-12.5} again. Hence $g(z)\in s(\ffi)L^1(M)s(\ffi)$ for $-1/2<\Re z<1/2$
by Theorem 11.22 (while it is immediately seen by applying $\theta_s$). Moreover, from Lemma
\ref{L-11.17} and Theorem \ref{T-11.25}\,(c) it follows that $g(z)$ is an
$s(\ffi)L^1(M)s(\ffi)$-valued analytic function in $-1/2<\Re z<1/2$. For every
$y\in s(\ffi)Ms(\ffi)$ note that
\begin{align*}
\tr(yg(it))&=\tr(yh_\ffi^{1/2}\sigma_t^\ffi(y)h_\ffi^{1/2}) \\
&=\tr(yh_\ffi^{1/2}\Delta_\ffi^{it}(xh_\ffi^{1/2}))\quad\mbox{(by \eqref{F-12.10})} \\
&=\tr(yh_\ffi^{1/2}f(it)),\qquad t\in\bR.
\end{align*}
By analytic continuation this implies that $\tr(yg(z))=\tr(yh_\ffi^{1/2}f(z))$ for all $z$ with
$0\le\Re z<1/2$. Taking $z=p$ with $0\le p<1/2$ one has
$$
\tr(yh_\ffi^{1/2}h_\ffi^pxh_\ffi^{{1\over2}-p})
=\tr(yh_\ffi^{1/2}\Delta_\ffi^p(xh_\ffi^{1/2})),\qquad0\le p<1/2.
$$
Noting that $h_\ffi^pxh_\ffi^{{1\over2}-p}$ and $\Delta_\ffi^p(xh_\ffi^{1/2})$ are in
$s(\ffi)L^2(M)s(\ffi)$ and $s(\ffi)Mh_\ffi^{1/2}$ is dense in $s(\ffi)L^2(M)s(\ffi)$, we see
that \eqref{F-12.9} holds for $0\le p<1/2$ in the case $\psi=\ffi$. When $p=1/2$, note that
$\Delta_\ffi^{1/2}xh_\ffi^{1/2}=J(s(\ffi)xh_\ffi^{1/2})=h_\ffi^{1/2}xs(\ffi)$.

Next, prove the case of general $\psi,\ffi$. Define the balanced functional
$\theta\in(M^{(2)})_*^+$ as in the proof of Proposition \ref{P-12.3}. From the previous case
$\psi=\ffi$ applied to $\theta$ we have
\begin{align*}
\Delta_\theta^{it}\Xi
&=h_\theta^{it}\Xi h_\theta^{-it},\qquad\ \ \Xi\in L^2(M^{(2)}),\ t\in\bR, \\
\Delta_\theta^{p/2}Xh_\theta^{1/2}
&=h_\theta^pXh_\theta^{{1\over2}-p},\qquad X\in M^{(2)},\ 0\le p\le1/2.
\end{align*}
In view of the above description of $L^p(M^{(2)})$ before the projection, note that the
element of $L^1(M^{(2)})$ corresponding to $\theta$ is
$h_\theta=\begin{bmatrix}h_\ffi&0\\0&h_\psi\end{bmatrix}$. With \eqref{F-12.4} apply the above
formulas to $\Xi=\begin{bmatrix}0&0\\\xi&0\end{bmatrix}$ with $\xi\in L^2(M)$ and
$X=\begin{bmatrix}0&0\\x&0\end{bmatrix}$ with $x\in M$; then \eqref{F-12.8} and \eqref{F-12.9}
are given.
\end{proof}

The next theorem gives a somewhat explicit description of the positive powers
$\Delta_{\psi,\ffi}^p$ of $\Delta_{\psi,\ffi}$ represented on $L^2(M)$. The proof of the part
$p>1/2$ is indebted to Jen\v cov\'a \cite{Je3} while the part $0\le p\le1/2$ is from
\cite[Lemma A.3]{Hi}.

\begin{thm}\label{T-12.7}
For every $\psi,\ffi\in M_*^+$ and every $p\ge0$, the domain of $\Delta_{\psi,\ffi}^p$ defined
on $L^2(M)$ coincides with
$$
\cD_p(\psi,\ffi):=\{\xi\in L^2(M):h_\psi^p\xi s(\ffi)=\eta h_\ffi^p
\ \mbox{for some $\eta\in L^2(M)s(\ffi)$}\},
$$
where $h_\psi^p\xi s(\ffi)=\eta h_\ffi^p$ means equality as elements of $\widetilde N$, see
Sec.~11.1. Moreover, if $\xi,\eta$ are given as above, then
\begin{align}\label{F-12.11}
\Delta_{\psi,\ffi}^p\xi=\Delta_{\psi,\ffi}^p(\xi s(\ffi))=\eta.
\end{align}
\end{thm}

\begin{proof}
First, we prove the case $0\le p\le1/2$. When $p=0$, the result is clear, for
$\Delta_{\psi,\ffi}^0=s(\psi)Js(\ffi)J$ so that $\Delta_{\psi,\ffi}^0\xi=s(\psi)\xi s(\ffi)$
for all $\xi\in L^2(M)$. So assume that $0<p\le1/2$. Let $\xi\in L^2(M)$ and
$\eta\in L^2(M)s(\ffi)$ be given with $h_\psi^p\xi s(\ffi)=\eta h_\ffi^p$ and so
$s(\ffi)\xi^*h_\psi^p=h_\ffi^p\eta^*$. For every $x\in M$ one has, thanks to \eqref{F-12.9},
\begin{align*}
\<\xi,\Delta_{\psi,\ffi}^p(xh_\ffi^{1/2})\>
&=\Bigl\<\xi,h_\psi^pxh_\ffi^{{1\over2}-p}\Bigr\>
=\tr\Bigl(s(\ffi)\xi^*h_\psi^pxh_\ffi^{{1\over2}-p}\Bigr) \\
&=\tr\Bigl(h_\ffi^p\eta^*xh_\ffi^{{1\over2}-p}\Bigr) \\
&=\tr(\eta^*xh_\ffi^{1/2})\quad\mbox{(by Proposition \ref{P-11.20})} \\
&=\<\eta,xh_\ffi^{1/2}\>.
\end{align*}
This immediately extends to
$$
\<\xi,\Delta_{\psi,\ffi}^p\zeta\>=\<\eta,\zeta\>,\qquad
\zeta\in Mh_\ffi^{1/2}+L^2(M)(1-s(\ffi)).
$$
Since $Mh_\ffi^{1/2}+L^2(M)(1-s(\ffi))$ is a core of $\Delta_{\psi,\ffi}^{1/2}$, it is also
a core of $\Delta_{\psi,\ffi}^p$ for $0<p\le1/2$, see \cite[Lemma 4]{Ar}. Hence we find that
$\xi\in\cD(\Delta_{\psi,\ffi}^p)$ and \eqref{F-12.11} holds.

Conversely, assume that $\xi\in\cD(\Delta_{\psi,\ffi}^p)$ and $\Delta_{\psi,\ffi}^p\xi=\eta$;
then $\Delta_{\psi,\ffi}^p(\xi s(\ffi))=\eta\in L^2(M)s(\ffi)$. Since
$Mh_\ffi^{1/2}+L^2(M)(1-s(\ffi))$ is a core of $\Delta_{\psi,\ffi}^p$, there exists a sequence
$\{x_n\}$ in $M$ such that
$$
\|x_nh_\ffi^{1/2}-\xi s(\ffi)\|_2\,\longrightarrow\,0,\qquad
\|\Delta_{\psi,\ffi}^p(x_nh_\ffi^{1/2})-\eta\|_2\,\longrightarrow\,0.
$$
Let $\eta_n:=\Delta_{\psi,\ffi}^p(x_nh_\ffi^{1/2})$; then
$\eta_n=h_\psi^px_nh_\ffi^{{1\over2}-p}$ thanks to \eqref{F-12.9} again. Hence one has
\begin{align}\label{F-12.12}
\eta_nh_\ffi^p=h_\psi^px_nh_\ffi^{1/2}.
\end{align}
By H\"older's inequality\footnote{
Unfortunately, this version of H\"older's inequality is not proved in these lecture notes, see
Remark \ref{R-11.23}.}
one has
\begin{align}
\|\eta_nh_\ffi^p-\eta h_\ffi^p\|_{2/(1+2p)}
&\le\|\eta_n-\eta\|_2\|h_\ffi^p\|_{1/p}\,\longrightarrow\,0, \label{F-12.13}\\
\|h_\psi^px_nh_\ffi^{1/2}-h_\psi^p\xi s(\ffi)\|_{2/(1+2p)}
&\le\|h_\psi^p\|_{1/p}\|x_nh_\ffi^{1/2}-\xi s(\ffi)\|_2\,\longrightarrow\,0. \label{F-12.14}
\end{align}
Combining \eqref{F-12.12}--\eqref{F-12.14} gives $h_\psi^p\xi s(\ffi)=\eta h_\ffi^p$. Thus,
$\cD(\Delta_{\psi,\ffi}^p)=\cD_p(\psi,\ffi)$ has been shown in the case $0\le p\le1/2$.

Define
$$
\cD_\infty(\psi,\ffi)
:=\bigl\{\xi\in L^2(M):t\in\bR\mapsto\Delta_{\psi,\ffi}^{it}\xi\in L^2(M)
\ \mbox{extends to an entire function}\bigr\}.
$$
By a familiar regularization technique with Gaussian kernels (as in the proof of Lemma
\ref{L-2.13} and also used in the last part of the proof here), it is seen that
$\cD_\infty(\psi,\ffi)$ is dense in $L^2(M)$ and is a core of $\Delta_{\psi,\ffi}^p$ for any
$p\ge0$. Let $\xi\in\cD_\infty(\psi,\ffi)$. By Lemma \ref{L-11.17} and Theorem
\ref{T-11.25}\,(c) we see that the $\widetilde N$-valued functions
$z\mapsto h_\psi^z\xi s(\ffi)$ and $z\mapsto(\Delta_{\psi,\ffi}^z\xi)h_\ffi^z$ are analytic in
$\Re z>0$. By the above proved case these functions coincide for $z=p\in(0,1/2]$, so they must
be equal for all $z$ with $\Re z>0$. This shows that
$\cD_\infty(\psi,\ffi)\subset\cD_p(\psi,\ffi)$ for all $p>0$ and \eqref{F-12.11} holds with
$\eta=\Delta_{\psi,\ffi}^p\xi$ for every $\xi\in\cD_\infty(\psi,\ffi)$ and any $p>0$.

Now, let $p>1/2$ and let $T_p$ be the operator with domain $\cD_p(\psi,\ffi)$ defined by
$T_p\xi:=\eta$, which is clearly a linear operator on $L^2(M)$ (note here that
$\eta\in L^2(M)s(\ffi)$ is uniquely determined for each $\xi\in\cD_p(\psi,\ffi)$). By the
previous paragraph note that $T_p\xi=\Delta_{\psi,\ffi}^p\xi$ for all
$\xi\in\cD_\infty(\psi,\ffi)$. Let us show that $T_p$ is a closed operator with a core
$\cD_\infty(\psi,\ffi)$. Then the result follows, for $\cD_\infty(\psi,\ffi)$ is also a core
of $\Delta_{\psi,\ffi}^p$.

So let $\{\xi_n\}$ be a sequence in $\cD_p(\psi,\ffi)$ such that $\xi_n\to\xi$ and
$T_p\xi_n\to\eta$ in $L^2(M)$. Then by Theorem \ref{T-11.25}\,(c), $\xi_n\to\xi$ and
$T_p\xi_n\to\eta$ in $\widetilde N$. Hence, from Theorem \ref{T-4.12} it follows that
$$
h_\psi^p\xi s(\ffi)=\lim_nh_\psi^p\xi_ns(\ffi)=\lim_n(T_p\xi_n)h_\ffi^p=\eta h_\ffi^p
$$
in $\widetilde N$ with the measure topology, so that $\xi\in\cD_p(\psi,\ffi)$ and
$T_p\xi=\eta$. Hence $T_p$ is closed. To show that $\cD_\infty(\psi,\ffi)$ is a core of $T_p$,
let $\xi\in\cD_p(\psi,\ffi)$ and $\eta=T_p\xi$. For $n\in\bN$ set
\begin{align*}
\xi_n&:=\sqrt{n\over\pi}\int_{-\infty}^\infty
e^{-nt^2}\Delta_{\psi,\ffi}^{it}\xi\,dt+\xi(1-s(\ffi)), \\
\eta_n&:=\sqrt{n\over\pi}\int_{-\infty}^\infty
e^{-nt^2}\Delta_{\psi,\ffi}^{it}\eta\,dt.
\end{align*}
Then $\xi_n\in\cD_\infty(\psi,\ffi)$ and $\xi_n\to\xi$, $\eta_n\to\eta$ in $L^2(M)$. By
\eqref{F-12.8} note that
\begin{align}\label{F-12.15}
h_\psi^p\Delta_{\psi,\ffi}^{it}\xi=h_\psi^{it}h_\psi^p\xi h_\ffi^{-it}
=h_\psi^{it}\eta h_\ffi^ph_\ffi^{-it}=(\Delta_{\psi,\ffi}^{it}\eta)h_\ffi^p,\qquad t\in\bR.
\end{align}
For each $n$, to see that $T_p\xi_n=\eta_n$, one can take sequences $\xi_{n,k},\eta_{n,k}$
($k\in bN$) of Riemann sums
\begin{align*}
\xi_{n,k}&:=\sqrt{n\over\pi}\sum_{l=1}^{m_k}\bigl(t_l^{(k)}-t_{l-1}^{(k)}\bigr)
e^{-n\bigl(t_l^{(k)}\bigr)^2}\Delta_{\psi,\ffi}^{it_l^{(k)}}\xi+\xi(1-s(\ffi)), \\
\eta_{n,k}&:=\sqrt{n\over\pi}\sum_{l=1}^{m_k}\bigl(t_l^{(k)}-t_{l-1}^{(k)}\bigr)
e^{-n\bigl(t_l^{(k)}\bigr)^2}\Delta_{\psi,\ffi}^{it_l^{(k)}}\eta,
\end{align*}
with $-\infty<t_0^{(k)}<t_1^{(k)}<\dots<t_{m_k}^{(k)}<\infty$, such that
$\|\xi_{n,k}-\xi_n\|_2\to0$ and $\|\eta_{n,k}-\eta_n\|_2\to0$ as $k\to\infty$. Then it follows
from \eqref{F-12.15} that
$$
h_\psi^p\xi_ns(\ffi)=\lim_kh_\psi^p\xi_{n,k}s(\ffi)=\lim_k\eta_{n,k}h_\ffi^p=\eta_nh_\ffi^p
$$
in $\widetilde N$, so that $T_p\xi_n=\eta_n$ for all $n$. Hence $\cD_\infty(\psi,\ffi)$ is a
core of $T_p$, as desired.
\end{proof}

\begin{remark}\label{R-12.8}\rm
For the case $p=1$, Theorem \ref{T-12.7} says that $\Delta_{\psi,\ffi}\xi=\eta$ when
$\xi,\eta\in L^2(M)$ satisfy $h_\psi\xi s(\ffi)=\eta h_\ffi$ with $\eta s(\ffi)=\eta$. The
condition might be formally written as $\eta=h_\psi\xi h_\ffi^{-1}$, so we may write
$\Delta_{\psi,\ffi}=L_{h_\psi}R_{h_\ffi^{-1}}$ in a formal sense. This is the same expression
as in \eqref{F-12.6}.
\end{remark}

\subsection{Connes' cocycle derivatives (cont.)}

In this section, as a continuation of Sec.~7.2, we discuss more about Connes' cocycle
derivatives here restricted to bounded functionals $\ffi,\psi\in M_*^+$ while not necessarily
faithful unlike Sec.~7.2.

For each $\ffi,\psi\in M_*^+$ consider the \emph{balanced functional}
$\theta=\theta(\ffi,\psi)$ on $M^{(2)}=M\otimes\bM_2$ defined by
$\theta\bigl(\sum_{i,j=1}^2x_{ij}\otimes e_{ij}\bigr):=\ffi(x_{11})+\psi(x_{22})$,
$x_{ij}\in M$, where $e_{ij}$ ($i,j=1,2$) are the matrix units of $\bM_2(\bC)$. The $\theta$ has
already been treated in the proof of Proposition \ref{P-12.3}. Since the support projection of
$\theta$ is $s(\theta)=s(\ffi)\otimes e_{11}+s(\psi)\otimes e_{22}$, note that
$[x_{ij}]_{i,j=1}^2=\sum_{i,j=1}^2x_{ij}\otimes e_{ij}\in M^{(2)}$ belongs to
$s(\theta)M^{(2)}s(\theta)$ if and only if
\begin{align}\label{F-12.16}
x_{11}\in s(\ffi)Ms(\ffi),\quad x_{12}\in s(\ffi)Ms(\psi),\quad
x_{21}\in s(\psi)Ms(\ffi),\quad x_{22}\in s(\psi)Ms(\psi).
\end{align}
We then define the modular automorphism group $\sigma_t^\theta$ on $s(\theta)Ns(\theta)$ as well
as $\sigma_t^\ffi$ on $s(\ffi)Ms(\ffi)$ and $\sigma_t^\psi$ on $s(\psi)Ms(\psi)$.

\begin{lemma}\label{L-12.9}
Let $\ffi$, $\psi$ and $\theta$ be as above.
\begin{itemize}
\item[\rm(1)] $s(\ffi)\otimes e_{11},s(\psi)\otimes e_{22}\in(s(\theta)Ns(\theta))^\theta$
(the centralizer of $\theta|_{s(\theta)Ns(\theta)}$, see Definition \ref{D-2.15}).
\item[\rm(2)] $\sigma_t^\theta(x\otimes e_{11})=\sigma_t^\ffi(x)\otimes e_{11}$ for all
$x\in s(\ffi)Ms(\ffi)$ and $t\in\bR$.
\item[\rm(3)] $\sigma_t^\theta(x\otimes e_{22})=\sigma_t^\psi(x)\otimes e_{11}$ for all
$x\in s(\psi)Ms(\psi)$ and $t\in\bR$.
\item[\rm(4)] $\sigma_t^\theta(s(\psi)Ms(\ffi)\otimes e_{21})\subset
s(\psi)Ms(\ffi)\otimes e_{21}$.
\end{itemize}
\end{lemma}

\begin{proof}
(1)\enspace
Note that, for every $[x_{ij}]\in s(\theta)Ns(\theta)$,
$$
\theta\biggl(\begin{bmatrix}s(\ffi)&0\\0&0\end{bmatrix}
\begin{bmatrix}x_{11}&x_{12}\\x_{21}&x_{22}\end{bmatrix}\biggr)
=\ffi(s(\ffi)x_{11})=\ffi(x_{11}s(\ffi))
=\theta\biggl(\begin{bmatrix}x_{11}&x_{12}\\x_{21}&x_{22}\end{bmatrix}
\begin{bmatrix}s(\ffi)&0\\0&0\end{bmatrix}\biggr)
$$
and similarly
$$
\theta\biggl(\begin{bmatrix}0&0\\0&s(\psi)\end{bmatrix}
\begin{bmatrix}x_{11}&x_{12}\\x_{21}&x_{22}\end{bmatrix}\biggr)
=\theta\biggl(\begin{bmatrix}x_{11}&x_{12}\\x_{21}&x_{22}\end{bmatrix}
\begin{bmatrix}0&0\\0&s(\psi)\end{bmatrix}\biggr).
$$
Hence (1) follows.

(2)\enspace
By (1) with Proposition \ref{P-2.16} there is a strongly* continuous automorphism group
$\alpha_t$ on $s(\ffi)Ms(\ffi)$ such that
$$
\sigma_t^\theta(x\otimes e_{11})=\alpha_t(x)\otimes e_{11},\qquad
x\in s(\ffi)Ms(\ffi),\ t\in\bR.
$$
When $x,y\in s(\ffi)Ms(\ffi)$, the KMS condition of $\theta$ for $\sigma_t^\theta$ with
$x\otimes e_{11}$ and $y\otimes e_{11}$ induces that of $\ffi$ for $\alpha_t$ with $x$ and
$y$. Hence Theorem \ref{T-2.14} implies that $\alpha_t=\sigma_t^\ffi$ on $s(\ffi)Ms(\ffi)$.

(3) is similar to (2).

(4)\enspace
For every $x\in s(\psi)Ms(\ffi)$, by (1) with Proposition \ref{P-2.16} we have
\begin{align*}
\sigma_t^\theta\biggl(\begin{bmatrix}0&0\\x&0\end{bmatrix}\biggr)
&=\sigma_t^\theta\biggl(\begin{bmatrix}0&0\\0&s(\psi)\end{bmatrix}
\begin{bmatrix}0&0\\x&0\end{bmatrix}\begin{bmatrix}s(\ffi)&0\\0&0\end{bmatrix}\biggr) \\
&=\begin{bmatrix}0&0\\0&s(\psi)\end{bmatrix}
\sigma_t^\theta\biggl(\begin{bmatrix}0&0\\x&0\end{bmatrix}\biggr)
\begin{bmatrix}s(\ffi)&0\\0&0\end{bmatrix}\in s(\psi)Ms(\ffi)\otimes e_{21}.
\end{align*}
\end{proof}

\begin{definition}\label{D-12.10}\rm
Let $\ffi,\psi\in M_*^+$ and $\theta=\theta(\ffi,\psi)$. By Lemma \ref{L-12.9}\,(4) there
exists a strongly* continuous map $t\in\bR\mapsto u_t\in s(\psi)Ms(\ffi)$ such that
$$
\sigma_t^\theta(s(\psi)s(\ffi)\otimes e_{21})=u_t\otimes e_{21},\qquad t\in\bR.
$$
The map $t\mapsto u_t$ is called \emph{Connes' cocycle (Radon-Nikodym) derivative} of
$\psi$ with respect to $\ffi$, and denoted by $(D\psi:D\ffi)_t$. Note that the definition here
extends Definition \ref{D-7.7} when $\ffi,\psi$ are faithful.
\end{definition}

The next proposition specifies the relation between Connes' cocycle derivative $(D\psi:D\ffi)_t$
and Araki's relative modular operator $\Delta_{\psi,\ffi}$.

\begin{prop}\label{P-12.11}
For every $\ffi,\psi\in M_*^+$,
\begin{align}
(D\psi:D\ffi)_tJs(\ffi)J&=\Delta_{\psi,\ffi}^{it}\Delta_\ffi^{-it}, \label{F-12.17}\\
(D\psi:D\ffi)_tJs(\psi)J&=\Delta_\psi^{it}\Delta_{\ffi,\psi}^{-it}, \label{F-12.18}\\
(D\psi:D\ffi)_t\Delta_\ffi^{it}&=\Delta_{\psi,\ffi}^{it} \label{F-12.19}
\end{align}
for all $t\in\bR$.
\end{prop}

\begin{proof}
Consider the balanced functional $\theta=\theta(\ffi,\psi)$ on $M^{(2)}=M\otimes\bM_2$ and
the associated $\Delta_\theta$ represented in the standard form of $M^{(2)}$ as described in
Example \ref{E-3.11} and in the proof of Proposition \ref{P-12.3}. From \eqref{F-12.1} and
\eqref{F-12.2} the support projection of $\Delta_\theta$ is
\begin{align}\label{F-12.20}
s(\Delta_\theta)=\begin{bmatrix}s(\ffi)Js(\ffi)J&0&0&0\\0&s(\ffi)Js(\psi)J&0&0\\
0&0&s(\psi)Js(\ffi)J&0\\0&0&0&s(\psi)Js(\psi)J\end{bmatrix}.
\end{align}
For every $X=[x_{ij}]\in s(\theta)M^{(2)}s(\theta)$ (see \eqref{F-12.16}) we write
\begin{align}
\sigma_t^\theta(X)&=\Delta_\theta^{it}\begin{bmatrix}
x_{11}&0&x_{12}&0\\0&x_{11}&0&x_{12}\\x_{21}&0&x_{22}&0\\0&x_{21}&0&x_{22}\end{bmatrix}
\Delta_\theta^{-it} \nonumber\\
&=\begin{bmatrix}\Delta_\ffi^{it}x_{11}\Delta_\ffi^{-it}&0&
\Delta_\ffi^{it}x_{12}\Delta_{\psi,\ffi}^{-it}&0\\
0&\Delta_{\ffi,\psi}^{it}x_{11}\Delta_{\ffi,\psi}^{-it}&0&
\Delta_{\ffi,\psi}^{it}x_{12}\Delta_\psi^{-it}\\
\Delta_{\psi,\ffi}^{it}x_{21}\Delta_\ffi^{-it}&0&
\Delta_{\psi,\ffi}^{it}x_{22}\Delta_{\psi,\ffi}^{-it}&0\\
0&\Delta_\psi^{it}x_{21}\Delta_{\ffi,\psi}^{-it}&0&
\Delta_\psi^{it}x_{22}\Delta_\psi^{-it}\end{bmatrix}. \label{F-12.21}
\end{align}
In articular, letting $X=s(\psi)s(\ffi)\otimes e_{21}$ we find that
\begin{align*}
(D\psi:D\ffi)_t|_{s(\ffi)Js(\ffi)J}
&=\Delta_{\psi,\ffi}^{it}s(\psi)s(\ffi)\Delta_\ffi^{-it}|_{s(\ffi)Js(\ffi)J}, \\
(D\psi:D\ffi)_t|_{s(\ffi)Js(\psi)J}
&=\Delta_\psi^{it}s(\psi)s(\ffi)\Delta_{\ffi,\psi}^{-it}|_{s(\ffi)Js(\psi)J},
\end{align*}
where $s(\ffi)Js(\ffi)J$ and $s(\ffi)Js(\psi)J$ arise from \eqref{F-12.20}. By taking account
of the supports of $\Delta_{\psi,\ffi}$, $\Delta_\ffi$, $\Delta_\psi$ and $\Delta_{\ffi,\psi}$
we have \eqref{F-12.17} and \eqref{F-12.18}. It is immediate to see that \eqref{F-12.17} gives
\eqref{F-12.19} as well.
\end{proof}

\begin{remark}\label{R-12.12}\rm
Since $x\in s(\ffi)Ms(\ffi)\mapsto xJs(\ffi)J\in s(\ffi)Ms(\ffi)Js(\ffi)J$ is a *-isomorphism
and $(D\psi:D\ffi)_t\in s(\psi)Ms(\ffi)$, from \eqref{F-12.17} and \eqref{F-12.18} we may a bit
roughly write
\begin{align*}
(D\psi:D\ffi)_t&=\Delta_{\psi,\ffi}^{it}\Delta_\ffi^{-it}
\qquad\mbox{if $s(\psi)\le s(\ffi)$}, \\
(D\psi:D\ffi)_t&=\Delta_\psi^{it}\Delta_{\ffi,\psi}^{-it}
\qquad\ \ \mbox{if $s(\ffi)\le s(\psi)$},
\end{align*}
Similarly, from \eqref{F-12.21} we may also write
$$
\sigma_t^\psi(x)=\Delta_{\psi,\ffi}^{it}x\Delta_{\psi,\ffi}^{-it},
\quad x\in s(\psi)Ms(\psi),\qquad\mbox{if $s(\psi)\le s(\ffi)$}.
$$
\end{remark}

We next show the expression of $(D\psi:D\ffi)_t$ in terms of Haagerup's $L^1(M)$-elements,
which is quite convenient to derive properties of Connes' cocycle derivative.

\begin{lemma}\label{L-12.13}
Let $\psi,\ffi\in M_*^+$ with corresponding $h_\psi,h_\ffi\in L^1(M)$. Then:
\begin{align}\label{F-12.22}
(D\psi:D\ffi)_t=h_\psi^{it}h_\ffi^{-it},\qquad t\in\bR,
\end{align}
with the same convention as in Lemma \ref{L-12.5} and Proposition \ref{P-12.6}.
\end{lemma}

\begin{proof}
In view of the description of $L^p(M\otimes\bM_2)$ before Proposition \ref{P-12.6}, since the
element of $L^1(N)$ corresponding to $\theta(\ffi,\psi)$ is
$h_\theta=\begin{bmatrix}h_\ffi&0\\0&h_\psi\end{bmatrix}$, it follows from \eqref{F-12.7} that,
for $[x_{ij}]\in s(\theta)Ns(\theta)$,
\begin{align}
\sigma_t^\theta\biggl(\begin{bmatrix}x_{11}&x_{12}\\x_{21}&x_{22}\end{bmatrix}\biggr)
&=\begin{bmatrix}h_\ffi&0\\0&h_\psi\end{bmatrix}^{it}
\begin{bmatrix}x_{11}&x_{12}\\x_{21}&x_{22}\end{bmatrix}
\begin{bmatrix}h_\ffi&0\\0&h_\psi\end{bmatrix}^{-it} \nonumber\\
&=\begin{bmatrix}h_\ffi^{it}x_{11}h_\ffi^{-it}&h_\ffi^{it}x_{12}h_\psi^{-it}\\
h_\psi^{it}x_{21}h_\ffi^{-it}&h_\psi^{it}x_{22}h_\psi^{-it}\end{bmatrix}, \label{F-12.23}
\qquad t\in\bR.
\end{align}
Therefore,
$$
\begin{bmatrix}0&0\\(D\psi:D\ffi)_t&0\end{bmatrix}
=\sigma_t^\theta\biggl(\begin{bmatrix}0&0\\s(\psi)s(\ffi)&0\end{bmatrix}\biggr)
=\begin{bmatrix}0&0\\h_\psi^{it}h_\ffi^{-it}&0\end{bmatrix}
$$
so that \eqref{F-12.22} follows.
\end{proof}

\begin{remark}\label{R-12.14}\rm
Indeed, the assertions of Lemma \ref{L-12.9} are immediately seen from the expression in
\eqref{F-12.23} and Lemma \ref{L-12.5}.
\end{remark}

\begin{example}\label{E-12.15}\rm
Assume that $M$ is a semifinite von Neumann algebra with a faithful semifinite normal trace
$\tau$. As explained in Example \ref{E-11.11}, the Haagerup $L^1$-space $L^1(M)$ in this case
is identified with the conventional $L^1$-space $L^1(M,\tau)$ with respect to $\tau$. More
precisely, for each $\ffi\in M_*$, $h_\ffi$ in $L^1(M)$ and the Radon-Nikodym derivative
${d\ffi\over d\tau}\in L^1(M,\tau)$ are in the relation that
$h_\ffi={d\psi\over d\tau}\otimes e^{-t}$. Hence, for every $\ffi,\psi\in M_*^+$
we have
$$
(D\psi:D\ffi)_t
=\biggl({d\psi\over d\tau}\biggr)^{it}\biggl({d\ffi\over d\tau}\biggr)^{-it}.
$$
In particular, when $B=B(\cH)$ with the usual trace $\Tr$,
$(D\psi:D\ffi)_t=D_\psi^{it}D_\ffi^{-it}$ where $D_\ffi,D_\psi$ are the density (trace-class)
operators representing $\ffi,\psi\in B(\cH)_*^+$.
\end{example}

In the following theorems we present important properties of the Connes cocycle derivative
$(D\psi:D\ffi)_t$, which were first given by Connes \cite{Co1,Co2} for the case of the
faithful semifinite normal weights. The properties (ii) and (iv) appeared in Theorem \ref{T-7.6}
for f.s.n.\ weights $\ffi,\psi$.

\begin{thm}\label{T-12.16}
Let $\ffi,\psi\in M_*^+$ and assume that $s(\psi)\le s(\ffi)$. Then $u_t:=(D\psi:D\ffi)_t$
satisfies the following properties:
\begin{itemize}
\item[\rm(1)] $u_tu_t^*=s(\psi)=u_0$ and $u_t^*u_t=\sigma_t^\ffi(s(\psi))$ for all $t\in\bR$.
In particular, $u_t$'s are partial isometries with the final projection $s(\psi)$.
\item[\rm(2)] $u_{s+t}=u_s\sigma_s^\ffi(u_t)$ for all $s,t\in\bR$ (\emph{cocycle identity}).
\item[\rm(3)] $u_{-t}=\sigma_{-t}^\ffi(u_t^*)$ for all $t\in\bR$.
\item[\rm(4)] $\sigma_t^\psi(x)=u_t\sigma_t^\ffi(x)u_t^*$ for all $t\in\bR$ and
$x\in s(\psi)Ms(\psi)$.
\item[\rm(5)] For every $x\in s(\ffi)Ms(\psi)$ and $y\in s(\psi)Ms(\ffi)$, there exists a
continuous bounded function $F$ on $0\le\Im z\le1$, that is analytic in $0<\Im z<1$, such that
$$
F(t)=\psi(u_t\sigma_t^\ffi(y)x),\qquad
F(t+i)=\ffi(xu_t\sigma_t^\ffi(y)),\qquad t\in\bR.
$$
\end{itemize}
Furthermore, $u_t$ ($t\in\bR$) is uniquely determined by a strongly* continuous map
$t\in\bR\mapsto u_t\in M$ satisfying the above (1), (2), (4) and (5). (Note that (3)
follows from (1) and (2).)
\end{thm}

\begin{proof}
First note that $s(h_\psi)=s(\psi)\le s(\ffi)=s(h_\ffi)$. In the proof below we repeatedly
use \eqref{F-12.22} as well as \eqref{F-12.7}. We compute
$$
u_tu_t^*=(h_\psi^{it}h_\ffi^{-it})(h_\psi^{it}h_\ffi^{-it})^*
=h_\psi^{it}h_\ffi^{-it}h_\ffi^{it}h_\psi^{-it}
=h_\psi^{it}s(\ffi)h_\psi^{-it}=h_\psi^{it}h_\psi^{-it}=s(\psi),
$$
$$
u_0=s(\psi)s(\ffi)=s(\psi),
$$
$$
u_t^*u_t=h_\ffi^{it}h_\psi^{-it}h_\psi^{it}h_\ffi^{-it}
=h_\ffi^{it}s(\psi)h_\ffi^{-it}=\sigma_t^\ffi(s(\psi)),
$$
$$
u_s\sigma_s^\ffi(u_t)
=h_\psi^{is}h_\ffi^{-is}h_\ffi^{is}h_\psi^{it}h_\ffi^{-it}h_\ffi^{is}
=h_\psi^{i(s+t)}h_\ffi^{-i(s+t)}=u_{s+t},
$$
$$
\sigma_{-t}^\ffi(u_t^*)=h_\ffi^{-it}h_\ffi^{it}h_\psi^{-it}h_\ffi^{it}
=h_\psi^{-it}h_\ffi^{it}=u_{-t},
$$
and for any $x\in s(\psi)Ms(\psi)$,
$$
u_t\sigma_t^\ffi(x)u_t^*
=h_\psi^{it}h_\ffi^{-it}h_\ffi^{it}xh_\ffi^{-it}h_\ffi^{it}h_\psi^{-it}
=h_\psi^{it}xh_\psi^{-it}=\sigma_t^\psi(x).
$$
Hence all the properties in (1)--(4) have been shown.

(5)\enspace
Let $x\in s(\ffi)Ms(\psi)$ and $y\in s(\psi)Ms(\ffi)$. When $X=x\otimes e_{12}$ and
$Y=y\otimes e_{21}$, we write by \eqref{F-12.23}
\begin{align*}
\theta(\sigma_t^\theta(Y)X)
&=\theta\biggl(\begin{bmatrix}0&0\\0&h_\psi^{it}yh_\ffi^{-it}x\end{bmatrix}\biggr)
=\theta\biggl(\begin{bmatrix}0&0\\0&u_t\sigma_t^\ffi(y)x\end{bmatrix}\biggr)
=\psi(u_t\sigma_t^\ffi(y)x), \\
\theta(X\sigma_t^\theta(Y))
&=\theta\biggl(\begin{bmatrix}xh_\psi^{it}yh_\ffi^{-it}&0\\0&0\end{bmatrix}\biggr)
=\theta\biggl(\begin{bmatrix}xu_t\sigma_t^\ffi(y)&0\\0&0\end{bmatrix}\biggr)
=\ffi(xu_t\sigma_t^\ffi(y)).
\end{align*}
Hence the assertion follows from the KMS condition of $\theta$ for $\sigma^\theta$.

To prove the unicity assertion, assume that $t\in\bR\mapsto u_t\in M$ is a strongly*
continuous map satisfying (1)--(4). For each $t\in\bR$ define a map $\sigma_t$ from
$s(\theta)Ns(\theta)$ into itself by
\begin{align*}
\sigma_t\biggl(\begin{bmatrix}x_{11}&x_{12}\\x_{21}&x_{22}\end{bmatrix}\biggr)
&:=\begin{bmatrix}\sigma_t^\ffi(x_{11})&\sigma_t^\ffi(x_{12})u_t^*\\
u_t\sigma_t^\ffi(x_{21})&\sigma_t^\psi(x_{22})\end{bmatrix} \\
&\ =\begin{bmatrix}1&0\\0&u_t\end{bmatrix}
\begin{bmatrix}\sigma_t^\ffi(x_{11})&\sigma_t^\ffi(x_{12})u_t^*\\
u_t\sigma_t^\ffi(x_{21})&\sigma_t^\ffi(x_{22})\end{bmatrix}
\begin{bmatrix}1&0\\0&u_t^*\end{bmatrix}
\end{align*}
for $[x_{ij}]\in s(\theta)Ns(\theta)$, where the last equality is due to (4).
For $X=[x_{ij}]$, $Y=[y_{ij}]\in s(\theta)Ns(\theta)$ we find by (1) and (2) that
$\sigma_t(X)\sigma_t(Y)=\sigma_t(XY)$, $\sigma_t(X)^*=\sigma_t(X^*)$, $\sigma_0(X)=X$ and
\begin{align*}
\sigma_s(\sigma_t(X))
&=\begin{bmatrix}\sigma_s^\ffi(\sigma_t^\ffi(x_{11}))&
\sigma_s^\ffi(\sigma_t^\ffi(x_{12})u_t^*)u_s^*\\
u_s\sigma_s^\ffi(u_t\sigma_t^\ffi(x_{21}))&
u_s\sigma_s^\ffi(u_t\sigma_t^\ffi(x_{22})u_t^*)u_s^*\end{bmatrix} \\
&=\begin{bmatrix}\sigma_{s+t}^\ffi(x_{11})&\sigma_{s+t}^\ffi(x_{12})u_{s+t}^*\\
u_{s+t}\sigma_{s+t}^\ffi(x_{21})&u_{s+t}\sigma_{s+t}^\ffi(x_{22})u_{s+t}^*\end{bmatrix}
=\sigma_{s+t}(X).
\end{align*}
Hence $\sigma_t$ ($t\in\bR$) is a strongly* continuous one-parameter automorphism group on
$s(\theta)Ns(\theta)$. Furthermore, we write
\begin{align*}
\theta(\sigma_t(Y)X)
&=\theta\biggl(\begin{bmatrix}\sigma_t^\ffi(y_{11})&\sigma_t^\ffi(y_{12})u_t^*\\
u_t\sigma_t^\ffi(y_{21})&\sigma_t^\psi(y_{12})\end{bmatrix}
\begin{bmatrix}x_{11}&x_{12}\\x_{21}&x_{22}\end{bmatrix}\biggr) \\
&=\ffi(\sigma_t^\ffi(y_{11})x_{11})+\ffi(\sigma_t^\ffi(y_{12})u_t^*x_{21})
+\psi(u_t\sigma_t^\ffi(y_{21})x_{12})+\psi(\sigma_t^\psi(y_{22})x_{22}), \\
\theta(X\sigma_t(Y))
&=\ffi(x_{11}\sigma_t^\ffi(y_{11}))+\psi(x_{21}\sigma_t^\ffi(y_{12})u_t^*)
+\ffi(x_{21}u_t\sigma_t^\ffi(y_{21}))+\psi(x_{22}\sigma_t^\psi(y_{22})).
\end{align*}
By property (v) there are continuous bounded functions $F,G$ on $0\le\Im z\le1$, that are
analytic in $0<\Im z<1$, such that
\begin{align*}
F(t)&=\psi(u_t\sigma_t^\ffi(y_{21})x_{12}),\qquad
F(t+i)=\ffi(x_{12}u_t\sigma_t^\ffi(y_{21})), \\
G(t)&=\psi(u_t\sigma_t^\ffi(y_{12}^*)x_{21}^*),\qquad
G(t+i)=\ffi(x_{21}^*u_t\sigma_t^\ffi(y_{12}^*)).
\end{align*}
Set $\widetilde G(z):=\overline{G(\overline z+i)}$ for $0\le\Im z\le1$, which is continuous
bounded on $0\le\Im z\le1$ and analytic in $0<\Im z<1$. Then
$$
\widetilde G(t)=\overline{G(t+i)}=\ffi(\sigma_t^\ffi(y_{12}x_t^*x_{21}),\qquad
G(t+i)=\overline{G(t)}=\psi(x_{21}\sigma_t^\ffi(y_{12})u_t^*).
$$
From these boundary conditions with $F$ and $\widetilde G$ as well as the KMS conditions of
$\ffi$ for $\sigma_t^\ffi$ and of $\psi$ for $\sigma_t^\psi$, it follows that $\theta$
satisfies the KMS condition for $\sigma_t$. Therefore, Theorem \ref{T-2.14} implies that
$\sigma_t=\sigma_t^\theta$, from which
$$
u_t\otimes e_{21}=u_t\sigma_t^\ffi(s(\psi))\otimes e_{21}
=\sigma_t(s(\psi)\otimes e_{21})=\sigma_t^\theta(s(\psi)\otimes e_{21})
=(D\psi:d\ffi)_t\otimes e_{21}
$$
so that $u_t=(D\psi:D\ffi)_t$ for all $t\in\bR$.
\end{proof}

\begin{prop}\label{P-12.17}
Let $\ffi,\psi,\chi\in M_*^+$.
\begin{itemize}
\item[\rm(1)] $(D\psi:D\ffi)_t^*=(D\ffi:D\psi)_t$ for all $t\in\bR$.
\item[\rm(2)] If either $s(\psi)\le s(\ffi)$ or $s(\chi)\le s(\ffi)$, the
$(D\psi:D\ffi)_t(D\ffi:D\chi)_t=(D\psi:D\chi)_t$ for all $t\in\bR$ (\emph{chain rule}).
\item[\rm(3)] If $s(\psi)\le s(\ffi)$ and $s(\chi)\le s(\ffi)$, then
$(D\psi:D\ffi)_t=(D\chi:D\ffi)_t$ for all $t\in\bR$ if and only if $\psi=\chi$.
\item[\rm(4)] For every $\alpha\in\Aut(M)$,
$(D(\psi\circ\alpha):D(\ffi\circ\alpha))_t=\alpha^{-1}((D\psi:D\ffi)_t)$ for all $t\in\bR$.
\end{itemize}
\end{prop}

\begin{proof}
(1)\enspace
By Lemma \ref{L-12.13} one has
$$
(D\psi:D\ffi)_t^*=(h_\psi^{it}h_\ffi^{-it})^*
=h_\ffi^{it}h_\psi^{-it}=(D\ffi;D\psi)_t.
$$

(2)\enspace
Since $s(h_\psi)=s(\psi)\le s(\ffi)$ or $s(h_\chi)=s(\chi)\le s(\ffi)$, one has
$$
(D\psi:D\ffi)_t(D\ffi:D\chi)_t=h_\psi^{it}h_\ffi^{-it}h_\ffi^{it}h_\chi^{-it}
=h_\psi^{it}s(\ffi)h_\chi^{-it}=h_\psi^{it}h_\chi^{-it}=(D\psi:D\chi)_t.
$$

(3)\enspace
Since $s(\psi),s(\chi)\le s(\ffi)$, one has
\begin{align*}
(D\psi:D\ffi)_t=(D\chi:D\ffi)_t\ \mbox{for all $t\in\bR$}
&\,\iff\,h_\psi^{it}h_\ffi^{-it}=h_\chi^{it}h_\ffi^{-it}\ \mbox{for all $t\in\bR$} \\
&\,\iff\,h_\psi^{it}=h_\chi^{it}\ \mbox{for all $t\in\bR$} \\
&\,\iff\,h_\psi=h_\chi\ \iff\ \psi=\chi.
\end{align*}

(4)\enspace
Let $\theta:=\theta(\ffi,\psi)$, the balanced functional on $M^{(2)}:=M\otimes\bM_2(\bC)$.
Note that $\alpha\otimes\id_2\in\Aut(M^{(2)})$,
$\theta\circ(\alpha\otimes\id_2)=\theta(\ffi\circ\alpha,\psi\circ\alpha)$, and
$s(\theta\circ(\alpha\otimes\id_2))=(\alpha^{-1}\otimes\id_2)(s(\theta))$. Hence
$\alpha^{-1}\otimes\id_2$ is a *-isomorphism from $s(\theta)M^{(2)}s(\theta)$ onto
$s(\theta\circ(\alpha\otimes\id_2))M^{(2)}s(\theta\circ(\alpha\otimes\id_2))$. As in Lemma
\ref{L-9.2}\,(1) (or more directly by using the KMS condition with Theorem \ref{T-2.14}) we
have
$$
\sigma_t^{\theta\circ(\alpha\otimes\id_2)}
=(\alpha^{-1}\otimes\id_2)\circ\sigma_t^\theta\circ(\alpha\otimes\id_2),\qquad t\in\bR.
$$
Therefore, we have
\begin{align*}
(D(\psi\circ\alpha):D(\ffi\circ\alpha))_t\otimes e_{21}
&=\sigma_t^{\theta\circ(\alpha\otimes\id_2)}
(s(\psi\circ\alpha)s(\ffi\circ\alpha)\otimes e_{21}) \\
&=(\alpha^{-1}\otimes\id_2)(\sigma_t^\theta((\alpha\otimes\id_2)
(\alpha^{-1}(s(\psi)s(\ffi))\otimes e_{21}))) \\
&=(\alpha^{-1}\otimes\id_2)(\sigma_t^\theta(s(\psi)s(\ffi)\otimes e_{21}) \\
&=\alpha^{-1}((D\psi:D\ffi)_t),\qquad t\in\bR.
\end{align*}
\end{proof}

\begin{prop}\label{P-12.18}
Let $\ffi,\psi\in M_*^+$ with $s(\psi)\le s(\ffi)$. The following conditions are
equivalent:
\begin{itemize}
\item[\rm(i)] $\psi\circ\sigma_t^\ffi=\psi$ for all $t\in\bR$;
\item[\rm(ii)] $(D\psi:D\ffi)_t\in(s(\ffi)Ms(\ffi))^\ffi$ (the centralizer of
$\ffi|_{s(\ffi)Ms(\ffi)}$) for all $t\in\bR$;
\item[\rm(iii)] $(D\psi:D\ffi)_t\in(s(\psi)Ms(\psi))^\psi$ for all $t\in\bR$;
\item[\rm(iv)] $t\in\bR\mapsto(D\psi:D_\ffi)_t$ is a one-parameter group of unitaries in
$s(\psi)Ms(\psi)$;
\item[\rm(v)] $h_\psi^{is}h_\ffi^{it}=h_\ffi^{it}h_\psi^{is}$ for all $s,t\in\bR$;
\item[\rm(vi)] $h_\psi h_\ffi=h_\ffi h_\psi$ as elements of $\widetilde N$ (the
$\tau$-measurable operators affiliated with $N$).
\end{itemize}
\end{prop}

\begin{proof}
First, in view of \cite[Theorem VIII.13]{RS} one can easily verify that (v)\,$\iff$\,(vi) and
they are also equivalent to that all the spectral projections of $h_\psi$ and $h_\ffi$
commute.

(i)\,$\iff$\,(v).\enspace
By \eqref{F-12.7}, (i) implies that
$$
\tr(h_\ffi^{-it}h_\psi h_\ffi^{it}x)=\tr(h_\psi h_\ffi^{it}xh_\ffi^{-it})
=\tr(h_\psi x),\qquad x\in M,\ t\in\bR,
$$
so that $h_\ffi^{it}h_\psi h_\ffi^{it}=h_\psi$ and hence
$h_\ffi^{it}(s(\ffi)+h_\psi)^{-1}=(s(\ffi)+h_\psi)^{-1}h_\ffi^{it}$ in
$s(\ffi)Ms(\ffi)$ for all $t\in\bR$. This implies (v). The argument can be reversed to
see the reverse implication.

(ii)\,$\iff$\,(v).\enspace
By \eqref{F-12.7} and Lemma \ref{L-12.13}, (ii) is rewritten as
$h_\ffi^{is}(h_\psi^{it}h_\ffi^{-it})h_\ffi^{-is}=h_\psi^{it}h_\ffi^{-it}$ for all
$s,t\in\bR$, which is equivalent to (v) by multiplying $h_\ffi^{i(s+t)}$ from the right. 

(iii)\,$\iff$\,(v).\enspace
Similarly to the proof of (ii)\,$\iff$\,(v), (iii) means that
$h_\psi^{it}h_\ffi^{-it}\in s(\psi)Ms(\psi)$ and
$h_\psi^{is}(h_\psi^{it}h_\ffi^{it})h_\psi^{-is}=h_\psi^{it}h_\ffi^{it}$, or equivalently,
$s(\psi)h_\ffi^{it}h_\psi^{-is}=h_\psi^{-is}h_\ffi^{it}$ for all $s,t\in\bR$. Letting
$t=0$ in the first condition gives $s(\psi)s(\ffi)=s(\psi)s(\ffi)s(\psi)$ so that
$s(\psi)s(\ffi)=s(\ffi)s(\psi)$. Hence we see that (iii) and (v) are equivalent.

(ii)\,$\implies$\,(iv).\enspace
Let $u_t:=(D\psi:D\ffi)_t$. From Theorem \ref{T-12.16}\,(2) it follows from (ii) that
$u_{s+t}=u_s\sigma_s^\ffi(u_t)=u_su_t$ for all $s,t\in\bR$. Letting $t=0$ gives
$u_s=u_su_0=u_ss(\psi)$ by Theorem \ref{T-12.16}\,(1). Hence $u_s^*u_s=u_s^*u_ss(\psi)$ so that
$u_s^*u_s\le s(\psi)$. Since (ii) implies that
$\ffi(u_s^*u_s)=\ffi(u_su_s^*)=\ffi(s(\psi))$, we have $\ffi(s(\psi)-u_s^*u_s)=0$ and
hence $u_s^*u_s=s(\psi)=u_su_s^*$. Therefore, $t\in\bR\mapsto u_t$ is a one-parameter unitary
group in $s(\psi)Ms(\psi)$.

(iv)\,$\implies$\,(ii).\enspace
From (iv), $u_s^*u_s=u_su_s^*=u_0=s(\psi)$. For every $s,t\in\bR$, by Theorem
\ref{T-12.16}\,(1) and (2) we have
$$
\sigma_s^\ffi(u_t)=\sigma_s^\ffi(s(\psi)u_t)=u_s^*u_s\sigma_s^\ffi(u_t)
=u_s^*u_{s+t}=u_s^*u_su_t=s(\psi)u_t=u_t,
$$
which implies (ii).
\end{proof}

\begin{definition}\label{D-12.19}\rm
We say that $\psi$ \emph{commutes} with $\ffi$ if the equivalent conditions of Theorem
\ref{P-12.18} hold. Condition (i) is often used to define the commutativity for normal positive
functionals (also semifinite normal weights), but (v) and (vi) are quite natural definitions
of commutativity that might be available for any $\ffi,\psi\in M_*^+$ without
$s(\psi)\le s(\ffi)$.
\end{definition}

We finish the section with \emph{Connes' inverse theorem} in \cite[Theorem 1.2.4]{Co2}
without proof.

\begin{thm}[\cite{Co2}]\label{T-12.20}
Let $\ffi$ is a faithful semifinite normal weight on $M$. Let $t\in\bR\mapsto u_t\in M$ is
a strongly* continuous map satisfying
\begin{align}
u_{s+t}&=u_s\sigma_s^\ffi(u_t),\qquad t\in\bR, \label{F-12.24}\\
u_{-t}&=\sigma_{-t}^\ffi(u_t^*),\qquad\ \,t\in\bR.\label{F-12.25}
\end{align}
Then there exists a unique semifinite normal weight $\psi$ on $M$ such that
$u_t=(D\psi:D\ffi)_t$ for all $t\in\bR$.
\end{thm}

It is an easy exercise to show that \eqref{F-12.24} and \eqref{F-12.25} imply
\begin{align}\label{F-12.26}
\begin{cases}\mbox{$u_t$'s are partial isometries and} \\
u_tu_t^*=u_0,\quad u_t^*u_t=\sigma_t^\ffi(u_0),\qquad t\in\bR,\end{cases}
\end{align}
and conversely \eqref{F-12.24} and \eqref{F-12.26} imply \eqref{F-12.25}. When
$u_t=(D\psi:D\ffi)_t$ with $\ffi,\psi\in M_*^+$, we have shown the properties
\eqref{F-12.24}--\eqref{F-12.26} in Theorem \ref{T-12.16}\,(1)—(3).

In \cite{Co2} $u_t$'s are assumed to be unitaries in $M$. In this case, \eqref{F-12.25} is
redundant and $\psi$ is faithful as well. The above version without $u_t$'s being unitaries
is taken from \cite[Theorem 5.1]{St}. Note that even when $\ffi\in M_*^+$, $\psi$ in the
theorem cannot be in $M_*^+$ in general.\footnote{
This fact suggests us that the von Neumann algebra
theory cannot be self-completed when we stick to functionals in $M_*$, so the (semifinite)
normal weight theory is unavoidable.}


\section{Spatial derivatives}


In this section we give a concise account on the notion of spatial derivatives due to Connes
\cite{Co4}. The notion is a kind of Radon-Nikodym derivative like the relative modular
operator in Sec.~12.1. But the spatial derivative is defined for a functional in $M_*^+$ (or a
semifinite normal weight on $M$) with respect to an f.s.n.\ weight on the commutant $M'$,
unlike the relative modular operator defined for two functionals in $M_*^+$. A merit of the
spatial derivative is that it is defined in any representing Hilbert space for $M$, the reason
for the term `spatial', while the relative modular operator is given in a standard form of $M$.

Let $M$ be a von Neumann algebra on a Hilbert space $\cH$. Let $\Psi$ be a faithful semifinite
normal weight on the commutant $N:=M'$ (on $\cH$). Let $\cH_\Psi$ be the Hilbert space of
the GNS construction of $N$ associated with $\Psi$, i.e., the completion of $\fN_\Psi:=
\{y\in N:\Psi(y^*y)<\infty\}$ with respect to the inner product $\<x,y\>_\Psi:=\Psi(x^*y)$,
$x,y\in\fN_\Psi$, and $\eta_\Psi:\fN_\Psi\to\cH_\Psi$ be the canonical injection.

\begin{definition}\label{D-13.1}\rm
A vector $\xi\in\cH$ is said to be \emph{$\Psi$-bounded} if there is a constant $C_\xi<\infty$
such that $\|y\xi\|\le C_\xi\|\eta_\Psi(y)\|$ for all $y\in\fN_\Psi$. Define
$D(\cH,\Psi):=\{\xi\in\cH:\mbox{$\Psi$-bounded}\}$. For $\xi\in D(\cH,\Psi)$ a bounded operator
$R^\Psi(\xi):\cH_\Psi\to\cH$ is defined as
\begin{align}\label{F-13.1}
R^\Psi(\xi)\eta_\Psi(y):=y\xi,\qquad y\in\fN_\Psi.
\end{align}
\end{definition}

\begin{lemma}\label{L-13.2}
With the above notations the following hold:
\begin{itemize}
\item[\rm(1)] $D(\cH,\Psi)$ is an $M$-invariant dense subspace of $\cH$.
\item[\rm(2)] For every $\xi\in D(\cH,\Psi)$, $R^\Psi(\xi)$ is $N$-linear, i.e.,
$R^\Psi(\xi)\pi_\Psi(x)=xR^\Psi(\xi)$ for all $x\in N$, where $\pi_\Psi$ is the GNS
representation of $N$ on $\cH_\Psi$.
\item[\rm(3)] For every $\xi\in D(\cH,\Psi)$, $R^\Psi(\xi)R^\Psi(\xi)^*$ belongs to $N'=M$.
\item[\rm(4)] If $\xi\in D(\cH,\Psi)$ and $R^\Psi(\xi)=0$, then $\xi=0$.
\end{itemize}
\end{lemma}

\begin{proof}
(1)\enspace
It is clear that $D(\cH,\Psi)$ is an $M$-invariant subspace of $\cH$. To prove the denseness of
$D(\cH,\Psi)$, let $e$ be the orthogonal projection onto $\overline{D(\cH,\Psi)}$. Since
$exe=xe$ for all $x\in M$, one has $e\in M'=N$. Suppose that $e\ne1$; then $\Psi(1-e)>0$. Then
by Theorem \ref{T-7.2} there is an $\omega\in N_*^+$ such that $\omega\le\Psi$ and 
$\omega(1-e)>0$. Since $\Psi$ is written as $\Psi=\sum_i\omega_{\zeta_i}$ for some
$(\zeta_i)\subset\cH$ by Theorem \ref{T-7.2} (where $\omega_\zeta:=\<\zeta,\cdot\zeta\>$, an
vector functional), there is a $\zeta\in\cH$ such that $\omega_\zeta\le\Psi$ and
$(1-e)\zeta\ne0$. Since
$$
\|y\zeta\|^2=\omega_\zeta(y^*y)\le\Psi(y^*y)=\|\eta_\Psi(y)\|^2,\qquad y\in\fN_\Psi,
$$
it follows that $\zeta\in D(\cH,\Psi)$, so $(1-e)\zeta=0$, a contradiction. Hence $e=1$
follows.

(2) is immediate since for every $\xi\in D(\cH,\Psi)$ one has
$$
R^\Psi(\xi)\pi_\Psi(x)\eta_\Psi(y)=R^\Psi(\xi)\eta_\Psi(xy)=xy\xi
=xR^\Psi(\xi)\eta_\Psi(y),\qquad x\in N,\ y\in\fN_\Psi.
$$

(3)\enspace
Since (2) gives $\pi_\Psi(x)R^\Psi(\xi)^*=R^\Psi(\xi)^*x$ for all $x\in N$, one has
$$
xR^\Psi(\xi)R^\Psi(\xi)^*=R^\Psi(\xi)\pi_\Psi(x)R^\Psi(\xi)^*
=R^\Psi(\xi)R^\Psi(\xi)^*x,\qquad x\in N.
$$

(4)\enspace
If $\xi\in D(\cH,\Psi)$ and $R^\Psi(\xi)=0$, then $y\xi=0$ for all $y\in\fN_\Psi$. This gives
$\xi=0$, for $\Psi$ is semifinite, see Definition \ref{D-7.1}.
\end{proof}

The next lemma is essential to define the spatial derivative, see \cite{Co4} (also
\cite[\S7.3]{St} and \cite[Chapter III]{Te}) for details.

\begin{lemma}\label{L-13.3}
Let $\cH$ and $\Psi$ be as stated above. Let $\ffi$ be a (not necessarily faithful) semifinite
normal weight $\ffi$ on $M$. Define the function $q:D(\cH,\Psi)\to[0,\infty]$ by
$$
q(\xi):=\ffi(R^\Psi(\xi)R^\Psi(\xi)^*),\qquad\xi\in D(\cH,\Psi).
$$
Then we have:
\begin{itemize}
\item[\rm\rm(1)] $\cD(q):=\{\xi\in D(\cH,\Psi):q(\xi)<\infty\}$ is dense in $\cH$.
\item[\rm\rm(2)] $q:\cD(q)\to[0,\infty)$ is a positive quadratic form in the sense of
Definition \ref{D-A.8}\,(1).
\item[\rm\rm(3)] $q$ is lower semicontinuous on $D(\cH,\Psi)$.
\end{itemize}
\end{lemma}

\begin{proof}
(2) is easy, so we prove (1) and (3).

(1)\enspace
Since $D(\cH,\Psi)$ is dense in $\cH$ and there is a net $\{u_\gamma\}\subset\fN_\ffi$ such
that $0\le u_\gamma\nearrow1$, it suffices to prove that if $y\in\fN_\ffi^*$ and
$\xi\in D(\cH,\Psi)$, then $q(y\xi)<+\infty$. For such $y,\xi$, since
$$
R^\Psi(y\xi)R^\Psi(y\xi)^*=yR^\Psi(\xi)R^\Psi(\xi)^*y^*\le\|R^\Psi(\xi)R^\Psi(\xi)^*\|yy^*,
$$
we have $q(y\xi)\le\|R^\Psi(\xi)R^\Psi(\xi)^*\|\ffi(yy^*)<+\infty$, as desired.

(3)\enspace
Theorem \ref{T-7.2} says that $\ffi(x)=\sup\{\omega(x):\omega\in M_*^+,\,\omega\le\ffi\}$ for
$x\in M_+$. Moreover, each $\omega\in M_*^+$ is written as $\omega=\sum_i\omega_{\xi_i}$ for
some $\{\xi_i\}\subset\cH$ with $\sum_i\|\xi_i\|^2<+\infty$. So it suffices to show the result
in the case where $\ffi(x)=\<\xi_0,x\xi_0\>$, $x\in M_+$, for some $\xi_0\in\cH$. Then for
every $\xi\in D(\cH,\Psi)$, we have
$$
q(\xi)^{1/2}=\|R^\Psi(\xi)^*\xi_0\|
=\sup\{|\<R^\Psi(\xi)^*\xi_0,\eta_\Psi(y)\>|:y\in N,\,\Psi(y^*y)\le1\},
$$
$$
\<R^\Psi(\xi)^*\xi_0,\eta_\Psi(y)\>=\<\xi_0,R^\Psi(\xi)\eta_\Psi(y)\>
=\<\xi_0,y\xi\>,\qquad y\in\fN_\Psi.
$$
Since the last term is continuous in $\xi$, the result has been shown.
\end{proof}

\begin{definition}\label{D-13.4}\rm
Let $\cH$, $\Psi$ be as above. For every semifinite normal weight $\ffi$ on $M$ let
$q:\cD(q)\to[0,\infty)$ be a positive quadratic form given in Lemma \ref{L-13.3}. By Lemma
\ref{L-13.3} and Theorem \ref{T-A.11} we have the closure $\overline q$ of $q$, which is a
closed positive quadratic form that is the smallest closed extension of $q$. Then by Theorem
\ref{T-A.10} there exist a unique positive self-adjoint operator $A$ on $\cH$ such that
$\cD(A^{1/2})=\cD(\overline q)$ and
$$
\|A^{1/2}\xi\|^2=\overline q(\xi),\qquad\xi\in\cD(\overline q).
$$
This $A$ is denoted by $d\ffi/d\Psi$ and called the \emph{spatial derivative} of $\ffi$ with
respect to $\Psi$. Note (see Remark \ref{R-A.12}) that $d\ffi/d\Psi$ is the largest positive
self-adjoint operator $A$ on $\cH$ satisfying:
\begin{align}\label{F-13.2}
\mbox{for every $\xi\in D(\cH,\Psi)$},\quad
\ffi(R^\Psi(\xi)R^\Psi(\xi)^*)
=\begin{cases}\|A^{1/2}\xi\|^2 & \text{if $\xi\in\cD(A^{1/2})$}, \\
\infty & \text{otherwise}.\end{cases}
\end{align}
Also, note that $\cD(q)$ is a core of $(d\ffi/d\Psi)^{1/2}$. In particular, when
$\ffi\in M_*^+$, $D(\cH,\Psi)=\cD(q)\subset\cD((d\ffi/d\Psi)^{1/2})$.
\end{definition}

\begin{example}\label{E-13.5}\rm
Assume that $M=B(\cH)$, a type I factor, so that $N=M'=\bC1$. Consider the state
$\Psi(\lambda1)=\lambda$ on $N$. Any $\ffi\in M_*^+$ is given as $\ffi(x)=\Tr(D_\ffi x)$ with
the density (positive trace-class) operator $D_\ffi$ on $\cH$; $D_\ffi$ is also denoted by
$d\ffi/d\Tr$, the Radon-Nikodym derivative of $\ffi$ with respect to $\Tr$. It is clear that
$\cH_\Psi=\bC$ and $D(\cH,\Psi)=\cH$. For any $\xi\in\cH$, $R^\Psi(\xi)1=\xi$ and hence
$R^\Psi(\xi)R^\Psi(\xi)^*=|\xi\>\<\xi|$ (i.e., the rank one projection onto $\bC\xi$).
Therefore, we find that
$$
\ffi(R^\Psi(\xi)R^\Psi(\xi)^*)=\Tr(D_\ffi|\xi\>\<\xi|)=\|D_\ffi^{1/2}\xi\|^2,
\qquad\xi\in\cH,
$$
which means that $d\ffi/d\Psi=D_\ffi$ ($=d\ffi/d\Tr$).
\end{example}

In the special case where $M$ is represented in its standard form and $\ffi,\psi\in M_*^+$,
the next proposition says that the spacial derivative $d\ffi/d\psi'$ (where $\psi'\in(M')_*^+$
is defined by the same representing vector as $\psi$) is nothing but the relative modular
operator $\Delta_{\ffi,\psi}$. Although this is stated in \cite[p.~72]{OP} in an implicit
way without proof, it does not seem widely known. From the proposition we may consider that
spatial derivatives properly extend relative modular operators to the setting of an arbitrary
representation $M\subset B(\cH)$. 

\begin{prop}\label{P-13.4}
Let $(M,\cH,J,\cP)$ be a standard form of a von Neumann algebra $M$ and $\ffi,\psi\in M_*^+$
with $\psi$ faithful. Define a faithful $\psi'\in(M')_*^+$ by $\psi'(x'):=\psi(Jx'^*J)$ for
$x'\in M'$, i.e., $\psi'(x')=\<\xi_0,x'\xi_0\>$ with $\xi_0\in\cP$ such that
$\psi(x)=\<\xi_0,x\xi_0\>$ for $x\in M$. Then
$$
{d\ffi\over d\psi'}=\Delta_{\ffi,\psi}.
$$
\end{prop}

\begin{proof}
By Theorems \ref{T-3.13} and \ref{T-11.30} we may assume that
$$
(M,\cH,J,\cP)=(M,L^2(M),J=\,^*,L^2(M)_+),
$$
where $M$ is represented by the left
multiplication on Haagerup's $L^2$-space $L^2(M)$. We use the linear bijection
$\omega\in M_*\mapsto h_\omega\in L^1(M)$ and the functional $\tr$ on $L^1(M)$, see Definition
\ref{D-11.6}. Note that $M'=JMJ$ is identified with the opposite von Neumann algebra
$M^o=\{x^o:x\in M\}$ with the reverse product $x^oy^o=(yx)^o$, which is represented by the
right multiplication $\pi_r(x^o)\xi:=\xi x$, $\xi\in L^2(M)$. In fact, we can write
$\pi_r(x^o)=Jx^*J$ for $x\in M$, which is an isomorphism from $M^o$ onto $M'$. By this
isomorphism, $\psi'$ on $M'$ corresponds to $\psi'(x^o)=\psi(x)$ on $M^o$ (here we use the same
notation $\psi'$ on $M^o$ as $\psi'$ on $M'$). Define $\eta_{\psi'}:M^o\to L^2(M)$ by
$\eta_{\psi'}(x^o):=h_\psi^{1/2}x$, $x\in M$. Then for every $x,y\in M$ we write
\begin{align*}
\<\eta_{\psi'}(x^o),\eta_{\psi'}(y^o)\>&=\<h_\psi^{1/2}x,h_\psi^{1/2}y\>
=\tr(x^*h_\psi y)=\tr(h_\psi yx^*) \\
&=\psi(yx^*)=\psi'((yx^*)^o)=\psi'((x^o)^*y^o),
\end{align*}
$$
\eta_{\psi'}(x^oy^o)=\eta_{\psi'}((yx)^o)=h_\psi^{1/2}yx
=\eta_{\psi'}(y^o)x=\pi_r(x^o)\eta_{\psi'}(y^o).
$$
Since $\psi\in M_*^+$ and so $\psi'\in(M^o)_*^+$ are faithful, note that
$\overline{h_\psi^{1/2}M}=L^2(M)$ and $\fN_{\psi'}=M^o$. Hence $(L^2(M),\pi_r,\eta_{\psi'})$ is
identified with the GNS representation of $M^o$ associated with $\psi'$.

For every $\xi\in D(L^2(M),\psi')$ we have a bounded operator $R^{\psi'}(\xi):L^2(M)\to L^2(M)$
such that
\begin{align}\label{F-13.3}
R^{\psi'}(\xi)\eta_{\psi'}(y^o)=\pi_r(y^o)\xi=\xi y,\qquad y\in M.
\end{align}
Furthermore, for every $x,y\in M$,
\begin{align*}
R^{\psi'}(\xi)\pi_r(x^o)\eta_{\psi'}(y^o)
&=R^{\psi'}(\xi)\eta_{\psi'}((yx)^o)=\xi yx \\
&=(R^{\psi'}(\xi)\eta_{\psi'}(y^o))x=\pi_r(x^o)R^{\psi'}(\xi)\eta_{\psi'}(y^o),
\end{align*}
so that $R^{\psi'}(\xi)\in\pi_r(M^o)'=(JMJ)'=M''=M$. Hence we put $a:=R^{\psi'}(\xi)\in M$.
Since \eqref{F-13.3} is rewritten as $\xi y=ah_\psi^{1/2}y$ for all $y\in M$, we have
$\xi=ah_\psi^{1/2}$. On the other hand, if $\xi=ah_\psi^{1/2}$ with $a\in M$, then we have
$$
\|\pi_r(y^o)\xi\|_2=\|\xi y\|_2=\|ah_\psi^{1/2}y\|_2\le\|a\|\,\|\eta_{\psi'}(y^o)\|_2,
\qquad y\in M,
$$
implying that $\xi\in D(L^2(M),\psi')$. Therefore, we obtain
\begin{align}\label{F-13.4}
D(L^2(M),\psi')=Mh_\psi^{1/2}.
\end{align}

Now, for every $\xi=ah_\psi^{1/2}\in D(L^2(M),\psi')$ with $a\in M$, it follows from the
above argument that $R^{\psi'}(\xi)=a$ and so
\begin{align}\label{F-13.5}
\ffi(R^{\psi'}(\xi)R^{\psi'}(\xi)^*)=\ffi(aa^*)=\tr(h_\ffi aa^*)=\|a^*h_\ffi^{1/2}\|_2^2.
\end{align}
From the definition of the relative modular operator $\Delta_{\ffi,\psi}$ it also follows that
\eqref{F-13.4} is a core of $\Delta_{\ffi,\psi}^{1/2}$ and
\begin{align}\label{F-13.6}
\|\Delta_{\ffi,\psi}^{1/2}\xi\|_2^2=\|J\Delta_{\ffi,\psi}^{1/2}ah_\psi^{1/2}\|_2^2
=\|a^*h_\ffi^{1/2}\|_2^2.
\end{align}
By \eqref{F-13.5} and \eqref{F-13.6} we have
$$
\ffi(R^{\psi'}(\xi)R^{\psi'}(\xi)^*)=\|\Delta_{\ffi,\psi}^{1/2}\xi\|_2^2,
\qquad\xi\in D(L^2(M),\psi'),
$$
which implies due to Definition \ref{D-13.4} that $d\ffi/d\psi'=\Delta_{\ffi,\psi}$.
\end{proof}

In the rest of the section, following the approach in \cite[Chap.~III]{Te}, we extend the
notion of spatial derivatives $d\ffi/d\Psi$ to general (not necessarily semifinite) normal
weights $\ffi$ on $M$. The approach is apparently more tractable than Connes' original approach
in \cite{Co4}. The idea in \cite{Te} is to define $d\ffi/d\Psi$ as a generalized positive
operator on $\cH$, i.e., an element of $\widehat{B(\cH)}_+$, see Sec.~8.1.

We first reformulate $R^\Psi(\xi)R^\Psi(\xi)^*$ in Definition \ref{D-13.4} as an element of
$\widehat M_+$ for any $\xi\in\cH$. To do this, for each $\xi\in\cH$ we consider $R^\Psi(\xi)$
to be a densely-defined operator defined by \eqref{F-13.1} with
$\cD(R^\Psi(\xi))=\fN_\Psi\subset\cH_\Psi$. Then we have the adjoint operator $R^\Psi(\xi)^*$
whose domain $\cD(R^\Psi(\xi)^*)$ is a (not necessarily dense) subspace of $\cH$.

\begin{lemma}\label{L-13.7}
For each $\xi\in\cH$ there exists a unique element $\theta^\Psi(\xi)\in\widehat M_+$ such that,
for every $\eta\in\cH$,
\begin{align}\label{F-13.7}
\theta^\Psi(\xi)(\omega_\eta)=\begin{cases}
\|R^\Psi(\xi)^*\eta\|^2 & \text{if $\eta\in\cD(R^\Psi(\xi)^*)$}, \\
\infty & \text{otherwise},\end{cases}
\end{align}
where $\omega_\eta(x):=\<\eta,x\eta\>$, $x\in M$.
\end{lemma}

\begin{proof}
Let $\xi\in\cH$ and $p$ be the orthogonal projection from $\cH$ onto
$\overline{\cD(R^\Psi(\xi)^*)}$. In the present situation, the proof of Lemma \ref{L-13.2}\,(2)
shows that
\begin{align}\label{F-13.8}
xR^\Psi(\xi)\subset R^\Psi(\xi)\pi_\Psi(x),\qquad x\in M'\,(=N),
\end{align}
and so
\begin{align}\label{F-13.9}
\pi_\Psi(x)R^\Psi(\xi)^*\subset R^\Psi(\xi)^*x,\qquad x\in M'.
\end{align}
This implies that $x\cD(R^\Psi(\xi)^*)\subset\cD(R^\Psi(\xi)^*)$ so that $xp=pxp$ for all
$x\in M'$. Therefore, $p\in M$ follows. Consider $R^\Psi(\xi)^*$ as a closed densely-defined
operator from $p\cH$ to $\cH_\Psi$. Then $|R^\Psi(\xi)^*|^2$ can be defined as a positive
self-adjoint operator on $p\cH$. From \eqref{F-13.8} and \eqref{F-13.9} one can further see
that $u^*|R^\Psi(\xi)^*|^2u=|R^\Psi(\xi)^*|^2$ for all unitaries $u\in M'$ (an exercise). Hence
it follows that $|R^\Psi(\xi)^*|^2$ is affiliated with $pMp$. So one can take the spectral
decomposition $|R^\Psi(\xi)^*|^2=\int_0^\infty\lambda\,de_\lambda$ with $e_\lambda\in M$ and
$p=\lim_{\lambda\to\infty}e_\lambda$ so that
$$
\|R^\Psi(\xi)^*\eta\|^2=\|\,|R^\Psi(\xi)^*|\eta\|^2
=\int_0^\infty\lambda\,d\|e_\lambda\eta\|^2,\qquad\eta\in\cD(R^\Psi(\xi)^*).
$$
Then, in view of the proof of Theorem \ref{T-8.3} (also see Example \ref{E-8.2}\,(1)), an
element $\theta^\Psi(\xi)\in\widehat M_+$ is defined as
$$
\theta^\Psi(\xi)(\ffi)=\int_0^\infty\lambda\,d\ffi(e_\lambda)+\infty\ffi(1-p),
\qquad\ffi\in M_*^+,
$$
or equivalently, for every $\eta\in\cH$,
\begin{align*}
\theta^\Psi(\xi)(\omega_\eta)
&=\int_0^\infty\lambda\,d\|e_\lambda\eta\|^2+\infty\|(1-p)\eta\|^2 \\
&=\begin{cases}\|R^\Psi(\xi)^*\eta\|^2 & \text{if $\eta\in\cD(R^\Psi(\xi)^*)$}, \\
\infty &\text{otherwise}.\end{cases}
\end{align*}
The uniqueness of $\theta^\Psi(\xi)\in\widehat M_+$ as above is immediate.
\end{proof}

Note that if $\xi\in D(\cH,\Psi)$, then $\theta^\Psi(\xi)=R^\Psi(\xi)R^\Psi(\xi)^*\in M_+$.

\begin{lemma}\label{L-13.8}
For every $\xi\in\cH$ and $x\in M$,
$$
\theta^\Psi(x^*\xi)=x^*\theta^\Psi(\xi)x,
$$
where $(x^*\theta^\Psi(\xi)x)(\ffi):=\theta^\Psi(\xi)(x\ffi x^*)$ for $\ffi\in M_*^+$.
\end{lemma}

\begin{proof}
Let $\xi\in\cH$ and $x\in M$. That $R^\Psi(x^*\xi)=x^*R^\Psi(\xi)$ is immediately seen, which
implies that $R^\Psi(x^*\xi)^*=R^\Psi(\xi)^*x$. Hence for every $\eta\in\cH$, by expression
\eqref{F-13.7} we have
$$
\theta^\Psi(x^*\xi)(\omega_\eta)=\theta^\Psi(\xi)(\omega_{x\eta})
=\theta^\psi(\xi)(x\omega_\eta x^*)=(x^*\theta^\Psi(\xi)x)(\omega_\eta),
$$
implying the assertion. 
\end{proof}

\begin{lemma}\label{L-13.9}
Let $\ffi$ be a (not necessarily semifinite) normal weight on $M$. Define
$q_\ffi:\cH\to[0,\infty]$ by
$$
q_\ffi(\xi):=\ffi(\theta^\Psi(\xi)),\qquad\xi\in\cH.
$$
(See Proposition \ref{P-8.4} for this definition.) Then $q_\ffi$ is a lower semicontinuous
positive form in the sense of Definition \ref{D-A.13} of Appendix A.
\end{lemma}

\begin{proof}
We need to show the following:
\begin{itemize}
\item[(1)] $q_\ffi(\lambda\xi)=|\lambda|^2q_\ffi(\xi)$ for all $\xi\in\cH$ and $\lambda\in\bC$,
\item[(2)] $q_\ffi(\xi_1+\xi_2)+q_\ffi(\xi_1-\xi_2)=2q_\ffi(\xi_1)+2q_\ffi(\xi_2)$ for all
$\xi_1,\xi_2\in\cH$,
\item[(3)] $q_\ffi$ is lower semicontinuous on $\cH$.
\end{itemize}

(1) is easy. We first prove (2) and (3) in the case $\ffi=\omega_\eta$ with $\eta\in\cH$. For
(2) let us prove that
\begin{align}\label{F-13.10}
q_\ffi(\xi_1+\xi_2)+q_\ffi(\xi_1-\xi_2)\le2q_\ffi(\xi_1)+2q_\ffi(\xi_2).
\end{align}
If $\eta\not\in\cD(R^\Psi(\xi_1)^*)$ or $\eta\not\in\cD(R^\Psi(\xi_2)^*)$, then
$\mbox{the RHS of \eqref{F-13.10}}=\infty$, so \eqref{F-13.10} holds trivially. So assume that
$\eta\in\cD(R^\Psi(\xi_1)^*)$ and $\eta\in\cD(R^\Psi(\xi_2)^*)$. Since
$R^\Psi(\xi_1\pm\xi_2)=R^\Psi(\xi_1)\pm R^\Psi(\xi_2)$, one has
$R^\Psi(\xi_1)^*\pm R^\Psi(\xi_2)^*\subset R^\Psi(\xi_1\pm\xi_2)^*$, so that
$\eta\in\cD(R^\Psi(\xi_1\pm\xi_2)^*)$. Therefore, one has
\begin{align*}
&\|R^\Psi(\xi_1+\xi_2)^*\eta\|^2+\|R^\Psi(\xi_1-\xi_2)^*\eta\|^2 \\
&\qquad=\|R^\Psi(\xi_1)^*\eta+R^\Psi(\xi_2)^*\eta\|^2
+\|R^\Psi(\xi_1)^*\eta-R^\Psi(\xi_2)^*\eta\|^2 \\
&\qquad=2\|R^\Psi(\xi_1)^*\eta\|^2+2\|R^\Psi(\xi_1)^*\eta\|^2,
\end{align*}
which is \eqref{F-13.10}. Replacing $\xi_1,\xi_2$ with $\xi_1+\xi_2,\xi_1-\xi_2$ in
\eqref{F-13.10} we also have
$$
q_\ffi(2\xi_1)+q_\ffi(2\xi_2)\le2q_\ffi(\xi_1+\xi_2)+q_\ffi(\xi_1-\xi_2),
$$
which is the reverse inequality of \eqref{F-13.10}. Thus (2) has been shown.

Next, from \eqref{F-13.7} it follows that, for every $\xi\in\cH$,
\begin{align*}
q_\ffi(\xi)=\theta^\Psi(\xi)(\omega_\eta)
&=\sup\{|\<\eta,R^\Psi(\xi)\eta_\Psi(y)\>|^2:y\in\fN_\Psi,\,\|\eta_\Psi(y)\|\le1\} \\
&=\sup\{|\<\eta,y\xi\>|^2:y\in\fN_\Psi,\,\|\eta_\Psi(y)\|\le1\}.
\end{align*}
Since $\xi\in\cH\mapsto|\<\eta,y\xi\>|^2$ is continuous, (3) follows in the case
$\ffi=\omega_\eta$.

Now, let $\ffi$ be an arbitrary normal weight on $M$. By Theorem \ref{T-7.2} we can write
$\ffi=\sum_i\omega_{\eta_i}$ for some $(\eta_i)\subset\cH$. Since
$$
q_\ffi(\xi)=\ffi(\theta^\Psi(\xi))=\sum_i\omega_{\eta_i}(\theta^\Psi(\xi)),
\qquad\xi\in\cH,
$$
the properties (2) and (3) follow from those in the case $\ffi=\omega_\eta$ proved above.
\end{proof}

\begin{definition}\label{D-13.10}\rm
Let $\cH,\Psi$ be as above. For every normal weight $\ffi$ on $M$ define
$q_\ffi:\cH\to[0,\infty]$ as in Lemma \ref{L-13.9}. Then by Lemma \ref{L-13.9} and Theorem
\ref{T-A.15} there exists a unique positive self-adjoint operator $A$ on
$\cK:=\overline{\cD(q_\ffi)}$, where $\cD(q_\ffi):=\{\xi\in\cH:q_\ffi(\xi)<\infty\}$, such that
$$
q_\ffi(\xi)=\begin{cases}\|A^{1/2}\xi\|^2 & \text{if $\xi\in\cD(A^{1/2})$}, \\
\infty & \text{otherwise}.\end{cases}
$$
We denote $A$ by $d\ffi/d\Psi$ and call it the \emph{spatial derivative} of $\ffi$ with respect
to $\Psi$. Taking the spectral decomposition $A=\int_0^\infty\lambda\,dE_\lambda$ with
$P=\lim_{\lambda\to\infty}E_\lambda$, the orthogonal projection onto $\cK$, and setting $\infty$
on $\cK^\perp$, we can consider $d\ffi/d\Psi$ as an element of $\widehat{B(\cH)}_+$ as
$$
{d\ffi\over d\Psi}(\omega_\xi)
=\int_0^\infty\lambda\,d\|E_\lambda\xi\|^2+\infty\omega_\xi(1-P),\qquad\xi\in\cH,
$$
see Theorem \ref{T-8.3}. Thus, $d\ffi/d\Psi$ is uniquely determined as an element of
$\widehat{B(\cH)}_+$ such that
$$
{d\ffi\over d\Psi}(\omega_\xi)=\ffi(\theta^\Psi(\xi)),\qquad\xi\in\cH.
$$
\end{definition}

\begin{remark}\label{R-13.11}\rm
Assume that $\ffi$ is a semifinite normal weight on $M$. Then for every $x\in\fN_\Psi$ and
$\xi\in D(\cH,\Psi)$ we have
\begin{align*}
q_\ffi(x^*\xi)&=\ffi(\theta^\Psi(x^*\xi))
=\ffi(x^*\theta^\Psi(\xi)x)\quad\mbox{(by Lemma \ref{L-13.8})} \\
&\le\|\theta^\Psi(\xi)\|\ffi(x^*x)<\infty.
\end{align*}
Since $\{x^*\xi:x\in\fN_\Psi,\,\xi\in D(\cH,\Psi)\}$ is dense in $\cH$, it follows that
$\cD(q_\ffi)$ is dense in $\cH$ so that $d\ffi/d\Psi$ is a positive self-adjoint operator on
$\cH$. Indeed, the next theorem holds true.
\end{remark}

\begin{thm}\label{T-13.12}
Assume that $\ffi$ is a semifinite normal weight on $M$. Then both $d\ffi/d\Psi$ in Definitions
\ref{D-13.4} and \ref{D-13.10} coincide
\end{thm}

To prove the theorem, we need a lemma.

\begin{lemma}\label{L-13.13}
For every $\xi\in\cH$ there exists a sequence $\xi_n\in D(\cH,\Psi)$ such that $\xi_n\to\xi$
and $q_\ffi(\xi_n)\to q_\ffi(\xi)$ as $n\to\infty$ for all normal weights $\ffi$ on $M$.
\end{lemma}

\begin{proof}
As in the proof of Lemma \ref{L-13.7} write the spectral resolution of $\theta^\Psi(\xi)$ as
$\theta^\Psi(\xi)=\int_0^\infty\lambda\,de_\lambda+\infty(1-p)$, where $p$ is the orthogonal
projection onto $\overline{\cD(R^\Psi(\xi)^*)}$. For each $n\in\bN$, since
$R^\Psi(e_n\xi)^*=R^\Psi(\xi)^*e_n$ (see the proof of Lemma \ref{L-13.8}), note that
$R^\Psi(e_n\xi)^*$ is bounded, i.e. $e_n\xi\in D(\cH,\Psi)$. By Lemma \ref{L-13.2}\,(1) choose
a sequence $\zeta_n\in D(\cH,\Psi)$ such that $\zeta_n\to\xi$. Then one has
$(1-p)\zeta_n\in D(\cH,\Psi)$ by Lemma \ref{L-13.2}\,(1). For each $n\in\bN$ define
$$
\xi_n:=e_n\xi+(1-p)\zeta_n\in D(\cH,\Psi);
$$
then $\xi_n\to p\xi+(1-p)\xi=\xi$ as $n\to\infty$.

Let $\ffi$ be a normal weight on $M$, and prove that $q_\ffi(\xi_n)\to q_\ffi(\xi)$. If
$q_\ffi(\xi)=\infty$, then this is clear by the lower semicontinuity of $q_\ffi$ (Lemma
\ref{L-13.9}). Now, assume that $q_\ffi(\xi)<\infty$. By Theorem \ref{T-7.2} write
$\ffi=\sum_i\omega_{\eta_i}$ for some $(\eta_i)\subset\cH$. Since
$q_\ffi(\xi)=\sum_i\theta^\Psi(\xi)(\omega_{\eta_i})<\infty$, one has
$\theta^\Psi(\xi)(\omega_{\eta_i})<\infty$ and so $\eta_i\in p\cH$ for all $i$. This implies
that $p\ffi p=\sum_ip\omega_{\eta_i}p=\sum_i\omega_{\eta_i}=\ffi$. Furthermore, by Lemma
\ref{L-13.8},
$$
p\theta^\Psi(\xi_n)p=\theta^\Psi(p\xi_n)=\theta^\Psi(e_n\xi)
=e_n\theta^\psi(\xi)e_n\nearrow p\theta^\Psi(\xi)p.
$$
Therefore, one has
$$
q_\ffi(\xi_n)=\ffi(\theta^\Psi(\xi_n))=\ffi(p\theta^\Psi(\xi_n)p)
\nearrow\ffi(p\theta^\Psi(\xi)p)=\ffi(\theta^\Psi(\xi))=q_\ffi(\xi).
$$
\end{proof}

\begin{proof}[Proof of Theorem \ref{T-13.12}]
Let $d\psi/d\Psi$ denote the spatial derivative in Definition \ref{D-13.10}, which is a
positive self-adjoint operator as noted in Remark \ref{R-13.11}. In view of the description in
Definition \ref{D-13.4}, it suffices to show that $d\ffi/d\Psi$ is the largest positive
self-adjoint operator on $\cH$ satisfying \eqref{F-13.2}. It is clear that $d\ffi/d\Psi$
satisfies \eqref{F-13.2}.

Now, let $A$ be any positive self-adjoint operator on $\cH$ satisfying \eqref{F-13.2}. Let us
prove that $A\le d\ffi/d\Psi$ (in the sense of Definition \ref{D-A.2}). By Lemma \ref{L-13.13},
for every $\xi\in\cH$ there exists a sequence $\xi_n\in D(\cH,\Psi)$ such that $\xi_n\to\xi$
and $q_\ffi(\xi_n)\to q_\ffi(\xi)$. Since the function $q_A:\cH\to[0,\infty]$ defined by
$$
q_A(\xi):=\begin{cases}\|A^{1/2}\xi\|^2 & \text{if $\xi\in\cD(A^{1/2})$}, \\
\infty & \text{otherwise},\end{cases}
$$
is lower semi-continuous (see Theorem \ref{T-A.15}), we find that
$$
q_A(\xi)\le\liminf_{n\to\infty}q_\ffi(\xi_n)=q_\ffi(\xi)
=\bigg\|\biggl({d\ffi\over d\Psi}\biggr)^{1/2}\xi\bigg\|^2,
$$
whenever $\xi\in\cD((d\ffi/d\Psi)^{1/2})$. This implies that $A\le d\ffi/d\Psi$.
\end{proof}

The following are basic properties of $d\ffi/d\Psi$.

\begin{prop}\label{P-13.14}
Let $\ffi$, $\ffi_1$ and $\ffi_2$ be normal weights on $M$ and $a\in M$. Then the following
hold as elements of $\widehat{B(\cH)}_+$:
\begin{itemize}
\item[\rm(1)] $\dis{d(\ffi_1+\ffi_2)\over d\Psi}={d\ffi_1\over d\Psi}+{d\ffi_2\over d\Psi}$.
\item[\rm(2)] $\dis{d(a\ffi a^*)\over d\Psi}=a\biggl({d\ffi\over d\Psi}\biggr)a^*$.
\item[\rm(3)] If $\ffi_1\le\ffi_2$, then $d\ffi_1/d\Psi\le d\ffi_2/d\Psi$.
\item[\rm(4)] Let $(\ffi_i)$ be an increasing net of normal weights on $M$. If
$\ffi_i\nearrow\ffi$, then $d\ffi_i/d\Psi\nearrow d\ffi/d\Psi$.
\end{itemize}
\end{prop}

\begin{proof}
The assertions in (1)--(3) follow from the following computations for every $\xi\in\cH$.

(1)\enspace
\begin{align*}
{d(\ffi_1+\ffi_2)\over d\Psi}(\omega_\xi)&=(\ffi_1+\ffi_2)(\theta^\Psi(\xi))
=\ffi_1(\theta^\Psi(\xi))+\ffi_2(\theta^\Psi(\xi)) \\
&={d\ffi_1\over d\Psi}(\omega_\xi)+{d\ffi_2\over d\Psi}(\omega_\xi)
=\biggl({d\ffi_1\over d\Psi}+{d\ffi_2\over d\Psi}\biggr)(\omega_\xi).
\end{align*}

(2)\enspace
\begin{align*}
{d(a\ffi a^*)\over d\Psi}(\omega_\xi)&=\ffi(a^*\theta^\Psi(\xi)a)
=\ffi(\theta^\Psi(a^*\xi))\quad\mbox{(by Lemma \ref{L-13.8})} \\
&={d\ffi\over d\Psi}(\omega_{a^*\xi})={d\ffi\over d\Psi}(a^*\omega_\xi a)
=\biggl(a\biggl({d\ffi\over d\Psi}\biggr)a^*\biggr)(\omega_\xi).
\end{align*}

(3)\enspace
$$
{d\ffi_1\over d\Psi}(\omega_\xi)=\ffi_1(\theta^\Psi(\xi))\le\ffi_2(\theta^\Psi(\xi))
={d\ffi_2\over d\Psi}(\omega_\xi).
$$

(4)\enspace
That $d\ffi_i/d\Psi\nearrow$ follows from (3). For every $\xi\in\cH$ we have
$$
{d\ffi\over d\Psi}(\omega_\xi)=\ffi(\theta^\Psi(\xi))
=\sup_i\ffi_i(\theta^\Psi(\xi))=\sup_i{d\ffi_i\over d\Psi}(\omega_\xi),
$$
where the above second equality is seen as follows: With the spectral resolution
$\theta^\Psi(\xi)=\int_0^\infty\lambda\,de_\lambda+\infty p$ in Theorem \ref{T-8.3} we have
\begin{align*}
\ffi(\theta^\Psi(\xi))
&=\sup_n\ffi\biggl(\int_0^n\lambda\,d_\lambda\biggr)+\infty\ffi(p) \\
&=\sup_n\sup_i\biggl[\ffi_i\biggl(\int_0^n\lambda\,de_\lambda\biggr)
+\infty\ffi_i(p)\biggr] \\
&=\sup_i\sup_n\biggl[\ffi_i\biggl(\int_0^n\lambda\,de_\lambda\biggr)
+\infty\ffi_i(p)\biggr]
=\sup_i\ffi_i(\theta^\Psi(\xi)).
\end{align*}
\end{proof}

\begin{remark}\label{R-13.15}\rm
In the above (1) the sum $d\ffi_1/d\Psi+d\ffi_2/d\Psi$ is equivalently considered as the form
sum in Example \ref{E-A.16}. When $\ffi_1$ and $\ffi_2$ are semifinite in (3),
$d\ffi_1/d\Psi\le d\ffi_2/d\Psi$ is equivalent to the inequality as positive self-adjoint
operators in Definition \ref{D-A.2}. When $\ffi$ is semifinite (hence so are all $\ffi_i$),
$d\ffi_i/d\Psi\nearrow d\ffi/d\Psi$ is equivalent to the convergence in the strong resolvent
sense in Definition \ref{D-A.6}.
\end{remark}

\begin{prop}\label{P-13.16}
Let $\ffi$ be a semifinite normal weight on $M$. Then the support projection of $d\ffi/d\Psi$
coincides with $s(\ffi)$.
\end{prop}

\begin{proof}
Let $e=s(\ffi)\in M$. For each $\xi\in\cH$ note that $\xi$ is in the kernel of $d\ffi/d\Psi$
if and only if $(d\psi/d\Psi)(\omega_\xi)=0$, i.e., $\ffi(\theta^\Psi(\xi))=0$. Write
$\theta^\Psi(\xi)=\int_0^\infty\lambda\,de_\lambda+\infty(1-p)$ as in the proof of Lemma
\ref{L-13.7}; then we find that
\begin{align*}
&\ffi(\theta^\Psi(\xi))
=\sup_n\ffi\biggl(\int_0^n\lambda\,de_\lambda\biggr)+\infty\ffi(1-p)=0 \\
&\quad\iff\,\ffi\biggl(\int_0^n\lambda\,de_\lambda\biggr)=0\ \ (n\in\bN)
\quad\mbox{and}\quad\ffi(1-p)=0 \\
&\quad\iff\,e\biggl(\int_0^n\lambda\,de_\lambda\biggr)e=0\ \ (n\in\bN)
\quad\mbox{and}\quad e(1-p)e=0 \\
&\quad\iff\,e\theta^\Psi(\xi)e=0
\,\iff\,\theta^\Psi(e\xi)=0\quad\mbox{(by Lemma \ref{L-13.8})} \\
&\quad\iff\,R^\Psi(e\xi)=0\,\iff\,e\xi=0.
\end{align*}
The above last equivalence is seen similarly to Lemma \ref{L-13.2}\,(4). Hence the assertion
follows.
\end{proof}

The following are the most significant properties of the spatial derivative $d\ffi/d\Psi$ given
in \cite{Co4} in connection with the modular automorphism group and Connes' cocycle derivative,
which we record here without proof.

\begin{thm}\label{T-13.17}
Let $\ffi$, $\ffi_1$ and $\ffi_2$ be semifinite normal weights on $M$, and assume that $\ffi$
and $\ffi_2$ are faithful. Then:
\begin{itemize}
\item[\rm\rm(1)] $\dis\sigma_t^\ffi(x)=\biggl({d\ffi\over d\Psi}\biggr)^{it}x
\biggl({d\ffi\over d\Psi}\biggr)^{-it}$ for all $x\in M$ and $t\in\bR$.
\item[\rm\rm(2)] $\dis\sigma_t^\Psi(y)=\biggl({d\ffi\over d\Psi}\biggr)^{-it}y
\biggl({d\ffi\over d\Psi}\biggr)^{it}$ for all $y\in N=M'$ and $t\in\bR$.
\item[\rm\rm(3)] $(D\ffi_1:D\ffi_2)_t=\dis\biggl({d\ffi_1\over d\Psi}\biggr)^{it}
\biggl({d\ffi_2\over d\Psi}\biggr)^{-it}$ for all $t\in\bR$, where $(D\ffi_1:D\ffi_2)_t$ is
Connes' cocycle derivative and $(d\ffi_1/d\Psi)^{it}$ is defined with restriction to the
support.
\item[\rm\rm(4)] $\dis{d\Psi\over d\ffi}=\biggl({d\ffi\over d\Psi}\biggr)^{-1}$.
\end{itemize}
\end{thm}

\begin{remark}\label{R-13.18}\rm
The expression in the above (3) has a resemblance to \eqref{F-12.22}. That was originally shown
in \cite{Co4} when both $\ffi_1,\ffi_2$ are faithful. The case where $\ffi_1$ is not necessarily
faithful is given in \cite[\S7]{St}; however the proof there does not seem transparent. Below
we give, for completeness, its more explicit proof based on the case of faithful $\ffi_1,\ffi_2$
and Propositions \ref{P-13.14} and \ref{P-13.16}.

\begin{proof}
Let $e:=s(\ffi_1)$. Choose a semifinite normal weight $\ffi_0$ on $M$ with $s(\ffi_0)=1-e$,
and take a faithful $\ffi:=\ffi_1+\ffi_0$. By (3) for faithful weights in \cite{Co4} we have
$$
(D\ffi:D\ffi_2)_t=\biggl({d\ffi\over d\Psi}\biggr)^{it}
\biggl({d\ffi_2\over d\Psi}\biggr)^{-it},\qquad t\in\bR.
$$
From Propositions \ref{P-13.14}\,(1) and \ref{P-13.16} it follows that
$$
e\biggl({d\ffi\over d\Psi}\biggr)^{it}=\biggl({d\ffi_1\over d\Psi}\biggr)^{it},
\qquad t\in\bR.
$$
So we may prove prove that
$$
e(D\ffi:D\ffi_2)_t=(D\ffi_1:D\ffi_2)_t,\qquad t\in\bR.
$$
To do this, let $\theta=\theta(\ffi_2,\ffi)$ and $\theta_1=\theta(\ffi_2,\ffi_1)$ be the
balanced weights on $\bM_2(M)=M\otimes\bM_2(\bC)$, see \eqref{F-7.1}. Note that
$s(\theta_1)=\begin{bmatrix}1&0\\0&e\end{bmatrix}$.
Let $E_0:=\begin{bmatrix}0&0\\0&1-e\end{bmatrix}$. For every
$X=[x_{ij}]\in\fM_\theta=\fN_\theta^*\fN_\theta$, i.e.,
$x_{11}\in\fM_{\ffi_2}$, $x_{12}\in\fN_{\ffi_2}^*\fN_\ffi$,
$x_{21}\in\fN_\ffi^*\fN_{\ffi_2}$, $x_{22}\in\fM_\ffi$ (Lemma \ref{L-7.4}\,(3)), we have
\begin{align*}
\theta(E_0X)
&=\theta\biggl(\begin{bmatrix}0&0\\(1-e)x_{21}&(1-e)x_{22}\end{bmatrix}\biggr) \\
&=\ffi((1-e)x_{22})=\ffi_0(x_{22})=\ffi(x_{22}(1-e)) \\
&=\theta\biggl(\begin{bmatrix}0&x_{12}(1-e)\\0&x_{22}(1-e)\end{bmatrix}\biggr)
=\theta(XE_0),
\end{align*}
so that $E_0$ and hence $1-E_0=\begin{bmatrix}1&0\\0&e\end{bmatrix}$ belong to the centralizer
of $\theta$. Therefore, from \cite[Theorem VIII.2.6]{Ta3} (see Proposition \ref{P-2.16} for a
faithful normal functional) it follows that
$$
\sigma_t^\theta\biggl(\begin{bmatrix}1&0\\0&e\end{bmatrix}\biggr)
=\begin{bmatrix}1&0\\0&e\end{bmatrix},\qquad t\in\bR,
$$
which further implies that $\sigma_t^{\theta_1}(X)=\sigma_t^\theta(X)$ for all
$X\in s(\theta_1)\bM_2(M)s(\theta_1)$ and all $t\in\bR$. We hence have
\begin{align*}
\begin{bmatrix}0&0\\ (D\ffi_1:D\ffi_2)_t &0\end{bmatrix}
&=\sigma_t^{\theta_1}\biggl(\begin{bmatrix}0&0\\e&0\end{bmatrix}\biggr)
=\sigma_t^{\theta_1}\biggl(\begin{bmatrix}1&0\\0&e\end{bmatrix}
\begin{bmatrix}0&0\\1&0\end{bmatrix}\begin{bmatrix}1&0\\0&e\end{bmatrix}\biggr) \\
&=\sigma_t^\theta\biggl(\begin{bmatrix}1&0\\0&e\end{bmatrix}
\begin{bmatrix}0&0\\1&0\end{bmatrix}\begin{bmatrix}1&0\\0&e\end{bmatrix}\biggr)
=\begin{bmatrix}1&0\\0&e\end{bmatrix}\sigma_t^\theta\biggl(
\begin{bmatrix}0&0\\1&0\end{bmatrix}\biggr)\begin{bmatrix}1&0\\0&e\end{bmatrix} \\
&=\begin{bmatrix}1&0\\0&e\end{bmatrix}
\begin{bmatrix}0&0\\ (D\ffi:d\ffi_2)_t &0\end{bmatrix}\begin{bmatrix}1&0\\0&e\end{bmatrix}
=\begin{bmatrix}0&0\\e(D\ffi:D\ffi_2)_t&0\end{bmatrix}.
\end{align*}
\end{proof}
\end{remark}


\appendix
\section{Positive self-adjoint operators and positive quadratic forms}

In this appendix we summarize basic facts on positive self-adjoint operators and positive
quadratic forms, some of which are effectively used in the main body of the lecture notes.

\subsection{Positive self-adjoint operators}

Let $\cH$ be a Hilbert space. Let $A$ be a \emph{positive self-adjoint} operator on $\cH$,
which is represented in the spectral decomposition
$$
A=\int_0^\infty\lambda\,dE_\lambda
$$
with a unique spectral measure $dE_\lambda$ on $[0,\infty)$. The \emph{domain} of $A$ is given
as
$$
\cD(A)=\Bigl\{\xi\in\cH:\int_0^\infty\lambda^2\,d\|E_\lambda\xi\|^2<+\infty\Bigr\},
$$
and for $\xi\in\cD(A)$ we have
$$
A\xi=\int_0^\infty\lambda\,dE_\lambda\xi=\lim_{n\to\infty}\int_0^n\lambda\,dE_\lambda\xi
\ \ \mbox{(strongly)},
\qquad\|A\xi\|^2=\int_0^\infty\lambda^2\,d\|E_\lambda\xi\|^2.
$$
For a complex-valued Borel function $f$ on $[0,\infty)$, the \emph{Borel functional calculus}
$f(A)$ is defined by
$$
f(A):=\int_0^\infty f(\lambda)\,dE_\lambda,
$$
whose domain is
$$
\cD(f(A))=\Bigl\{\xi\in\cH:\int_0^\infty|f(\lambda)|^2\,d\|E_\lambda\xi\|^2<+\infty\Bigr\}.
$$
The operator $f(A)$ is self-adjoint if $f$ is real-valued, and bounded if $f$ is bounded.
For instance,
$$
A^{1/2}=\int_0^\infty\lambda^{1/2}\,dE_\lambda
$$
and
$$
(\alpha1+A)^{-1}=\int_0^\infty{1\over\alpha+\lambda}\,dE_\lambda,\quad
A(1+\alpha A)^{-1}=\int_0^\infty{t\over1+\alpha\lambda}\,dE_\lambda\qquad
\mbox{for $\alpha>0$}.
$$

The next lemma summarizes equivalent conditions on order between two positive self-adjoint
operators.

\begin{lemma}\label{L-A.1}
Let $A,B$ be positive self-adjoint operators on $\cH$. Then the following conditions are
equivalent:
\begin{itemize}
\item[\rm(i)] $\cD(B^{1/2})\subset\cD(A^{1/2})$ and $\|A^{1/2}\xi\|^2\le\|B^{1/2}\|^2$ for
all $\xi\in\cD(B)^{1/2}$;
\item[\rm(ii)] there exists a core $\cD$ of $B^{1/2}$ such that $\cD\subset\cD(A^{1/2})$ and
$\|A^{1/2}\xi\|^2\le\|B^{1/2}\xi\|^2$ for all $\xi\in\cD$;
\item[\rm(iii)] $(\alpha1+B)^{-1}\le(\alpha1+A)^{-1}$ for all $\alpha>0$;
\item[\rm(iv)] $(1+B)^{-1}\le(1+A)^{-1}$;
\item[\rm(v)] $A(1+\alpha A)^{-1}\le B(1+\alpha B)^{-1}$ for some (equivalently, any)
$\alpha>0$.
\end{itemize}
\end{lemma}

\begin{proof}
(i)\,$\implies$\,(ii) and (iii)\,$\implies$\,(iv) are trivial.

(ii)\,$\implies$\,(i).\enspace
Let $\cD_0$ be a core of $B^{1/2}$ for which condition (ii) holds. For any
$\xi\in\cD(B^{1/2})$ there exists a sequence $\{\xi_n\}$ in $\cD_0$ such that $\xi_n\to\xi$
and $B^{1/2}\xi_n\to B^{1/2}\xi$. Since
$$
\|A^{1/2}\xi_n-A^{1/2}\xi_m\|=\|A^{1/2}(\xi_n-\xi_m)\|\le\|B^{1/2}(\xi_n-\xi_m)\|
\ \longrightarrow\ 0\quad\mbox{as $n,m\to\infty$},
$$
it follows that $A^{1/2}\xi_n$ converges to some $\eta\in\cH$, and hence we have
$\xi\in\cD(A^{1/2})$ and $A^{1/2}\xi=\eta$. Therefore,
$$
\|A^{1/2}\xi\|=\lim_n\|A^{1/2}\xi_n\|\le\lim_n\|B^{1/2}\xi_n\|=\|B^{1/2}\xi\|.
$$

(i)\,$\implies$\,(iii).\enspace
Since (i) implies the same condition for $\alpha^{-1}A$ and $\alpha^{-1}B$ for any
$\alpha>0$, we may prove the case $\alpha=1$ (i.e., (iv)). Since
$$
\|(1+A)^{1/2}\xi\|^2=\|\xi\|^2+\|A^{1/2}\xi\|^2\le\|\xi\|^2+\|B^{1/2}\xi\|^2
=\|(1+B)^{1/2}\xi\|^2
$$
for all $\xi\in\cD(B^{1/2})$, there is a contraction $C$ on $\cH$ such that
$$
(1+A)^{1/2}\xi=C(1+B)^{1/2}\xi,\qquad\xi\in\cD(B^{1/2}).
$$
For any $\eta\in\cH$ let $\xi:=(1+B)^{-1/2}\eta\in\cD((1+B)^{1/2})=\cD(B^{1/2})$. Then,
since $\eta=(1+B)^{1/2}\xi$, we have
$$
(1+A)^{-1/2}C\eta=(1+A)^{-1/2}C(1+B)^{1/2}\xi=\xi=(1+B)^{-1/2}\eta,
$$
so that $(1+A)^{-1/2}C=(1+B)^{-1/2}$. Therefore,
\begin{align*}
(1+B)^{-1}&=\bigl[(1+A)^{-1/2}C\bigr]\bigl[(1+A)^{-1}C\bigr]^* \\
&=(1+A)^{-1/2}CC^*(1+A)^{-1/2}\le(1+A)^{-1}.
\end{align*}

(iv)\,$\implies$\,(i).\enspace
Since (iv) means that $\|(1+B)^{-1/2}\xi\|\le\|(1+A)^{-1/2}\xi\|$ for all $\xi\in\cH$, there
is a contraction $C$ on $\cH$ such that $(1+B)^{-1/2}=C(1+A)^{-1/2}=(1+A)^{-1/2}C^*$. Since
$\cD(B^{1/2})=\cD((1+B)^{1/2})=\cR((1+B)^{-1/2})$, the range of $(1+B)^{-1/2}$, for any
$\xi\in\cD(B^{1/2})$ there is an $\eta\in\cH$ such that
$\xi=(1+B)^{-1/2}\eta=(1+A)^{-1/2}C^*\eta$. Then $\xi\in\cD((1+A)^{1/2})=\cD(A^{1/2})$ and
$(1+A)^{1/2}\xi=C^*\eta$. Therefore,
\begin{align*}
\|\xi\|^2+\|A^{1/2}\xi\|^2&=\|(1+A)^{1/2}\xi\|^2=\|C^*\eta\|^2\le\|\eta\|^2 \\
&=\|(1+B)^{1/2}\xi\|^2=\|\xi\|^2+\|B^{1/2}\xi\|^2,
\end{align*}
so that $\|A^{1/2}\|^2\le\|B^{1/2}\xi\|^2$.

(iv)\,$\iff$\,(v).\enspace
Since $t(1+\alpha t)^{-1}=\alpha^{-1}-\alpha^{-2}(\alpha^{-1}+t)^{-1}$, one has
$$
A(1+\alpha A)^{-1}=\alpha^{-1}1-\alpha^{-2}(\alpha^{-1}1+A)^{-1},\qquad\alpha>0,
$$
from which (iv)\,$\iff$\,(v) is seen immediately.
\end{proof}

\begin{definition}\label{D-A.2}\rm
For positive self-adjoint operators $A,B$ we write $A\le B$ if the equivalent conditions
in Lemma \ref{L-A.1} hold.
\end{definition}

By Lemma \ref{L-A.1} and Definition \ref{D-A.2} we have the following:

\begin{prop}\label{P-A.3}
For densely-defined closed operators $A,B$ on $\cH$ the following are equivalent:
\begin{itemize}
\item[\rm(i)] $A^*A\le B^*B$;
\item[\rm(ii)] there exists a core $\cD$ of $B$ such that $\cD\subset\cD(A)$ and
$\|A\xi\|\le\|B\xi\|$ for all $\xi\in\cD$.
\end{itemize}
\end{prop}

The next lemma gives equivalent conditions on convergence of positive self-adjoint operators.

\begin{lemma}\label{L-A.4}
Let $A_n$ ($n\in\bN$) and $A$ be positive self-adjoint operators on $\cH$. Then the following
conditions are equivalent:
\begin{itemize}
\item[\rm(i)] $(1+A_n)^{-1}\to(1+A)^{-1}$ strongly;
\item[\rm(ii)] $(i+A_n)^{-1}\to(i+A)^{-1}$ strongly;
\item[\rm(iii)] $A_n(1+\alpha A_n)^{-1}\to A(1+\alpha A)^{-1}$ strongly for some
(equivalently, any) $\alpha>0$.
\end{itemize}
Moreover, assume that all $A_n$ and $A$ are non-singular, so we write $A_n=e^{H_n}$ and
$A=e^H$, where $H_n:=\log A_n$ and $H:=\log A$. Then the above conditions are also equivalent
to the following:
\begin{itemize}
\item[\rm(iv)] $(i+H_n)^{-1}\to(i+H)^{-1}$ strongly;
\item[\rm(v)] $A_n^{it}\to A^{it}$ strongly for all $t\in\bR$.
\end{itemize}
\end{lemma}

To prove the above lemma, we state the next result due to Kadison \cite{Kad} without proof.
For details, see \cite{Kad} or \cite[Theorem A.2]{St}.

\begin{lemma}[Kadison]\label{L-A.5}
Let $f$ be a complex-valued continuous function on a subset $S\subset\bC$ such that
$\overline{(\overline{S}\setminus S)}\cap S=\emptyset$ and
$$
\sup_{z\in S}{|f(z)|\over1+|z|}<+\infty.
$$
Then $f$ is strong-operator continuous, i.e., the functional calculus $f(A)$ from
$\{A\in B(\cH):\mbox{normal},\,\sigma(A)\subset S\}$ to $B(\cH)$ is strongly continuous.
\end{lemma}

\begin{proof}[Proof of Lemma \ref{L-A.4}]
(i)\,$\iff$\,(ii).\enspace
Consider a homeomorphism
$$
\phi_1:[0,1]=\Bigl\{{1\over1+t}:0\le t\le\infty\Bigr\}\,\longrightarrow\,
\Bigl\{{1\over i+t}:0\le t\le\infty\Bigr\},
$$
$$
\phi_1\Bigl({1\over1+t}\Bigr):={1\over i+t}\quad
\mbox{(where $\phi_1(0):=0$ for $t=\infty$)}.
$$
Since $(i+A)^{-1}=\phi_1((1+A)^{-1})$ and $(1+A)^{-1}=\phi_1^{-1}((i+A)^{-1})$, the result is
seen from Lemma \ref{L-A.5}.

(i)\,$\iff$\,(iii).\enspace
For each $\alpha>0$ consider a homeomorphism
$$
\phi_2:[0,1]\,\longrightarrow\,
[0,\alpha^{-1}]=\Bigl\{{t\over1+\alpha t}:0\le t\le\infty\Bigr\},
$$
$$
\phi_2\Bigl({1\over1+t}\Bigr):={t\over1+\alpha t}\quad
\mbox{(where $\phi_2(0):=\alpha^{-1}$ for $t=\infty$)}.
$$
Since $A(1+\alpha A)^{-1}=\phi_2((1+A)^{-1})$ and
$(1+A)^{-1}=\phi_2^{-1}(A(1+\alpha A)^{-1})$, the result follows from Lemma \ref{L-A.5}.

Next, assume that all $A_n$ and $A$ are non-singular. Let $H_n:=\log A_n$ and $H:=\log A$.

(i)\,$\implies$\,(iv) follows as above by considering a continuous function
$$
\phi_3:[0,1]=\Bigl\{{1\over1+t}:0\le t\le\infty\Bigr\}\,\longrightarrow\,
\Bigl\{{1\over i+\log t}:0\le t\le\infty\Bigr\},
$$
$$
\phi_3\Bigl({1\over 1+t}\Bigr):={1\over i+\log t}\quad
\mbox{(where $\phi_3(1)=\phi_3(0):=0$ for $t=0,\infty$)}.
$$

(iv)\,$\implies$\,(v).\enspace
By considering a homeomorphism
$$
\phi_4:\Bigl\{{1\over i+t}:-\infty\le t\le\infty\Bigr\}\,\longrightarrow\,
\Bigl\{{1\over i-t}:-\infty\le t\le\infty\Bigr\},
$$
$$
\phi_4\Bigl({1\over i+t}\Bigr):={1\over i-t}\quad
\mbox{(where $\phi_4(0):=0$ for $t=\pm\infty$)},
$$
it follows from (iv) that $(i-H_n)^{-1}\to(i-H)^{-1}$ strongly as well. Let $C_0(\bR)$ denote
the Banach space of complex continuous functions $\phi$ on $\bR$ vanishing at infinity
($\lim_{t\to\pm\infty}f(t)=0$) with sup-norm. By the Stone-Weierstrass theorem, the set of
polynomials of $(i\pm t)^{-1}$ is dense in $C_0(\bR)$. Hence for any $\psi\in C_0(\bR)$ and
$\eps>0$ one can choose a polynomial $p(i+t)^{-1},(i-t)^{-1})$ such that
$\|\psi(t)-p((i+t)^{-1},(i-t)^{-1})\|_\infty<\eps$, so that
$$
\|\psi(H)-p((i+H)^{-1},(i-H)^{-1})\|<\eps,\qquad
\|\psi(H_n)-p((i+H_n)^{-1},(i-H_n)^{-1})\|<\eps
$$
for all $n\in\bN$. For any $\xi\in\cH$ with $\|\xi\|=1$, there is an $n_0$ such that
$$
\|p((i+H_n)^{-1},(i-H_n)^{-1})\xi-p((i+H)^{-1},(i-H)^{-1})\xi\|<\eps,\qquad n\ge n_0.
$$
Therefore,
$$
\|\psi(H_n)\xi-\psi(H)\xi\|<3\eps,\qquad n\ge n_0,
$$
which implies that $\psi(H_n)\to\psi(H)$ strongly for every $\psi\in C_0(\bR)$.

Now, for each fixed $s\in\bR$ set
$$
f(t):=e^{ist},\qquad t\in\bR.
$$
For $k\in\bN$ set
$$
\phi_k(x):=\begin{cases}1 & (|t|\le k-1), \\
k-|t| & (k-1\le|t|\le k), \\
0 & (|t|\ge k),\end{cases}
$$
and $f_k(t):=f(t)\phi_k(t)$. For every $\xi\in\cH$ and $\eps>0$, there is a $k_0$ such
that $\|\phi_{k_0}(H)\xi-\xi\|<\eps$. Moreover, there is an $n_0$ such that
$$
\|\phi_{k_0}(H_n)\xi-\phi_{k_0}(H)\xi\|<\eps,\qquad n\ge n_0,
$$
so that
$$
\|\phi_{k_0}(H_n)\xi-\xi\|<2\eps,\qquad n\ge n_0.
$$
Hence for every $n\ge n_0$ one has
\begin{align*}
&\|e^{isH_n}\xi-e^{isH}\xi\| \\
&\quad\le\|f(H_n)\xi-f_{k_0}(H_n)\xi\|
+\|f_{k_0}(H_n)\xi-f_{k_0}(H)\xi\|
+\|f_{k_0}(H)\xi-f(H)\xi\| \\
&\quad=\|f(H_n)(\xi-\phi_{k_0}(H_n)\xi)\|
+\|f_{k_0}(H_n)\xi-f_{k_0}(H)\xi\|
+\|f(H)(\phi_{k_0}(H)\xi-\xi)\| \\
&\quad\le\|\xi-\phi_{k_0}(H_n)\xi\|
+\|f_{k_0}(H_n)\xi-f_{k_0}(H)\xi\|
+\|\phi_{k_0}(H)\xi-\xi\| \\
&\quad\le3\eps+\|f_{k_0}(H_n)\xi-f_{k_0}(H)\xi\|
\end{align*}
so that, thanks to $f_{k_0}\in C_0(\bR)$,
$$
\limsup_{n\to\infty}\|e^{isH_n}\xi-e^{isH}\xi\|\le3\eps,
$$
which implies that $A_n^{is}=e^{isH_n}\to A^{is}=e^{isH}$ strongly for every $s\in\bR$.

(iv)\,$\implies$\,(i).\enspace
The proof is similar to that of (iv)\,$\implies$\,(v) by replacing $f(t)=e^{ist}$ with
$f(t)=(1+e^t)^{-1}$, $t\in\bR$, so the details are omitted.

(v)\,$\implies$\,(iv).\enspace
The proof is from \cite[Theorem VIII.21]{RS}. Letting $dF_t:=dE_{e^t}$ we write
$H=\int_0^\infty\log\lambda\,dE_\lambda=\int_{-\infty}^\infty t\,dF_t$. We have
\begin{align*}
(i+H)^{-1}&=\int_{-\infty}^\infty{1\over i+t}\,dF_t
=\int_{-\infty}^\infty\biggl(i\int_0^\infty e^{-s}e^{ist}\,ds\biggr)\,dF_t \\
&=i\int_0^\infty e^{-s}\biggl(\int_{-\infty}^\infty e^{ist}\,dF_t\biggr)\,ds
=i\int_0^\infty e^{-s}e^{isH}\,ds,
\end{align*}
and similarly $(i+H_n)^{-1}=i\int_0^\infty e^{-s}e^{isH_n}\,ds$. Hence it is immediate to see
the conclusion.
\end{proof}

\begin{definition}\label{D-A.6}\rm
If the equivalent conditions in Lemma \ref{L-A.4} hold, then we say that $A_n$ converges to
$A$ \emph{in the strong resolvent sense} or just \emph{strongly}.
\end{definition}

The following theorem may be a reformulation of famous Stone's representation theorem from the
viewpoint of the analytic generator. This may be considered as a particular case of
Cior\v anescu and Zsid\'o's theorem \cite{CZ}, mentioned in Theorem \ref{T-8.9} in the case of
a one-parameter isometry group on a von Neumann algebra. For a strongly continuous
one-parameter unitary group $U_t=e^{itH}$ on $\cH$, where $H$ is a self-adjoint generator, the
$H$ is usually obtained in the real analytic method as $H=\lim_{t\to0}(U_t-1)/it$, while the
following theorem says that the analytic generator $A=e^H$ can be obtained in a complex
analytic method.

\begin{thm}\label{T-A.7}
Let $A$ be a non-singular self-adjoint operator on $\cH$ with the spectral decomposition
$A=\int_0^\infty\lambda\,dE_\lambda$. Let $\alpha>0$. Then, for every $\xi\in\cH$, the
following conditions are equivalent:
\begin{itemize}
\item[\rm(i)] $\xi\in\cD(A^\alpha)$.
\item[\rm(ii)] $\int_0^\infty\lambda^{2\alpha}\,d\|E_\lambda\xi\|^2<+\infty$.
\item[\rm(iii)] There exists an $\cH$-valued bounded weakly continuous function $f$ on
$-\alpha\le\Im z\le0$, weakly analytic in $-\alpha<\Im z<0$, such that $f(t)=A^{it}\xi$
for all $t\in\bR$.
\item[\rm(iv)] There exists an $\cH$-valued bounded strongly continuous function $f$ on
$-\alpha\le\Im z\le0$, strongly analytic in $-\alpha<\Im z<0$, such that $f(t)=A^{it}\xi$
for all $t\in\bR$.
\end{itemize}

Furthermore, in either condition {\rm(iii)} or {\rm(iv)}, the function $f(z)$ is unique and
$A\xi=f(-i\alpha)$ holds.
\end{thm}

\begin{proof}
Write $d\mu(\lambda):=d\|E_\lambda\xi\|^2$, a finite positive measure on $(0,\infty)$,
and $d\nu(s):=d\mu(e^s)$ for $s\in\bR$. Without loss of generality we may assume that
$\alpha=1$, by replacing $A$ with $A^\alpha$. 

(i)\,$\iff$\,(ii) is a well-known fact, and (iv)\,$\implies$\,(iii) is trivial.

(i)\,$\implies$\,(iv).\enspace
From (i) one can define
$$
f(z):=A^{iz}\xi=\int_0^\infty\lambda^{iz}\,dE_\lambda\xi,\qquad-1\le\Im z\le0,
$$
which is bounded on $-1\le\Im z\le0$ since
$$
\|f(z)\|^2=\int_0^\infty|\lambda^{iz}|^2\,d\mu(\lambda)
\le\int_0^\infty\lambda^2\,d\mu(\lambda)<+\infty.
$$
Assume that $-1\le\Im z_n\le0$ ($n\in\bN$) and $z_n\to z_0$. Since
$$
|\lambda^{iz_n}-\lambda^{iz_0}|^2\le2(|\lambda^{iz_n}|^2+|\lambda^{iz_0}|^2)
\le4\lambda^2
$$
and $\lambda^{iz_n}\to\lambda^{iz_0}$ for all $\lambda\in(0,\infty)$, the Lebesgue convergence
theorem gives
$$
\|f(z_n)-f(z_0)\|^2=\int_0^\infty|\lambda^{iz_n}-\lambda^{iz_0}|^2\,d\mu(\lambda)
\,\longrightarrow\,0.
$$
Hence $f(z)$ is strongly continuous on $-1\le\Im z\le0$.

Next, assume that $-1<\Im z_0<0$, and choose a $\delta>0$ such that
$-1+2\delta\le\Im z_0\le-2\delta$. For each $\lambda\in(0,\infty)$ and any $k\in\bC$ with
$0<|k|\le\delta$, by the mean value theorem (applied to $\Re\lambda^{i(z_0+tk)}$ and
$\Im\lambda^{i(z_0+tk)}$ in $t\in[0,1]$), there are $\zeta_1,\zeta_2$ in the segment joining
$z_0$ and $z_0+k$ such that
$$
{\lambda^{i(z_0+k)}-\lambda^{iz_0}\over k}
=\Re(\lambda^{iz})'(\zeta_1)+i\,\Im(\lambda^{iz})'(\zeta_2)
=\Re(i\lambda^{i\zeta_1}\log\lambda)+i\,\Im(i\lambda^{i\zeta_2}\log\lambda)
$$
so that
\begin{align}\label{F-A.1}
\bigg|{\lambda^{i(z_0+k)}-\lambda^{iz_0}\over k}-i\lambda^{iz_0}\log\lambda\bigg|
\le2\biggl(\max_{|\zeta-z_0|\le\delta}|\lambda^{i\zeta}-\lambda^{iz_0}|\biggr)|\log\lambda|.
\end{align}
Note that if $|\zeta-z_0|\le\delta$, then $\Im\zeta\le\Im z_0+\delta\le-\delta$ and
$\Im\zeta\ge\Im z_0-\delta\ge-1+\delta$. Hence, when $0<\lambda\le1$, one has
\begin{align}\label{F-A.2}
\biggl(\max_{|\zeta-z_0|\le\delta}|\lambda^{i\zeta}-\lambda^{iz_0}|\biggr)|\log\lambda|
\le\biggl(\max_{\Im\zeta\le-\delta}(\lambda^{-\Im\zeta}+\lambda^{-\Im z_0})\biggr)
|\log\lambda|
\le2\lambda^\delta|\log\lambda|,
\end{align}
and when $\lambda\ge1$, one has
\begin{align}\label{F-A.3}
\biggl(\max_{|\zeta-z_0|\le\delta}|\lambda^{i\zeta}-\lambda^{iz_0}|\biggr)|\log\lambda|
&\le\biggl(\max_{\Im\zeta\ge-1+\delta}(\lambda^{-\Im\zeta}+\lambda^{-\Im z_0})\biggr)
\log\lambda \nonumber\\
&\le2\lambda^{1-\delta}\log\lambda=2\,{\log\lambda\over\lambda^\delta}\lambda.
\end{align}
Set
$$
K_0:=\sup_{0<\lambda\le1}\lambda^\delta|\log\lambda|<+\infty,\qquad
K_1:=\sup_{\lambda\ge1}{\log\lambda\over\lambda^\delta}<+\infty.
$$
$$
\phi(\lambda):=K_0+K_1\lambda,\qquad\lambda>0.
$$
Then $\int_0^\infty\phi(\lambda)^2\,d\mu(\lambda)<+\infty$ by condition (ii), and we have by
\eqref{F-A.1}--\eqref{F-A.3},
$$
\bigg|{\lambda^{i(z_0+k)}-\lambda^{iz_0}\over k}-i\lambda^{iz_0}\log\lambda\bigg|
\le4\phi(\lambda),\qquad\lambda>0.
$$
Moreover, since $|\lambda^{iz_0}\log\lambda|\le\phi(\lambda)$ for $\lambda>0$ similarly to
\eqref{F-A.2} and \eqref{F-A.3}, it follows that
$\int_0^\infty|\lambda^{iz_0}\log\lambda|^2\,d\mu(\lambda)<+\infty$ so that
$\eta_0:=\int_0^\infty i\lambda^{iz_0}\log\lambda\,dE_\lambda\xi\in\cH$ is defined. We then
have
$$
\bigg\|{f(z_0+k)-f(z_0)\over k}-\eta_0\bigg\|^2
=\int_0^\infty\bigg|{\lambda^{i(z_0+k)}-\lambda^{iz_0}\over k}
-i\lambda^{iz_0}\log\lambda\bigg|^2\,d\mu(\lambda)\,\longrightarrow\,0
$$
as $\delta\ge|k|\to0$ by the Lebesgue convergence theorem, which implies that $f'(z_0)=\eta_0$.
Hence $f(z)$ is strongly analytic in $-1<\Im z<0$.

(iii)\,$\implies$\,(i).\enspace
Let $f(z)$ be as given in (iii), and set $\ffi(z):=\<\xi,f(z)\>$, $-1\le\Im z\le0$, which is
bounded and continuous on $-1\le\Im z\le0$ and analytic in $-1<\Im z<0$. Note that
\begin{align}\label{F-A.4}
\ffi(t)=\int_0^\infty\lambda^{it}\,d\mu(\lambda)=\int_{-\infty}^\infty e^{ist}\,d\nu(s),
\qquad t\in\bR,
\end{align}
that is, $\ffi(t)=\widehat\nu(-t)$ where $\widehat\nu$ is the Fourier transform of $\nu$. For
$n\in\bN$ consider an entire function $G_n(z):={n\over\sqrt{2\pi}}e^{-n^2z^2/2}$, $z\in\bC$.
Then it is well-known that $\int_{-\infty}^\infty G_n(t)\,dt=1$ and
\begin{align}\label{F-A.5}
\widehat G_n(s):=\int_{-\infty}^\infty e^{-ist}G_n(s)\,ds=e^{-{s^2\over2n^2}},
\qquad s\in\bR.
\end{align}
For any $\delta\in(0,1/2)$ take the contour integral of $G_n(z)\ffi(-i-z)$ along the rectangle
joining $R-i\delta$, $-R-i\delta$, $-R-i(1-\delta)$ and $R-i(1-\delta)$, and then take the
limit as $R\to\infty$ to obtain
$$
\int_{-\infty}^\infty G_n(t-i\delta)\ffi(-t-i(1-\delta))\,dt
=\int_{-\infty}^\infty G_n(t-i(1-\delta))\ffi(-t-i\delta)\,dt.
$$
By the Lebesgue convergence theorem, letting $\delta\searrow0$ gives
\begin{align*}
\int_{-\infty}^\infty G_n(t)\ffi(-t-i)\,dt
&=\int_{-\infty}^\infty G_n(t-i)\ffi(-t)\,dt \\
&=\int_{-\infty}^\infty G_n(t-i)\biggl(\int_{-\infty}^\infty e^{-ist}\,d\nu(s)\biggr)\,dt
\quad\mbox{(by \eqref{F-A.4})} \\
&=\int_{-\infty}^\infty\biggl(\int_{-\infty}^\infty G_n(t-i)e^{-ist}\,dt\biggr)\,d\nu(s) \\
&=\int_{-\infty}^\infty\biggl(\int_{-\infty}^\infty G_n(t)e^{-is(t+i)}\,dt\biggr)d\nu(s) \\
&=\int_{-\infty}^\infty\biggl(\int_{\infty}^\infty G_n(t)e^{-ist}\,dt\biggr)e^s\,d\nu(s)
=\int_{-\infty}^\infty\widehat G_n(s)e^s\,d\nu(s).
\end{align*}
By \eqref{F-A.5} letting $n\to\infty$ yields
$$
\ffi(-i)=\int_{-\infty}^\infty e^s\,d\nu(s)=\int_0^\infty\lambda\,d\mu(\lambda),
$$
which implies that $\int_0^\infty\lambda\,d\mu(\lambda)<+\infty$. Hence $\xi\in\cD(A^{1/2})$,
and it follows from the above proof of (i)\,$\implies$\,(iv) that $f(z)=A^{iz}\xi$ for
$-1/2\le\Im z\le0$.

Next. for any $\eta\in\cD(A)$ let $g(z):=A^{iz}\eta$ for $-1\le\Im z\le0$. From
(i)\,$\implies$\,(iv), $g(z)$ is bounded and strongly continuous on $-1\le\Im z\le0$ and
strongly analytic in $-1<\Im z<0$. Define
$$
F(z):=\<g(\overline z-i),f(z)\>,\qquad-1\le\Im z\le0.
$$
It is easy to verify that $F(z)$ is bounded continuous on $-1\le\Im z\le0$ and analytic in
$-1<\Im z<0$.
For any $r\in(0,1/2)$ we have
$$
F(-ir)=\<g(-i(1-r)),f(-ir)\>=\<A^{1-r}\eta,A^r\xi\>=\<A\eta,\xi\>=F(0),
$$
which implies that $F(z)\equiv F(0)$ for $-1\le\Im z\le0$. Hence we have
$$
\<A\eta,\xi\>=F(0)=F(-i)=\<\eta,f(-i)\>.
$$
Since this holds for all $\eta\in\cD(A)$, we have
$\xi\in\cD(A)$ and $A\xi=f(-i)$. Therefore, (i) follows and the last statement has been shown
as well.
\end{proof}

\subsection{Positive quadratic forms}

An important aspect of positive self-adjoint operators on $\cH$ is their correspondence to
closed (densely-defined) positive quadratic forms on $\cH$. We begin with the definition.

\begin{definition}\label{D-A.8}\rm
\begin{itemize}
\item[\rm(1)] A function $q:\cD(q)\to[0,\infty)$, where $\cD(q)$ is a dense subspace of
$\cH$, is called a \emph{positive quadratic form} on $\cH$ if
\begin{itemize}
\item[\rm(a)] $q(\lambda\xi)=|\lambda|^2q(\xi)$,\quad$\xi\in\cD(q)$, $\lambda\in\bC$,
\item[\rm(b)] $q(\xi+\eta)+q(\xi-\eta)=2q(\xi)+2q(\eta)$,\quad$\xi,\eta\in\cD(q)$.
\end{itemize}
\item[\rm(2)] The above $q$ is said to be \emph{closed} if $\{\xi_n\}\subset\cD(q)$,
$\xi\in\cH$, $\|\xi_n-\xi\|\to0$ and $q(\xi_n-\xi_m)\to0$ as $n,m\to\infty$, then
$\xi\in\cD(q)$ and $q(\xi_n-\xi)\to0$.
\end{itemize}
\end{definition}

\begin{lemma}\label{L-A.9}
Let $q:\cD(q)\to[0,\infty)$ be a positive quadratic form on $\cH$. Define
\begin{align}\label{F-A.6}
q(\xi,\eta):={1\over 4}\sum_{k=0}^3i^kq(\eta+i^k\xi),\qquad\xi,\eta\in\cD(q),
\end{align}
Then $q(\xi,\eta)$ is a positive sesquilinear form on $\cD(q)$ such that $q(\xi)=q(\xi,\xi)$
for all $\xi\in\cD(q)$. Hence $q(\xi)^{1/2}$ is a semi-norm on $\cD(q)$, so that
$$
|q(\xi_1)^{1/2}-q(\xi_2)^{1/2}|\le q(\xi_1-\xi_2)^{1/2},\qquad\xi_1,\xi_2\in\cD(q).
$$
Moreover, if $\xi,\xi_n\in\cD(q)$ ($n\in\bN$) and $q(\xi_n-\xi)\to0$, then $q(\xi_n)\to q(\xi)$.
\end{lemma}

\begin{proof}
The well-known Jordan-von Neumann theorem says that a (semi-)norm on a complex vector space
comes from a (semi-)inner product if and only if the norm satisfies the parallelogram law.
Condition (b) is the parallelogram law for $q(\xi)^{1/2}$ and \eqref{F-A.6} is the usual
polarization formula to define a semi-inner product. Even when $q(\xi)^{1/2}$ is not assumed
to be a semi-norm, the usual proof of the Jordan-von Neumann theorem can work to prove that
$q(\xi,\eta)$ is s sesquilinear form on $\cD(q)$. Although the details are left to an exercise,
a point we have to check here is that the function $\lambda>0\mapsto q(\xi,\lambda\eta)$ is
continuous. For this, note that, for every $\xi,\eta\in\cD(q)$, the function
$\lambda>0\mapsto q(\lambda\eta+\xi)$ is mid-point convex and locally bounded by (a) and (b),
so it is continuous, see, e.g., \cite[\S3.18]{HLP}. The remaining assertions of the lemma are
now obvious.
\end{proof}

The equivalence of (i) and (iii) in the next theorem is the most fundamental representation
result for positive quadratic forms, see \cite[Chap.~6, \S2.6]{Ka2} for more details. The
equivalence of (i) and (ii) was first given in \cite[Theorem 2]{Si} with stating that it is
attributed to Kato.

\begin{thm}\label{T-A.10}
Let $q$ be a positive quadratic form on $\cH$. Then the following conditions are equivalent:
\begin{itemize}
\item[\rm\rm(i)] $q$ is closed.
\item[\rm\rm(ii)] $\widetilde q$ is lower semicontinuous on $\cH$, where $\widetilde q$ is the
extension of $q$ as
\begin{align}\label{F-A.7}
\widetilde q(\xi):=\begin{cases}q(\xi) & \text{if $\xi\in\cD(q)$}, \\
\infty & \text{if $\xi\in\cH\setminus\cD(q)$}.\end{cases}
\end{align}
\item[\rm\rm(iii)] There exists a positive self-adjoint operator $A$ on $\cH$ such that
$\cD(A^{1/2})=\cD(q)$ and
$$
q(\xi)=\|A^{1/2}\xi\|^2,\qquad\xi\in\cD(A^{1/2}).
$$
\end{itemize}

Moreover, $A$ in condition (iii) is unique.
\end{thm}

\begin{proof}
(iii)\,$\implies$\,(i).\enspace
Assume that $q$ is as given in (iii) by a positive self-adjoint operator $A$ with the
spectral decomposition $A=\int_0^\infty\lambda\,dE_\lambda$. First, let us confirm that $q$
is a positive quadratic form. The property (a) of Definition \ref{D-A.8}\,(1) is obvious. For
$\xi,\eta\in\cD(A^{1/2})$ one has
\begin{align*}
q(\xi+\eta)+q(\xi-\eta)&=\int_0^\infty\lambda\,d\|E_\lambda(\xi+\eta)\|^2
+\int_0^\infty\lambda\,d\|E_\lambda(\xi-\eta)\|^2 \\
&=2\int_0^\infty\lambda\,d\|E_\lambda\xi\|^2+2\int_0^\infty\lambda\,d\|E_\lambda\eta\|^2
=2q(\xi)+2q(\eta),
\end{align*}
since $\|E_\lambda(\xi-\eta)\|^2+\|E_\lambda(\xi-\eta)\|^2
=2\|E_\lambda\xi\|^2+2\|E_\lambda\eta\|^2$. So (b) holds as well. Moreover, the closedness of
$q$ immediately follows from that of $A^{1/2}$.

(i)\,$\implies$\,(iii).\enspace
Assume that $q$ is a closed positive quadratic form on $\cH$. It is easy to see (an exercise)
that $\cD(q)$ becomes a Hilbert space with the inner product
$$
\<\xi,\eta\>_q:=\<\xi,\eta\>+q(\xi,\eta),\qquad\xi,\eta\in\cD(q),
$$
where $q(\xi,\eta)$ is given in \eqref{F-A.6}. For every $\zeta\in\cH$, since
$|\<\zeta,\eta\>|\le\|\zeta\|\,\|\eta\|\le\|\zeta\|\,\|\eta\|_q$, where
$\|\eta\|_q:=\<\eta,\eta\>_q^{1/2}$, the linear functional $\eta\in\cD(q)\mapsto\<\zeta,\eta\>$
is bounded on $\cD(q)$. By the Riesz theorem there is a $B\zeta\in\cD(q)$ such that
\begin{align}\label{F-A.8}
\<\zeta,\eta\>=\<B\zeta,\eta\>_q=\<B\zeta,\eta\>+q(B\zeta,\eta),\qquad\eta\in\cD(q).
\end{align}
It follows from \eqref{F-A.8} that $B$ is an injective linear operator from $\cH$ into
$\cD(q)$. If $\eta\in\cD(q)$ satisfies $\<B\zeta,\eta\>_q=0$ for all $\zeta\in\cH$, then
$\<\zeta,\eta\>=0$ for all $\zeta\in\cH$ and so $\eta=0$. This means that $\cR(B)$ (the range
of $B$) is dense in $\cD(q)$ and so dense in $\cH$. Hence we have a densely-defined operator
$B^{-1}:\cD(B^{-1})=\cR(B)\to\cH$. Define
$$
A\xi:=B^{-1}\xi-\xi,\qquad\xi\in\cD(A):=\cR(B).
$$
It then follows from \eqref{F-A.8} that
\begin{align}\label{F-A.9}
q(\xi,\eta)=\<B^{-1}\xi,\eta\>-\<\xi,\eta\>=\<A\xi,\eta\>,\qquad
\xi\in\cD(A)\subset\cD(q),\ \eta\in\cD(q).
\end{align}
Hence $\<A\xi,\xi\>=q(\xi,\xi)\ge0$ for all $\xi\in\cD(A)$, so $A$ is positive. This implies
that $A$ and hence $1+A=B^{-1}$ are symmetric. For every $\zeta\in\cH$, letting $\xi:=B\zeta$,
we have $(1+A)\xi=\zeta$ and
$$
\<B\zeta,\zeta\>=\<\xi,(1+A)\xi\>=\<(1+A)\xi,\xi\>=\<\zeta,B\zeta\>,
$$
which means that $B$ is symmetric. The well-known Hellinger-Toeplitz theorem (see
\cite[p.~84]{RS}, an easy corollary of the closed graph theorem) says that a symmetric
operator defined on the whole $\cH$ is a bounded self-adjoint operator. Hence $B$ is a
bounded self-adjoint operator, which implies that $B^{-1}=1+A$ is self-adjoint and so $A$ is
self-adjoint (and positive). Moreover, by \eqref{F-A.9} we have $q(\xi)=\|A^{1/2}\xi\|^2$ for
all $\xi\in\cD(A)\subset\cD(A^{1/2})$

For every $\xi\in\cD(q)$, since $\cD(A)$ is dense in $\cD(q)$ with $\<\cdot,\cdot\>_q$, there
is a sequence $\xi_n\in\cD(A)$ such that $\|\xi_n-\xi\|\to0$, $q(\xi_n-\xi)\to0$ and
$\|A^{1/2}\xi_n-A^{1/2}\xi_m\|^2=q(\xi_n-\xi_m)\to0$. Since $A^{1/2}$ is closed,
$\xi\in\cD(A^{1/2})$ and $\|A^{1/2}\xi_n-A^{1/2}\xi\|\to0$, so that
$$
\|A^{1/2}\xi\|^2=\lim_n\|A^{1/2}\xi_n\|^2=\lim_nq(\xi_n)=q(\xi)
$$
by Lemma \ref{L-A.9} for the above last equality. On the other hand, for every
$\xi\in\cD(A^{1/2})$, since $\cD(A)$ is a core of $A^{1/2}$, there is a sequence
$\xi_n\in\cD(A)\subset\cD(q)$ such that $\|\xi_n-\xi\|\to0$, $\|A^{1/2}\xi_n-A^{1/2}\xi\|\to0$
and $q(\xi_n-\xi_m)=\|A^{1/2}\xi_n-A^{1/2}\xi_m\|^2\to0$. Since $q$ is closed, $\xi\in\cD(q)$
and $q(\xi_n-\xi)\to0$. Therefore, $\cD(A^{1/2})=\cD(q)$ and (iii) follows.

(ii)\,$\implies$\,(i).\enspace
Assume (ii). Let $\{\xi_n\}\subset\cD(q)$ and $\xi\in\cH$ be such that $\|\xi_n-\xi\|\to0$ and
$q(\xi_n-\xi_m)\to0$ as $n,m\to\infty$. For any $\eps>0$ choose an $n_0$ such that
$q(\xi_n-\xi_m)\le\eps$ for all $n,m\ge n_0$. From (ii) we find that, for every $n\ge n_0$,
$$
\widetilde q(\xi_n-\xi)\le\liminf_m\widetilde q(\xi_n-\xi_m)\le\eps.
$$
This in particular implies that $\xi_n-\xi\in\cD(q)$ so that $\xi\in\cD(q)$. Hence
$q(\xi_n-\xi)\le\eps$ for all $n\ge n_0$. Therefore, $q(\xi_n-\xi)\to0$ and (i) follows.

(iii)\,$\implies$\,(ii).\enspace
Assume (iii). To show (ii), it suffices to prove that
\begin{align}\label{F-A.10}
\widetilde q(\xi)=\sup\{|\<A\zeta,\xi\>|^2:\zeta\in\cD(A),\,\<A\zeta,\zeta\>=1\}
\end{align}
for all $\xi\in\cH$. Let $\zeta\in\cD(A)\subset\cD(A^{1/2})$ with $\<A\zeta,\zeta\>=1$. For
every $\xi\in\cD(q)=\cD(A^{1/2})$ one has
$$
|\<A\zeta,\xi\>|^2=|\<A^{1/2}\zeta,A^{1/2}\xi\>|^2
\le\|A^{1/2}\zeta\|^2\|A^{1/2}\xi\|^2=\<A\zeta,\zeta\>q(\xi)=\widetilde q(\xi).
$$
If $\xi\in\cH\setminus\cD(q)$, then $|\<A\zeta,\xi\>|^2\le\infty=\widetilde q(\xi)$.
Therefore, $\widetilde q(\xi)\ge\mbox{the RHS of \eqref{F-A.10}}$ for all $\xi\in\cH$.

On the other hand, taking the spectral decomposition $A=\int_0^\infty\lambda\,dE_\lambda$,
for every $\xi\in\cH$ one has
$$
\<AE_n\xi,\xi\>=\int_0^n\lambda\,d\|E_\lambda\xi\|^2
\ \nearrow\ \int_0^\infty\lambda\,d\|E_\lambda\xi\|^2=\widetilde q(\xi).
$$
If $\<AE_n\xi,\xi\>=0$ for all $n$, then $\widetilde q(\xi)=0$. Otherwise, letting
$\zeta_n:=\<AE_n\xi,\xi\>^{-1/2}E_n\xi\in\cD(A)$ for $n$ large, one has
$$
\<A\zeta_n,\zeta_n\>=\<AE_n\xi,\xi\>^{-1}\<AE_n\xi,\xi\>=1,
$$
$$
|\<A\zeta_n,\xi\>|^2=|\<AE_n\xi,\xi\>^{-1/2}\<AE_n\xi,\xi\>|^2
=\<AE_n\xi,\xi\>\ \nearrow\ \widetilde q(\xi).
$$
Therefore, \eqref{F-A.10} holds.

Finally, the uniqueness of $A$ in condition (iii) is immediately seen from Lemma \ref{L-A.1}.
\end{proof}

Let $p,q$ be positive quadratic forms on $\cH$. It is said that $p$ is an extension of $q$
if $\cD(p)\supset\cD(q)$ and $p(\xi)=q(\xi)$ for all $\xi\in\cD(q)$. It is said that $q$ is
\emph{closable} if $q$ has a closed extension. The proof of the next theorem was given in
\cite{Co4}.

\begin{thm}\label{T-A.11}
Let $q:\cD(q)\to[0,\infty)$ be a positive quadratic form on $\cH$. Then the following
conditions are equivalent:
\begin{itemize}
\item[\rm\rm(i)] $q$ is closable.
\item[\rm\rm(ii)] if $\{\xi_n\}\subset\cD(q)$, $\|\xi_n\|\to0$ and $q(\xi_n-\xi_m)\to0$ as
$n,m\to\infty$, then $q(\xi_n)\to0$.
\item[\rm\rm(iii)] $q$ is lower semicontinuous on $\cD(q)$.
\end{itemize}

In this case, there exists the smallest closed extension $\overline q$ of $q$.
\end{thm}

\begin{proof}
(i)\,$\implies$\,(ii).\enspace
Assume that $q$ is closable, and let $p$ be a closed extension of $q$. Let
$\{\xi_n\}\subset\cD(q)$, $\|\xi_n\|\to0$ and $q(\xi_n-\xi_m)\to0$ as $n,m\to\infty$. Then
$q(\xi_n)=p(\xi_n-0)\to0$ by Definition \ref{D-A.8}\,(2) for $p$.

(ii)\,$\implies$\,(i).\enspace
Assume that (ii) holds, and define
\begin{align}
\cD(\overline q):=\{\xi\in\cH:\ &\mbox{there exists a sequence $\xi_n\in\cD(q)$ such that} 
\nonumber\\
&\mbox{$\|\xi_n-\xi\|\to0$ and $q(\xi_n-\xi_m)\to0$ as $n,m\to\infty$}\} \label{F-A.11}
\end{align}
and
\begin{align}\label{F-A.12}
\overline q(\xi):=\lim_{n\to\infty}q(\xi_n)\qquad\mbox{for $\xi\in\cD(\overline q)$}.
\end{align}
Here, the above limit $\lim_nq(\xi_n)$ exists by Lemma \ref{L-A.9}. Moreover, this limit is
independent of the choice of $\{\xi_n\}$; indeed, if $\{\eta_n\}\subset\cD(q)$ is another
sequence as above, then $\|\xi_n-\eta_n\|\to0$ and
$q((\xi_n-\eta_n)-(\xi_m-\eta_m))\le2q(\xi_n-\xi_m)+2q(\eta_n-\eta_m)\to0$,
so condition (ii) implies that $q(\xi_n-\eta_n)\to0$ and so $\lim_nq(\xi_n)=\lim_nq(\eta_n)$.
It is clear that $\cD(\overline q)\supset\cD(q)$ and $\overline q(\xi)=q(\xi)$ for all
$\xi\in\cD(q)$.

Let us prove that $\overline q$ is a closed positive quadratic form on $\cH$. For every
$\xi,\eta\in\cD(\overline q)$ choose $\{\xi_n\},\{\eta_n\}\subset\cD(q)$ such that
$\|\xi_n-\xi\|\to0$, $\|\eta_n-\eta\|\to0$ and $q(\xi_n-\xi)\to0$, $q(\eta_n-\eta)\to0$.
Then for every $\lambda,\mu\in\bC$,
$$
\|(\lambda\xi_n+\mu\eta_n)-(\lambda\xi+\mu\eta)\|\ \longrightarrow\ 0,
$$
$$
q((\lambda\xi_n+\mu\eta_n)-(\lambda\xi_n+\mu\eta_n))^{1/2}
\le|\lambda|q(\xi_n-\xi_m)^{1/2}+|\mu|q(\eta_n-\eta_m)^{1/2}\,\longrightarrow\,0,
$$
so that $\lambda\xi+\mu\eta\in\cD(\overline q)$. Hence $\cD(\overline q)$ is a dense subspace
of $\cH$. Since
$$
\overline q(\lambda\xi)=\lim_nq(\lambda\xi_n)=|\lambda|^2\lim_nq(\xi_n)
=|\lambda|^2\overline q(\xi),
$$
\begin{align*}
\overline q(\xi+\eta)+\overline q(\xi-\eta)
&=\lim_n\{q(\xi_n+\eta_n)+q(\xi_n-\eta_n)\} \\
&=\lim_n\{2q(\xi_n)+2q(\eta_n)\}=2\overline q(\xi)+2\overline q(\eta),
\end{align*}
it follows that $\overline q$ is a positive quadratic form on $\cH$. Furthermore, note that
\begin{align}\label{F-A.13}
\lim_n\overline q(\xi_n-\xi)=\lim_n\lim_m q(\xi_n-\xi_m)=0,\qquad\xi\in\cD(\overline q).
\end{align}

To show the closedness of $\overline q$, let $\{\eta_n\}\subset\cD(\overline q)$,
$\eta\in\cH$, $\|\eta_n-\eta\|\to0$ and $\overline q(\eta_n-\eta_m)\to0$ as $n,m\to\infty$.
For each $n$ choose a $\xi_n\in\cD(q)$ such that $\|\xi_n-\eta_n\|<1/n$ and
$\overline q(\xi_n-\eta_n)<1/n$ (thanks to \eqref{F-A.13}). Then $\|\xi_n-\eta\|\to0$ and
\begin{align*}
q(\xi_n-\xi_m)^{1/2}&=\overline q(\xi_n-\xi_m)^{1/2} \\
&\le\overline q(\xi_n-\eta_n)^{1/2}+\overline q(\eta_n-\eta_m)^{1/2}
+\overline q(\eta_m-\xi_m)^{1/2}\,\longrightarrow\,0\quad\mbox{as $n,m\to\infty$}.
\end{align*}
Hence we have $\eta\in\cD(\overline q)$ and
\begin{align*}
\limsup_n\overline q(\eta_n-\eta)^{1/2}
&\le\limsup_n\{\overline q(\eta_n-\xi_n)^{1/2}+\overline q(\xi_n-\eta)^{1/2}\}  \\
&=\limsup_n\lim_mq(\xi_n-\xi_m)^{1/2}=0,
\end{align*}
since $\overline q(\xi_n-\eta)=\lim_mq(\xi_n-\xi_m)$. Therefore, $\overline q(\eta_n-\eta)\to0$,
so that $\overline q$ is a closed extension of $q$.

(i)\,$\implies$\,(iii).\enspace
Let $p$ be a closed extension of $q$. Then $\widetilde p$ is lower semicontinuous on $\cH$
by Theorem \ref{T-A.10}. Since  $q=\widetilde p$ on $\cD(q)$, it follows that $q$ is lower
semicontinuous on $\cD(q)$.

(iii)\,$\implies$\,(ii).\enspace
Assume (iii). Let $\{\xi_n\}\subset\cD(q)$, $\|\xi\|\to0$ and $q(\xi_n-\xi_m)\to0$ as
$n,m\to\infty$. For every $\eps>0$ choose an $n_0$ such that $q(\xi_n-\xi_m)\le\eps$ for
all $n,m\ge n_0$. From (iii) we find that
$$
q(\xi_n)\le\liminf_mq(\xi_n-\xi_m)\le\eps,\qquad n\ge n_0,
$$
so that $q(\xi_n)\to0$ and (ii) follows.

Finally, we show that $\overline q$ given in the above proof of (ii)\,$\implies$\,(i) is the
smallest closed extension of $q$. For this, let $p$ be any closed extension of $q$. If
$\xi\in\cD(\overline q)$, then by the definition of $\overline q$ in \eqref{F-A.11} and
\eqref{F-A.12} there is a sequence $\xi_n\in\cD(q)\subset\cD(p)$ such that $\|\xi_n-\xi\|\to0$
and $p(\xi_n-\xi_m)=q(\xi_n-\xi_m)\to0$, so $\xi\in\cD(p)$ and $p(\xi_n-\xi)\to0$. Hence
$\overline q(\xi)=\lim_nq(\xi_n)=p(\xi)$, so the result follows.
\end{proof}

\begin{remark}\label{R-A.12}\rm
Let $q$ be a positive quadratic form on $\cH$ and assume that $q$ is lower semicontinuous on
$\cD(q)$. Then by Theorems \ref{T-A.11} and \ref{T-A.10} there exists a unique positive
self-adjoint operator $A_1$ on $\cH$ such that $\cD(A_1^{1/2})=\cD(\overline q)$ and
$$
\|A_1^{1/2}\xi\|^2=\overline q(\xi),\qquad\xi\in\cD(\overline q).
$$
In particular,
\begin{align}\label{F-A.14}
\cD(A_1^{1/2})\supset\cD(q)\qquad\mbox{and}\qquad
\|A_1^{1/2}\xi\|^2=q(\xi),\quad\xi\in\cD(q).
\end{align}
Furthermore, note that $\cD(q)$ is a core of $A_1^{1/2}$ and $A_1$ is the \emph{largest} (in
the sense of Definition \ref{D-A.2}) positive self-adjoint operator on $\cH$ satisfying
\eqref{F-A.14}. Indeed, that $\cD(q)$ is a core of $A_1^{1/2}$ is immediately seen from
\eqref{F-A.11} and \eqref{F-A.13}. Let $A$ be any positive self-adjoint operator satisfying
\eqref{F-A.14}. For every $\xi\in\cD(A_1^{1/2})=\cD(\overline q)$ there is a sequence
$\xi_n\in\cD(q)$ such that $\|\xi_n-\xi\|\to0$ and $q(\xi_n)\to\overline q(\xi)$. Hence one has
$$
\|A^{1/2}\xi\|^2\le\liminf_n\|A^{1/2}\xi\|^2=\lim_nq(\xi_n)
=\overline q(\xi)=\|A_1^{1/2}\xi\|^2,
$$
implying that $A\le A_1$ in sense of Definition \ref{D-A.2}.
\end{remark}

In view of condition (ii) of Theorem \ref{T-A.10}, it is convenient to reformulate positive
quadratic forms as those defined on the whole $\cH$ but allowed to have the value $\infty$ in
the following way:

\begin{definition}\label{D-A.13}\rm
We call a function $q:\cH\to[0,\infty]$ a \emph{positive form} on $\cH$ if the properties (a)
and (b) of Definition \ref{D-A.8}\,(1) hold for all $\xi,\eta\in\cH$, with the convention
$0\infty=0$. Here we use the term ``positive form" instead of ``positive quadratic form"
to distinguish the present definition from Definition \ref{D-A.8}\,(1).
\end{definition}

\begin{lemma}\label{L-A.14}
Let $q:\cH\to[0,\infty]$ be a positive form on $\cH$, and set
$\cD(q):=\{\xi\in\cH:q(\xi)<\infty\}$. Then $\cK:=\overline{\cD(q)}$ is a closed subspace of
$\cH$ and $q|_{\cD(q)}$ is a positive quadratic form on $\cK$.

Conversely, let $q:\cD(q)\to[0,\infty)$ be a positive quadratic form on a closed subspace of
$\cH$, and set $\widetilde q:\cH\to[0,\infty]$ by \eqref{F-A.7}. Then $\widetilde q$ is a
positive form.
\end{lemma}

\begin{proof}
For the first assertion, it suffices to show that $\cD(q)$ is a subspace of $\cH$. But this is
clear from (a) and (b) for $q$. Next, let $q$ and $\widetilde q$ be as stated in the second
assertion. We shows that $\widetilde q$ satisfies (a) and (b) for all $\xi,\eta\in\cH$.
Obviously, when $\lambda=0$, both sides of (a) are $0$ for all $\xi\in\cH$. When $\lambda\ne0$,
both sides of (a) are $\infty$ for all $\xi\in\cH\setminus\cD(q)$. Hence (a) holds for all
$\xi\in\cH$. When $\xi\not\in\cD(q)$ or $\eta\not\in\cD(q)$, it follows from (b) for $q$ that
$\xi+\eta\not\in\cD(q)$ or $\xi-\eta\not\in\cD(q)$, so both sides of (b) are $\infty$. Hence
(b) holds for all $\xi,\eta\in\cH$.
\end{proof}

The next theorem is a reformulation of Theorem \ref{T-A.10}.

\begin{thm}\label{T-A.15}
Let $q$ be a positive form on $\cH$ in the sense of Definition \ref{D-A.13}. Set
$\cD(q):=\{\xi\in\cH:q(\xi)<\infty\}$ and $\cK:=\overline{\cD(q)}$. Then the following
conditions are equivalent:
\begin{itemize}
\item[\rm\rm(i)] $q|_{\cD(q)}$ is a closed positive quadratic form on $\cK$ in the sense of
Definition \ref{D-A.8}.
\item[\rm\rm(ii)] $q$ is lower semicontinuous on $\cH$.
\item[\rm\rm(iii)] There exists a positive self-adjoint operator $A$ on $\cK$ such that
$\cD(A^{1/2})=\cD(q)$ and
$$
q(\xi)=\begin{cases}\|A^{1/2}\xi\|^2 & \text{if $\xi\in\cD(A^{1/2})$}, \\
\infty & \text{if $\xi\in\cH\setminus\cD(A^{1/2})$}.\end{cases}
$$
\end{itemize}

Moreover, $A$ in condition (iii) is unique.
\end{thm}

\begin{proof}
It is clear that (ii) is equivalent to the lower semicontinuity of $q|_\cK$. Hence the theorem
immediately follows from Lemma \ref{L-A.14} and Theorem \ref{T-A.10} applied to a positive
quadratic form $q|_{\cD(q)}$ on $\cK$.
\end{proof}

We write $q=q_A$ for the positive form determined by $A$ as in the above (iii).

\begin{example}\label{E-A.16}\rm
We here recall the notion of form sums of two positive self-adjoint operators; the idea goes
back to \cite{Ka1}. Let $A,B$ be positive self-adjoint operators on some respective closed
subspaces of $\cH$. Define $\cD(q):=\cD(A^{1/2})\cap\cD(B^{1/2})$ and
$$
q(\xi):=\begin{cases}\|A^{1/2}\xi\|^2+\|B^{1/2}\xi\|^2 & \text{if $\xi\in\cD(q)$}, \\
\infty & \text{if $\xi\in\cH\setminus\cD(q)$}.\end{cases}
$$
Then it is immediate to see that $q$ is a positive form on $\cH$ satisfying condition (ii) of
Theorem \ref{T-A.15}. Hence there exists a unique positive self-adjoint operator $C$ on
$\overline{\cD(q)}$ such that $\cD(C^{1/2})=\cD(q)$ and
$$
\|C^{1/2}\xi\|^2=\|A^{1/2}\xi\|^2+\|B^{1/2}\xi\|^2,
\qquad\xi\in\cD(A^{1/2})\cap\cD(B^{1/2}).
$$
That is, $C$ is determined by the equality $q_C=q_A+q_B$. The $C$ is denoted by $A\,\dot+\,B$
and called the \emph{form sum} of $A$ and $B$. 
\end{example}

\end{document}